\AtBeginDocument{}
\AtBeginDocument{\renewcommand{\bT}{\mathbf{T}}}
\documentclass[12pt]{amsart}
\usepackage{macros}

\begin{document}

\title{Quasi-coherent sheaves and $\mathcal{D}$-modules in Derived Differential Supergeometry}

\author{David Carchedi}

\maketitle

\begin{abstract}
Derived geometry provides powerful tools to handle non-transverse intersections and singular moduli problems arising in geometry and theoretical physics. While derived algebraic geometry has been extensively developed, classical field theories---formulated as variational problems involving sections of smooth fiber bundles over manifolds---naturally require the language of differential geometry, infinite-dimensional analysis (e.g., Fréchet manifolds), and  additional geometric structures on spacetime, such as smooth metrics. Moreover, field theories incorporating fermionic matter fields necessitate extending the framework to include supermanifolds. This article is the first in a sequence aimed at rigorously modeling the derived space of solutions to the field equations of Lagrangian gauge theories as derived $\Ci$-stacks. 

While this article does not explicitly discuss physics or field theory, it develops foundational aspects of derived differential geometry which are useful in their own right and contribute to the further development of the field. Moreover, these results provide essential groundwork for subsequent papers rigorously constructing derived spaces of solutions to Euler-Lagrange equations. We establish foundational results extending existing work on derived manifolds into supergeometric and infinite-dimensional contexts, and explicitly relate these constructions to differential operators and PDE theory. This paper an excerpt from a larger manuscript currently in preparation and is made available now to disseminate key foundational developments.

We begin by extending key results about derived manifolds to the supergeometric context, explicitly defining derived supermanifolds. This generalization is crucial for subsequent articles addressing fermionic matter fields. We also carefully clarify the role of infinite-dimensional manifolds in derived $\Ci$-geometry, providing two \emph{different} embeddings of the category of convenient manifolds of Fr\"olicher-Kriegl-Michor into derived stacks, expanding on work of Steffens. 

Moreover, we develop a novel approach to quasi-coherent sheaves, characterizing them as an explicit full subcategory of sheaves of modules on the large site. This formulation provides a concrete and explicit description of quasi-coherent sheaves over general stacks, in contrast to the traditional abstract definition as a limit of $\i$-categories arising from an affine cover. It is particularly well-suited to studying examples such as de Rham stacks, which provide a natural geometric setting for analyzing PDEs. In this direction, we introduce a topos-theoretic definition of infinite jet bundles and nonlinear differential operators via the de Rham stack, and prove that this construction recovers the classical infinite jet bundle of a fiber bundle, including its Fr\'echet manifold structure, extending work of Khavkine-Schreiber.

Finally, we establish an equivalence between the $\i$-category of quasi-coherent sheaves on the de Rham stack $\cM_{dR}$ of any supermanifold $\cM$ and the $\i$-category of $\mathcal{D}_{\cM}$-modules, giving a direct link to linear PDEs. Due to fundamental differences from the algebraic setting---such as the lack of compact generation, the fact that all module sheaves over affine $\Ci$-schemes are quasi-coherent, and that quasi-coherent sheaves form a full subcategory of module sheaves rather than vice versa---this equivalence does not formally follow from analogous results in derived algebraic geometry.
\end{abstract}

\tableofcontents

\section{Introduction}
Given two submanifolds $X$ and $Y$ of a smooth manifold $M,$ if they intersect transversely, $X \times_M Y$ is a closed submanifold. However, without transversality, things go terribly wrong; an arbitrary closed subset can be obtained this way. Even if the intersection is a manifold, it often does not have the expected cohomological properties, e.g. in cobordism theory. The theory of \emph{Derived Differential Geometry} (DDG) is designed to fix these deficiencies and is analogous to DAG. Such a theory provides a geometric framework in which arbitrary pullbacks of smooth maps between smooth manifolds exist as smooth objects that behave well with respect to cohomology, and agree with the ordinary fibered product in the case the maps are transverse. Objects which can be constructed iteratively in this manner are called \emph{derived manifolds}. 

There are various approaches to derived $\Ci$-geometry, \cite{spivak,joyce,derivjustden,dg2,Nuiten,univ,dgder,steffens1, pardon}, but a common theme is, in some way shape or form, derived $\Ci$-geometry is derived \emph{algebraic} geometry over $\Ci$-rings. $\Ci$-rings are commutative $\R$-algebras with extra structure, that allow one to not only interpret polynomials, but also smooth functions. The prototypical example is the ring $\Ci\left(M\right)$ of smooth functions on a manifold. $\Ci$-rings are algebras for a Lawvere theory, and as such there is well established way of discussing derived versions of such algebras as simplicial $\Ci$-rings. Roytenberg and the author offered an alternative equivalent differential graded approach using dg-$\Ci$-algebras \cite{dg2}, and many foundational results about the derived algebraic geometry of such algebras was developed by Nuiten in their thesis \cite{Nuiten}, and developed even further by Steffens \cite{steffens1}. Together with Steffens, we defined derived manifolds via a universal property \cite{univ} and proved they were equivalent to the affine derived schemes of finite type for the algebraic geometry of $\Ci$-rings. In subsequent work, we went on to prove that there is an equivalent description in terms of differential graded manifolds \cite{dgmander}. 

In this article, we define derived supermanifolds in an analogous way as in our work with Steffens, and carefully extend the above two results to this setting. To do so, we use the $2$-sorted algebraic theory of $\SCi$-rings we developed with Roytenberg \cite{dg1,dg2}. We also clarify recent work of Steffens related to the integration of derived geometry with the convenient manifolds of Fr\"{o}licher-Kriegl-Michor.

We go on two develop a theory of modules and quasi-coherent sheaves. This definition is the same as that of Nuiten \cite{Nuiten} and Steffens \cite{steffens1}, but we present a novel approach which is particularly convenient for working with derived stacks which are not affine $\Ci$-schemes. Our approach allows one to view quasi-coherent sheaves over $\cX$ concretely as certain $\O_{\cX}$-module sheaves over the \emph{big} site of $\cX,$ and is based in part on key ideas from \cite{qcglob}.

We then go on to discuss infinite-jet bundles and $\mathcal{D}$-modules in this context. Both of these are discussed in terms of Simpson's \emph{de Rham stack} \cite{simpson}. The case of infinite-jet bundles largely follows work of Khavkine-Schreiber \cite{jett} carefully adapted to the derived setting. We first show that the infinite jet bundle of any fiber bundle over a manifold, complete with its Frech\'et manifold structure, can be reconstructed purely from the de Rham stack. We then study quasi-coherent sheaves over the de Rham stack. Our main result is that for any supermanifold $\cM,$ there is an equivalence of $\i$-categories between $\mathcal{D}$-modules over $\cM$ and quasi-coherent sheaves over the de Rham stack $\cM_{dR}.$  This mirrors classical equivalences established by Gaitsgory and Rozenblyum \cite{GaitsRoz}, but the method of proof is quite different. On one hand, there is a technical challenge that $\QC\left(\cX\right)$ is rarely compactly generated, and there is no good notion of coherent sheaf, making approaches based on ind-coherent sheaves difficult to apply. On the other hand, our framework naturally takes into account the geometry of jets in that the notion of differential operators can be formulated in terms of the Frech\'et manifold structure on the infinite jet bundle.

The field of derived differential is under rapid development, but there is still much to be done. One striking omission from this article is a reformulation in the language of dg-manifolds. This is work in progress.

\subsection{Organization and Main Results:}

In Section \ref{sec:algthy}, we briefly review the general theory of multi-sorted algebraic theories and their homotopy algebras. The main example of interest is the $2$-sorted theory of $\SCi$-algebras of \cite{dg1,dg2}.

In Section \ref{sec:ci}, we summarize important basic properties of $\Ci$-algebras and extend them to $\SCi$-algebras.

In Section \ref{sec:derived_mfd} we review the universal property of derived manifolds from \cite{univ}, and its relationship with $\Ci$-algebras. We also extend this definition to supermanifolds in the obvious way:

\begin{definition}(Definition \ref{dfn:dsmfd})
The $\i$-category  $\mathsf{DSMfd}$ of derived supermanifolds is the (essentially unique) $\i$-category with finite limits, together with a functor $i:\mathsf{SMfd} \to \mathsf{DSMfd}$ which preserves transverse pullbacks and the terminal object, such that for all $\i$-categories $\sC$ with finite limits, composition along $i$ induces an equivalence of $\i$-categories
$$\Fun^{\mathbf{lex}}\left(\mathsf{DSMfd},\sC\right) \stackrel{\sim}{\longlongrightarrow} \Fun^\pitchfork\left(\SMfd,\sC\right)$$
between the $\i$-category of finite limit preserving functors from $\mathsf{DSMfd}$ to $\sC$ and the $\i$-category of functors from $\SMfd$ to $\sC$ which preserves transverse pullbacks and the terminal object.
\end{definition}

In Section \ref{sec:homotopical}, we discuss the homotopy theory and homological algebra of homotopical $\SCi$-algebras. We review the model structure on dg-$\Ci$-superalgebras, discuss the localization of such algebras, and the construction of graded Koszul complexes. As an application, we prove that the morphism of algebraic theories $$\mathbf{SCom}_\R \to \SCi$$ is unramified, and use this to prove:

\begin{corollary}(Corollary \ref{cor:supdman})
There is a canonical equivalence of $\i$-categories
$$\mathsf{DSMfd} \simeq \left(\Alg_{\SCi}\left(\Spc\right)^{\mathbf{fp}}\right)^{op}$$ between the $\i$-category of derived supermanifolds and the opposite of the $\i$-category of finitely presented $\SCi$-algebras in $\Spc.$
\end{corollary}

In Section \ref{sec:schemes} we discuss the spectrum functor $\Speci$ for $\SCi$-algebras, and define derived $\SCi$-schemes, and prove that for a supermanifold $\mathcal{M},$ $\mathcal{M} \simeq \Speci\left(\Ci\left(\cM\right)\right),$ as a locally ringed space. In Section \ref{sec:infdimmfd} we explore the link between infinite-dimensional manifolds and $\Ci$-schemes. In particular, we prove:

\begin{proposition} (Proposition \ref{prop:convchar})
For a locally convex vector space $E$ or convenient vector space, $\left(E,\Ci_E\right)$ is an affine $\Ci$-scheme if and only if $E$ smoothly regular, realcompact, and every germ of a $\Ci$-function has a global representative. Similarly, for $E$ a convenient vector space, $\left(E_{c^\i},\Ci_C\right)$ is an affine $\Ci$-scheme if and only if $E_{c^\i}$ smoothly regular and realcompact, where $E_{c^\i}$ is $E$ equipped with the $c^\i$-topology.
\end{proposition}

which has the following corollary:

\begin{corollary}(Corollary \ref{cor:allinfmfds})
The following classes of locally convex or convenient vector spaces are affine $\Ci$-schemes when equipped with their appropriate sheaf of smooth functions:
\begin{itemize}
\item[1.] Any smoothly regular locally convex vector space which admits smooth partitions of unity. This includes:
\begin{itemize}
\item[a)] Any metrizable locally convex vector space for which smooth functions separate disjoint closed sets.
\item[b)] All Hilbert spaces.
\item[c)] All Nuclear Frech\'et spaces.
\end{itemize}
\item[2.] All realcompact smoothly regular convenient vector spaces (with the $c^\i$-topology).
\item[3.] Any smoothly regular locally convex vector space which is Lindel\"{o}f. Any smoothly regular convenient vector space which is Lindel\"{o}f for its $c^\i$-topology.
\end{itemize}
\end{corollary}

We show that nonetheless, the induced functor from locally convex vector spaces to $\Ci$-schemes is only faithful and not full unless restricted to convenient vector spaces. We also examine the fully faithful inclusion of convenient vector spaces into the topos of sheaves on affine $\Ci$-schemes induced by first taking their functor of points on finite dimensional manifolds, and then left Kan extending. We show that for a manifold $M$, these two embeddings agree if and only if $M$ is finite dimensional. Finally, we prove that the infinite jet bundle of a submersion over a finite dimensional manifold $M,$ when regarded as a $\Ci$-scheme, is the limit in $\Ci$-schemes of the associated finite jet-bundles.

In Section \ref{sec:derived_topos}, we discuss various adjunctions that arise on the $\i$-topos of sheaves over a suitably chosen site of affine $\SCi$-schemes, in the spirit of Schreiber's differential cohesion \cite{Urs}. We emphasize that the $\i$-topoi arising as sheaves over small sites of affine $\SCi$-schemes are \emph{NOT} cohesive however.

In Section \ref{sec:orbifold}, we specialize our theory of higher orbifolds and Deligne-Mumford stacks \cite{higherme} to the derived $\SCi$-setting. We establish some technical results we use later in our study of quasi-coherent sheaves.

In Section \ref{sec:modandqc}, we define modules and quasi-coherent sheaves in the setting of derived $\SCi$-geometry. Our first result is the following:

\begin{theorem}(Theorem \ref{thm:stabmod})
For any dg-$\SCi$-superalgebra $\cA,$ there is a canonical equivalence of $\i$-categories
$$\Stab\left(\dgc/\cA\right) \simeq \Mod_{\A}$$ between the stabilization of the slice category over $\cA,$ and the $\i$-category of unbounded dg-modules over $\cA.$
\end{theorem}

We go on to discuss square-zero extensions, derivations, and tangent $\i$-categories. In Sections \ref{sec:shvmod} and \ref{sec:funmod}, we define module sheaves for $\SCi$-ringed $\i$-topoi and use their functoriality properties to give a very general version of a module-spectrum functor.

In Section \ref{sec:qc}, we develop the theory of quasi-coherent sheaves. In $\Ci$-geometry, the theory of quasi-coherent sheaves is a bit different than that of algebraic geometry, as was already noticed by Joyce \cite{joycesch}. For $X=\left(\Speci\left(\A\right),\O_\A\right)$ an affine $\Ci$-scheme, \emph{every} $\O_\A$-module is quasi-coherent, but not every $\A$-module is global sections of an $\O_\A$-module. Quasi-coherent sheaves on $\Speci\left(\A\right)$ are equivalent to a reflective subcategory of $\A$-modules called \emph{complete} $\A$-modules $\widehat{\Mod}_{\A}$. It follows that it is a symmetric monoidal with an inherited tensor product $\widehat{\otimes}$ from that of $\Mod_{\Speci\left(\A\right)}.$ Whereas our definition of quasi-coherent sheaf agrees with that of Nuiten \cite{Nuiten} and Steffens \cite{steffens1}, we present a novel approach which is particularly convenient for working with derived stacks which are not affine $\Ci$-schemes. Our approach allows one to view quasi-coherent sheaves over $\cX$ concretely as certain $\O_{\cX}$-module sheaves over the \emph{big} site of $\cX,$ and is based in part on key ideas from \cite{qcglob}. Denoting $\bH$ the $\i$-topos of sheaves on a small site of affine $\SCi$-schemes, the pair $\left(\bH,\R\right)$ is a $\SCi$-ringed $\i$-topos, and we denote $\R$ when regarded as a $\SCi$-ring object as $\O.$ If $\cX \in \bH$ is any stack, $\left(\bH/\cX,\O_{\cX}\right)$ is an $\SCi$-ringed $\i$-topos. We prove:

\begin{proposition}(Proposition \ref{prop:lamlam})
For any stack $\cX \in \bH,$ $\QC\left(\cX\right)$ can be realized as a coreflective subcategory of $\Mod_{\O_\cX}$.
\end{proposition}

Moreover, for any quasi-coherent sheaf $\cF$ on derived stack $\cX$, realized as an $\O_\cX$-module, and any morphism $f:\Speci\left(\A\right),$ $\cF\left(f\right)$ is a \emph{complete} $\A$-module and we have the following theorem:

\begin{theorem}(Theorem \ref{thm:qccond})
Let $\cX \in \bH$ be an arbitrary derived stack. Then the following conditions are equivalent for an $\O_\cX$-module $\cF$:
\begin{itemize}
	\item[i)] $\cF$ is quasi-coherent (i.e. in the essential image of the fully-faithful functor $\Lambda^\star_\cX$).
	\item[ii)] For all morphisms
	$$\xymatrix{\Speci\left(\cB\right) \ar[rr]^-{\varphi} \ar[rd]_-{g} & & \Speci\left(\A\right) \ar[ld]^-{f}\\
		& \cX & }$$ the canonical map
	$$\cB \underset{\A} {\widehat{\otimes}} \cF\left(f\right) \to \cF\left(g\right)$$ is an equivalence.
	\end{itemize}
	\end{theorem}
	
	In Section \ref{sec:deRham}, we introduce the de Rham stack in the derived $\SCi$-setting. We use it to define formal neighborhoods, and the notions of formally \'etale and formal smoothness.
	
	In Section \ref{sec:group}, we review the concept of $\i$-group object introduced by Nikolaus-Schreiber-Stevenson in \cite{pb}. There are no new results here.
	
	In Section \ref{sec:jet} we introduce jet bundles and differential operators synthetically in terms of the de Rham stack. The main definitions are an adaption of ideas and concepts from \cite{jett} of Khavkine-Schrieber. Given any derived stack $\cX,$ it comes with a canonical map 
$$\eta:\cX \to \cX_{dR},$$	
to its de Rham stack, which equivalently is a geometric morphism between their slice topoi
$$\eta:\bH/cX \to \bH/\cX_{dR}.$$ Given an object $\pi:E \to \cX$ in $\bH/\cX,$ which one can regard as a generalized bundle, the \emph{infinite jet bundle} of $E$ is $\J^\i E:=\eta^*\eta_*E.$ (or more precisely $\eta^*\eta_*\pi$). To justify this definition, we prove:

\begin{theorem}(Corollary \ref{cor:jetfrech})
Let $\pi:E \to M$ be a fiber bundle over a finite dimensional manifold. Then $\J^\i E$ is the usual infinite-jet bundle with its Frech\'et manifold structure, regarded as a $\Ci$-scheme. 
\end{theorem}

(A similar result in the non-derived context was proven in \cite{jett}.) We then discuss the notion of non-linear differential operators in this context and how they induce morphisms between respective stacks of sections. The $\i$-jet functor $\J^\i$ has the canonical structure of a comonad, and this closely related to Vinogradov's category of formally integrable PDEs \cite{vino1,vino2}. The linearized version is related to $\mathcal{D}$-modules.

In Section \ref{sec:monad} we start by discussing the general theory of monads and comonads and their interaction with higher topos theory. Along the way we prove the $\i$-categorical analogue of a well-known result in classical category theory:

\begin{lemma}(Lemma \ref{lem:AlgKL})
	Given a monad $T$ on an $\i$-category $\sC,$ there is a pullback diagram of $\i$-categories	$$\xymatrix{\Alg_T\left(\sC\right) \ar@{^(_->}[r] \ar[d] & \Psh\left(\Kl\left(T\right)\right) \ar[d]^-{\left(F_\Kl\right)^\ast} \\
		\sC \ar@{^(_->}[r]^-{y} & \Psh\left(\sC\right),}$$
		where $\Kl\left(T\right)$ is the Kleisli $\i$-category of free $T$-algebras.
	\end{lemma}
	
	As an application we prove the following:

\begin{corollary}(Corollary \ref{cor:pbqc})
Let $f:E \to F$ be an epimorphism in $\bH$. Then the following diagram of $\i$-categories is a pullback:
$$\xymatrix@C=1.5cm{\QC\left(F\right) \ar[r]^-{\Lambda^\star\left(F\right)} \ar[d]_-{f^\star} & \Mod_{\O_F} \ar[d]^-{f^\star}\\
\QC\left(E\right) \ar[r]_-{\Lambda^\star\left(E\right)} & \Mod_{\O_E}.}$$
I.e. one can identify quasi-coherent sheaves on $\QC\left(F\right)$ as those $\O_F$-modules $\cF$ whose pullback $f^\star \cF$ to $F$ are quasi-coherent.
\end{corollary}

The rest of Section \ref{sec:monad} contains a series of technical $\i$-categorical results necessary to prove the main result about $\mathcal{D}$-modules.

Section \ref{sec:DMod} is dedicated to the theory of $\mathcal{D}$-modules. The $\i$-jet comonad $\J^\i$ canonically lifts to a comonad on $\Mod_{\O_{\cX}}$ allowing one to study linear differential operators systematically, and to define the ring of differential operators abstractly in such a way as to recover the usual ring in the case of a manifold or supermanifold:

\begin{proposition}(Propositions \ref{prop:locfinjet} and \ref{prop:Dsuper})
Let $\cM$ be a supermanifold. Then $\cD_{\cM}$ is the sheaf which assigns each open subset $U$ of the core $|\cM|_0$ the ring of linear differential operators of $\Ci_{\cM}\left(U\right)$ which are \underline{locally} of finite order. In particular, for $M$ an even manifold, such differential operators are the same as maps of Frech\'et vector bundles $\J^\i M \to \underline{\R}.$
\end{proposition}

Finally, we prove the main result:

\begin{theorem}(Theorem \ref{thm:DMod})
Let $\cM$ be a supermanifold. Then there is a canonical equivalence of $\i$-categories
$$\QC\left(\cM_{dR}\right) \simeq \DMod\left(\cM\right).$$
\end{theorem}

\subsection*{Acknowledgments}
Firstly, we would like to thank Owen Gwilliam, with whom we have been working for many years now. This theory grew out of a desire to formalize many ideas that came up in our discussions about field theory. A large part of this manuscript was produced, in some form, while supported by a joint grant with Gwilliam from the National Science Foundation, Award No. 1811864, and we wish to thank the NSF for their support. We would also like to thank Tobias Barthel, Joseph Bernstein, Christian Blohmann, Mikhail Kapranov, Jacob Lurie, Joost Nuiten, Emily Riehl, Dmitry Roytenberg, Nick Rozenblyum, Pelle Steffens, and Peter Teichner for useful conversations.

\newpage

\section{Algebraic theories}\label{sec:algthy}

\begin{definition}
An \textbf{algebraic theory} is an $\infty$-category $\bT$ with finite products. A morphism of algebraic theories $$\bT \to \bT'$$ is a finite product preserving functor. Denote the associated $\i$-category by $\mathbf{AlgThy}.$
\end{definition}

\begin{example}
Let $\bbS$ be a set. Regarding this set as a discrete category, let $\mathbf{T}_\mathbb{S}$ be the
free completion of $\bbS$ with respect to finite products. Concretely, the objects are finite families
$$\left(s_i \in \mathbb{S}\right)_{i\in I}$$
and morphisms
$$\left(s_i \in \mathbb{S}\right)_{i\in I} \to \left(t_j \in \mathbb{S}\right)_{j\in J}$$
are functions of finite sets $f:J\to I$ such that
$$s_{f\left(j\right)}= t_j$$
for all $j\in J$.
Then $\mathbf{T}_\mathbb{S}$ is the \textbf{algebraic theory of $\bbS$-sorts}.

If $\bbS$ is a singleton set, then $\mathbf{T}_\mathbb{S}\cong \mathbf{FinSet}^{op}$ is the opposite category of finite sets. We denote this algebraic theory by $\bT_{obj}.$
\end{example}

\begin{definition}\label{definition:algebra}
Let $\sC$ be an $\i$-category with finite products. The $\i$-category $\Alg_{\bT}\left(\sC\right)$ of \textbf{$\bT$-algebras in $\sC$} is the full subcategory of $\Fun\left(\bT,\sC\right)$ on those functors which preserve finite products.

Of particular importance is the cases when $\sC=\Set$ and when $\sC=\Spc$. We will call an object of $\Alg_{\bT}\left(\Spc\right)$ a $\bT$-algebra, and an object of $\Alg_{\bT}\left(\Set\right)$ a $0$-truncated $\bT$-algebra.
\end{definition}

\begin{proposition}
Let $\sC$ be a presentable $\i$-category. Then $\Alg_{\bT}\left(\sC\right)$ is also presentable.
\end{proposition}

\begin{proof}
This follows immediately from  \cite[Lemmas 5.5.4.17, 5.5.4.18, 5.5.4.19]{htt}.
\end{proof}

\begin{example}
For any $\i$-category $\sC$ with finite products, there is a canonical equivalence $\Alg_{\bT_{obj}.}\left(\sC\right)\simeq \sC$ given by evaluation at $*$. More generally, for any set $\bbS,$ there is a canonical equivalence $\Alg_{T_{\bbS}}\left(\sC\right)\simeq \sC^{\bbS}.$ For example, if $\bbS=\left\{1,2,\ldots,n\right\},$ a $\bT_{\bbS}$-algebra in $\sC$ is the same data as the choice of $n$ objects of $\sC.$
\end{example}

\begin{definition}
Given an algebraic theory $\bT,$ an object $r \in \bT_0$ is called a \textbf{generator} if every object in $\bT$ is equivalent to $r^n,$ for some $n \ge 0.$ More generally, a subset $\bbS \subseteq \bT_0$ is said to be a set of generators for $\bT$ if every object in $\bT$ is equivalent to a finite product of objects in $\bbS.$
\end{definition}

\begin{definition}
An \textbf{$\bbS$-sorted Lawvere theory} is a skeletal algebraic theory $\bT$ together with a chosen a set of generators $\bbS \subseteq \bT_0.$ Equivalently, this is the same as a faithful map of algebraic theories $$\sigma_\bT:\bT_{\bbS}\to \bT.$$ A morphism $\varphi:\bT \to \bT'$ of $\bbS$-sorted Lawvere theories is a morphism in $\bbS/\mathbf{AlgThy},$ i.e. a finite product preserving functor making the following diagram commute
$$\xymatrix{& \bT_{\bbS} \ar[ld]_-{\sigma_\bT} \ar[rd]^-{\sigma_{\bT'}} & \\ \bT \ar[rr]^-{\varphi} && \bT'.}$$
\end{definition}

\begin{remark}
Since $\bT$ is further assumed skeletal, we may assume, without loss of generality, that the set of objects of $\bT$ is $\mathbb{N}^\bbS.$ Often, the set $\bbS$ is assumed to be finite. When $\bbS$ has $n$ elements, the Lawvere theory is said to be $n$-sorted. In this paper, we will only need to discuss $1$-sorted and $2$-sorted Lawvere theories.
\end{remark}

Most types of algebraic objects have an associated algebraic theory. For example:

\begin{example}
Let $k$ be a commutative ring and let $\Comk$ be the the opposite category of finitely generated free $k$-algebras. Making this a skeletal category, we may represent it as the category of affine $k$-schemes of the form $\bbA_k^n,$ for $n \ge 0.$ Then, $\Comk$ is a $1$-sorted Lawvere theory with $\bbA_k^1$ as a generator.

A $\Comk$-algebra in $\Set$ is an ordinary commutative $k$-algebra. More precisely, given a $\Comk$-algebra $$\mathcal{A}:\Comk \to \Set,$$ the set $\mathcal{A}\left(\bbA_k^1\right)$ has an induced structure of a commutative $k$-algebra. This is because $\bbA_k^1$ is a commutative $k$-algebra object in schemes (and hence in $\Comk$) and any finite product preserving functor preserves commutative $k$-algebra objects. For example, the maps
\begin{eqnarray*}
\bbA_k^1 \times \bbA_k^1 = \bbA_k^2 &\to & \bbA_k^1\\
\left(x,y\right) &\mapsto & x+y
\end{eqnarray*}
and 
\begin{eqnarray*}
\bbA_k^1 \times \bbA_k^1 = \bbA_k^2 &\to & \bbA_k^1\\
\left(x,y\right) &\mapsto & x\cdot y
\end{eqnarray*}
define addition and multiplication for the ring object $\bbA_k^1.$ Conversely, given a commutative $k$-algebra $B,$ the functor $$\Hom\left(\blank,B\right)=\Hom\left(\Spec\left(B\right),\blank\right)$$ preserves finite products, hence is a $\Comk$-algebra in $\Set$. These constructions are categorically inverse to one another.
\end{example}

\begin{example}
Similarly, let $k$ be a commutative ring and let $\SComk$ be the the opposite category of supercommutative $k$-algebras free on finitely many even and odd generators. This can be identified with the category of superschemes over $\Spec\left(k\right)$ of the form $\bbA_k^{n|m},$ for $n,m \ge 0.$ The category $\SComk$ is a $2$-sorted Lawvere theory with two generators; $\bbA_k^1=\bbA_k^{1|0}$ is the \emph{even} generator, and 
$$\bbA_k^1=\bbA_k^{0|1}=\Spec\left(k\left[\theta\right]\right),$$ (with $\theta$ an element of degree $1$) is the \emph{odd} generator.

Similarly to above, the category of $\SComk$-algebras in $\Set$ is canonically isomorphic to the category of supercommutative $k$-algebras. In slightly more detail,
a $\SComk$-algebra $$\mathcal{A}:\SComk \to \Set,$$ determines a supercommutative algebra whose set of even elements is $\mathcal{A}\left(\bbA_k^{1|0}\right)$ and whose set of odd elements is $\mathcal{A}\left(\bbA_k^{0|1}\right).$ From this point of a view, a supercommutative algebra does not have an underlying set, but instead has an underlying $\Z_2$-graded set.
\end{example}

This works well for other types of algebraic gadgets, such as groups, monoids, abelian groups, modules etc.

\begin{example}
Denote by $\bT_\Ab$ the opposite of the full subcategory of abelian groups on those of the form $\Z^n$ for some $n.$ Then $\bT_{\Ab}$ is an algebraic theory with $\Z$ as a generator. A classical $\bT_\Ab$ algebra is an abelian group.

Moreover, if $\sC$ is any category, or $\i$-category, with finite products, then $\Alg_{\bT_{\Ab}}\left(\sC\right)$ is the category of abelian group objects in $\sC,$ and hence we write
$$\Ab\left(\sC\right):=\Alg_{\bT_{\Ab}}\left(\sC\right)$$

More generally, the above holds with the category of abelian groups ($\Z$-modules) replaced with modules for any commutative ring $k$, resulting in an algebraic theory $\bT_{\mathbf{Mod}_k}.$
\end{example}

The following two examples are of fundamental importance for derived $\Ci$-geometry:

\begin{example}\label{ex:Ci}
Let $\Cart$ denote the full subcategory of the category of smooth manifolds $\Mfd$ on those manifolds of the form $\R^n$ for some $n.$ Then $\Cart$ is an algebraic theory with $\R$ as a generator. $\Cart$-algebras in $\Set$ are also known as \emph{$\Ci$-rings.}

Note that if $k=\R,$ $\ComR$ may also be seen as the category of manifolds of the form $\R^n$ whose morphisms are given by polynomials. Hence, there is an evident functor $\ComR \to \Cart$ which is a morphism of $1$-sorted Lawvere theories.
\end{example}

\begin{example}\label{ex:SCi}
Similarly, let $\SCart$ denote the full subcategory of the category of supermanifolds $\SMfd$ on those of the form $\R^{p|q}$ for $p,q \ge 0.$ Then $\SCart$ is a $2$-sorted Lawvere theory with the real line $\R^1=\R^{1|0}$ as the even generator, and the odd line $\R^{0|1}$ as the odd generator. $\SCart$-algebras in $\Set$ are introduced in \cite{dg1} and are also known as \emph{$\Ci$-superalgebras}.

Similarly to above, there is an evident morphism of $2$-sorted Lawvere theories 
$$\SComR \to \SCart.$$
\end{example}

We will review some of the important aspects of (super) $\Ci$-algebras in Section \ref{sec:ci}.

\subsection{Categorical properties of algebraic theories}

Let $\bT$ be an algebraic theory. Then as any co-representable functor preserves finite products, the Yoneda embedding restricts to a fully faithful functor $j:\bT \hookrightarrow \Alg_{\bT}\left(\Spc\right)^{op}.$ 
\begin{theorem} \cite[Proposition 5.5.8.10]{htt} \label{theorem:5.5.8.10}
\begin{itemize}
\item [a)] $\Alg_{\bT}\left(\Spc\right)$ is a localization of $\Fun\left(\bT,\Spc\right).$
 \item[b)] The $\i$-category $\Alg_{\bT}\left(\Spc\right)$ is compactly generated. In particular, $$\Ind\left(\Alg_{\bT}\left(\Spc\right)^{\mathbf{fp}}\right)\simeq \Alg_{\bT}\left(\Spc\right),$$ where $\Alg_{\bT}\left(\Spc\right)^{\mathbf{fp}}$ denotes the compact objects.
 \item[c)] The inclusion $j:\bT \hookrightarrow \Alg_{\bT}\left(\Spc\right)^{op}$ preserves finite products.
\end{itemize}
\end{theorem}

The functor $j$ also has the following universal property:

\begin{theorem} \cite[Proposition 5.5.8.15]{htt} \label{theorem:5.5.8.15}
Let $\sC$ be an $\i$-category which admits sifted limits. Then, composition with $j$ induces an equivalence of $\i$-categories
$$\Fun^{sift}\left(\Alg_{\bT}\left(\Spc\right)^{op},\sC\right) \stackrel{\sim}{\longlongrightarrow} \Fun\left(\bT,\sC\right),$$
between the $\i$-category of functors $\Alg_{\bT}\left(\Spc\right)^{op} \to \sC$ which preserves sifted limits, to the $\i$-category $\Fun\left(\bT,\sC\right),$ and if $\sC$ is complete, it also induces an equivalence of $\i$-categories
$$\Fun^{R}\left(\Alg_{\bT}\left(\Spc\right)^{op},\sC\right) \stackrel{\sim}{\longlongrightarrow} \Fun^{\Pi\!\!}\left(\bT,\sC\right)=\Alg_{\bT}\left(\sC\right)$$
between the $\i$-category of functors $\Alg_{\bT}\left(\Spc\right)^{op} \to \sC$ which preserves small limits, and the $\i$-category of $\bT$-algebras in $\sC.$ The inverse of both equivalences is given by right Kan extension.
\end{theorem}

We have the following corollary:

\begin{corollary}\label{cor:1.16}
A functor $F:\bT \to \Spc$ is a $\bT$-algebra if and only if $F$ is a sifted colimit of corepresentable functors
$$F \simeq \underset{\alpha} \colim j\left(t_\alpha\right).$$
\end{corollary}

\begin{lemma} \label{lem:refalg}
If $\sC$ is presentable, then $\Alg_{\bT}\left(\sC\right)$ is reflective in $\Fun\left(\bT,\sC\right).$
\end{lemma}

\begin{proof}
Since both $\i$-categories are presentable, by the adjoint functor theorem, it suffices to prove that the inclusion preserves limits. But this follows immediately from the fact that limits commute with finite products.
\end{proof}

By an analogous proof, we also have:

\begin{lemma} \label{lem:presift}
For any $\i$-category $\sC$ in which sifted colimits commute with finite products (e.g. an $\i$-topos, or more generally any $\i$-category monadic over an $\i$-topos), then inclusion $$\Alg_{\bT}\left(\sC\right) \hookrightarrow \Fun\left(\bT,\sC\right)$$ preserves sifted colimits.
\end{lemma}

\begin{proposition} \label{prop:monadic}
Let $\sC$ be a presentable $\i$-category in which sifted colimits commute with finite products (e.g. an $\i$-topos, or more generally any $\i$-category monadic over an $\i$-topos), and suppose that $\bT$ is an $\bbS$-sorted Lawvere theory with $\bbS$ finite. Then $\Alg_{\bT}\left(\sC\right)$ is monadic over $\sC^{\bbS}.$
\end{proposition}

\begin{proof}
Since both $\sC$ and $\Alg_{\bT}\left(\sC\right)$ are presentable, it suffices to prove that the functor 
$$R:\Alg_{\bT}\left(\sC\right) \to \sC^{\bbS}$$ given by $R=\underset{s \in \bbS} \prod \mathbf{ev}_s,$ preserves limits and sifted colimits, and is conservative. It is clearly conservative. Notice that $R$ factors as:

$$\Alg_{\bT}\left(\sC\right) \hookrightarrow \Fun\left(\bT,\sC\right) \stackrel{\underset{s \in \bbS}\prod \mathbf{ev}_s}{\longlongrightarrow} \sC^{\bbS}.$$ The first functor preserves limits by Lemma \ref{lem:refalg}, and the second since limits in functor categories are computed object-wise. The first functor preserves sifted colimits by Proposition \ref{lem:presift}, and the second does since colimits in the functor category are computed object-wise.
\end{proof}

\begin{definition}
Let $\bT$ be an $\bbS$-sorted Lawvere theory. Denote the left adjoint to
$$\underset{s \in \bbS} \prod \mathbf{ev}_s:\Alg_{\bT}\left(\Spc\right) \to \Spc^{\bbS}$$
by $$\bT\left\{\blank\right\}:\Spc^{\bbS} \to \Alg_{\bT}\left(\Spc\right).$$ $\bT$-algebras in the essential image of $\bT\left\{\blank\right\}$ are called \textbf{free $\bT$-algebras}.
\end{definition}

\begin{remark}
The $\bT$-algebras in the essential image of $j:\bT^{op} \hookrightarrow \Alg_{\bT}\left(\Spc\right)$ are also in the essential image of $\bT\left\{\blank\right\}$; they are precisely the \emph{finitely generated} free $\bT$-algebras. For example, if $\bT$ is $1$-sorted, with generator $r,$ then $j\left(r^n\right) \simeq \bT\left\{ \left\{1,2,\ldots,n\right\} \right\}$ is the free algebras on $n$ generators. We usually denote this as $\bT\left\{x_1,x_2,\ldots ,x_n\right\}.$
\end{remark}

\begin{example}
Suppose that $\bT$ is a $2$-sorted Lawvere theory, with generators $r,s\in \bT.$ We cannot speak of the free $\bT$-algebra on $n$-generators, since $\Alg_{\bT}\left(\Set\right)$ is not monadic over $\Set,$ but rather over $\Set^2.$ But we can speak of the free $\bT$-algebra on $n$ generators of \emph{type $r$} and $m$-generators of \emph{type $s$}, we denote it by $\bT\left\{x_1,x_2,\ldots,x_n ; \eta_1,\eta_2,\ldots \eta_m\right\}.$ The most important example for this paper will be $\bT=\SCi,$ the $2$-sorted Lawvere theory of super $\Ci$-algebras, of Example \ref{ex:SCi}. We call the sort corresponding to the generator $\R$ \emph{even} and the sort corresponding to the generator $\R^{0|1}$ \emph{odd}. We denote the free $\SCi$-algebra on $n$ even and $m$ odd generators by $\Ci\left\{x_1,x_2,\ldots x_n ; \eta_1,\eta_2 \ldots \eta_m \right\}.$ It the $\SCi$-algebra of smooth functions on the supermanifold $\R^{n|m}=\R^n \times \left(\R^{0|1}\right)^m.$
\end{example}

Notice that if $f$ is a morphism of algebraic theories, and $A:\bT' \to \sC$ is a $\bT'$-algebra in $\sC,$ then the composition $$\bT \stackrel{f}{\longrightarrow} \bT' \stackrel{A}{\longrightarrow} \sC$$ is a $\bT$-algebra in $\sC,$ which we denote by $f^*A.$ This induces a functor:
$$f^*:\Alg_{\bT'}\left(\sC\right) \to \Alg_{\bT}\left(\sC\right).$$

\begin{lemma} \label{prop:presift}
Let $\sC$ be an $\i$-category in which sifted colimits commute with finite products and let $f:\bT \to \bT'$ be a morphism of algebraic theories. Then $$f^*:\Alg_{\bT'}\left(\sC\right) \to \Alg_{\bT}\left(\sC\right)$$ preserves limits and sifted colimits.
\end{lemma}

\begin{proof}
Notice that if $f:\bT \to \bT'$ is a morphism of algebraic theories, then we have a commutative diagram
$$\xymatrix{\Alg_{\bT'}\left(\sC\right) \ar[rd] \ar[r]^-{f^\ast} & \Alg_{\bT}\left(\sC\right) \ar[d]\\
& \sC.}$$
Since both $\i$-categories are monadic over $\sC$ by Proposition \ref{prop:monadic}, and conservative functors reflect limits and colimits that they preserve, the result follows.
\end{proof}

\begin{corollary}\label{cor:completioncor}
If $\sC$ is a presentable $\i$-category in which sifted colimits commute with finite limits then $f^*$ has a left adjoint $$f_!:\Alg_{\bT}\left(\sC\right) \to \Alg_{\bT'}\left(\sC\right).$$
\end{corollary}

When $\sC$ is the $\i$-category of spaces $\Spc,$ this can be described more explicitly:
We can take the left Kan-extension along the Yoneda embedding:
$$\xymatrix@C=3cm{\PShv\left(\bT^{op}\right) \ar@{-->}[r]^-{\Lan_{y_{\bT}}\left(y_{\bT'} \circ f^{op}\right)} & \PShv\left(\bT'^{op}\right)\\
\bT^{op} \ar@{^{(}->}[u]_-{y_{\bT}} \ar[r]^-{f^{op}} & \bT'^{op} \ar@{^{(}->}[u]_-{y_{\bT'}}.}$$
This is the unique colimit preserving functor which agrees with $f$ on representables. Corollary \ref{cor:1.16} then implies that $\Lan_{y_{\bT}}\left(y_{\bT'} \circ f^{op}\right)$ restricts to a functor between the respective categories of algebras. For example, if $f$ is a morphism of Lawvere theories, $f_!$ can then be characterized as the unique colimit preserving functor sending each finitely generated free algebra $\bT\left\{x_1,x_2,\ldots x_n\right\}$ to $\bT'\left\{x_1,x_2,\ldots x_n\right\}.$ Since every algebra is a sifted colimit of finitely generated free ones, this uniquely determines the functor $f_!.$

\begin{corollary}
If $f:\bT \to \bT'$ is a morphism of algebraic theories which has a right adjoint $g,$ then functor 
$$f_!:\Alg_{\bT}\left(\Spc\right) \to \Alg_{\bT}\left(\Spc\right)$$ has both a left and right adjoint, and in particular preserves small limits and colimits.
\end{corollary}

\begin{proof}
Since $f$ is a morphism of algebraic theories, $f_!$ has $f^*$ as a right adjoint. Since $g$ is a right adjoint, it preserves limits, and hence is an morphism of algebraic theories. Consequently, we have two induced adjunctions

$$\xymatrix{\Alg_{\bT}\left(\Spc\right) \ar@<-0.5ex>[r]_-{f_!} & \Alg_{\bT'}\left(\Spc\right) \ar@<-0.5ex>[l]_-{f^\ast}}$$
and 
$$\xymatrix{\Alg_{\bT'}\left(\Spc\right)\ar@<-0.5ex>[r]_-{g_!}
& \Alg_{\bT}\left(\Spc\right).\ar@<-0.5ex>[l]_-{g^{\ast}}}$$
Using the inclusion $\Alg\left(\Spc\right) \hookrightarrow \PShv\left(\bT^{op}\right),$ and similarly for $\bT'$, we can conclude that $f_! \simeq g^*.$ Indeed, let $t \in \bT$ and $s \in  \bT'.$ Then we have the following string of natural equivalences:
\begin{eqnarray*}
g^*\left(j\left(t\right)\right)\left(s\right) & \simeq & \Map_{\bT}\left(t,g\left(s\right)\right)\\
&\simeq& \Map_{\bT'}\left(f\left(t\right),s\right)\\
&\simeq& j\left(f\left(t\right)\right)\left(s\right)\\
&\simeq& f_!\left(j\left(t\right)\right)\left(s\right).
\end{eqnarray*}
Finally, we are done since $g^*$ has $g_!$ as a left adjoint.
\end{proof}

\subsection{Finitely presented algebras}
\begin{definition}
\begin{enumerate}
     \item A $\bT$-algebra $X \in \Alg_{\bT}\left(\Spc\right)$ is \textbf{finitely generated} if the functor  $$\Alg_{\bT}\left(\Spc\right)\rightarrow \Spc$$ corepresented by $X$ preserves small filtered colimits consisting only of monomorphisms. Denote the full subcategory on the finitely generated algebras by $\Alg_{\bT}\left(\Spc\right)^{\mathbf{fg}}$.
    \item A $\bT$-algebra $X \in \Alg_{\bT}\left(\Spc\right)$ is \textbf{finitely presented} if the functor  $$\Alg_{\bT}\left(\Spc\right)^{\fp}\rightarrow \Spc$$ corepresented by $X$ preserves small filtered colimits; that is, if $X$ is a compact object. Denote the full subcategory on the finitely presented algebras by $\Alg_{\bT}\left(\Spc\right)^{\fp}$.
\end{enumerate}
\end{definition}

\begin{lemma}\label{lem:retcol} \cite[Lemma 3.21]{univ}
The subcategory $\Alg_{\bT}\left(\Spc\right)^{\fp}$ is the smallest subcategory of $\Alg_{\bT}\left(\Spc\right)$ containing the finitely generated free algebras and closed under finite colimits and retracts. Moreover, any finitely presented algebra is a retract of a finite colimit of finitely generated free algebras.
\end{lemma}

\begin{theorem}\label{thm:finenv}\cite[Theorem 3.22]{univ}
Let $\sC$ be any idempotent complete $\i$-category with finite limits. Then composition with $j^{\fp}$ induces an equivalence of $\i$-categories
$$\Fun^{\lex}\left(\left(\Alg_{\bT}\left(\Spc\right)^{\fp}\right)^{op},\sC\right) \stackrel{\sim}{\longlongrightarrow} \Fun^{\Pi\!\!}\left(\bT,\sC\right)=\Alg_{\bT}\left(\sC\right)$$
whose inverse is given by right Kan extension.
\end{theorem}

\subsection{Homotopical algebras}

When $\bT$ is a $1$-category, then the projective model structure on $\Fun\left(\bT,\Set^{\Delta^{op}}\right)$ restricts to the subcategory of simplicial $\bT$-algebras
$$\Alg_{\bT}\left(\Set\right)^{\Delta^{op}}\cong \Alg_{\bT}\left(\Set^{\Delta^{op}}\right).$$

We have the following result, originally due to Bergner:

\begin{theorem}\cite[Theorem 5.1]{Be},\cite[Corollary 5.5.9.2]{htt} \label{theorem:bergner}
There is an equivalence of $\i$-categories $$N_{hc}\left(\Alg_{\bT}\left(\Set\right)^{\Delta^{op}}_{proj.}\right) \simeq \Alg_{\bT}\left(\Spc\right)$$
between the homotopy coherent nerve of the category of simplicial $\bT$-algebras, endowed with the projective model structure, and the $\i$-category of $\bT$-algebras in $\Spc.$
\end{theorem}

\begin{example}
By the above, $\Alg_{\bT_{\Ab}}\left(\Spc\right)$ is equivalent to the $\i$-category of simplicial abelian groups, which by the Dold-Kan correspondence is equivalent to the $\i$-category of non-positively graded cochain complexes of abelian groups $\Ch^{\le 0}\left(\Ab\right).$
\end{example}

\begin{example}
Similarly, by the monoidal Dold-Kan correspondence, if $k$ is a field of characteristic zero, $\Alg_{\Comk}\left(\Spc\right)$ can be identified with the $\i$-category $dg\mbox{-}\left(\Comk\right)_{\le 0}$ of commutative non-positively graded differential $k$-algebras  (using cohomological grading conventions).
\end{example}

\section{$\Ci$-algebras}\label{sec:ci}

\begin{definition}
A \emph{classical} \textbf{$\Ci$-algebra} is a $\Ci$-algebra in $\Set,$ where $\Ci$ is the algebraic theory of Example \ref{ex:Ci}. Similarly, a \emph{classical} \textbf{$\SCi$-algebra} is a $\SCi$-algebra in $\Set,$ where $\SCi$ is the $2$-sorted algebraic theory of Example \ref{ex:SCi}.
\end{definition}

The prototypical example of a $\Ci$-ring is the ring of smooth functions $\Ci\left(M\right)$ on a smooth manifold $M.$ This ring has an $n$-ary operation  associated to each smooth function $f:\R^n \to \R$:
\begin{eqnarray*}
\Ci\left(M\right)^n\cong \Hom_{\Mfd}\left(M,\R^n\right) &\to& \Ci\left(M\right)\\
\left(\varphi_1,\ldots,\varphi_n\right)=\varphi:M \to \R^n &\mapsto& f\circ \varphi.
\end{eqnarray*}
As an algebra for the algebraic theory $\Ci,$ $\Ci\left(M\right)$ is the functor
$$\Hom_{\Mfd}\left(M,\blank\right):\Ci \to \Set,$$ and the underlying set of $\Ci\left(M\right)$ is recovered as the value of this functor on $\R.$ Recall that $\Ci$ is the full subcategory of $\Mfd$ on the manifolds of the form $\R^n$ for some $n \ge 0,$ and that this is an algebraic theory with a single generator $r:=\R.$ The free $\Ci$-algebra on $n$ generators is $j\left(r^n\right)=\Hom_{\Mfd}\left(\R^n,\blank\right),$ i.e. 
$$\Ci\{x_1,x_2,\ldots,x_n\} \cong \Ci\left(\R^n\right).$$

An analogous story holds for the $2$-sorted algebraic theory $\SCi$. The prototypical example of a $\SCi$-algebra, also called a \emph{super $\Ci$-ring}, is the superalgebra of smooth functions on a supermanifold $\cM.$ As a functor, this algebra is simply $\Hom_{\SMfd}\left(\cM,\blank\right).$ The main difference between this theory and the theory of $\Ci$-algebras is that there are two generators--- the even generator $\R$ and the odd generator $\R^{0|1}.$  If $\cM$ is a supermanifold, then the even elements $\Ci\left(\cM\right)_0=\Hom_{\SMfd}\left(\cM,\R\right)$ of its superalgebra of smooth functions are smooth functions to $\R,$ whereas its odd elements $\Ci\left(\cM\right)_1=\Hom_{\SMfd}\left(\cM,\R^{0|1}\right)$ correspond to smooth functions to the odd line $\R^{0|1}.$ In this formalism, there are no in-homogeneous elements, and so there is an underlying $\Z_2$-graded set of $\Ci\left(\cM\right).$ The free $\SCi$-algebra on $n$ even and $m$ odd generators is
$$\Ci\{x_1,\ldots,x_n;\eta_1,\ldots,\eta_m\}\cong \Ci\left(\R^{n|m}\right)\cong \Ci\left(\R^n\right) \underset{\R} \otimes \wedge^\bullet\left(\R^m\right),$$ where $\wedge^\bullet\left(\R^m\right)$ is the exterior algebra with its Koszul grading and the tensor product is as supercommutative $\R$-algebras.

One of the main reasons that $\Ci$-algebras are reasonable to work with is the following theorem:

\begin{theorem}\cite[Proposition 1.2]{MSIA}, \cite[Corollary 2.116]{dg1}
Let $I$ be an ideal of the underlying $\R$-algebra $\A_\sharp$ of a classical $\Ci$-algebra $\A.$ Then $I$ is a $\Ci$-congruence and $\A_\sharp/I$ carries the canonical structure of a $\Ci$-algebra making $\A \to \A_\sharp/I=:\A/I$ a $\Ci$-homomorphism. Similarly, for $\A$ a classical $\SCi$-algebra with underlying supercommutative $\R$-algebra $\A_\sharp,$ and $I$ a $\Z_2$ \emph{homogeneous} ideal of $\A_\sharp.$
\end{theorem}

The category of classical $\Ci$-algebras is presentable, so, in particular, it has coproducts. Coproducts in the category of commutative $k$-algebras are computed by tensor product: If $\A$ and $\B$ are commutative rings, their coproduct is $\A \underset{k} \otimes \B$--- the tensor product of $k$-modules equipped with an induced $k$-algebra structure. This is no longer the case for (classical) $\Ci$-algebras, and this is a good thing, since we have for $n,m > 0,$
$$\Ci\left(\R^n\right) \underset{\R} \otimes \Ci\left(\R^m\right) \subsetneq \Ci\left(\R^{n+m}\right).$$ But $\Ci\left(\R^n\right)$ and $\Ci\left(\R^m\right)$ are the free $\Ci$-algebras on $n$ and $m$ generators respectively, so their coproduct must be the free $\Ci$-algebra on $\left(n+m\right)$ generators, namely $\Ci\left(\R^{n+m}\right).$ We denote the coproduct in the category of classical $\Ci$-rings by $\oinfty.$ Then we have
$$\Ci\left(\R^n\right) \oinfty \Ci\left(\R^m\right) \cong \Ci\left(\R^{n+m}\right).$$ Similarly, we denote pushouts by
$$\xymatrix{\A \ar[d] \ar[r] & \cC \ar[d] \\ \B \ar[r] & \B \underset{\A} \oinfty \cC.}$$

The story is completely analogous for classical $\SCi$-algebras, and we denote the coproduct again by $\oinfty,$ which is justified since the inclusion $$\Alg_{\Ci}\left(\Set\right) \hookrightarrow \Alg_{\SCi}\left(\Set\right)$$ preserve coproducts.

The tensor product $\oinfty$ is remarkably well-behaved.

\begin{proposition}\cite[Chapter 1, Theorem 2.8]{MSIA}
If $f:M \to N$ and $g:L \to N$ are transverse maps of smooth manifolds, then
$$\Ci\left(M \times_N L\right) \cong \Ci\left(M\right) \underset{\Ci\left(N\right)} \oinfty \Ci\left(L\right).$$
\end{proposition}

\begin{lemma}\label{lem:whit}
Every supermanifold is a retract of an open subset of $\R^{n|m}$ for some $n$ and $m.$
\end{lemma}
\begin{proof}
Let $M$ be the underlying even manifold of a supermanifold $\cM.$ By Batchelor's theorem \cite{Batchelor}, $\cM \cong \Pi E$ for $E$ a finite dimensional vector bundle over $M.$ $E$ in turn is a direct summand of a trivial bundle $V$ of rank $m$ for some $m.$ By the tubular neighborhood theorem combined with Whitney's embedding theorem, write $M$ as a retract of an open subset $U$ of $\R^n.$ Then $\Pi V\cong M \times \R^{0|m}$ is a retract of $U \times \R^{0|m}$ which is an open subset of $\R^{n|m}.$ Since $E$ is a retract of $V,$ $\Pi E$ is a retract of $\Pi V,$ and the result follows.
\end{proof}

\begin{theorem}
If $f:\cM \to \cN$ and $g:\cL \to \cN$ are transverse maps of supermanifolds, then
$$\Ci\left(\cM \times_\cN \cL\right) \cong \Ci\left(\cM\right) \underset{\Ci\left(\cN\right)} \oinfty \Ci\left(\cL\right).$$
\end{theorem}

\begin{proof}
Analyzing the proof of \cite[Chapter 1, Theorem 2.8]{MSIA}, it consists of two separate facts. Firstly, that the above holds when $g$ is a closed embedding, and secondly that it holds when $\cN$ is a point (i.e. binary products are sent to coproducts). The proof of the first fact goes through with virtually no modification in the supergoemetric setting. The proof of the second fact follows formally from the fact that every smooth manifold $M$ is a retract of an open subset of $\R^n.$ So we are done by Lemma \ref{lem:whit}.
\end{proof}

The following fact is fundamental:

\begin{proposition}
The canonical functors
$$\Ci\left(\blank\right):\Mfd \to \Alg_{\Ci}\left(\Set\right)$$
$$\Ci\left(\blank\right):\SMfd \to \Alg{\SCi}\left(\Set\right)$$
are fully faithful.
\end{proposition}

We combine these results into the following theorem for reference

\begin{theorem}\cite[Theorem 4.51]{univ} \label{thm:4.51}
The canonical functor $\SCi:\SMfd \to \dgc^{op}$ is fully faithful and preserves transverse pullbacks.
\end{theorem}
\subsubsection{$\Ci$-completion}
Denote by $\tau:\mathsf{SComm}_{\R} \to \SCi$ the evident morphism of algebraic theories from the theory of supercommutative $\R$-algebras to the theory of $\SCi$-algebras. By Corollary \ref{cor:completioncor} there is an induced adjunction between their associated categories of algebras
$$\xymatrix@C=2.5cm{\mathsf{SCAlg}_{\R} \ar@<+1ex>[r]^-{\widehat{\left(\blank\right)}}  & \SCi\mathsf{Alg} \ar@<+1ex>[l]^-{\left(\blank\right)_\sharp}}$$
where $\left(\blank\right)_\sharp=\tau^*$ and $\widehat{\left(\blank\right)}=\tau_!.$
\begin{definition}
Given a supercommutative $\R$-algebra $\A,$ its \textbf{$\Ci$-completion} is $\widehat{\A}.$ Given a (homogeneous) ideal $I$ of $\A,$ denote by $\widehat{I}$ the extension of $I$ along
$$\A \to \A \widehat{A}_\sharp,$$ the unit of the adjunction $\tau_! \dashv \tau^*.$
\end{definition}

\begin{example}
Let $\A$ be a finitely generated (classical) supercommutative $\R$-algebra. Then $\A$ has a presentation as $\R\left[x_1,\ldots,x_n;\theta_1,\ldots,\theta_m\right]/I,$ for a homogeneous ideal. The $\Ci$-completion of $\A$ is given as
$$\widehat{A}=\Ci\left(\R^{m|n}\right)/\widehat{I}.$$
\end{example}

The $\Ci$-completion can sometimes be destructive, as the following example illustrates:

\begin{example}
Consider $\mathbb{C}$ as a commutative $\R$-algebra. Notice $$\mathbb{C}\cong \R\left[x\right]/\left(x^2+1\right).$$ Since $\R\left[x\right]$ is free and $\widehat{\left(\blank\right)}$ is a left adjoint, $\widehat{\R\left[x\right]}=\Ci\{x\}=\Ci\left(\R\right).$ The extension of the principal ideal $\left(x^2+1\right)$ is still a principal ideal and generated by the same element $x^2+1.$ But $x^2+1$ is invertible on $\R,$ as it is never zero, so is a unit in $\Ci\left(\R\right),$ hence $\widehat{\left(x^2+1\right)}=\Ci\{x\},$ and 
$\widehat{\mathbb{C}}=0.$
\end{example} 
\subsubsection{Localizations}
\begin{definition}
Let $\A$ be a classical $\SCi$-algebra, and let $a \in \A_0.$ A morphism $\varphi:\A \to \B$ of $\SCi$-rings is a \textbf{localization of $\A$ at $a$} if $\B$ fits into a pushout diagram
$$\xymatrix{\Ci\left(\R\right) \ar[r]^-{\tilde a} \ar[d] & \A \ar[d]^-{\varphi}\\ \Ci\left(\R \setminus \left\{0\right\}\right) \ar[r] & \B.}$$
In this case we write $\cB=\A\left[a^{-1}\right].$ If $S$ is a multiplicatively closed subset of $\A_0,$ then the \textbf{localization of $\A$ at $S$} is the filtered colimit
$$\A_S:=\underset{a \in S} \colim \A\left[a^{-1}\right].$$ The localization of $\A$ at a prime ideal $P$ is by definition the localization of $\A$ at $\A\setminus P.$
\end{definition}

\begin{lemma}[{Lemma \cite[5.1.13]{Nuiten}}]\label{lem:5.1.13}
Any localization $\A \to \A\left[a^{-1}\right]$ is flat.
\end{lemma}

\begin{corollary}
Any localization of a $\Ci$-algebra is flat.
\end{corollary}

\subsection{Reduced $\Ci$-algebras}

\begin{definition}
Let $\A$ be a classical $\Ci$-algebra, and let $x:\A \to \R$ be an $\R$-point of $\A.$ Let $\mathfrak{m}_x$ be the maximal ideal $\mathfrak{m}_x:=\ker\left(x\right).$ Denote the natural morphism to the localization by $$\germ_x:\A \to \A_x$$ where $\A_x$ is the $\Ci$-localization of $\A$ at $\mathfrak{m}_x.$ Given an element $a \in \A,$ its \textbf{germ at $x$} is $\germ_x\left(a\right).$ Similarly, given an ideal $I$ of $\A,$ its \textbf{germ at $x$}, $I_x$, is the extension of $I$ along the homomorphism $\germ_x.$
\end{definition}

\begin{lemma}\label{lem:localq}
Let $I$ be an ideal of a classical $\Ci$-algebra $\A.$ Let $x:\A/I \to \R$ be an $\R$-point. Then
$$\left(\A/I\right)_x \cong \A_x/I_x.$$
\end{lemma}

\begin{proof}
Firstly, let $\pi:\A \to \A/I$ and let $a \in \A.$ Then
\begin{eqnarray*}
\left(\A/I\right)\left\{\pi\left(a\right)^{-1}\right\} & \cong & \left(\A/I\right) \underset{\Ci\left(\R\right)} \oinfty \Ci\left(\R\setminus \left\{0\right\}\right)\\
&\cong & \left(\A\underset{\Ci\left(\R\right)} \oinfty \Ci\left(\R\setminus \left\{0\right\}\right)\right)/\left(I^e\right)\\
&\cong & \A\left\{a^{-1}\right\}/I^e
\end{eqnarray*}
where $I^e$ is the extension of the ideal. The result now follows since any localization is a filtered colimit of localizations by a single element.
\end{proof}

\begin{definition}
Let $\A$ be a classical $\Ci$-algebra and let $I$ be an ideal of $\A.$ The ideal
$$I^{\wedge}=\left\{a \in \A\mspace{3mu}|\mspace{3mu} a_x \in I_x, \mspace{5mu} \forall \mspace{3mu} x:\A \to \R\right\}$$ is called the \textbf{germ-completion} of $I$. An ideal $I$ is called \textbf{germ-determined} if $I=I^{\wedge}.$ A classical $\Ci$-algebra is called germ-determined if its zero ideal is.
\end{definition}

\begin{definition}
An element $a$ in a classical $\Ci$-algebra $\A$ is \textbf{locally nilpotent} if for all $x:\A \to \R,$ $\germ_x\left(a\right)$ is nilpotent in $\A_x.$ Denote by $\Nil_{\Ci}\left(\A\right)$ the ideal of locally nilpotent elements.
\end{definition}

\begin{definition}
A classical $\Ci$-algebra is said to be \textbf{$\Ci$-reduced} if it has no non-zero locally nilpotent elements. 
\end{definition}

\begin{remark}
This is a different definition than \cite[Definition 2.1]{ringlocmr}. What we are defining should perhaps be called \emph{Archimidean reduced}
\end{remark}

\begin{remark}
Note that $\Nil_{\Ci}\left(\A\right)$ is germ-determined.
\end{remark}

\begin{lemma} \label{lem:nilradsame}
Let $\A$ be a finitely generated classical $\Ci$-algebra. Then $$\Nil_{\Ci}\left(\A\right)=\Nil\left(\A\right)^{\wedge},$$ i.e. an element $a \in \A$ is locally nilpotent if and only if it is in the germ-completion of the nilradical.
\end{lemma}

\begin{proof}
Note that, since $\Nil_{\Ci}\left(\A\right)$ is germ-determined, we can reduce to the case when $\A$ is germ-determined. Consider the nilradical $\Nil\left(\A\right).$ If it is not germ-determined, then it is not a complete-module in the sense of \cite[Definition 5.25]{joycesch}. However, we can identify $\Nil_{\Ci}\left(\A\right)$ as its completion --- that is global sections of $\mathbf{MSpec}\left(\Nil\left(\A\right)\right).$ Since $\mathbf{MSpec}\left(\Nil\left(\A\right)\right)$ is quasi-coherent, for all $b \in \A,$ 
$$\A\left\{b^{-1}\right\} \underset{\A} \otimes \Nil_{\Ci}\left(\A\right) \cong \Nil_{\Ci}\left(\A\left\{b^{-1}\right\}\right).$$ Let $U_z$ be the open subset of $\Speci\left(\A\right)$ corresponding to $\Spec\left(\A\left\{z^{-1}\right\}\right).$ It follows that for all $x:\A \to \R,$
\begin{eqnarray*}
\Nil_{\Ci}\left(\A_x\right) &\cong & \mathbf{MSpec}\left(\Nil_{\Ci}\left(\A\right)\right)_x\\
&\cong & \underset{z \notin \mathfrak{m}_x} \colim \mathbf{MSpec}\left(\Nil_{\Ci}\left(\A\right)\right)\left(U_z\right)\\
& \cong & \underset{z \notin \mathfrak{m}_x} \colim \A\left\{b^{-1}\right\} \underset{\A} \otimes \Nil_{\Ci}\left(\A\right)\\
& \cong & \left(\underset{z \notin \mathfrak{m}_x} \colim \A\left\{b^{-1}\right\}\right) \underset{\A} \otimes \Nil_{\Ci}\left(\A\right)\\
& \cong & \A_x \underset{\A} \otimes \Nil_{\Ci}\left(\A\right).
\end{eqnarray*}
Now, suppose that $a \in \Nil_{\Ci}\left(\A\right).$ Then it follows that $\mathbf{germ}_x\left(a\right) \in \Nil_{\Ci}\left(\A\right)_x$ for all $x:\A \to \R,$ and hence in $\Nil\left(\A\right)^{\wedge}.$ Conversely, suppose that $a \in \Nil\left(\A\right)^{\wedge}.$ Then, it also follows that for all $x,$ $\germ_x\left(a\right) \in \Nil_{\Ci}\left(\A_x\right),$ and hence $a$ is locally nilpotent.
\end{proof}

\begin{lemma} \label{lem:redfact}
Let $\A$ and $\B$ be $\Ci$-algebras, with $\B$ $\Ci$-reduced. Then for any map $$\varphi:\A \to \B,$$ $\Nil_{\Ci}\left(\A\right) \subseteq \ker\left(\varphi\right).$
\end{lemma}

\begin{proof}
Let $a \in \Nil_{\Ci}\left(\A\right).$ Let $x:\B \to \R.$ Consider the induced map
$$\varphi_x:\A_{x \circ \varphi} \to \B_x.$$ Then $\varphi_x\left(\germ_x\left(a\right)\right)$ is nilpotent in $\B_x.$ But $$\varphi_x\left(\germ_x\left(a\right)\right)=\varphi\left(a\right)_x.$$ So $\varphi\left(a\right)$ is locally nilpotent in $\B,$ and hence zero.
\end{proof}

\begin{lemma}
Let $\A$ be a finitely generated $\Ci$-algebra. Then $\A/\Nil_{\Ci}\left(\A\right)$ is $\Ci$-reduced.
\end{lemma}

\begin{proof}
Let $\pi:\A \to \A/\Nil_{\Ci}\left(\A\right),$ and let $\left[a\right]$ be a locally nilpotent element of $\A/\Nil_{\Ci}\left(\A\right),$ with $a \in \A.$ By Lemma \ref{lem:redfact}, it follows that for any $x:\A \to \R,$ there is a factorization $x=y \circ \pi,$ for $y:\A/\Nil_{\Ci}\left(\A\right) \to \R.$ Let $y$ be such a $\R$-point. Then $\germ_y\left(\left[a\right]\right)$ is nilpotent, so there exists an integer $n_y$ such that $\germ_y\left(\left[a\right]\right)^{n_y}=0$ in $\left(\A/\Nil_{\Ci}\left(\A\right)\right)_y.$ So $\germ_x\left(a^{n_y}\right)$ is in the kernel of $$\pi_x:\A_x \to \left(\A/\Nil_{\Ci}\left(\A\right)\right)_y$$ which is $\Nil_{\Ci}\left(\A\right)_x$ by Lemma \ref{lem:localq}. But this is the same as $\Nil_{\Ci}\left(\A_x\right)$ by the proof of Lemma \ref{lem:nilradsame}. Hence $\germ_x\left(a^{n_y}\right)$ is nilpotent in $\A_x.$ Since $x$ was arbitrary, we conclude that $a\in\Nil_{\Ci}\left(\A\right),$ and hence $\left[a\right]=0.$
\end{proof}

\begin{definition}
Let $\A$ be a classical $\Ci$-algebra. Its \textbf{reduction} is $\A_{red}:=\A/\Nil_{\Ci}\left(\A\right).$ If $\A$ is a classical $\SCi$-algebra, then its reduction is defined by $\A_{red}:=\left(\A_0\right)_{red}.$
\end{definition}

\begin{corollary}
The category of reduced finitely generated $\Ci$-algebras is reflective in finitely generated $\Ci$-algebras, with the left adjoint being $$\A \mapsto \A_{red}=\A/\Nil_{\Ci}\left(\A\right).$$
\end{corollary}

\section{Derived manifolds}\label{sec:derived_mfd}
\subsection{A universal property}
Morally, derived manifolds are objects that arise by iteratively taking derived    fibered products and retracts starting with ordinary manifolds and smooth maps. There have been many proposed models for what a derived manifold ought to be \cite{spivak,joyce,derivjustden,dg2}. The first was given by Spivak \cite{spivak}, and Joyce has a very well-developed theory of \emph{d-manifolds}, with important applications to symplectic topology. A common thread of all of these approaches is the prominent role played by $\Ci$-algebras. This is not a coincidence, as it turns out that derived manifolds and $\Ci$-algebras are intimately linked at the categorical level.

Whatever the $\i$-category $\mathsf{DMfd}$ is, it should receive a fully faithful functor $$i:\Mfd \hookrightarrow \mathsf{DMfd}$$ and $i$ should preserve transverse pullbacks and the terminal object. Moreover, $\mathsf{DMfd}$ should be closed under taking fibered products. One way to phrase this is to ask for $\mathsf{DMfd}$ to have finite limits. Finally, one would make sense of the idea that derived manifolds are completely determined by how they are built out of manifolds using fibered products. This leads to the following definition:

\begin{definition}
The $\i$-category  $\mathsf{DMfd}$ of derived manifolds is the (essentially unique) $\i$-category with finite limits, together with a functor $i:\mathsf{Mfd} \to \mathsf{DMfd}$ which preserves transverse pullbacks and the terminal object, such that for all $\i$-categories $\sC$ with finite limits, composition along $i$ induces an equivalence of $\i$-categories
$$\Fun^{\mathbf{lex}}\left(\mathsf{DMfd},\sC\right) \stackrel{\sim}{\longlongrightarrow} \Fun^\pitchfork\left(\Mfd,\sC\right)$$
between the $\i$-category of finite limit preserving functors from $\mathsf{DMfd}$ to $\sC$ and the $\i$-category of functors from $\Mfd$ to $\sC$ which preserves transverse pullbacks and the terminal object.
\end{definition}

\begin{remark}
It turns out one does not need to demand that $i$ be fully faithful, and that it is automatically satisfied.
\end{remark}

\subsection{An algebro-geometric description}

Of course, such a definition is not very tractable, however there is the following remarkable connection between derived manifolds and $\Ci$-algebras:

\begin{theorem} \cite[Theorem 5.3]{univ}
For all $\i$-categories $\sC$ with finite limits, composition with the canonical functor $\Ci \to \Mfd \to \mathsf{DMfd}$ induces an equivalence of $\i$-categories
$$\Fun^{\mathbf{lex}}\left(\mathsf{DMfd},\sC\right) \stackrel{\sim}{\longlongrightarrow} \Alg_{\Ci}\left(\sC\right).$$
\end{theorem}

In light of \ref{thm:finenv}, there is the following corollary, which provides a concrete model for derived manifolds:

\begin{corollary}
There is a canonical equivalence of $\i$-categories
$$\mathsf{DMfd} \simeq \left(\Alg_{\Ci}\left(\Spc\right)^{\mathbf{fp}}\right)^{op}$$ between the $\i$-category of derived manifolds and the opposite of the $\i$-category of finitely presented $\Ci$-algebras in $\Spc.$
\end{corollary}

We wish to develop a theory of derived geometry that also incorporates supermanifolds, as such, we want a supergeometric analogue of the above. Upon careful analysis of the proofs in \cite{univ}, such a result will follow essentially verbatim from the even case once two technical facts are established. The first has already been proven, namely Lemma \ref{lem:whit}. The second fact we need is that the the canonical morphism of algebraic theories $$\mathbf{SCom}_{\R}\to \SCi$$ is unramified. We postpone this to Theorem \ref{thm:unramified}. However, we record the result for derived supermanifolds here:

\begin{definition}\label{dfn:dsmfd}
The $\i$-category  $\mathsf{DSMfd}$ of derived supermanifolds is the (essentially unique) $\i$-category with finite limits, together with a functor $i:\mathsf{SMfd} \to \mathsf{DSMfd}$ which preserves transverse pullbacks and the terminal object, such that for all $\i$-categories $\sC$ with finite limits, composition along $i$ induces an equivalence of $\i$-categories
$$\Fun^{\mathbf{lex}}\left(\mathsf{DSMfd},\sC\right) \stackrel{\sim}{\longlongrightarrow} \Fun^\pitchfork\left(\SMfd,\sC\right)$$
between the $\i$-category of finite limit preserving functors from $\mathsf{DSMfd}$ to $\sC$ and the $\i$-category of functors from $\SMfd$ to $\sC$ which preserves transverse pullbacks and the terminal object.
\end{definition}

\begin{theorem}\label{thm:superuniv}
For all $\i$-categories $\sC$ with finite limits, composition with the canonical functor $\SCi \to \SMfd \to \mathsf{DSMfd}$ induces an equivalence of $\i$-categories
$$\Fun^{\mathbf{lex}}\left(\mathsf{DSMfd},\sC\right) \stackrel{\sim}{\longlongrightarrow} \Alg_{\SCi}\left(\sC\right).$$
\end{theorem}

\begin{corollary}\label{cor:supdman}
There is a canonical equivalence of $\i$-categories
$$\mathsf{DSMfd} \simeq \left(\Alg_{\SCi}\left(\Spc\right)^{\mathbf{fp}}\right)^{op}$$ between the $\i$-category of derived supermanifolds and the opposite of the $\i$-category of finitely presented $\SCi$-algebras in $\Spc.$
\end{corollary}

The upshot is the universal property of derived (super)manifolds leads us naturally to the study of derived algebraic geometry over $\Ci$-(super)algebras. Put another way, derived differential (super)geometry is precisely the derived algebraic geometry of $\Ci$-(super)algebras!

\subsection{A geometric description}
There is an alternative more down to earth description of the $\i$-category of derived manifolds. It turns out this $\i$-category arises naturally as a localization of the $1$-category of non-positively graded manifolds. More specifically, let $f:\cM \to \cN$ be a map of dg-manifolds (or dg-supermanifolds). We say
\begin{definition}\label{dfn:catfibdgman}
\begin{itemize}
\item[i)] $f$ is a \textbf{fibration} if it is a submersion of underlying graded (super)manifolds, or equivalently, if the induced $\End\left(\R^{0|1}\right)$-equivariant map $\mathfrak{M} \to \mathfrak{N}$ is a submersion of supermanifolds.
\item[ii)] $f$ is a \textbf{weak equivalence} if 
$$f^*:\Ci\left(\cN\right) \to \Ci\left(\cM\right)$$ is a weak equivalence of dg-$\SCi$-algebras, i.e. of underlying cochain complexes.
\end{itemize}
\end{definition}

Consider the canonical functor of $1$-categories
\begin{eqnarray*}
\Ci:\dgMan &\to& \left(\dgcni\right)^{op}\\
\M  &\mapsto& \Ci\left(\M\right)
\end{eqnarray*}

It sends weak equivalences to quasi-isomorphisms, hence there is an induced functor between $\i$-categories
$$\Ci:\dgMani \to \dgci^{op},$$
where $\dgMani$ is the $\i$-category obtained by formally inverting the weak equivalences of dg-manifolds, up to homotopy.
By \cite[Lemma 6.1]{dgmander}, the essential image lies in the homotopically finite generated $\Ci$-algebras, and hence by Corollary \ref{cor:supdman}, there an induced functor
$$\Psi:\dgMani \to \DMfd$$
from the $\i$-category of non-positively graded dg-manifolds to the $\i$-category of derived manifolds.

\begin{theorem}\label{thm:dgmander}\cite[Theorem 6.11]{dgmander}
The functor $\Psi:\dgMani \to \DMfd$ is an equivalence of $\i$-categories.
\end{theorem}

The extension of the proof of this fact to dg-supermanifolds is trivial and the details are left to the reader. We record it here:

There is a well-defined functor $$\Psi:\dgSMani \to \DSMfd$$
from the $\i$-category of non-positively graded dg-supermanifolds to the $\i$-category of derived manifolds.

\begin{theorem}\label{thm:dgsmander}
The functor $\Psi:\dgSMani \to \DSMfd$ is an equivalence of $\i$-categories.
\end{theorem}

Given this, much of this manuscript can be rephrased using the language of dg-manifolds. We hope to do so in the future.

\section{Homotopical $\Ci$-algebras}\label{sec:homotopical}

\begin{definition}
We denote the $\i$-category $\Alg_{\SCi}\left(\Spc\right)$ by $\dgc$. 
\end{definition}

By Theorem \ref{theorem:bergner}, any $\SCi$-algebra in the $\i$-category $\Spc$ of spaces can be modeled by a simplicial $\SCi$-algebra in sets. Since any $\SCi$-algebra has an underlying $
\Z_2$-graded abelian group, namely its even and odd additive groups, such a simplicial algebra has an underlying simplicial abelian group of even and odd elements respectively. It follows each of the respective underlying simplicial sets are Kan complexes. There is a model structure on the category of $\SCi$-algebras in $\Set^{\Delta^{op}}$ such that:
\begin{itemize}
\item[i)] $f:X \to Y$ is a \textbf{weak equivalence} if both
$$f\left(\R\right):X\left(\R\right) \to Y\left(\R\right)$$ and
$$f\left(\R^{0|1}\right):X\left(\R^{0|1}\right) \to Y\left(\R^{0|1}\right)$$ are weak homotopy equivalences of simplicial sets and
\item[ii)] $f:X \to Y$ is a \textbf{fibration} if both maps above are Kan fibrations of simplicial sets.
\end{itemize}
The $\i$-category associated to this model category is equivalent to $\Alg_{\SCi}\left(\Spc\right).$ Having this model structure at hand will sometimes be useful for concrete arguments and calculations.

\subsection{Differential Graded $\Ci$-algebras}
Let a $\Ci$-algebra $\A$ (in $\Spc$) be represented by a simplicial $\Ci$-algebra $X.$ As discussed, $X$ has the underlying structure of a simplicial abelian group. The classical Dold-Kan correspondence associates to this simplicial abelian group a cochain complex. In fact, this abelian group is an $\R$-module, so this becomes a cochain complex over $\R.$ One may wonder if one can transfer the $\Ci$-algebra structure to this cochain complex.
\begin{definition}
A \textbf{differential graded $\Ci$-algebra} (dg-$\Ci$-algebra), is a commutative differential graded $\R$-algebra $\left(\cA^{\bullet},d\right),$ together with the additional structure of a lift of the induced commutative $\R$-algebra structure on $\cA_0$ to the structure of a $\Ci$-algebra. A morphism $f:\left(\cA,d\right)\to \left(\cA',d'\right)$ between two such algebras is a morphism of differential graded $\R$-algebras such that the morphism $f_0:\cA_0 \to \cA'_0$ is a morphism of $\Ci$-algebras.

Similarly, a \textbf{differential graded $\Ci$-superalgebra} (dg-$\Ci$-superalgebra), is a commutative differential graded $\R$-superalgebra $\left(\cA^{\bullet},d\right),$ (here each $\cA_i$ is additionally $\Z_2$-graded) together with the additional structure of a lift of the induced supercommutative $\R$-algebra structure on $\cA_0$ to the structure of a $\SCi$-algebra. A morphism $f:\left(\cA,d\right)\to \left(\cA',d'\right)$ between two such algebras is a morphism of differential graded $\R$-algebras such that the morphism $f_0:\cA_0 \to \cA'_0$ is a morphism of $\SCi$-algebras.
\end{definition}

\begin{theorem}\cite[Theorem 6.5]{dg2} \label{thm:dgmodel}
There exists a cofibrantly generated, almost simplicial, model category structure on the category of dg-$\Ci$-algebras, unique with the property that $$f:\left(\cA,d\right)\to \left(\cA',d'\right)$$ is a weak equivalence (respectively fibration) if and only if the induced map of underlying cochain complexes is, with respect to the projective model structure on $\Ch_{\R},$ and similarly for the category of dg-$\Ci$-superalgebras, but with the category $\Ch^\bullet_{\R}$ replaced with that of $\Ch^\bullet\left(\mathbf{Vect_{\R}^{\Z_2}}\right)$--- the category of cochain complexes of $\Z_2$-graded real vector spaces. Moreover, there are induced model structures on non-positively (and non-negatively) graded $\Ci$-(super)algebras.
\end{theorem}

\begin{definition}
Denote by $\mathbf{dg}\Ci\Alg,$ $\mathbf{dg}\Ci\Alg_{\le 0},$ $\mathbf{dg}\Ci\Alg_{\ge 0},$ the $\i$-categories associated to the above model categories of appropriately bounded $dg$-$\Ci$-algebras. Similarly, denote by $\mathbf{dg}\SCi\Alg,$ $\mathbf{dg}\SCi\Alg_{\le 0},$ $\mathbf{dg}\SCi\Alg_{\ge 0},$ the analogously defined $\i$-categories of differential graded $\SCi$-algebras.
\end{definition}


\begin{remark}
A dg-$\Ci$-superalgbera $\left(\A,d\right)$ has both a $\Z$-grading and a $\Z_2$-grading, and they need not be compatible!
\end{remark}


Notice that for any dg-$\Ci$-superalgbera $\left(\A^\bullet,d\right),$ there is an underlying cochain complex of super vector spaces $\A^\bullet.$ This produces a forgetful functor
$$\mathbf{dg}\Ci\Alg \to \mathbf{dg}\mathsf{SVect},$$ with a left adjoint $\Sym_{\Ci}.$

Recall that if $V^\bullet$ is a $\Z$-graded vector space, and for each $i,$ we have a given basis $\left(e_i^\alpha\right),$ $\Sym\left(V^\bullet\right)$ can be described as the free graded commutative algebra with a generator of degree $i$ for each basis element $e_i^\alpha.$ If each $V^i$ is additionally $\Z_2$-graded, each of these generators is appropriately even or odd. If $\left(V^\bullet,d\right)$ is then a cochain complex, we can describe each linear map
$$d_i:V_i \to V_{i+1}$$
in terms of the basis elements as a matrix $M^i,$ where, using Einstein's summation convention,
$$d_i\left(e_i^\alpha\right)=M^i_{\alpha\beta}e^\beta_{i+1}.$$ $\Sym\left(\left(V^\bullet,d\right)\right)$ can then be described as $\Sym\left(V^\bullet\right)$ equipped with the differential
\begin{equation}\label{eq:differential}
D:=M^i_{\alpha\beta}e^{\beta}_{i+1}\frac{\partial}{\partial e^{\alpha}_i}.
\end{equation}

Similarly, $\Sym_{\Ci}\left(V^\bullet\right)$ is the free dg-$\SCi$-algebra on the variables $e^\alpha_i.$ If $V^\bullet=V$ is concentrated in degree $0,$ and has superdimension $n|m,$ then 
$$\Sym_{\Ci}\left(V\right)\cong \Ci\left(\R^{n|m}\right).$$ On the other extreme, if $V^0=0,$ then $$\Sym_{\Ci}\left(V^\bullet\right)\simeq \Sym\left(V^\bullet\right).$$ Combining these two facts implies that in general
$$\Sym_{\Ci}\left(V^\bullet\right) \simeq \Sym\left(\underset{i \ne 0} \bigoplus V^i\left[-i\right]\right) \underset{S\left(V^0\right)} \otimes \Ci\left(V_0\right).$$

For example, if $W^0=\R^{n|m},$ $W^{-1}=\R^{k|l},$ $W^{-2}=\R^{a|b},$ and all other components zero, then
$$\Sym_{\Ci}\left(W^\bullet\right)\cong \Ci\left(\R^{n|m}\right) \underset{\R} \otimes \R\left[x_1,x_2,\ldots,x_k,y_1,y_2,\ldots y_l,z_1,\ldots,z_a,w_1,\ldots w_b\right],$$ with each $x_i$ in cohomological degree $-1$ and even in its $\Z_2$-grading, each $y_i$ in cohomological degree $-1$ and odd in its $\Z_2$-grading, each $z_i$ in cohomological degree $-2$ and even in its $\Z_2$-grading, and each $w_i$ in cohomological degree $-2$ and odd in its $\Z_2$-grading. $\Sym_{\Ci}\left(W^\bullet\right)$ is the free $\Z$-graded $\SCi$-algebra on $n$ even generators of degree $0,$ $m$ odd generators of degree $0,$ $k$ even generators of degree $-1$ etc.

For a cochain complex $\left(V^\bullet,d\right),$ the underlying graded commutative algebra of $\Sym_{\Ci}\left(\left(V^\bullet,d\right)\right)$ is $\Sym_{\Ci}\left(V^\bullet\right),$ and the differential is given by the same formula (\ref{eq:differential}).

Finally, we note that in the projective model structure, every cochain complex is cofibrant, so the left derived functor of $\Sym_{\Ci}$ is itself.

\begin{definition}
A morphism $f:\left(\A,d\right) \to \left(\B,d\right)$ between  dg-$\SCi$-algebras is \textbf{quasi-free} if there exists a $\Z$-graded super vector space $W^\bullet$ such that $f$ is isomorphic to the inclusion of $\A$ into $\A \underset{\R}\otimes \Sym_{\Ci}\left(W^\bullet\right),$ after forgetting the differentials.
\end{definition}

\begin{proposition}\cite{dg2}
A morphism $f:\left(\A,d\right) \to \left(\B,d\right)$ between non-positively graded differential $\SCi$-algebras is a cofibration if and only if it is a retract of a quasi-free extension.
\end{proposition}

Cofibrations between non-connective dg-$\SCi$-algebras are slightly more subtle:

\begin{proposition}\cite{dg2}
A quasi-free extension $f:\left(\A,d\right) \to \left(\B,d\right)$ dg-$\SCi$-algebras is a cofibration if and only if it admits an exhaustive filtration
$$\A=\cB_{-1} \subset \cB_{0} \subset \cB_1 \subset \cdots \cB_k \subset \cdots$$
by sub-algebras such that for all $k \ge 0,$ 
\begin{itemize}
\item[1.] $\cB_k$ is obtained from $\cB_{k-1}$ by adjoining a set of generators $\{x_{k,\alpha}\}_{\alpha \in I_k},$ i.e. the filtration is by quasi-free extensions, and
\item[2.] $d\left(x_{k,\alpha}\right) \in \cB_{k-1}$ for all $\alpha \in I_k.$
\end{itemize}
\end{proposition}

This implies that the cofibrant objects in the model category $\dgchmod$ are retracts of quasi-free dg-$\Ci$-superalgebras by \cite[Corollary 6.20]{dg2}. In particular, each algebra $\Ci\left(\R^{p|q}\right),$ regarded as a dg-$\SCi$-algebra is cofibrant, and the collection of these algebras is closed under coproducts, and hence is also closed under derived $\Ci$-tensor products, as they coincide in this case. It follows that the canonical functor
\begin{eqnarray*}
\SCi & \to & \left(\mathbf{dg}\Ci\Alg_{\le 0}\right)^{op}\\
\R^{p|q} &\mapsto & \Ci\left(\R^{p|q}\right)
\end{eqnarray*}
preserves finite products. Hence, by Theorem \ref{theorem:5.5.8.15}, we deduce that there is a unique colimit preserving functor $$N_{\Ci}:\dgc \to \mathbf{dg}\Ci\Alg_{\le 0}$$ which sends each algebra of the form $\Ci\left(\R^{p|q}\right)$ to itself. By the adjoint functor theorem, it follows there exists a right adjoint $\Gamma^{\Ci}$. In fact, it follows from the Yoneda lemma that  $\Gamma^{\Ci}$ sends a dg-$\SCi$-algebra $\A$ to
\begin{eqnarray*}
\Gamma^{\Ci}\left(\A\right):\SCi &\to & \Spc\\
\R^{p|q} &\mapsto & \Map_{\mathbf{dg}\Ci\Alg_{\le 0}}\left(\Ci\left(\R^{p|q}\right),\A\right).
\end{eqnarray*}

\begin{theorem}[The $\Ci$ Dold-Kan Correspondence] \cite[Corollary 2.2.10]{Nuiten}
The above adjunction
is an adjoint equivalence of $\i$-categories.
\end{theorem}

The above theorem is a major convenience as it allows us to use techniques from homological algebra in our study of derived differential geometry.

\subsection{Localizations of homotopical $\SCi$-algebras}

\begin{definition}
Let $\A$ be a homotopical $\Ci$-algebra and let $a\in \pi_0\left( \A\right)$. We say that a map $f:\A\rightarrow \B$ such that $f\left( a\right)\in \pi_0\left( B\right)$ is invertible is a \textbf{localization of $A$ with respect to $a$} if for each $\cC \in \cialgsp$, the map $\Map_{\cialgsp}\left( \B,\cC\right)\rightarrow\Map_{\cialgsp}\left(\A,\cC\right)$ given by composition with $f$ induces an equivalence
\[\Map_{\cialgsp}\left(\B,\cC\right)\overset{\simeq}{\longrightarrow} \Map^0_{\cialgsp}\left( \A,\cC\right),\]
where $\Map^0_{\cialgsp}\left( \A,\cC\right)$ is the union of those connected components of $\Map_{\cialgsp}\left(\A,\cC\right)$ spanned by those maps $g$ such that $g\left( a\right)$ is invertible in $\pi_0\left( \cC\right)$.
\end{definition}

In the case of a classical $\Ci$-algebra (in sets), $\A,$ and some $a\in \A$, the above definition reduces to the usual $\Ci$ localization $\A\left[1/a\right]$ given up to equivalence by the pushout
\begin{equation*}
\begin{tikzcd}
\Ci\left( \R\right)\ar[r,"q_a"]\ar[d]& \A\ar[d]\\
\Ci\left( \R\setminus \{0\}\right)\ar[r] & \A\left[1/a\right]
\end{tikzcd}    
\end{equation*}
of $C^{\infty}$-rings. The localization of homotopical $C^{\infty}$-ring admits a similar characterization, for which we will need the following definition.
\begin{definition}\label{strongmap}
\begin{enumerate}
    \item A map $f:\A\rightarrow \B$ in $\Alg_{\ComR}\left(\Spc\right)$ is \textbf{strong} (in the sense of \cite[Definition 2.2.2.1]{ToeVez}) if the natural map
\[ \pi_n\left( \A\right)\underset{\pi_0\left( \A\right)}\otimes\pi_0\left( \B\right)\rightarrow \pi_n\left( \B\right)\]
is an isomorphism for all $n\geq 0$.
\item A map $f:\A\rightarrow \B$ of homotopical $C^{\infty}$-algebras is \textbf{strong} if $f^{\sharp}:\A^{\sharp}\rightarrow \B^{\sharp}$ is strong.
\end{enumerate}
\end{definition}
\begin{proposition}\label{localization}\label{prop:locz} \cite[Proposition 4.16]{univ}
Let $\A$ be a homotopical $\Ci$-algebra and let $a\in \pi_0\left( \A\right)$, and let $f:A\rightarrow B$ a map of simplicial $C^{\infty}$-rings. The following are equivalent:
\begin{enumerate}
    \item The map $f:A\rightarrow B$ exhibits $B$ as a localization with respect to $a$.
    \item For every $n\geq 0$, the induced map 
    \[\pi_n\left( \A^{\sharp}\right)\underset{\pi_0\left( A^{alg}\right)}\otimes\left( \pi_0\left( \A\right)\left[1/a\right]\right)^{alg}\rightarrow \pi_n\left( B^{alg}\right) \]
    is an equivalence; that is, $f$ is strong and the map of $C^{\infty}$-schemes corresponding to $\pi_0\left( A\right)\rightarrow \pi_0\left( B\right)$ is an open immersion. 
    \item $B$ fits into a pushout diagram
\begin{equation*}
\begin{tikzcd}
\Ci\left( \R\right)\ar[r,"q_a"]\ar[d]& A\ar[d,"f"]\\
\Ci\left( \R\setminus \{0\}\right)\ar[r] & B
\end{tikzcd}    
\end{equation*}
where $q_a$ is the unique up to homotopy map associated to $a\in \pi_0\left( A\right)$ (note that as a consequence, localizations always exist).
\end{enumerate}
\end{proposition}

We denote the $\Ci$-localization of a homotopical $\Ci$-ring $\A$ by an element $a \in \pi_0 \A$ by $\A\left[1/a\right].$

\subsection{$\SCi$-Koszul complexes}

There is one subtle notion needed to work with homological algebra in the $\Z_2$-graded setting. For example, in the free superalgebra on odd generators $\eta_1,\eta_2,\ldots \eta_n$, the sequence of generators is not a regular sequence in the usual sense, since it consists of zero divisors. However, it is an \emph{odd regular sequence} in the sense of \cite{Schmitt}:

\begin{definition}
An even element $q$ of a supercommutative ring $\A$ is \textbf{regular} if it is not a zero-divisor.

An odd element $p$ of a supercommutative ring $\A$ is \textbf{regular} if the sequence $$\A \stackrel{\cdot p}{\longrightarrow} \A \stackrel{\cdot p}{\longrightarrow} \A$$ is exact.

A sequence $\left(a_1,a_2,...,a_n\right)$ of (even or odd) elements in a supercommutative ring $\A$ will be called a \textbf{regular sequence}, if  the image of $a_{i+1}$ is a regular element of $\A/\left(a_1,\ldots,a_i\right)$ for all $i.$

If $\A$ is a $\Z$-graded commutative superalgebra, a sequence $\left(a_1,a_2,...,a_n\right)$ of elements (possibly of different degrees), is \textbf{regular} if it is a regular sequence in the underlying $\Z_2$-graded commutative algebra.
\end{definition}

\begin{remark}
Recall that if $a \in \A$ is of bidegree $||a||=\left(n,\varepsilon\right),$ it is of $\Z_2$-degree $n+\varepsilon$ in the underlying $\Z_2$-graded algebra.
\end{remark}

\begin{definition}
Let $\A$ be a dg-$\SCi$-algebra. Let $a_1,\ldots, a_k$ be a set of elements of $\A$ of bidegrees $\left(n_1,\varepsilon_1\right),\ldots,\left(n_k,\varepsilon_k\right)$ such that each $a_i$ is closed, i.e. $da_i=0.$ The \textbf{Koszul algebra} of $\A$ associated to this set is the ($1$-categorical) pushout
$$\xymatrix{\left(\Ci\{x_1,\ldots,x_k\},||x_i||=||a_i||,0\right) \ar[d]_-{\left(a_1,\ldots,a_k\right)} \ar[r] & \left(\Ci\{x_1,\ldots,x_k,\xi_1,\ldots,\xi_k\},d=x_i \frac{\partial}{\partial \xi_i},||\xi_i||=\left(n_i -1,\varepsilon_i\right)\right) \ar[d]\\
\A \ar[r] & K\left(\A,a_1,\ldots, a_k\right),}$$
in $\dgchmod.$
If $\A=\left(\A,d\right),$ then the underlying $\Z$-graded commutative superalgebra of $K\left(\A,a_1,\ldots, a_k\right)$ is
$K\left(\A,a_1,\ldots, a_k\right) \cong \A\{\xi_1,\ldots,\xi_k\},$ the free $\A$-algebra (in $\SCi$-algebras) on the generators $\xi_1,\ldots,\xi_k.$ Since all of these generators are strictly of negative degrees, we have that
$$\A\{\xi_1,\ldots,\xi_k\} \cong \A \underset{\R} \otimes \mathbb{R}\left[\xi_{1},\ldots,\xi_{k}\right],$$ where $\mathbb{R}\left[\xi_{1},\ldots,\xi_{k}\right]$ is the free commutative $\R$-algebra on the (graded) generators $\xi_{1},\ldots,\xi_k.$
Using Einstein summation conventions, the differential takes the form
$$D=D_{\A} + a^i \frac{\partial}{\partial \xi_i}.$$
\end{definition}




\begin{remark}
The Koszul algebra of $a_1,\ldots, a_k$ can be constructed inductively, so that $$K\left(\A,a_1,\ldots,a_n\right)=K\left(K\left(\A,a_1\right),a_2,\ldots,a_n\right)=\ldots= K\left(K\ldots\left(K\left(\A,a_1\right),a_2\right),\ldots,a_n\right).$$
\end{remark}

\begin{lemma}\label{lem:Kosz1}
Suppose that $\A$ is a dg-$\Ci$-algebra with zero differential, and let $a$ be a regular element of $\A$ of any degree. Then there is a quasi-isomorphism of dg-$\Ci$-algebras
$$K\left(\A,a\right) \stackrel{\sim}{\longrightarrow} \A/\left(a\right),$$
where $\left(a\right)$ is the homogeneous ideal generated by $a.$
\end{lemma}

\begin{proof}
\underline{Case 1}: $|a|$ is even. Then we have that $a$ is not a zero-divisor. Let $||a||=\left(i,\varepsilon\right).$ Let $\xi_{i-1}$ denote the added generator of bidegree $\left(i-1,\epsilon\right).$ Then $\xi_{i-1}$ is fermionic so squares to zero. Therefore, as a $\Z$-graded super vector space, we have
$$K\left(\A,a\right) \cong \A^\bullet \oplus \xi_{i-1} \cdot \A^{\bullet}\left[1-i\right].$$ We have
$$d_{k}:\A^{k} \oplus \xi_{i-1} \cdot  \A^{k-i+1} \to \A^{k+1} \oplus   \xi_{i-1} \cdot \A^{k-i}$$ is zero for $k=0,\ldots,i,$ and for $k <i,$
$$d_k\left(b_{k}+\xi_{i-1}c_{k-i+1}\right)=a \cdot c_{k-i+1}.$$ Since $a$ is not a zero-divisor, we have $$d_k\left(b_{k}+\xi_{i-1}c_{k-i+1}\right) = 0 \iff c_{k-1+i}= 0.$$ This means that $\ker\left(d_k\right)=\A^k.$ It's also clear that $\mathbf{Im}\left(d_{k-1}\right)=\left(a\right)_{k},$ that is, the $k^{th}$ component of the principal ideal $\left(a\right).$ It follows that the canonical map $K\left(\A,a\right) \to \A/\left(a\right)$ is a quasi-isomorphism.

\underline{Case 2}: $||a||$ is odd. Then we have that $a \cdot x = 0$ iff $x=a\cdot b$ for some $b.$ So we have the underlying graded algebra of $K\left(\A,a\right)$ is the polynomial algebra $\A\left[\xi_{i-1}\right]$ on a generator of \emph{bosonic} degree $\left(i-1,\varepsilon\right).$ An element of degree $k$ is polynomial of the form 
$x=c_k+\underset{j \ge 1} \sum c_{k-j\cdot\left(i-1\right)}\xi^j_{i-1},$ with each $c_{r} \in \A^{r}.$
So $$dx=\underset{j \ge 1} \sum c_{k-j\cdot\left(i-1\right)}\cdot a\cdot j\xi^{j-1}_{i-1}.$$ This means that $dx=0$ if and only if for all $j \ge 1,$ 
 $$c_{k-j\cdot\left(i-1\right)}\cdot a=0,$$
 but since $a$ is regular, this means that each $c_{k-j\cdot\left(i-1\right)}=0$ for $j \ge 1,$ i.e. $x=c_k,$ so $\ker\left(d_k\right)=\A_k.$ Again, we clearly have that $\mathbf{Im}\left(d_{k-1}\right)=\left(a\right)_{k},$ so it follows once more that the canonical map $K\left(\A,a\right) \to \A/\left(a\right)$ is a quasi-isomorphism.
 \end{proof}
 
\begin{proposition}\label{prop:kosz}
Suppose that $\A$ is a dg-$\Ci$-algebra with zero differential, and let $\left(a_1,a_2,\ldots a_n\right)$ be a regular sequence of elements of $\A.$ Then
$$K\left(\A,a_i\right) \stackrel{\sim}{\longrightarrow} \A/\left(a_1,a_2,\ldots,a_n\right).$$
\end{proposition}

\begin{proof}
We will prove this by induction on the number of elements in the sequence. For simplicity, we will assume we are in the $\Z$-graded rather than $\Z \times \Z_2$-graded setting. We have already established this for $n=1.$ Suppose now that it is true for all sequences of $n-1$ regular elements. We wish to show it also holds for those of length $n.$

Let $\left(a_1,a_2,\ldots a_n\right)$ be a regular sequence of elements of $\A.$ Firstly, notice that since we are not allowing elements of degree $1,$ the Koszul construction can never add generators of degree $0.$ Therefore, for any dg-$\Ci$-algebra $\cB$ one can identify the underlying cdga of the Koszul algebra of an element $b_i \in \cB$ of degree $i$ as the pushout
$$\xymatrix{\R\left[x_i\right] \ar[r] \ar[d]_-{b_i} & \R\left[x_i,\xi_{i-1}\right] \ar[d]\\
\cB \ar[r] & K\left(\cB,b_i\right).}$$
Applying this to the case that $\cB=K\left(\A,a_1,\ldots,a_{n-1}\right),$ we can consider the stack of pushout diagrams
$$\xymatrix{\R\left[x_i\right] \ar[r] \ar[d]_-{a_n} & \R\left[x_i,\xi_{i-1}\right] \ar[d]\\
K\left(\A,a_1,\ldots,a_{n-1}\right) \ar[r] \ar[d]_-{ \rotatebox[origin=c]{90}{$\sim$}} & K\left(\A/\left(\A,a_1,\ldots,a_{n-1}\right),b_i\right) \ar[d]\\
\A/\left(\A,a_1,\ldots,a_{n-1}\right) \ar[r] & K\left(\A/\left(a_1,\ldots,a_{n-1}\right),a_{n}\right),}$$
where $$K\left(A,a_1,\ldots,a_{n-1}\right) \to \A/\left(\A,a_1,\ldots,a_{n-1}\right)$$ is a quasi-isomorphism. Since the projective model structure on non-positively graded differential algebras over $\R$ is left proper, and the top horizontal map is a cofibration, it follows that the map $$K\left(a_1,\ldots,a_n\right) \to K\left(\A/\left(a_1,\ldots,a_{n-1}\right),a_{n}\right)$$ is also a quasi-isomorphism. Since the outer square is also a pushout diagram, and $a_n$ is regular in $\A/\left(a_1,\ldots,a_{n-1}\right),$ Lemma \ref{lem:Kosz1} implies that $K\left(\A/\left(a_1,\ldots,a_{n-1}\right),a_{n}\right) \to \A/\left(a_1,\ldots,a_{n}\right)$ is a quasi-isomorphism. This completes the proof.
\end{proof}

\begin{lemma}\label{lem:suplocsubm}
Suppose that $\cM$ is a supermanifold such that $\Ci\left(\cM\right)$ is cofibrant in $\dgchmod,$ e.g. $\R^{n|m}$ for some $n$ and $m.$ Let $\varphi=\left(f,\alpha\right):\cM \to \R^{p|q}$ be a submersion locally around the origin. Then the following diagram is a homotopy pushout in $\dgchmod$:
$$\xymatrix{\Ci\left(\R^{p|q}\right) \ar[r]^-{\varphi^\ast} \ar[d]_-{ev_0} & 
\Ci\left(\cM\right) \ar[d]^-{i^\ast}\\
\R \ar[r] & \Ci\left(\varphi^{-1}\left(0\right)\right),}$$
where $i:\varphi^{-1}\left(0\right) \hookrightarrow \cM$ is in the inclusion of the closed sub-supermanifold of zeros of $\varphi.$
\end{lemma}

\begin{proof}
The function $f$ corresponds, via its components, to $p$ smooth functions $$f_1,f_2,\ldots f_p:\cM \to \R$$ and $q$ smooth functions $$g_1,\ldots,q_q:\cM \to \R^{0|1}.$$ Hence we have $f_1,\ldots f_p \in \Ci\left(\cM\right)_0$ and $g_1,\ldots,g_q \in \Ci\left(\cM\right)_1.$ Consider the super Koszul complex $K\left(\Ci\left(\cM\right),f_1,\ldots,f_p,g_1,\ldots,g_q\right),$ which is a differential graded commutative superalgebra over $\R.$ As a $\Z$-graded superalgebra it is
$$K\left(\Ci\left(\cM\right),f_1,\ldots,f_p,g_1,\ldots,g_q\right)=\Ci\left(\cM\right)\left[\theta_1,\theta_2,\ldots,\eta_1,\eta_2,\ldots,\eta_q\right],$$ where each $\theta_i$ is of degree $\left(0,-1\right)$ and each $\eta_j$ is of degree $\left(1,-1\right)$ with respect to the $\Z_2 \times \Z$ grading. The differential is given by
$$d=\sum_{i=1}^p f_i \frac{\partial}{\partial \theta_i} + \sum_{j=1}^{q} g_j \frac{\partial}{\partial \eta_j}.$$ In particular, we see that $K\left(\Ci\left(\cM\right),f_1,\ldots,f_p,g_1,\ldots,g_q\right)_0=\Ci\left(\cM\right)$ which is a $\SCi$-algebra, so the super Koszul algebra is a dg-$\SCi$-algebra. 

Consider for a moment the special case $\cM=\R^{p|q}$ and $\varphi=id.$ Choose coordinates $\left(y_1,y_2,\ldots,y_p|\beta_1,\beta_2,\ldots,\beta_q\right).$ Then these coordinates, regarded as elements of $\Ci\left(\cM\right)=\Ci\left(\R^{p|q}\right)$ form a regular sequence. It follows that the super Koszul complex in this case is acyclic. Thus the we have an equivalence of homotopy pushouts
$$\R\!\!\!\!\!\! \underset{\Ci\left(\R^{p|q}\right)} \oinfty^{\!\!\!\!\!\!\!\!\mathbb{L}} \mspace{2mu}\Ci\left(\cM\right) \simeq K\left(\Ci\left(\R^{p|q}\right),y_1,\ldots,y_p,\beta_1,\ldots,\beta_q\right)\!\!\!\!\!\! \underset{\Ci\left(\R^{p|q}\right)} \oinfty^{\!\!\!\!\!\!\!\!\mathbb{L}}\mspace{2mu} \Ci\left(\cM\right).$$ Moreover, the natural homomorphism $$\Ci\left(\R^{p|q}\right) \to K\left(\Ci\left(\cM\right),y_1,\ldots,y_p,\beta_1,\ldots,\beta_q\right)$$ is a quasi-free extension, hence a cofibration.

Let us now return to the case of a general $\varphi.$ Since all dg-$\Ci$-algebras in the diagram
$$\xymatrix{\Ci\left(\R^{p|q}\right) \ar[r]^-{\varphi^\ast} \ar[d] & 
\Ci\left(\cM\right) \\
K\left(\Ci\left(\R^{p|q}\right),y_1,\ldots,y_p,\beta_1,\ldots,\beta_q\right) & }$$ 
are cofibrant, and the vertical arrow above is a cofibration, the ordinary pushout computes the homotopy pushout. But the ordinary pushout is precisely $K\left(\Ci\left(\cM\right),f_1,\ldots,f_p,g_1,\ldots,g_q\right),$ whereas for the homotopy pushout we can replace $K\left(\Ci\left(\R^{p|q}\right),y_1,\ldots,y_p,\beta_1,\ldots,\beta_q\right)$ with $\R$ since it is acyclic. So the desired homotopy pushout is given by the super Koszul complex $K\left(\Ci\left(\cM\right),f_1,\ldots,f_p,g_1,\ldots,g_q\right).$ Therefore, it will suffice to prove that $K\left(\Ci\left(\cM\right),f_1,\ldots,f_p,g_1,\ldots,g_q\right)$ is quasi-isomorphic to $\Ci\left(\varphi^{-1}\left(0\right)\right).$

Suppose that as a supermanifold we have $\cM=\left(M,\cO_{\cM}\right).$ Define a $1$-categorical sheaf $\O_{\varphi}$ on $M$ with values in $\dgchmod$ by 
$$\cO_{\varphi}\left(U\right):=\cO_{\cM}\left(U\right) \underset{\Ci\left(\cM\right)} \otimes K\left(\Ci\left(\cM\right),f_1,\ldots,f_p,g_1,\ldots,g_q\right),$$ where the tensor product is the underived tensor product of commutative differential graded superalgebras over $\R.$ In concrete terms this is adding graded coordinates $\theta_1,\ldots,\theta_p,\eta_1,\ldots,\eta_q$ to $\cM,$ and equipping the result with a vector field $\Z$-degree $+1.$ (This gives the structure of a differential graded supermanifold). By Batchelor's theorem, $\cO_\cM$ is given by sections of a vector bundle over $M,$ and hence $\cO_{\cM}$ and $\cO_\varphi$ are soft.

We claim that $\cO_\varphi$ is a homotopy sheaf of dg-$\SCi$-algebras. The model structure on dg-$\SCi$-algebras is produced via transfer from the projective model structure on non-positively graded cochain complexes on super vector spaces, and hence homotopy limits can be computed as bare cochain complexes. Therefore, it suffices to prove that $\cO_\varphi$ is a homotopy sheaf of cochain complexes (of super vector spaces). Let $\cV:=\left(V_\alpha\right)_\alpha$ be an open cover of an open subset $U$ of $M.$ Let $\mbox{\v{C}}_{\cV}:\Delta^{op} \to \mathbf{Op}\left(M\right)$ be the \v{C}ech nerve of the cover, regarded as a simplicial diagram in the poset of open subsets of $M.$ The homotopy sheaf condition demands the homotopy limit of the natural cosimplicial diagram $\cO_{\varphi} \circ \mbox{\v{C}}_{\cV}$ of cochain complexes agrees up to quasi-isomorphism with $\cO_{\varphi}\left(U\right).$ This homotopy limit can be computed as the totalization of the  natural double complex arising by applying cosimplicial Dold-Kan to the cosimplicial direction of $\cO_{\varphi} \circ \mbox{\v{C}}_{\cV}.$ The spectral sequence for the double complex, which computes the cohomology of the totalization, degenerates on the $E_2$ page since the sheaves of complexes involved are all soft. The result readily follows at this point from the ($1$-categorical) sheaf condition of $\cO_{\varphi}.$

Since being a submersion is a local property, we can find an open sub-supermanifold $\cU$ of the closed sub-supermanifold $\cN:=\varphi^{-1}\left(0\right)$ such that $\varphi|_{\cU}$ is globally a submersion. By the local form of submersions, we can find local coordinate charts $\left(\mathcal{U}_{\alpha} \subseteq \mathcal{U}\right)_\alpha,$ $\mathcal{U}_\alpha \cong \R^{n+p|m+q}$ with coordinates $\left(x_1,\ldots,x_n,y_1,\ldots y_p|\xi_1,\ldots,\xi_m,\beta_1,\ldots,\beta_q\right),$ such that in these coordinates we have
$$\varphi\left(x_1,\ldots,x_n,y_1,\ldots y_p|\xi_1,\ldots,\xi_m,\beta_1,\ldots,\beta_q\right)=\left(y_1,\ldots,y_p|\beta_1,\ldots,\beta_q\right).$$ In particular, $\mathcal{U_\alpha}\cap\varphi^{-1}\left(0\right),$ is exactly the space $$\{\left(x_1,\ldots,x_n,y_1,\ldots y_p|\xi_1,\ldots,\xi_m,\beta_1,\ldots,\beta_q\right)\mspace{3mu}|\mspace{3mu}\left(y_1,\ldots,y_p|\beta_1,\ldots,\beta_p\right)=0\}.$$ Under this change of coordinates, we have that $\cO_{\varphi}\left(\mathcal{U}_{\alpha}\right)$ is equivalent to the Koszul complex $$K\left(\Ci\{y_1,\ldots,y_p|\beta_1,\ldots,\beta_q\},y_1,y_2,\ldots,y_p,\beta_1,\beta_2,\ldots, y_q\right).$$ However, $y_1\ldots,y_p,\beta_1,\ldots, y_q$ is a regular sequence, hence
\begin{eqnarray*}
K\left(\Ci\{y_1,\ldots,y_p|\beta_1,\ldots,\beta_q\},y_1,y_2,\ldots,y_p,\beta_1,\beta_2,\ldots, y_q\right) &\simeq& \Ci\{y_1,\ldots,y_p|\beta_1,\ldots,\beta_q\}/\left(y,\beta\right)\\
&\simeq& \Ci\left(\mathcal{U}_\alpha\right)/\left(f\mspace{3mu}, f|_{\cN \cap \cU_\alpha}=0\right)\\
&\simeq & \Ci\left(\cN \cap \cU_\alpha\right).
\end{eqnarray*}
By the homotopy sheaf property, we conclude that $\cO_{\varphi}\left(\mathcal{U}\right)\simeq \Ci\left(\cN\right).$

Finally, notice that on one hand since $f_1,\ldots,f_p$ are invertible away from $\cN,$ $\cO_{\varphi}$ is acyclic away from $\cN,$ and hence, again by the sheaf property, the restriction map $$K\left(\Ci\left(\cM\right),f_1,\ldots,f_p,g_1,\ldots,g_q\right)=\cO_\varphi\left(\cM\right)\to \cO_\varphi\left(\mathcal{U}\right)$$ is a quasi-isomorphism. On the other hand, we have already proved that $K\left(\Ci\left(\cM\right),f_1,\ldots,f_p,g_1,\ldots,g_q\right)$ computes the desired homotopy pushout.
\end{proof}

In the process of this proof, we have proven the following:

\begin{proposition}\label{prop:pushlocus}
Suppose that $\cM$ is a supermanifold such that $\Ci\left(\cM\right)$ is cofibrant in $\mathbf{dg}\Ci\Alg_{\le 0},$ e.g. $\R^{n|m}$ for some $n$ and $m.$ Let $\varphi=\left(f,\alpha\right):\cM \to \R^{p|q}$ be any smooth map. Then the following diagram is a homotopy pushout in $\dgchmod$:
$$\xymatrix{\Ci\left(\R^{p|q}\right) \ar[r]^-{\varphi^\ast} \ar[d]_-{ev_0} & 
\Ci\left(\cM\right) \ar[d]\\
\R \ar[r] & K\left(\Ci\left(\cM\right),f_1,\ldots,f_p,g_1,\ldots,g_q\right),}$$
where $\varphi=\left(f_1,\ldots,f_p|g_1,\ldots,g_q\right).$
\end{proposition}

It is worth remarking that the morphism
$$\Ci\left(\R^{p|q}\right) \to K\left(\Ci\left(\R^{p|q}\right),y_1,\ldots,y_p,\beta_1,\ldots,\beta_q\right)$$ used in the proof is also a cofibration in the projective model structure on non-positively graded commutative differential graded superalgebras over $\R.$ This model structure is moreover left proper \cite[Theorem 4.17]{white}, hence we also get have the following:

\begin{proposition}\label{prop:pushlocuscalg}
Suppose that $\cM$ is a supermanifold. Let $\varphi=\left(f,\alpha\right):\cM \to \R^{p|q}$ be any smooth map. Then the following diagram is a homotopy pushout in the projective model structure on non-positively graded commutative differential graded superalgebras over $\R.$
$$\xymatrix{\Ci\left(\R^{p|q}\right) \ar[r]^-{\varphi^\ast} \ar[d]_-{ev_0} & 
\Ci\left(\cM\right) \ar[d]\\
\R \ar[r] & K\left(\Ci\left(\cM\right),f_1,\ldots,f_p,g_1,\ldots,g_q\right),}$$
where $\varphi=\left(f_1,\ldots,f_p|g_1,\ldots,g_q\right).$
\end{proposition}

\begin{proof}
This follows from the above discussion together with the fact that the above diagram is a $1$-categorical pushout diagram of differential graded commutative superalgebras.
\end{proof}

\begin{theorem}\label{thm:unramified}
Let $\varphi:\R^{n|m} \to \R^{p|q}$ be any smooth map of supermanifolds. Then for any $l,k,$ the following diagram is pushout in $\Alg_{\mathbf{SCom}_{\R}}\left(\Spc\right)$:
$$\xymatrix{\Ci\left(\R^{n+p|m+q}\right) \ar[r]^{pr^\ast} \ar[d]_-{\Gamma_f^\ast} & \Ci\left(\R^{n+p+l|m+q+k}\right) \ar[d]\\
\Ci\left(\R^{n|m}\right) \ar[r] & \Ci\left(\R^{n+l|m+k}\right)}$$
where $\Gamma_\varphi$ is the graph of $\varphi.$ I.e. the canonical morphism $$\mathbf{SCom}_\R \to \SCi$$ of algebraic theories is \emph{unramified} in the sense of \cite[Definition 3.31]{univ}.
\end{theorem}

\begin{proof}
Suppose once again that $\varphi=\left(f_1,\ldots,f_p|g_1,\ldots,g_q\right).$ Consider the super Koszul complex $K\left(\Ci\{t_1,\ldots,t_p|\xi_1,\ldots,\xi_q\},t_1,\ldots,t_p,\xi_1,\ldots,\xi_q\right) \simeq \R.$ Introduce coordinates $$\left(x_1,\ldots,x_n,y_1,\ldots,y_p|\alpha_1,\ldots,\alpha_m,\beta_1\ldots \beta_q\right)$$ on $\R^{n+m|p+q}.$ Define the function
$$h:\R^{n+p|m+q} \to \R^{p|q}$$
$$\resizebox{6in}{!}{$
h\left(x_1,\ldots,x_n,y_1,\ldots,y_p|\alpha_1,\ldots,\alpha_m,\beta_1\ldots \beta_q\right) =\left(y_1-f_1\left(x,\alpha\right),\ldots,y_p-f_p\left(x,\alpha\right)|\beta_1-g_1\left(x,\alpha\right),\ldots,\beta_q-g_q\left(x,\alpha\right)\right).$}$$
Write $h=\left(r,s\right).$
Then $\Gamma_\varphi^*$ factors as
$$\Ci\left(\R^{n+p|m+q}\right) \rightarrowtail K\left(\Ci\left(\R^{n+p|m+q}\right),r_1,\ldots,r_p,s_1,\ldots, s_q\right)  \stackrel{a}{\longrightarrow} \Ci\left(\R^{n|m}\right),$$ where the first arrow is a cofibration and the second is induced by the identification $$H^0\left(K\left(\Ci\left(\R^{n+p|m+q}\right),r_1,\ldots,r_p,s_1,\ldots, s_q\right)\right)\cong \Ci\left(\R^{n|m}\right),$$ since we have e.g. $\left[y_i\right]=\left[f_i\right]$ in Koszul cohomology. Notice that we also have a $1$-categorical pushout diagram
$$\resizebox{6in}{!}{$
\xymatrix{\Ci\{t_1,\ldots,t_p|\xi_1,\ldots,\xi_q\} \ar[d] \ar[r]^-{h^\ast} & \Ci\left(\R^{n+p|m+q}\right) \ar[d]\\
K\left(\Ci\{t_1,\ldots,t_p|\xi_1,\ldots,\xi_q\},t_1,\ldots,t_p,\xi_1,\ldots,\xi_q\right) \ar[r] & K\left(\Ci\left(\R^{n+p|m+q}\right),r_1,\ldots,r_p,s_1,\ldots, s_q\right).}$}$$
The vertical map is a cofibration and all the algebras are cofibrant, hence this is also a homotopy pushout. Moreover, since $h$ is a submersion, it follows from Lemma \ref{lem:suplocsubm} that $K\left(\Ci\left(\R^{n+p|m+q}\right),r_1,\ldots,r_p,s_1,\ldots, s_q\right)$ is quasi-isomorphic to smooth functions on the zero-locus of $h,$ which is readily identified with the $0^{th}$ cohomology of the super Koszul complex. Hence the morphism $a$ above is a quasi-isomorphism. The above pushout diagram is a homotopy pushout both of differential graded commutative superalgbras and of dg-${\Ci}$-superalgebras. By gluing pushout diagrams, we have the follow diagram
$$\resizebox{6.8in}{!}{$\xymatrix{\Ci\{t_1,\ldots,t_p|\xi_1,\ldots,\xi_q\} \ar@/^2.0pc/[rr]^-{\left(h\circ pr\right)^{\ast}} \ar[d] \ar[r]^-{h^\ast} & \Ci\left(\R^{n+p|m+q}\right) \ar[d] \ar[r]^-{pr^\ast} \ar[d] & \Ci\left(\R^{n+p+l|m+q+k}\right) \ar[d]\\
K\left(\Ci\{t_1,\ldots,t_p|\xi_1,\ldots,\xi_q\},t_1,\ldots,t_p,\xi_1,\ldots,\xi_q\right) \ar[r] & K\left(\Ci\left(\R^{n+p|m+q}\right),r_1,\ldots,r_p,s_1,\ldots, s_q\right) \ar[r] & K\left(\Ci\left(\R^{n+p+k|m+q+1}\right),r_1,\ldots,r_p,s_1,\ldots, s_q\right).}$}$$
all of whose squares are homotopy pushouts of differential graded commutative superalgebras and of dg-${\Ci}$-superalgebras. Since $pr \circ h$ is a submersion, Lemma \ref{lem:suplocsubm} implies that $K\left(\Ci\left(\R^{n+p+k|m+q+1}\right),r_1,\ldots,r_p,s_1,\ldots, s_q\right)$ is quasi-isomorphic to smooth functions on the zero locus of $pr \circ h,$ which once again is $0^{th}$ super Koszul cohomology, and hence $\Ci\left(\R^{n+l|m+k}\right).$ Finally, since $K\left(\Ci\left(\R^{n+p|m+q}\right),r_1,\ldots,r_p,s_1,\ldots, s_q\right) \stackrel{\sim}{\longrightarrow} \Ci\left(\R^{n|m}\right),$ the right most pushout square is also the homotopy pushout of 
$$\xymatrix{\Ci\left(\R^{n+p|m+q}\right) \ar[r]^{pr^\ast} \ar[d]_-{\Gamma_f^\ast} & \Ci\left(\R^{n+p+l|m+q+k}\right) \\
\Ci\left(\R^{n|m}\right)  &}$$
and so the result follows.
\end{proof}

This finishes the proof of Theorem \ref{thm:superuniv} and Corollary \ref{cor:supdman}.
\section{$C^\i$-schemes}\label{sec:schemes}

\begin{definition}
A \textbf{locally $\SCi$-ringed space} is a pair $\left(X,\cO_X\right)$ with $X$ a topological space and $\cO_X$ a sheaf with values in the $\i$-category $\dgc,$ such that for each point $x \in X,$ the stalk of $\left(\pi_0\left(\cO_X\right)\right)_0$ is a local ring with residue field $\R.$  Morphisms are pairs $$\left(  f,\alpha\right):\left(  X,\cO_X\right) \to \left(  Y,\cO_Y\right),$$ with $f$ a continuous map and $$\alpha:\cO_Y \to f_{*}\cO_X.$$ More formally, the $\i$-category $\Loc$ of locally $\SCi$-ringed space is the full subcategory on such spaces of the total space of the coCartesian fibration associated to to the functor
$$\Shv\left(  \blank,\dgc\right):\Top \to \widehat{\mathbf{Cat}}_\i$$ induced by push-forward of sheaves, where $\widehat{\mathbf{Cat}}_\i$ is the $\i$-category of large $\i$-categories.
\end{definition}

Consider the canonical functor $$\Gamma:\Loc \to \dgc$$ induced by taking global sections of each structure sheaf. This functor has a right adjoint $$\Speci:\dgc^{op} \to \Loc$$  \cite[Proposition 5.1.8]{Nuiten} which can be described as follows:

Let us first describe the underlined space $\underline{\Speci\left(\A\right)}$ of $\Speci\left(\A\right).$ Consider the canonical functor $\SCi \to \Top$ sending $\R^{p|q}$ to $\R^p.$ This functor preserves finite products, hence, by Theorem \ref{theorem:5.5.8.15} induces a unique limit-preserving functor
$$\underline{\Speci}:\dgc^{op} \to \Top.$$ Concretely, the underlying set of $\underline{\Speci}\left(\A\right)$ is the set of homomorphisms $$\Hom\left(\pi_0\left(\A\right)_0,\R\right).$$ Given an element $a \in \pi_0\left(\A\right)_0,$ consider the function
\begin{eqnarray*}
f_a:\Hom\left(\pi_0\left(\A\right)_0,\R\right) &\to &\R\\
\varphi & \mapsto & \varphi\left(a\right).
\end{eqnarray*}
Then the subsets of the form $U_a:=f_a^{-1}\left(\R\setminus \left\{0\right\}\right)$ form a basis for a topology on $\Hom\left(\pi_0\left(\A\right)_0,\R\right)$ which is moreover closed under finite intersections--- see \cite[Section 4.3.1]{univ}. Since this basis is closed under finite intersections, to describe a sheaf on $\underline{\Speci}\left(\A\right),$ it suffices to define a sheaf on the subcategory of all open subsets of $\underline{\Speci}\left(\A\right)$ on those of the form $U_a.$ Hence, there is a unique sheaf $\cO_\A$ arising as the sheafification of the presheaf $$U_a \mapsto \A\left[a^{-1}\right],$$ and we have
$$\Speci\left(\A\right)=\left(\underline{\Speci}\left(\A\right),\cO_\A\right).$$

\begin{remark}\label{rmk:4.19}
By \cite[Proposition 4.19]{univ}, if $\A$ is finitely presented, then all open subsets of $\Speci\left(\A\right)$ are of the form $U_a$ for some $a \in \pi_0\left(\A\right).$
\end{remark}

\begin{definition}
An \textbf{affine derived $\SCi$-schemes} is locally $\SCi$-ringed space in the image of $\Speci.$ Denote the full subcategory of $\Loc$ on these objects by $\Aff_{\SCi}.$
\end{definition}

\begin{definition}
An affine $\SCi$-scheme is \textbf{of finite type} if it is equivalent to $\Speci\left(\A\right),$ for $\A$ finitely presented. A $\SCi$-scheme is \textbf{locally of finite type} if it has a cover by finite type affine $\SCi$-schemes.
\end{definition}

\begin{remark}\label{rmk:pi0spec}
Let $\A$ be a dg-$\SCi$-algebra. By Theorem \ref{thm:dgsmander} and Corollary \ref{cor:supdman}, it follows that $\A$ is a finitely presented $\SCi$-ring if and only if it is quasi-isomorphic to $\Ci\left(\cM\right),$ for some dg-supermanifold $\cM.$ By \cite[Theorem 5.7]{dgmander}, we have that
$$\Speci\left(\A\right) \simeq \Speci\left(\Ci\left(\cM\right)\right) \simeq \left(\pi_0\left(\cM\right),\O_{\cM}|_{\pi_0\left(\cM\right)}\right).$$
\end{remark}

The $\Ci$-spectrum functor enjoys some nice properties:

\begin{proposition}
Let $\cM$ be a supermanifold, then $\Speci\left(\Ci\left(\cM\right)\right) \simeq \cM.$ Moreover, composite 
$$\SMfd \stackrel{\Ci}{\longhookrightarrow} \dgc^{op} \stackrel{\Speci}{\longlonglongrightarrow} \Loc$$ is fully faithful and preserves transverse pullbacks.
\end{proposition}

\begin{proof}
It is well known that for an ordinary smooth manifold $M,$ the set of $\R$-algebra homomorphisms $\Hom\left(\Ci\left(M\right),\R\right)$ is natural bijection with the points of $M,$ and all such homomorphisms are homomorphisms of $\Ci$-rings. Since any closed subset of a smooth manifold is the zero-set of a smooth function, an open subset is of the form $U_f,$ for $f:M \to \R,$ a smooth function. Hence, the topology also agrees with that of $M.$ It therefore suffices to check that for any open subset $U \subseteq \cM_0$ of the core of $\cM,$ we have $$\cO_{\SCi\left(\cM\right)}\left(U\right) \simeq \cO_{\cM}\left(U\right).$$ Notice that the following is a transverse pullback diagram in $\SMfd$:
$$\xymatrix{\cM|_{U} \ar[r] \ar[d] & \R\setminus\left\{0\right\} \ar[d] \\
\cM \ar[r]^-{f} & \R,}$$
where $U=U_f.$ By Theorem \ref{thm:4.51}, this implies that 
\begin{eqnarray*}
\Ci\left(\cM|_U\right) &\simeq & \Ci\left(\cM\right) \underset{\Ci\left(\R\right)} \oinfty \Ci\left(\R\setminus\left\{0\right\}\right)\\
&\simeq & \Ci\left(\cM\right)\left[f^{-1}\right]
\end{eqnarray*}
Since $\Ci\left(\cM|_U\right)=\cO_{\cM}\left(U\right),$ we can deduce that the presheaf defining $\cO_{\Speci\left(\Ci\left(\cM\right)\right)}$ is already a sheaf and is equivalent to $\cO_{\cM}.$
The second statement follows immediately from \cite[Theorem 16]{cinfsch}, since $\Speci$ agrees with Dubuc's spectrum for $\Ci$-rings for $0$-truncated algebras.
\end{proof}

Despite this, unlike in classical algebraic geometry, the spectrum functor $\Speci$ is not fully faithful in general. This was already realized for $\Ci$-rings in $\Set$ by Dubuc \cite{cinfsch}. It is fully faithful however when restricted to germ-determined finitely presented $\Ci$-rings by loc. cit.

\begin{definition}
A $\SCi$-algebra $\A$ is \textbf{complete} if the co-unit $$\varepsilon:\A \to \Gamma\left(\Speci\left(\A\right)\right)$$ is an equivalence. Denote the full subcategory on the complete algebras by $\dgc^{\wedge}.$
\end{definition}

\begin{proposition}\cite[Example 5.1.30]{Nuiten}, \cite[Corollary 4.45]{univ} \label{prop:compcomplete}
If $\A \in \dgc$ is finitely presented, then it is complete.
\end{proposition}

Note that the adjunction $\Gamma \dashv \Speci$ restricts to an adjunction
$$\xymatrix@C=2cm{\dgc^{op} \ar@<-0.5ex>[r]_-{\Speci} & \Aff_{\SCi \ar@<-0.5ex>[l]_-{\Gamma}}}$$

\begin{proposition}\cite[Proposition 5.1.26]{Nuiten}
The unit $$\eta:id_{\Aff_{\SCi}} \Rightarrow \Speci \circ \Gamma$$ is an equivalence.
\end{proposition}

The following corollary follows immediately:

\begin{corollary}
The composite
$$\left(\dgc^{\wedge}\right)^{op} \hookrightarrow \dgc^{op} \stackrel{\Speci}{\longlonglongrightarrow} \Loc$$ is fully faithful, with essential image $\Aff_{\SCi}.$ In particular $$\left(\dgc^{\wedge}\right)^{op} \simeq \Aff_{\SCi}.$$ Moreover, $\dgc^{\wedge}$ is a reflective subcategory of $\dgc$ with the left adjoint to the inclusion $i^{\wedge}$ given by $L^{\wedge}:=\Gamma \circ \Speci.$
\end{corollary}

\begin{corollary}
The above adjunction restricts to an equivalence
$$\Aff^{\mathbf{ft}}_{\SCi} \simeq \left(\Alg_{\Ci}\left(\Spc\right)^{\mathbf{fp}}\right)^{op}$$ between affine $\SCi$-schemes of finite type and the opposite $\i$-category of finitely presented $\SCi$-algebras. In particular, we have
$$\Aff^{\mathbf{ft}}_{\SCi} \simeq \left(\Alg_{\Ci}\left(\Spc\right)^{\mathbf{fp}}\right)^{op} \simeq \DSMfd \simeq \dgsmani.$$
\end{corollary}

\begin{definition}\label{dfn:open_covering}
A morphism $$\left(f,\alpha\right):\left(X,\cO_X\right) \to \left(Y,\cO_Y\right)$$ is an \textbf{open embedding} if $f:X \to Y$ is an open embedding of topological spaces, and the map $$f^*\cO_Y \to \cO_X$$ adjoint to $\alpha$ is an equivalence. An \textbf{open covering} of a locally $\SCi$-ringed space is a collection $$\left(\varphi_i:\left(X_i,\cO_{X_i}\right) \to \left(X,\cO_X\right)\right)_{i \in I}$$ of open embeddings such that the underlying maps of topological spaces are jointly surjective. The associated Grothendieck topology is called the \textbf{open covering} topology.
\end{definition}

\begin{definition}
A \textbf{derived $\SCi$-scheme} is a locally $\SCi$-ringed space $\left(X,\cO_X\right)$ such that there exists an open covering $$\left(f_i:\left(\Spec\left(\A_i\right) \to \left(X,\cO_X\right)\right)\right)_{i \in I}$$ with each $\A_i \in \dgc.$
\end{definition}

\begin{definition}
A collection of maps $\left(\varphi_\alpha:\A \to \A_\alpha\right)_{\alpha}$ in $\dgc$ is a \textbf{open covering} if the collection $\left(\Speci\left(\varphi_\alpha\right):\Speci\left(\A_{\alpha}\right) \to \Spec\left(\A\right)\right)_{\alpha}$ is in $\Loc.$ It follows immediately from the fact that $\Speci$ preserves pullbacks, that the collection of open coverings forms a Grothendieck topology on $\dgc^{op}.$
\end{definition}

\begin{remark}\label{rmk:open_covering}
Notice also that, since $L^{\wedge}$ preserves coproducts, and $\Speci \circ L^{\wedge} \circ i^{\wedge},$ the open covering topology restricts to $\left(\dgc^{\wedge}\right)^{op}.$ In fact, since $\Speci$ induces an equivalence $$\left(\dgc^{\wedge}\right)^{op} \simeq \Aff_{\SCi},$$ it follows that a collection of maps $\left(\varphi_\alpha:\A \to \A_\alpha\right)_{\alpha}$ in $\dgc^{op}$ is an open covering if and only if the collection of maps $\left(L^{\wedge}\left(\varphi_\alpha\right):L^{\wedge}\A \to L^{\wedge}\A_\alpha\right)_{\alpha}$ is.
\end{remark}

Let $\widehat{\Spc}$ be the $\i$-category of large spaces, and for any $\i$-category let $$\widehat{\Psh}\left(\sC\right):=\Fun\left(\sC^{op},\widehat{\Spc}\right),$$ (and similarly for sheaves). Then we have a composition of adjunctions

$$\xymatrix@C=1.5cm{\widehat{\Shv}\left(    \left(  \dgc^{\wedge}  \right)^{op}   \right) \ar@<-0.65ex>[r] & 
\widehat{\PShv}\left(    \left(  \dgc^{\wedge}   \right)^{op}  \right) \ar@<-0.65ex>[l]_-{a} \ar@<-0.65ex>[r]_-{\left(i^{\wedge}\right)_*} & 
\widehat{\PShv}\left(  \left(\dgc^{\wedge}\right)    ^{op}\right) \ar@<-0.65ex>[l]_-{\left(i^{\wedge}\right)^*}   }$$
with $a$ denoting sheafification with respect to the induced open covering topology, exhibiting $\widehat{\Shv}\left(    \left(  \dgc^{\wedge}  \right)^{op}   \right)$ as a left exact localization of $\widehat{\PShv}\left(  \left(\dgc^{\wedge}\right)    ^{op}\right).$ Unwinding the definitions, since $i^{\wedge} \simeq \left(L^{\wedge}\right)_!,$ it is in fact the localization with respect to sieves arising from covering families $\left(\varphi_\alpha:\A \to \A_\alpha\right)_{\alpha}$ in $\dgc^{op}$ such that $\left(L^{\wedge}\left(\varphi_\alpha\right):L^{\wedge}\A \to L^{\wedge}\A_\alpha\right)_{\alpha}$ is an open covering. But by Remark \ref{rmk:open_covering}, this is if and only if the original family was an open covering. This implies the following:

\begin{corollary}
Restriction along $i^{\wedge}$ induce an equivalence of $\i$-categories
$$\Shv\left(\dgc^{op}\right) \simeq \Shv\left(\left(\dgc^{\wedge}\right)^{op}\right),$$ where sheaves are taken with respect to the open covering topology.
\end{corollary}

\begin{remark}
The open covering topology on $\left(\dgc^{\wedge}\right)^{op}$ is sub-canonical. However, the open covering topology on $\dgc^{op}$ is \emph{NOT} sub-canonical. Indeed, since the former is sub-canonical, and $\left(i^{\wedge}\right)^* \simeq \left(L^{\wedge}\right)_!,$ it follows that the sheafification of the representable presheaf corresponding to $\A$ is the representable sheaf corresponding to $L^\wedge \A.$
\end{remark}

\begin{remark}\label{rmk:subcan}
By Proposition \ref{prop:compcomplete}, every finitely presented $\SCi$-algebra is complete. Moreover, by \cite[Remark 4.17]{univ}, these are closed under localizations. So, it follows from Remark \ref{rmk:4.19} that if $\A$ is finitely presented, all open subsets of $\Speci\left(\A\right)$ are of the form $\Speci\left(\A\left[a^{-1}\right]\right),$ for some $a \in \pi_0\left(\cM\right),$ and are thus finitely presented $\SCi$-schemes. It follows that the open cover pre-topology restricts to $\left(\Alg_{\SCi}\left(\Spc\right)^{\mathbf{fp}}\right)^{op}$ and is sub-canonical.
\end{remark}

\begin{proposition}
For $\cF:\dgc^{op} \to \Spc$ a sheaf with respect to the open covering topology, for any $\A \in \dgc,$ $$\cF\left(\A\right) \simeq \cF\left(L^{\wedge}\A\right).$$
\end{proposition}

\begin{proof}
By the above composite of adjunctions, it follows that any sheaf $\cF$ with respect to the open covering topology is in the essential image of $\left(i^{\wedge}\right)_*.$ Suppose that $$\cF=\left(i^{\wedge}\right)_*\cG.$$ Then for any $\A,$
\begin{eqnarray*}
\cF\left(\A\right) &\simeq & \Map\left(y\left(\A\right),\left(i^{\wedge}\right)_*\cG\right)\\
& \simeq & \Map\left(\left(i^{\wedge}\right)^*\left(y\left(\A\right)\right),\cG\right)\\
& \simeq & \Map\left(\left(L^{\wedge}\right)_!\left(y\left(\A\right)\right),\cG\right)\\
& \simeq & \Map\left(y\left(L^{\wedge}\A\right),\cG\right)\\
& \simeq & \cG\left(L^{\wedge}\A\right),
\end{eqnarray*}
The result now follows by replacing $\A$ with $i^{\wedge}L^{\wedge} \A.$
\end{proof}

In light of this, it is permissible to write $\cF\left(\Speci\left(\A\right)\right)$ instead of $\cF\left(\A\right),$ even when $\A$ is not complete, and we shall do so freely. One is free to interpret the occurrence of $\Speci\left(\A\right)$ in the above expression as a purely formal placeholder for $\A$ to remind one that they are working in the opposite category, or for $\cF$ to really mean $\Speci\cF$--- since $\Speci$ yields an equivalence between $\Aff_{\SCi}$ and $\dgc^{op}.$

Finally, we make the following observation:

\begin{remark}\label{rmk:restrict}
A presheaf $\cF:\dgc^{op} \to \Spc$ is a sheaf for the open covering topology if and only if for each $\A \in \dgc,$ $\cF|_{\underline{\Speci\left(\A\right)}}$ is, where the latter is the composite
\begin{alignat*}{8}
\mathbf{Bas}\left(\Speci\left(\A\right)\right)^{op} & \to & &\dgc^{op}& &&\stackrel{\cF}{\longrightarrow}&& &\Spc&\\
U_a &\mapsto & &\A\left[a^{-1}\right]&  &&\mapsto &&  &\cF\left( \A\left[a^{-1}\right]\right),&
\end{alignat*}
where $\mathbf{Bas}\left(\Speci\left(\A\right)\right)$ is the natural basis.
\end{remark}

\subsection{Infinite dimensional manifolds as $\Ci$-schemes}\label{sec:infdimmfd}
We will show that a very large class of infinite dimensional manifolds embed fully faithfully into $\Ci$-schemes. We start by discussing their model spaces, which are infinite dimensional vector spaces.

Recall that one of many equivalent ways of defining a locally convex vector space is as follows:

\begin{definition}
A \textbf{locally convex (real) vector space} (LCVS) is real vector space $V$ for which there exists a set $N_V=\left\{\rho_\alpha:V \to \mathbb{R}\right\}_{\alpha \in A}$ of semi-norms such that the addition $$+:V \times V \to V$$ and scalar multiplication $$m:\mathbb{R} \times V \to V$$ maps are continuous when $V$ is equipped with the initial topology with respect to the family $N_V$ of functions to $\R.$ If $N_V$ can be chosen to be countable, $V$ is called a \textbf{Frech\'et space}.
\end{definition}

\begin{remark}
Explicitly, given a family of semi-norms $N_V$ as above, the topology of $V$ is generated by the subbasis $$\left\{\rho_\alpha^{-1}\left(U\right)\mspace{3mu}:\mspace{3mu} U \subseteq \mathbb{R} \mbox{ open, }\alpha \in A\right\}.$$
\end{remark}


\begin{definition}
Suppose that $E$ and $F$ are LCV spaces. Let $U \subseteq E$ be an open subset and let $f:U \to F$ be a continuous function, and let $x \in U$ and $v \in E.$ The \textbf{G\^{a}teaux derivative} of $f$ at $x$ in the direction $v$ is
$$Df_x\left(v\right):=\underset{t \to 0} {\mathsf{lim}} \frac{f\left(x+t\cdot v\right)-f\left(x\right)}{t} \in F,$$
provided this limit exists.

We will say that such an $f$ is $\mathbf{C}_G^1$ if $Df_x$ exists for all $x \in U$ and the function
\begin{eqnarray*}
Df:U \times E &\to& F\\
\left(x,v\right) &\mapsto & Df_x\left(v\right)
\end{eqnarray*}
is continuous. This is essentially what is called \textbf{Michal-Bastiani differentiable} in the literature. For $k \ge 1,$ We will say that $f$ is $\mathbf{C}_G^{k+1}$ if $f$ is continuous and $Df$ is $\mathbf{C}^k.$ Finally, we will say that $f$ is $\Ci_G$ or G\^{a}teaux smooth if it is $\mathbf{C}_G^k$ for all $k \ge 0.$ Denote the set of G\^{a}teaux smooth functions from $U$ to $\mathbb{R}$ by $\Ci_G\left(U\right).$ This has a canonical structure of a $\Ci$-algebra in an obvious way.
\end{definition}

\begin{remark}
For Frech\'et or Banach spaces, this notion of smoothness agrees with Frech\'et smoothness (although for finite $n,$ the classes of $C^n$-functions can differ).
\end{remark}

\begin{remark}
The assignment $U \mapsto \Ci_G\left(U\right)$ is clearly a sheaf of $\Ci$-algebras on the underlying topological space of $E$. Moreover, for all open subsets $U,$ $\Ci_G\left(U\right)$ is closed under locally finite sums.
\end{remark}

Besides this notion of smoothness, there exists the notion of \emph{convenient smoothness} of Fr\"{o}licher and Kriegl:

\begin{definition}
A LCVS $E$ is a \textbf{convenient vector space} if a curve $\gamma:\R \to E$ is in $\Ci_G$ if and only if for all continuous linear functions $\ell:E \to \R,$ $$\ell \circ \gamma \in \Ci\left(\R\right).$$ For $E$ a convenient vector space, the \textbf{$c^\i$-topology} on $E$ is the final topology with respect to all smooth curves, i.e. $U \subseteq E$ is $c^\i$-open if and only if for all smooth curves $\gamma,$ $\gamma^{-1}\left(U\right) \subseteq \mathbb{R}$ is open. We denote $E$ equipped with the $c^\i$-topology by $E_{c^\i}$.

For a $c^\i$-open subset $U,$ a function $f:U \to \R$ is \textbf{conveniently smooth} if for all $\gamma:\R \to E$ smooth with $\gamma\left(\R\right) \subseteq U,$ $f \circ \gamma \in \Ci\left(\R\right).$ Denote the set of conveniently smooth functions on $U$ by $\Ci_C\left(U\right).$ More generally, a set-theoretic function $f:E \to F$ between convenient vector spaces is \textbf{conveniently smooth} if for all $\Ci_G$-smooth curves $\gamma:\R \to E,$ $f\circ \gamma$ is a $\Ci_G$-smooth curve in $F,$ equivalently, for all continuous linear functionals $\ell:E \to \R,$ $\ell \circ f \circ \gamma \in \Ci\left(\R\right).$
\end{definition}

\begin{remark}
The $c^\i$-topology on a convenient vector space $E$ is clearly finer than the locally convex topology, and may not agree with it, that is there can exists subsets $B \subseteq E$ such that for all smooth curves $\gamma,$ $\gamma^{-1}\left(B\right)$ is open, but $B$ is not open (for the topology generated by the semi-norms). In fact, $E_{c^\i}$ may fail to be a topological vector space with the $c^\i$-topology. Nonetheless, it is easily checked by using the sheaf condition on $\R,$ that the assignment to each $c^\i$-open subset $U$ the set of conveniently smooth functions $\Ci_C\left(U\right)$ is a sheaf for the $c^\i$-topology. It should be pointed out that a function $f:E \to F$ which is conveniently smooth may fail to be continuous for the given locally convex topologies, even when $F=\R,$ (but is of course continuous for their respective $c^\i$-topologies).
\end{remark}

There is a large class of convenient vector spaces for which the $c^\i$-topology \emph{does} coincide with the given locally convex topology \cite[Theorem 4.11]{mk}. These include any $E$ which is:

\begin{itemize}
\item a sequential space (equivalently the quotient of a metric space), e.g. any Frech\'et space.
\item the strong dual of a Frech\'et-Schwartz space.
\end{itemize}

\begin{proposition}\cite[Section 16.21]{mk}\label{prop:globbas}
For any convenient vector space $E,$ there is a basis for the $c^\i$-topology which consists of open subsets which are conveniently diffeomorphic to $E_{c^\i}$ itself.
\end{proposition}

\begin{definition}
Let $\left(X,\O_X\right)$ be a $\Ci$-ringed space with $\O_X$ a subsheaf of the sheaf of continuous functions. We say that it is \textbf{smoothly regular} if $X$ carries the final topology with respect to all functions of the form
\begin{eqnarray*}
X &\longrightarrow & \R\\
x &\mapsto & a\left(x\right),
\end{eqnarray*}
for some $a \in \Gamma\left(\O_X\right).$
\end{definition}

\begin{proposition}\cite[Proposition 14.4]{mk}
Let $E$ be a LCVS. Then $\left(E,\Ci_G\right)$ is smoothly regular if and only if the locally convex topology of $E$ is generated by semi-norms $\rho_\alpha$ for which
$$\rho_\alpha:\rho_\alpha^{-1}\left(\R\setminus\{0\}\right) \to \R$$ is in $\Ci_G$ for all $\alpha.$ (These could be different than the original semi-norms chosen, so this is a question of the existence of smooth ones.) Similarly, if $E$ is convenient, $\left(E_{c^\i},\Ci_C\right)$ is smoothly regular if the $c^\i$-topology on $E$ is generated by semi-norms $\rho_\alpha$ for which
$$\rho_\alpha:\rho_\alpha^{-1}\left(\R\setminus\{0\}\right) \to \R$$ is in $\Ci_C$ for all $\alpha.$
\end{proposition}

\begin{remark}
The above holds for Nuclear Frech\'et spaces, and in this case, the $c^\i$-topology coincides with the original locally convex one, and $C^\i_G=C^\i_C.$ On the other extreme, there are Banach spaces which do not even admit a $C^1$-norm.
\end{remark}

\begin{lemma}\label{lem:convglobat}
Let $E$ be a convenient vector space, and let $x \in E.$ Then any germ of a smooth function $\mathsf{germ}_x f$ defined on a $c^\i$-open neighborhood $U$ of $x$ has a representation as a conveniently smooth function $h:E \to \mathbb{R}$.
\end{lemma}

\begin{proof}
This follows since $x$ has a neighborhood basis of $c^\i$-open subsets which are $c^\i$-diffeomorphic to $E_{c^\i}$ by Proposition \ref{prop:globbas}.
\end{proof}

\begin{definition}
A locally convex space $E$ or convenient vector space is \textbf{realcompact} if the function of sets
\begin{eqnarray*}
E &\longrightarrow & \Hom_{\Ci\Alg}\left(\Ci\left(E\right),\mathbb{R}\right)\\
e &\mapsto & ev_e,
\end{eqnarray*}
is a bijection, where $ev_e\left(f\right)=f\left(e\right).$
\end{definition}

\begin{proposition}\label{prop:convchar}
For a locally convex vector space $E$ or convenient vector space, $\left(E,\Ci_G\right)$ or respectively $\left(E_{c^\i},\Ci_C\right)$ is an affine $\Ci$-scheme if and only if $E$ smoothly regular, realcompact, and every germ of a $\Ci$-function has a global representative.
\end{proposition}

\begin{proof}
One direction is obvious since any affine $\Ci$-scheme has these properties. Let us prove the converse. We will be agnostic as to whether $E$ is convenient or not, and $E$ will denote either $E$ with the locally convex topology, or the $c^\i$-topology.

Since $E$ is realcompact, we can identify $E=\Hom_{\Ci\Alg}\left(\Ci\left(E\right),\mathbb{R}\right)$ as a set. Since $E$ is smoothly regular, the topology of $E$ is generated by subsets of the form $f^{-1}\left(\R\setminus \{0\}\right),$ for $f \in \Ci\left(E\right),$ but this is the same topology as that on $\Hom_{\Ci\Alg}\left(\Ci\left(E\right),\mathbb{R}\right)$ coming from the $\Speci$-construction. It suffices to prove that the structure sheaves agree. For this, we will show that their stalks agree at each point $x \in E.$ Notice that for each $f \in \Ci\left(E\right),$ letting $U_f=f^{-1}\left(\R\setminus\{0\}\right),$ $f|_{U_f}$ is a unit, and hence there is an induced homomorphism
$$\Ci\left(E\right)\left[f^{-1}\right] \to \Ci\left(U\right).$$ Recall that the sheaf $\O_{\Ci\left(E\right)}$ is defined as the sheafification of a presheaf $\tilde \O_{\Ci\left(E\right)},$ and the left hand side is precisely what the presheaf $\tilde \O_{\Ci\left(E\right)}$ assigns $U_f.$ It follows that we get an induced map of sheaves $$\O_{\Ci\left(E\right)} \to \Ci_E.$$ It suffices to show that this morphism induces an isomorphism on stalks. The stalk at $x \in E$ of the right hand side is simply the $\Ci$-ring of germs of smooth functions at $x.$ The stalks on the left hand side can be computed directly from the presheaf. There is a neighborhood basis $\mathcal{U}_x$ of $x$ consisting of open subsets of the form $U_f=f^{-1}\left(\R\setminus\{0\}\right),$ ranging over all $f$ for which $f\left(x\right) \ne 0.$ Note that these are precise those $f$ such that $f \in \Ci\left(E\right) \setminus \mathfrak{m}_x.$ The stalk of $\tilde \O_{\Ci\left(E\right)}$ at $x$ is therefore the filtered colimit over all such $f$ of $\Ci\left(E\right)\left[f^{-1}\right],$ but this is just the $\Ci$-localization at the compliment of $\mathfrak{m}_x.$ We claim that the $\Ci$-localization of $\Ci\left(E\right)$ at this muliplicatively closed subset coincides with the algebraic localization. For this, it suffices to prove that the algebraic localization is a $\Ci$-algebra, but since $\mathfrak{m}_x$ is maximal, the algebraic localization is a local ring, so this is obvious. Elements of the algebraic localization are generalized fractions of the form  $\frac{g}{s},$ with $s$ a smooth function such that $s\left(x\right) \ne 0.$ We have that $\frac{g}{s}=0$ in the localization if and only if there exists $s'$ smooth such that $s'\left(x\right) \ne 0$ and such that $s' \cdot g =0.$ The induced map on stalks sends this generalized fraction to the quotient of the germ of $g$ at $x$ by the germ of $s$ at $x.$ So, $\frac{g}{s}$ is in the kernel if and only if $\mathsf{germ}_x\left(g\right)=0.$ This means that there exists an open neighborhood $U$ of $x$ such that $g|_U=0.$ Since $E$ is smoothly regular, we can shrink $U$ so that $U=f^{-1}\left(\R\setminus\{0\}\right),$ for some $f \in \Ci\left(E\right).$  But now $f \cdot g=0$, and $\germ_x\left(f\right) \ne 0.$ It follows that the homomorphism is injective at the level of stalks. It suffices to prove that it is surjective, but this is precisely the statement that every germ of a smooth function has a global representative, so we are done.
\end{proof}


\begin{corollary}\label{cor:allinfmfds}
The following classes of LCVS and convenient vector spaces are affine $\Ci$-schemes when equipped with their appropriate sheaf of smooth functions:
\begin{itemize}
\item[1.] Any smoothly regular LCVS which admits smooth partitions of unity. This includes:
\begin{itemize}
\item[a)] Any metrizable LCVS for which smooth functions separate disjoint closed sets.
\item[b)] All Hilbert spaces.
\item[c)] All Nuclear Frech\'et spaces.
\end{itemize}
\item[2.] All realcompact smoothly regular convenient vector spaces (with the $c^\i$-topology).
\item[3.] Any smoothly regular LCVS which is Lindel\"{o}f. Any smoothly regular convenient vector space which is Lindel\"{o}f for its $c^\i$-topology.
\end{itemize}
\end{corollary}

\begin{proof}
$1.$ follows since admitting smooth partitions of unity allows one to construct global representatives for germs. $2.$ follows since by Lemma \ref{lem:convglobat}, any such space admits global representatives for germs. $3.$ follows from \cite[Theorem 4.41]{joycesch}. Alternatively, by \cite[Theorem 16.10]{mk}, any such space has smooth partitions of unity, so the result follows from 1.

\end{proof}

\begin{remark}
Let $\left(E,\Ci_E\right)$ and $\left(F,\Ci_F\right)$ be locally convex vector spaces as above, regarded as affine $\Ci$-schemes. Since each structure sheaf is a subsheaf of the sheaf of continuous functions, a morphism $$\varphi:\left(E,\Ci_E\right) \to \left(F,\Ci_F\right)$$ of $\Ci$-ringed spaces is the same as a function of underlying sets $\varphi:E \to F$ such that $f \circ \varphi \in \Ci_G\left(E\right)$ for all $f \in \Ci_G\left(F\right).$ Denoting $\mathsf{LCVS_{s}}$ the full subcategory of locally convex real vector spaces and $\Ci_G$-smooth maps on those which are smoothly regular (i.e. there exists a family of semi-norms generating the topology which are $C^\i_G$ away from their zero-set), realcompact, and such that germs of smooth functions have global representatives, there is a functor
\begin{eqnarray*}
\mathsf{LCVS}_{s} &\to& \Aff_{\Ci}\\
E &\mapsto & \left(E,\Ci_E\right).
\end{eqnarray*}
This functor is faithful, but it is not full.

Consider now instead, the full subcategory of convenient vector spaces and conveniently smooth maps on those which are smoothly regular and realcompact, and denote it by $\mathsf{Conv}_s.$ Then the corresponding functor
\begin{eqnarray*}
\mathsf{Conv}_s &\to& \Aff_{\Ci}\\
E &\mapsto & \left(E_{c^\i},\Ci_E\right)
\end{eqnarray*}
\emph{is} fully faithful. To see this, we need to show that a function $\varphi:E \to F$ is conveniently smooth if and only if $f \circ \varphi \in \Ci_C\left(E\right)$ for all $f \in \Ci_C\left(F\right).$ If $\varphi$ is conveniently smooth, so is $f \circ \varphi,$ so one direction is clear. Conversely, suppose that $f \circ \varphi \in \Ci_C\left(E\right)$ for all $f \in \Ci_C\left(F\right).$ Let $\gamma:\R \to E$ be a smooth curve.  We need to show that $f \circ \gamma:\R \to F$ is a smooth curve. But since $F$ is convenient, this is if and only if for all $\ell:F \to \R$ linear and continuous, $\ell \circ \left(f \circ \gamma\right) \in \Ci\left(\R\right).$ Since $\ell \in \Ci_C\left(F\right),$ we have by hypothesis that $\ell \circ f \in \Ci_C\left(E\right),$ and thus $\gamma \circ \ell \circ f$ is also smooth.

Thus, if $E$ is not convenient, then even when $\left(E,\Ci_E\right)$ is an affine $\Ci$-scheme, it may not have the same functor of points as $E,$ regarding $E$ as an infinite-dimensional manifold.




\end{remark}

\begin{definition}
A \textbf{convenient manifold} is a topological space $M,$ together with charts with values in convenient vector spaces whose transition maps are conveniently smooth. We call a convenient manifold \textbf{admissible} if it has a chart with values in $\mathsf{Conv}_s$.
\end{definition}

\begin{corollary}\label{cor:frechff}
The category of admissible convenient manifolds embeds fully faithfully into $\Ci$-schemes. In particular, so does the full subcategory of Frech\'et manifolds locally modeled on nuclear Frech\'et spaces. We denote by
$$S_{\mathsf{Conv}}:\Mfd_{\mathsf{Conv}_s} \hookrightarrow \Ci\mathsf{Sch}$$
the fully faithful embedding.
\end{corollary}

\begin{remark}
By \cite[Theorem 4.41]{joycesch}, if $M$ is an admissible convenient manifold which is moreover Lindel\"{o}f, $S_{\mathsf{Conv}}\left(M\right)$ is moreover affine.
\end{remark}

Since $\Ci$-schemes embed fully faithfully into $\Shv\left(\dgc^{op}\right),$ this gives a fully faithful embedding of admissible convenient manifolds into sheaves via their functor of points. There is another way to fully faithful embed convenient manifolds into sheaves that works even without the admissibility assumption, but the embedding is different. Namely, the category of smoothly regular convenient manifolds embeds fully faithfully into the category of Fr\"{o}licher spaces \cite[Lemma 27.5]{mk}, and these in turn embed fully faithfully into the category of sheaves on (finite dimensional) manifolds $\Shv\left(\Mfd\right)$ (and moreover, factors through the inclusion of diffeological spaces). Finally, if $$sm:\Mfd \hookrightarrow \dgc^{op}$$ denotes the fully faithful inclusion, there is an induced restriction functor $$sm^*:\Shv\left(\dgc^{op}\right) \to \Shv\left(\Mfd\right)$$ which admits a left adjoint $sm_!,$ since the covers for a finite dimensional manifold in either site are the same. Since $sm$ is fully faithful, so is $sm_!,$ and by composition we get a fully faithful embedding of the category of smoothly regular convenient manifolds into the topos $\Shv\left(\dgc^{op}\right).$

Now suppose that $M$ is an \emph{admissible} convenient manifold. Then $S_{\mathsf{Conv}}\left(M\right)$ is a $\Ci$-scheme, and moreover, its functor of points on manifolds is $sm^*S_{\mathsf{Conv}}\left(M\right).$ So we wish to compare
$S_{\mathsf{Conv}}\left(M\right)$ with $sm_!sm^*S_{\mathsf{Conv}}\left(M\right),$ i.e. we wish to determine if the co-unit
$$\epsilon_M:sm_!sm^*S_{\mathsf{Conv}}\left(M\right) \to S_{\mathsf{Conv}}\left(M\right)$$ is an equivalence. Let $F \in \Shv\left(\Mfd\right)$ be arbitrary. Since we could equivalently take sheaves on $\Ci$---i.e. the full subcategory of manifolds on those which are of the form $\R^n$ for some $n,$ the canonical map
$$\underset{f:\R^n \to M, n \in \mathbb{N}} \coprod \R^n \to \tilde y\left(M\right)$$ is an effective epimorphism. Since $sm_!$ is a left-adjoint, we get that
$$\pi:\underset{f:\R^n \to M, n \in \mathbb{N}} \coprod \R^n \to sm_!sm^*S_{\mathsf{Conv}}\left(M\right)$$ is an effective epimorphism in $\Shv\left(\dgc^{op}\right).$ If $\epsilon_M$ were an equivalence, since $S_{\mathsf{Conv}}\left(M\right)$ is a $\Ci$-scheme, $\pi$ would need to admit local sections. That means around every $x \in M,$ there exists an open subset $U$ such the inclusion $U \hookrightarrow M$ factors through a map $f:\R^n \to M$ via a smooth map $\lambda:U \to \R^n.$ This would imply in particular that 
$$Tf_{\lambda\left(x\right)}:T_{\lambda\left(x\right)} \R^n \to T_x M$$ is surjective, which would imply that $M$ is finite dimensional.

\begin{corollary}
For an admissible convenient manifold, $$\epsilon_M:sm_!sm^*S_{\mathsf{Conv}}\left(M\right) \to S_{\mathsf{Conv}}\left(M\right)$$ is an equivalence if and only if $M$ is finite dimensional.
\end{corollary}

There are two main ways that infinite dimensional manifolds naturally arise. One is via the mapping space between two finite dimensional manifolds, or more generally, as the space of sections of some fiber bundle. Another, is by a pro-system of finite dimensional manifolds. Depending on which of the ways the manifold shows up, one may prefer to use one embedding over the other. On one hand, for the former case, \cite[Proposition 3.2.6]{ellst} implies that for $E \to M$ a proper submersion of finite dimensional manifolds, the relative mapping stack $\Gamma_M\left(E\right)$ is represented by the image under $sm_!$ of the nuclear Frech\'et manifold of sections. On the other hand, for the latter case, consider the following two examples:

\begin{example}
	The projective limit $\R^\i$ with its standard Frech\'et structure is the projective limit $\underset{k} \lim \R^k$ in the category of nuclear Frech\'et manifolds.
	\end{example}
	
	\begin{example}
	More generally, for any surjective submersion $$\pi:E \to M$$ of finite dimensional manifolds, its collection of jet bundles $\J^k\left(E,M\right)$ forms a projective system of finite dimensional manifolds, and the infinite jet bundle $\J^\i\left(E,M\right)$ with its standard Frech\'et structure is the projective limit in the category of nuclear Frech\'et manifolds. It is locally diffeomorphic to $\R^\i.$
	\end{example}
	
	

\begin{lemma}\label{lem:frecol}
Let $\R^\i$ be equipped with its standard Frech\'et structure. Then
$$\Ci\left(\R^\i\right) \simeq \Gamma\left(\Speci\left(\colim \Ci\left(\R^k\right)\right)\right).$$
\end{lemma}

\begin{proof}
Since $\underline{\Speci}:\dgc^{op} \to \Top$ preserves limits, we have an identification of topological spaces between $\R^\i$ and $\underline{\Speci}\left(\colim \Ci\left(\R^k\right)\right).$ The structures sheaf of $\Speci\left(\colim \Ci\left(\R^k\right)\right)$ is the sheafification of the presheaf that assigns an open subset $U$ of $\R^\i$ the set of set-theoretic functions $$U \to \R$$ which factor as $$U \to p_k\left(U\right) \stackrel{f}{\longrightarrow} \R,$$ with $p_k:\R^\i \to \R^k$ the projection, and with $f$ smooth. The sheafification is then all functions that are \emph{locally} of this form. But these are precisely the smooth functions from the inherited Frech\'et manifold structure on $U \subset \R^\i.$
\end{proof}

\begin{corollary}\label{cor:colimfre}
For any surjective submersion $$\pi:E \to M$$ of finite dimensional manifolds, 
$$\Ci\left(\J^\i\left(E,M\right)\right) \simeq \Gamma\left(\Speci\left(\colim \Ci\left(\J^k\left(E,M\right)\right)\right)\right).$$
\end{corollary}
\section{The derived topos}\label{sec:derived_topos}

\subsection{The derived site}
Let $\sD$ be a small subcategory of the the opposite category of $\dgc$ such that\footnote{We will also tacitly assume each algebra $\A$ in $\sD$ is complete. Otherwise, after choosing a subcategory $\sD$ satisfying $1)$-$4)$, one can replace each algebra with its completion.}

\begin{itemize}
\item[1)] $\sD$ contains all the compact objects--- i.e. (homotopically) finitely presented algebras
\item[2)] For all $\A \in \sD,$ and $n \in \mathbb{N} \cup \left\{-1\right\},$ $\tau_{\le n}\left(\A\right) \in \sD,$ where $\tau_{\le -1} \A$ is the reduction of $\pi_0\left(\A\right).$
\item[3)] $\sD$ is closed under open coverings
\item[4)] $\sD$ is closed under finite limits and retracts.
\end{itemize}

\begin{definition}
The $\i$-category $\sD$ together with the open covering Grothendieck topology is called the \textbf{derived site}. Moreover, the associated $\i$-topos $$\Shv\left(\sD\right)=:\bH$$ will be called the \textbf{derived topos}. Objects of this topos will be called \textbf{derived stacks}.
\end{definition}

\begin{proposition}
	$\bH$ is hypercomplete.
	\end{proposition}

\begin{proof}
	It suffices to prove that for each $D \in \sD,$ $\bH/D$ is hypercomplete. To this end, it suffices to show that each representable $D \in \bH$ is a hypersheaf. Let $\Speci\left(B\right)$ be a representable and let $\Speci\left(\A\right)$ be another. By Remark \ref{rmk:restrict}, it suffices to prove that the restriction of $\Speci\left(\B\right)$ to $\mathbf{Bas}\left(\Speci\left(\A\right)\right)$ is a hypersheaf. Without loss of generality, assume that $\cB$ is complete, so that this restriction can be identified with $\Map_{\dgc}\left(\cB,\cO_\A\left(\blank\right)\right).$ It follows from \cite[Lemma 5.1.25]{Nuiten} that $\cO_\A$ is a hypersheaf. The result now follows.
	\end{proof}
	
\subsection{Adjunctions and Monads}\label{sec:adj}

Denote by $$i:\sD^{0} \hookrightarrow \sD$$ the full subcategory of classical $\SCi$-algebras, and denote by $$l:\sD_{red} \hookrightarrow \sD_0$$ the full subcategory thereof on the reduced algebras. Finally, denote by $k:=i \circ l.$ Note that $i$ has a right adjoint $\pi_0$, which sends $\sp\left(A\right)$ to spec of its $0$-truncation $\sp\left(\pi_0\left(A\right)\right),$ and $l$ has a right adjoint $R$ which sends $\sp\left(A\right)$ to $\sp\left(A/\Nil_{\Ci}\left(A\right)\right)$. Denote the composite $R \circ \pi_0$ by $p$; it is right adjoint to $k.$ Note that $p$ sends $\sp\left(\A\right)$ to $\sp\left(\A_{red}\right).$

Note that all $6$ of these functors preserve covers, consequently there are are many induced adjunctions induced by Kan extensions:

  $$
    \xymatrix{
      \Shv\left(\sD^{0}\right)
      \ar@<+18pt>@{^{(}->}[rr]|-{i_!}
      \ar@<+9pt>@{<-}[rr]|-{i^\ast \simeq \left(\pi_0\right)_!}
      \ar@<+0pt>@{^{(}->}[rr]|-{i_\ast \simeq \left(\pi_0\right)^\ast}
      \ar@<-9pt>@{<-}[rr]|-{\left(\pi_0\right)_\ast}
      &&
      \;\Shv\left(\sD\right)
    }
    \,
  $$

$$
    \xymatrix{
      \Shv\left(\sD_{red}\right)
      \ar@<+18pt>@{^{(}->}[rr]|-{l_!}
      \ar@<+9pt>@{<-}[rr]|-{l^\ast \simeq R_!}
      \ar@<+0pt>@{^{(}->}[rr]|-{l_\ast \simeq R^\ast}
      \ar@<-9pt>@{<-}[rr]|-{R_\ast}
      &&
      \;\Shv\left(\sD_0\right)
    }
    \,
  $$

$$
    \xymatrix{
      \Shv\left(\sD_{red}\right)
      \ar@<+18pt>@{^{(}->}[rr]|-{k_!}
      \ar@<+9pt>@{<-}[rr]|-{k^\ast \simeq p_!}
      \ar@<+0pt>@{^{(}->}[rr]|-{k_\ast \simeq p^\ast}
      \ar@<-9pt>@{<-}[rr]|-{p_\ast}
      &&
      \;\Shv\left(\sD\right)
    }
    \,.
  $$

\begin{definition}
An derived stack $\cX$ is \textbf{classical} if it is in the essential image of $i_!$. Denote the full subcategory on the classical stacks by $\bHc.$
\end{definition}

Since $i$ is fully faithful so is $i_!,$ and hence $\bHc$ is a coreflective subcategory of $\bH.$ In particular, it is closed under colimits.

\begin{proposition}\label{prop:closedprod}
The subcategory $\bHc$ is closed under finite products.
\end{proposition}

\begin{proof}
$\bHc$ clearly contains the terminal object, so it suffices to check it is closed under binary products. Let $i_!\cX$ and $i_!\cY$ be classical stacks. Then 
$$i_!\cX \simeq \underset{\sp\left(A\right) \to \cX} \colim i_! \sp\left(A\right)$$ and $$i_!\cY \simeq \underset{\sp\left(B\right) \to \cY} \colim i_! \sp\left(B\right)$$ where in both colimits, $A$ and $B$ range over objects in $\cD_0.$ By abuse of notation, $$i_!\sp\left(A\right)=\sp\left(A\right)$$ (more accurately, $i_! y_{\cD_0}\left(\sp\left(A\right)\right) \simeq y_{\sD}\left(i\left(\sp\left(A\right)\right)\right)$). Hence
\begin{eqnarray*}
\cX \times \cY &\simeq& \left(\underset{\sp\left(A\right) \to \cX} \colim \sp\left(A\right)\right) \times \left(\underset{\sp\left(B\right) \to \cX} \colim \sp\left(B\right)\right)\\
&\simeq& \underset{\sp\left(A\right) \to \cX} \colim \mspace{5mu} \underset{\sp\left(B\right) \to \cY} \colim \left(\sp\left(A\right) \times \sp\left(B\right)\right)\\
&\simeq& \underset{\sp\left(A\right) \to \cX} \colim \mspace{5mu} \underset{\sp\left(B\right) \to \cY} \colim \left(\sp\left(A \oinfty B\right)\right)
\end{eqnarray*}
which is in $\bHc,$ since each $\sp\left(A \oinfty B\right)$ is and $\bHc$ is closed under colimits.
\end{proof}

\subsection{Spaces of sections}
Recall that for any $\i$-topos $\cE,$ and an object $E \in \cE,$ there is a canonical geometric morphism $$\pi_E:\cE/E \to \cE.$$ The inverse image functor is given by pullback:
\begin{eqnarray*}
\pi_E^*:\cE &\to & \cE/E\\ 
F &\mapsto & F \times E \to E.
\end{eqnarray*}
This functor has both a left and a right adjoint:
$$\left(\pi_E\right)_! \dashv \pi_E^* \dashv \left(\pi_E\right)_*.$$
The left adjoint is the functor:
\begin{eqnarray*}
\left(\pi_E\right)_!:\cE/E \to \cE\\
f:F \to E & \mapsto & F.
\end{eqnarray*}

Notice that if $M \in \sD,$ that by the Yoneda lemma it follows that for all $N \in \sD$ and $f:E \to M,$
\begin{eqnarray*}
\left(\pi_M\right)_*\left(f\right)\left(N\right) & \simeq & \Map_{\bH}\left(N,\left(\pi_M\right)_*\left(f\right)\right)\\
&\simeq & \Map_{\bH/M}\left(\pi^*_M\left(N\right),f\right)\\
& \simeq & \Map_{\bH/M} \left(N \times M \to M,f\right),
\end{eqnarray*}
where the latter is the space of sections
$$\xymatrix@C=2.5cm@C=2cm{& E \ar[d]^-{f}\\
N \times M \ar[r]_-{pr_N} \ar@{-->}[ur]^-{\sigma} & M.}$$
\begin{definition}
Given $f:E \to M$ in $\bH,$ the \textbf{space of sections of $f$} is the stack
$$\Sec_M\left(f\right):=\left(\pi_M\right)_*\left(f\right)$$
\end{definition}

\begin{remark}
Notice that for $E,F \in \bH,$ the internal Hom $$F^E=\underline{\Map}_{\bH}\left(E,F\right)=\Sec\left(\pi_F^*E\right),$$ and conversely, if $f:E \to F,$ there is a pullback diagram
$$\xymatrix{\Sec_F\left(E\right) \ar[r] \ar[d] &\ar[d] F^E \ar[d]^-{F^f}\\
\left\{id_F\right\} \ar[r] & F^F.}$$
\end{remark}

\section{Derived Orbifolds} \label{sec:orbifold}
This section will follow \cite{higherme}. (See also \cite[Chapter 20]{LurieSAG}, where these ideas are expanded upon from a different point of view.)

\begin{definition}\label{dfn:sctopos}
A \textbf{$\SCi$-ringed $\i$-topos} is a pair $\left(\cE,\O_{\cE}\right)$ with $\O_{\cE}$ a $\SCi$-algebra in $\cE,$ i.e. a finite product preserving functor $$\O_\cE:\Ci \to \cE.$$ Equivalently, by abuse of notation, this is an $\i$-topos $\cE$ equipped with a geometric morphism $$\O_{\cE}:\cE \to \cB\SCi,$$ where $\cB\SCi$ is the classifying $\i$-topos of the algebraic theory $\SCi.$
\end{definition}

The $\i$-category of $\SCi$-ringed $\i$-topoi, denoted $\Topi^{\SCi},$ may be defined informally as the $\i$-category whose objects are $\SCi$-ringed $\i$-topoi, and whose morphisms $$\left(\cE,\cO_{\cE}\right) \to \left(\cF,\O_{\cF}\right)$$ consists of a geometric morphism $$f:\cE \to \cF,$$ together with a morphism of sheaves of $\SCi$-algebras over $\cF$ $$\alpha:\O_{\cF} \to f_*\O_{\cE}.$$ This can be rigorously constructed as a Cartesian fibration over $\Topi,$ c.f. \cite[Section 2.4]{higherme}.

\begin{remark}
By Theorem \ref{theorem:5.5.8.15}, specifying a finite product preserving functor $$\SCi \to \cE$$ is the same as specifying a limit preserving functor $$\Alg_{\SCi}\left(\Spc\right)^{op} \to \cE,$$ which is the same as a colimit preserving functor $$\Alg_{\SCi}\left(\Spc\right) \to \cE^{op},$$ which is completely determined by its right adjoint (which must exist) which is a limit preserving functor $$\cE^{op} \to \Alg_{\SCi}\left(\Spc\right),$$ i.e. a sheaf on $\cE$ with values in $\dalg.$ If $\cE\simeq \Shv\left(\sC,J\right)$ for a Grothendieck site $\left(\sC,J\right),$ then this is the same data as a sheaf on $\left(\sC,J\right)$ with values in $\dalg.$ Applying this to the case that the Grothendieck site is the poset of open subsets of a topological space $X$ with the standard open covered topology, a finite product preserving functor $$\SCi \to \Shv\left(X\right)$$ is the same as a sheaf of $\SCi$-rings on $X.$ So if $\left(X,\O_X\right)$ is any $\SCi$-ringed space, then $\left(\Shv\left(X\right),\O_X\right)$ is an $\SCi$-ringed $\i$-topos.
\end{remark}

\begin{definition}
Given a $\SCi$-ringed $\i$-topos $\left(\cE,\O\right)$ and an object $E$ of $\cE,$ the slice topos $\cE/E$ inherits a canonical structure of a $\SCi$-ringed topos as $\left(\cE/E,\O_E:=\O|_E\right),$ where more formally $\O|_E:=e^*\O,$ where $e:\cE/E \to \cE$ is the canonical (\'etale) geometric morphism.
We denote $\O|_E$ by $\O_E$ and call it the \textbf{structure sheaf} of $E.$
	\end{definition}

\begin{definition}
An \textbf{$\SCi$-Deligne-Mumford stack} is a $\SCi$-ringed $\i$-topos $\left(\cX,\O_{\cX}\right)$ such that there exists a collection of objects $X_\alpha \in \cX$ for which the unique map $$\underset{\alpha} \coprod X_\alpha \to 1$$ is an effective epimorphism, together with equivalences $$\left(\cX/X_\alpha,\O_\cX|_{X_\alpha}\right) \simeq \left(\Shv\left(\Speci\left(\A_\alpha\right)\right),\O_{\A_\alpha}\right)$$ for $\A_\alpha \in \sD.$ If the $\A_\alpha$ can be chosen to be of finite presentation (i.e. $\Speci\left(\A_\alpha\right)$ is a derived manifold), then $\left(\cX,\O_{\cX}\right)$ is called a \textbf{derived orbifold}. Denote the full subcategory of $\Topi^{\SCi}$ spanned by the $\SCi$-Deligne-Mumford stacks by $\DMsCi.$
\end{definition}

\begin{remark}
The canonical functor
\begin{eqnarray*}
\sD &\longrightarrow & \Topi^{\SCi}\\
\A &\mapsto & \left(\Shv\left(\Speci\left(\A\right)\right)\right)
\end{eqnarray*}
is fully faithful. Moreover, the induced functor
$$\DMsCi \to \Shv\left(\sD\right)=\bH$$ is fully faithful by \cite[Theorem 2.2]{higherme}. We hence can identify $\DMsCi$ with its image in $\bH.$
\end{remark}

\begin{definition}
Let $\left(\cX,\O_{\cX}\right)$ be an $\SCi$-ringed $\i$-topos and let $X \in \cX.$ Then there is a canonical morphism $$\left(\cX/X,\O_{\cX}|_{X}\right) \to \left(\cX,\O_{\cX}\right).$$ Any morphism $\left(\cY,\O_{\cY}\right) \to \left(\cX,\O_{\cX}\right)$ equivalent to one of the above form is called \textbf{\'etale}.
\end{definition}

Denote by $\DMsCi^{\et}$ the subcategory of $\DMsCi$ of $\SCi$-Deligne-Mumford stacks and \'etale morphisms.

\begin{proposition}\cite[Remark 1.7]{higherme}
For any $\SCi$-Deligne-Mumford stack $\left(\cX,\O_{\cX}\right),$ there is a canonical equivalence of $\i$-categories
\begin{eqnarray*}
\cX & \longrightarrow & \DMsCi^{\et}/\left(\cX,\O_{\cX}\right)\\
X & \mapsto & \left(\cX/X,\O_{\cX}|_{X}\right) \to \left(\cX,\O_{\cX}\right)
\end{eqnarray*}
\end{proposition}

\begin{remark}
It follows from \cite[Corollary 6.3.5.9]{htt} that for $\cX$ any $\SCi$-Deligne-Mumford stack, $\DMsCi^{\et}/\left(\cX,\O_{\cX}\right)$ can be identified with a full subcategory of $\bH/\cX,$ namely by those morphisms $\cY \to \cX$ which are representable by an \'etale map of $\SCi$-Deligne-Mumford stacks. In particular, for any $\A \in \sD,$ the $\i$-topos $\Shv\left(\Speci\left(\A\right)\right)$ can be identified with the full subcategory of $\bH/\Speci\left(\A\right)$ spanned by \'etale maps from an $\SCi$-Deligne-Mumford stack.
\end{remark}

Consider the canonical inclusion $$\sD^{op} \hookrightarrow \dgc.$$ Since the open covering topology is sub-canonical on $\sD$ (as we are assuming all algebras in $\sD$ are complete), this inclusion is a sheaf $\cO$ on $\sD$ with values in $\dgc.$ Thus $\left(\bH,\cO\right)$ is a $\SCi$-ringed $\i$-topos.

\begin{remark}
The above sheaf $\cO$ is represented by $\mathbb{R}$ (and $\mathbb{R}^{0|1}$ as a sheaf of $\mathbb{Z}_2$-graded sets in the super setting).
\end{remark}

\begin{remark}
Let $\widetilde{\cX} \in \bH$ be an $\SCi$-Deligne-Mumford stack (which is the functor of points of say $\left(\cX,\O_{\cX}\right).$) Then, identifying $\cX$ with the full subcategory $\DMsCi^{\et}/\cX$ of $\bH/\cX,$ we have that $$\left(\cX,\O_{\cX}\right) \simeq \left(\DMsCi^{\et}/\cX,\O_{\bH}|_{\DMsCi^{\et}/\cX}\right).$$
\end{remark}

Denote by $\sD^{\et}$ the subcategory of $\sD$ spanned by the \'etale maps. The open covered topology on $\sD$ restricts to $\sD^{\et}.$ Denote by 
$$j:\sD^{\et} \to \sD$$ the canonical inclusion. Then the restriction functor $$j^*:\bH=\Shv\left(\sD\right) \to \Shv\left(\sD^{\et}\right)$$ admits a left adjoint $j_!$ \cite[Definition 3.4]{higherme} called the \textbf{\'etale prolongation functor}.

\begin{theorem}\cite[Definition 3.9]{higherme}
The essential image of $j_!$ is precisely (the functor of points of) the $\SCi$-Deligne-Mumford stacks.
\end{theorem}

\begin{definition}
Denote by $\mathbb{U}=j_!\left(1\right).$ This is the \textbf{universal $\SCi$-Deligne-Mumford stack}.
\end{definition}

\begin{theorem}\cite[Theorems 3.6,3.7]{higherme}
The $\i$-category $\DMsCi^{\et}$ itself is (equivalent to) an $\i$-topos and it is in fact the underlying $\i$-topos of the universal $\SCi$-Deligne-Mumford stack. Moreover, this topos is canonically equivalent to $\Shv\left(\sD^{\et}\right).$
\end{theorem}

\begin{remark}\cite[Proposition 3.2]{higherme}
The equivalence 
$$\DMsCi^{\et} \simeq \Shv\left(\sD^{\et}\right)$$ sends an $\SCi$-Deligne-Mumford stack $\cX$ to the sheaf $y^{\et}\left(\cX\right)$ on $\sD^{\et}$ which sends an object $D$ to the space of \'etale morphisms from $D$ to $\cX.$ Moreover, $j_!y^{\et}\left(\cX\right)\simeq \cX.$
\end{remark}

\begin{proposition}
For any $\SCi$-Deligne-Mumford stack $\left(\cX,\O_{\cX}\right),$ there is a canonical morphism of $\SCi$-ringed $\i$-topoi
$$\left(\bH/\cX,\O_{\bH}|_{\cX}\right) \to \left(\cX,\O_{\cX}\right)$$
\end{proposition}

\begin{proof}
By \cite[Proposition 2.4]{higherme}, there is a canonical geometric morphism
$$\Lambda:\bH/\mathbb{U} \to \Shv\left(\sD^{\et}\right)$$ whose inverse image functor is
\begin{eqnarray*}
\Lambda^*:\Shv\left(\sD^{\et}\right) &\longrightarrow & \bH/\mathbb{U}\\
\cF &\mapsto & j_!\left(\cF \to 1\right).
\end{eqnarray*}
So, by \cite[Proposition 6.3.5.8]{htt}, for any $\cF \in \Shv\left(\sD^{\et}\right),$ there is a pullback diagram of $\i$-topoi
$$\xymatrix{\left(\bH/\mathbb{U}\right)/\Lambda^*\left(\cF\right) \ar[d] \ar[r] & \Shv\left(\sD^{\et}\right)/\cF \ar[d]\\
\bH/\mathbb{U} \ar[r]_-{\Lambda} & \Shv\left(\sD^{\et}\right).}$$
Now, $\Lambda^*\left(\cF\right)=j_!\cF \to j_!1=\mathbb{U},$ so $\left(\bH/\mathbb{U}\right)/\Lambda^*\left(\cF\right) \simeq \bH/j_!\left(\cF\right).$ So, if $\cF=y^{\et}\left(\cX\right),$ then $$\left(\bH/\bU\right)/\Lambda^*\left(\cF\right)\simeq \bH/\cX.$$
Notice moreover that 
$$\Shv\left(\sD^{\et}\right)/y^{\et}\left(\cX\right) \simeq \DMsCi^{\et}/\left(\cX,\O_{\cX}\right) \simeq \cX.$$ The result now follows.
\end{proof}

\begin{remark}\label{rmk:lambda}
Denote by $\Lambda_{\cX}:\bH/\cX \to \cX$ the geometric morphism from the previous proposition. Then, $\Lambda^{\cX}_*$ is the restriction functor from the large topos $\bH/\cX$ to the small topos $\cX.$ In particular, $\Lambda^{\cX}_*\O_{\bH}|_{\cX}=\O_{\cX}.$ Hence the above geometric morphism has a canonical upgrade to a morphism of $\SCi$-ringed $\i$-topoi.
\end{remark}

\begin{lemma}\label{lem:coldmsame}
There exists a unique colimit preserving functor
$$\bH \to \Topi^{\SCi}$$ which sends each $\SCi$-Deligne-Mumford stack to itself.
\end{lemma}

\begin{proof}
Firstly, by \cite[Lemma 3.1]{higherme}, there is a canonical equivalence $$\bH \simeq \Shv\left(\DMsCi\right),$$ where a functor $\cF:\DMSCi^{op} \to \Spc$ is a sheaf if and only if for each $\SCi$-Deligne-Mumford stack $\left(\cX,\O_{\cX}\right),$ the composite
$$\cX^{op} \simeq \left(\DMSCi^{\et}/\left(\cX,\O_{\cX}\right)\right)^{op} \longrightarrow \left(\DMSCi\right)^{op} \stackrel{\cF}{\longrightarrow} \Spc$$ preserves small limits. A colimit preserving functor $$\bH \to \Topi^{\SCi}$$ is the same as a limit preserving functor $$\bH^{op} \to \left(\Topi^{\SCi}\right)^{op},$$ i.e. a sheaf with values in $\left(\Topi^{\SCi}\right)^{op}$, which can then be identified with a functor
$$\DMSCi \to \Topi^{\SCi}$$ such that each composite
$$\cX \simeq \DMSCi^{\et}/\left(\cX,\O_{\cX}\right) \longrightarrow \Topi^{\SCi}$$ preserves colimits. Now, observe that for each $\SCi$-Deligne-Mumford stack $\left(\cX,\O_{\cX}\right),$
the composite
$$\cX \simeq \left(\Topi^{\SCi}\right)^{\et}/\left(\cX,\O_{\cX}\right) \to \Topi^{\SCi}$$ preserves colimits, by \cite[Proposition 3.1]{higherme}.
\end{proof}

\begin{proposition}\label{lem:wdd}
There is a unique colimit-preserving functor
$$\bH \to \Fun\left(\Delta\left[1\right],\Topi^{\SCi}\right)$$
which sends each $\SCi$-Deligne-Mumford stack to the map of $\SCi$-ringed $\i$-topoi $$\Lambda_{\cX}:\left(\bH/\cX,\O_{\bH}|_{\cX}\right) \to \left(\cX,\O_{\cX}\right)$$ of Remark \ref{rmk:lambda}.
\end{proposition}

\begin{proof}
Firstly, let us show that there is a functor
\begin{eqnarray*}
\DMSCi & \longrightarrow& \Fun\left(\Delta\left[1\right],\Topi^{\SCi}\right)\\
\left(\cX,O_{\cX}\right) &\mapsto & \Lambda_{\cX}:\left(\bH/\cX,\O_{\bH}|_{\cX}\right) \to \left(\cX,\O_{\cX}\right).
\end{eqnarray*}
Consider the functor 
$$\DMSCi \to \Topi^{\SCi}$$ from Lemma \ref{lem:coldmsame}. 

By \cite[Theorem 3.6]{htt}, for each such $\SCi$-Deligne-Mumford stack, there is a unique \'etale map $\left(\cX,\O_{\cX}\right) \to \left(\mathbb{U},\O_{\mathbb{U}}\right).$ It follows that there is a canonical extension of the above functor to a functor
$$\DMSCi \to \Topi^{\SCi}/\left(\mathbb{U},\O_{\mathbb{U}}\right).$$ Furthermore, base change along the map
$$\Lambda:\left(\bH/\mathbb{U},\O_{\bH}|_{\mathbb{U}}\right) \to \left(\mathbb{U},\O_{\mathbb{U}}\right)$$ induces a functor 
\begin{eqnarray*}
\Topi^{\SCi}/\left(\mathbb{U},\O_{\mathbb{U}}\right) &\longrightarrow & \Fun\left(\Delta\left[1\right],\Topi^{\SCi}\right)\\
\left(\cB,\O_{\cB}\right) \to \left(\mathbb{U},\O_{\mathbb{U}}\right) & \mapsto & \left(\cB,\O_{\cB}\right) \times_{\left(\mathbb{U},\O_{\mathbb{U}}\right)} \left(\bH/\mathbb{U},\O_{\bH}|_{\mathbb{U}}\right) \to \left(\cB,\O_{\cB}\right).
\end{eqnarray*}
By composition, we get a functor $$F:\DMSCi \to \Fun\left(\Delta\left[1\right],\Topi^{\SCi}\right).$$ Unwinding the definitions, it can be identified with the desired one. 

To finish the proof, we need to verify that this functor is a cosheaf, or, in other words, the restriction to each $\cX$ preserves colimits. This will be the case if and only if both $\ev_0 \circ F$ and $\ev_1 \circ F$ preserve colimits. $\ev_1 \circ F$ can be identified with the canonical functor $$\cX \simeq \left(\Topi^{\SCi}\right)^{\et}/\left(\cX,\O_{\cX}\right) \to \Topi^{\SCi}$$ which we already saw preserves colimits. The second can be identified with the composite
$$\cX \stackrel{\Lambda_{\cX}^*}{\longlongrightarrow} \bH/\cX\simeq \left(\Topi^{\SCi}\right)^{\et}/\left(\bH/\cX\right) \to \Topi^{\SCi}$$ which also preserves colimits by \cite[Proposition 3.1]{higherme}.
\end{proof}

\section{Modules and Quasi-coherent sheaves}\label{sec:modandqc}

\subsection{Modules and square-zero extensions}

Recall that for a commutative ring $R,$ the canonical functor
\begin{eqnarray*}
\Ab\left(\mathbf{Com}\Alg/R\right) &\to & \Mod_R\\
\left(\pi_B:B \to R\right) &\mapsto & \ker\left(\pi\right)\\
\end{eqnarray*}
from the category of abelian group objects in the slice category to the category of $R$-modules is an equivalence of categories.
Consider the forgetful/free monadic adjunction
$$\xymatrix@C=2cm{\Ab\left(\mathbf{Com}\Alg/R\right) \ar@<-0.5ex>[r]_-{U} & \mathbf{Com}/R \ar@<-0.5ex>[l]_-{F}}$$

Under the equivalence $\Mod_R\simeq\Ab\left(\mathbf{Com}\Alg/R\right) \simeq \Mod_{R},$
for a module $M,$ $U\left(M\right)$ is the square-zero extension $$R\oplus \epsilon M \to R.$$ Also, $F\left(id_R\right)$ is isomorphic to the module of K\"{a}hler differentials $\Omega^1\left(R\right),$ since for any module $M,$
\begin{eqnarray*}
\Hom\left(F\left(id_R\right),M\right) &\cong & \Hom\left(id_R,R\oplus \epsilon M \to R\right)\\
& \cong & \mathfrak{Der}\left(A,M\right),
\end{eqnarray*}
which is precisely the universal property of the module $\Omega^1\left(R\right).$

We now generalize this story for algebras in $\dgc.$

\begin{definition}
Let $\cA \in \dgc.$ An \textbf{$\A$-module} is by definition a module for the underlying supercommutative dg-algebra $\A_\sharp$ in unbounded cochain complexes of $\Z_2$-graded real vector spaces. Denote this $\i$-category by $\Mod_\A.$ Denote the subcategory on algebras whose underlying cochain complex is concentrated in non-positive degrees by $\Mod^{\le 0}_\A.$
\end{definition}

Recall the following theorem:
\begin{lemma}\cite[Theorem 1.1]{schwedeship} \label{thm:boundedok}
For $\cA$ any $dg$-algebra, corresponding to a simplicial ring $\tilde \cA,$ there is an equivalence
$$\Mod_{\cA}^{\le 0} \simeq \Mod_{\tilde \cA}.$$
\end{lemma}

\begin{theorem}\label{thm:stabmod}
For any $\cA \in \dgc,$ there is a canonical equivalence of $\i$-categories
$$\Stab\left(\dgc/\cA\right) \simeq \Mod_{\A}.$$
\end{theorem}
\
\begin{proof}
Let $\cB$ be an $\SCi$-algebra with underlying supercommutative $\R$-algebra $\cB_{\sharp}.$ Since $\SCi$-completion preserves the terminal algebra $\R,$ the forgetful functor
$$\left(\blank\right)_\sharp:\left({\dgc}_{\cB}\right)_\ast \to \left({\mathsf{SCAlg}}_{\cB}\right)_\ast$$ from augmented $\cB$-algebras in $\SCi$-algebras, to augmented $\cB$-algebras in $\mathsf{SCAlg}_{\R}$ has a left adjoint, sending
$$\xymatrix@C=2cm{\cB_{\sharp} \ar@<5pt>[r]^{a} & \cA \ar@<5pt>[l]^{\pi}}$$ to
$$\xymatrix@C=2cm{\cB \ar@<5pt>[r]^{\eta_{\A} \circ a} & \widehat{\cA}, \ar@<5pt>[l]^{\widehat{\pi}}}$$
where $\eta$ is the unit of the adjunction between the $\i$-categories of un-augmented objects.  As both of these $\i$-categories are presentable, there is an induced adjunction between their stabilizations
$$\xymatrix@C=1.5cm{ \left({\mathsf{SCAlg}}_{\cB}\right)_{\ast} \ar@<+1ex>[r]^-{\Sigma^\infty_{+,{alg}}} &\Stab\left(\left({\mathsf{SCAlg}}_{\cB}\right)_\ast\right) \ar@<+1ex>[l]^-{\Omega^\infty_{alg}} \ar@<+1ex>[r]^-{\widehat{\left(\blank\right)}^\infty} & \Stab\left(\left({\dgc}_{\cB}\right)_\ast\right) \ar@<+1ex>[l]^-{{\left(\blank\right)_\sharp}^\infty} \ar@<-1ex>[r]_-{{\Omega}^\infty_{\Ci}} & \left({\dgc}_{\cB}\right)_\ast. \ar@<-1ex>[l]_-{\Sigma^\infty_{\Ci}} }$$

Consider the adjunction
$$\xymatrix{
\Mod_{\cB}^{\le 0} \ar@<+1ex>[r]^-{\Sym} &  \left({\mathsf{SCAlg}}_{\cB}\right)_{\ast} \ar@<+1ex>[l]^-{I}}$$
where $$\xymatrix{I:\A \ar@<+1ex>[r]^-{\pi} & \cB \ar@<+1ex>[l]^-{a}  & \ar@{|->}[r] & 0\times_{\cB} \A}$$ is the right derived functor of the augmentation ideal functor.
By \cite[Theorem 3.2.3]{schwede}, the induced adjunction $\Sym^\infty \dashv I^\i$ between their stabilizations is an equivalence. Since $\Stab\left(\Mod^{\le 0}_{\cB}\right) \simeq \Mod_{\cB},$ this implies $$\Stab\left( \left({\mathsf{SCAlg}}_{\cB}\right)_{\ast}\right)\simeq \Mod_{\cB}.$$ It therefore suffices to prove that $\widehat{\left(\blank\right)}^\infty \dashv \left(\blank\right)_\sharp$ is an adjoint equivalence. 

Notice that \cite[Corollary 3.1.4, Theorem 3.2.3]{schwede} imply that the essential image of 
$$\Sym^\i\Sigma^\infty_+:\Mod^{\le 0}_{\cB} \to \Stab\left(\left({\mathsf{SCAlg}}_{\cB}\right)_{\ast}\right)$$ 
is equivalent to the connective $\left({\mathsf{SCAlg}}_{\cB}\right)_{\ast}$-spectra objects. Let $M$ be a connective $\cB$-module in $\Mod^{\le 0}_{\cB}.$ To ease notation, let $F:=\widehat{\left(\blank\right)}^\infty$ and $G:=\left(\blank\right)_\sharp^\infty.$
We claim that the unit
$$\Sigma\Sym^\i \Sigma^\infty_+ M \to GF\Sigma\Sym^\i \Sigma^\infty_+ M,$$ is an equivalence.  Since $F$ and $\Sym^\i$ preserve colimits, we have
\begin{eqnarray*}
\Sigma\Sym^\i \Sigma^\infty_+ M &\simeq& \Sym^\i \Sigma^\infty_+ M\left[1\right].\\
&\simeq& \Sigma^\infty_{+,alg} \Sym M\left[1\right].
\end{eqnarray*}
It follows that 
$$F\Sigma\Sym^\i \Sigma^\infty_+ M \simeq \Sigma^\infty_{+,\Ci} \Sym_{\Ci} M\left[1\right].$$
For a general $\cB$-module $N,$
$$\Sym_{\Ci}\left(N\right)=\Sym_{\Ci}\left(N_0\right)  \underset{\Sym\left(N_0\right)} \otimes \Sym\left(N\right).$$
Since $M\left[1\right]_0=\R,$ we therefore have that $\left(\Sym_{\Ci} M\left[1\right]\right)_\sharp = \Sym M\left[1\right].$ It follows that the unit is an equivalence, as claimed. So the component of the unit is an equivalence for all $\left({\mathsf{SCAlg}}_{\cB}\right)_{\ast}$-spectra objects of the form $\Sigma N,$ for $N$ connective. By \cite[Corollary 6.2.2.16]{LurieHA}, since $\left(\blank\right)_\sharp^\i$ preserves sequential colimits, we can identify $G$ with the pointwise application of $\left(\blank\right)_\sharp^\i$ to spectra objects, and can conclude that $G$ preserves sequential colimits as well. Since every spectrum object is a sequential colimit of such objects of the form $\Sigma N$ with $N$ connective, and $F$ and $G$ preserve sequential colimits, it follows that the unit is an equivalence everywhere.

We now wish to show that that counit $\beta$ is also an equivalence. Notice that $G$ is conservative since $\left(\blank\right)_\sharp$ is. Therefore, is suffices to show that for any $\left(\dgc_{\cB}\right)_{\ast}$-spectrum object $X,$ $G\left(\beta_X\right)$ is an equivalence. By the triangle identity for adjunctions, we have that the following diagram commutes up to homotopy
$$\xymatrix@C=2cm{GX & \ar[l]_-{G\beta_X} GFG X\\ & GX \ar[lu]^-{id} \ar[u]}$$
where the vertical arrow is the unit of $GX,$ hence an equivalence. Therefore, $G\left(\beta_X\right)$ is an equivalence.
\end{proof}

Since $\dgc/\A$ has finite limits, there is an adjunction
$$\xymatrix{\left(\dgc/\A\right)_* \ar@<-0.5ex>[r]_-{\Sigma^\i} & \Stab\left(\dgc/\cA\right) \simeq \Mod_{\A} \ar@<-0.5ex>[l]_-{\Omega^\i}},$$
with $\Sigma^\i \dashv \Omega^\i.$

Concretely, given an $\A$-module $M,$ under the identification above, $\Omega^\i\left(M\right)$ is the square zero extension of $\A$ by $\tau_{\le 0} M.$ Consequently, we denote the composite
$$\Mod^{\le 0}_{\A} \hookrightarrow \Mod_{\A} \stackrel{\simeq}{\longrightarrow} \Stab\left(\dgc/\cA\right) \stackrel{\Omega^\i}{\longrightarrow} \left(\dgc/\A\right)_* \to \dgc/\cA$$ by $\A\left[\blank\right].$ 

For a given $\A$-module $M,$ $\A\left[M\right]\simeq \A \oplus M,$ as an $\A$-modules, and the projection map is the map making $\A\left[M\right]$ an object in $\dgc/A.$ Moreover, the canonical section $$a \mapsto \left(a,0\right)$$ is the map making it a pointed object. We use the following notation
$$\xymatrix{\A\left[M\right] \ar[r]^-{\pi_M} & \A \ar@/^1.65pc/^{0} [dlu]}.$$

Notice that the canonical functor $$\left(\dgc/\A\right)_* \to \dgc/\cA$$ has a left adjoint
\begin{eqnarray*}
\left(\blank\right)_{+}:\dgc/\cA &\to & \left(\dgc/\A\right)_*\\
\left(\B \to \A \right) & \to & \left(\A \to \B\oinfty \A \to \A\right).
\end{eqnarray*}

Denote the composite
$$\dgc/\cA \stackrel{\left(\blank\right)_+}{\longlongrightarrow} \left(\dgc/\A\right)_* \stackrel{\Sigma^\i}{\longrightarrow} \Stab\left(\left(\dgc/\A\right)_*\right) \stackrel{\sim}{\longrightarrow} \Mod_{\A}$$
by $\Omega^1_{\left(\blank\right)}\left(\A\right).$ We will return to this functor in Section \ref{sec:derivations}. For now, we just observe that by the composition of adjoints, we have
\begin{equation}\label{eqn:1form}
\Omega^1_{\left(\blank\right)}\left(\A\right) \dashv \A\left[\blank\right].
\end{equation}

\begin{remark}
Since the functor $\A\left[\blank\right]$ is a right adjoint, it preserves limits. In particular, we have that for $M_1,M_2,\ldots,M_n$ $\A$-modules
$$\A\left[\bigoplus \limits^{n}_{i=1} M_i\right] \simeq \A\left[M_1\right] \times_{\A} \A\left[M_2\right] \ldots \times_{\A} \A\left[M_n\right].$$
\end{remark}

\subsection{Derivations and $1$-forms}\label{sec:derivations}

\begin{definition}
Given an $\A$-module $M,$ the \textbf{space of $\Ci$-derivations of $\A$ with values in $M$} is the mapping space $$\Der\left(\A,M\right):=\Map_{\dgc/A}\left(id_{\A},\A\left[M\right]\right).$$ More generally, given a map $f:\B \to \A,$ the \textbf{space of $\Ci$-derivations of $\A$ along $f$ with values in $M$} is the space
$$\Der\left(\A,M\right)_f:=\Map_{\dgc/A}\left(f,\A\left[M\right]\right).$$
\end{definition}

\begin{remark}
Given a derivation $\varphi$ of $\A$ with values in $M,$
$$\xymatrix{\A\left[M\right] \ar[r]^-{\pi_M} & \A \ar@/^1.65pc/^{\varphi} [dlu]},$$ as a map of $\A$-modules,
\begin{eqnarray*}
\varphi:\A &\to& \A\left[M\right]\simeq \A \oplus M\\
a & \mapsto & \left(a,\cD_\varphi\left(a\right)\right),
\end{eqnarray*}
for a unique map of $\A$-modules $\cD_\varphi:\A \to M.$ The map $\cD_\varphi$ is a derivation in the classical sense, but the converse need not be true; a derivation of $\A$ as a $\Ci$-algebra needs to satisfy more conditions than simply being a derivation of the underlying $\R$-algebra $\A_{\sharp}.$ See \cite[Proposition 2.11]{dg2}. However for nice enough $\SCi$-rings (such as the algebra of functions on a supermanifold) any derivation on the sense of $\R$-algebras is a $\Ci$-derivation \cite[Corollary 2.18]{dg2}.
\end{remark}

\begin{definition}
For $n \in \Z,$ let $\R\left[n\right]$ denote the graded vector space with $\R$ in degree $-n,$ and otherwise zero, regarded as a cochain complex with zero differential. Let $\mathbb{E} \in \Mod_{\A},$ and let $$\mathbb{E}\left[n\right]:=\mathbb{E} \underset{\R} \otimes \R\left[n\right].$$
\end{definition}

\begin{remark}
In concrete terms, we have $\left(\mathbb{E}\left[n\right]\right)^k=\mathbb{E}^{n+k}$ with the same differential as before, i.e. you shift the entire complex to the left by $n.$
\end{remark}

\begin{remark}
Since $\Mod_{\A}$ is a stable $\i$-category, it comes equipped with a suspension functor 
$$\Sigma:\Mod_{\A} \stackrel{\sim}{\longrightarrow} \Mod_{\A}.$$ Strictly speaking, there is a difference between $\Sigma\mathbb{E}$ and $\mathbb{E}\left[1\right].$ The functor $\Sigma$ is given by $\R\left[1\right] \underset{\R} \otimes \left(\blank\right).$ For a given $\A$-module $\mathbb{E}^\bullet$, and any integer $k,$ we have
$$\left(\Sigma\mathbb{E}^\bullet\right)^k=\mathbb{E}^{k+1}=\left(\mathbb{E}^\bullet\left[1\right]\right)^k,$$ but the differentials are different. The differential on $\mathbb{E}^\bullet\left[1\right]^k=\mathbb{E}^{k+1}$ is simply $d_{k+1},$ whereas the differential for $\Sigma \mathbb{E}^{\bullet}$ is $-d_{k+1}.$ Clearly, these two $\A$-modules are canonically equivalent. Also, of course, if $\mathbb{E}$ has zero differential (for example is $\R\left[m\right]$, for some $m$), then we do have $\Sigma^n\mathbb{E}=\mathbb{E}\left[n\right].$

The functor $\Omega:\Mod_{\A} \stackrel{\sim}{\longrightarrow} \Mod_{\A},$ is defined to be $$\Sigma^{-1}=\R\left[-1\right] \underset{\R} \otimes \left(\blank\right)$$ and is inverse to $\Sigma.$ For any cochain complex, the following is a (homotopy) pullback diagram
$$\xymatrix{\Omega\left(\mathbb{E}\right) \ar[r] \ar[d] & 0 \ar[d] \\ 0 \ar[r] & \mathbb{E},}$$  hence the notation.
\end{remark}

\begin{definition}
For all $n \in \Z,$ let $$\A\left[\xi_n\right]:=\A\left[\A\left[-n\right]\right],$$ i.e. the square zero extension associated to the $\A$-module $\A\left[-n\right].$ Concretely, 
$$\A\left[\xi_n\right]=\A\left[\mathbf{x}\right]/\mathbf{x}^2,\mspace{10mu} |\mathbf{x}|=n.$$
In the supercommutative case, $\A$-modules are $\Z \times \Z_2,$ graded, so we mean that $\mathbf{x}$ is in degree $\left(n,0\right).$ We also define $\Pi\A\left[\xi_n\right]:=\A\left[\Pi\A\left[-n\right]\right],$ where $\Pi$ is the parity-reversal operator. Concretely, 
$$\Pi\A\left[\xi_n\right]=\A\left[\mathbf{x}\right],\mspace{10mu} |\mathbf{x}|=\left(n,1\right).$$ (We do not need to impose that $\mathbf{x}^2=0$; this is automatic since $\mathbf{x}$ has odd $\Z_2$-parity)
\end{definition}

\begin{remark}
Since $\A\left[\blank\right]$ is a right adjoint, it follows that for any $n \in Z,$ we have a pullback diagram
$$\xymatrix{\A\left[\xi_{n+1}\right] \ar[r] \ar[d] & \A \ar[d]^-0\\ \A \ar[r]^-0 & \A\left[\xi_n\right],}$$ and similarly for the odd versions.
\end{remark}

\begin{definition}
Let $n \in \mathbb{N}.$ Define the \textbf{$n$-shifted first order $1$-disk} be the $\Ci$-scheme $$\mathbb{D}_{\left(1\right)}\left[n\right]:=\Speci\left(\R\left[\xi_{n}\right]\right).$$ In the setting of supergeometry, we also denote the above by $$\mathbb{D}_{\left(1\right)}\left[n\right]^{\left(1,0\right)}:=\Speci\left(\R\left[\xi_{n}\right]\right)$$ for clarity, and call it the \textbf{$n$-shifted first order \emph{even} $1$-disk}, and we denote 
$$\mathbb{D}_{\left(1\right)}\left[n\right]^{\left(0,1\right)}:=\Speci\left(\Pi\R\left[\xi_{n}\right]\right)$$ and call it the \textbf{$n$-shifted first order \emph{odd} $1$-disk}
\end{definition}

\begin{remark}
The first order odd $1$ disk is the same as the odd line
$$\mathbb{D}_{\left(1\right)}^{\left(1,0\right)} = \R^{\left(0|1\right)}.$$
\end{remark}

Notice that the canonical functor $$\left(\dgc/\A\right)_* \to \dgc/\cA$$ has a left adjoint
\begin{eqnarray*}
\left(\blank\right)_{+}:\dgc/\cA &\to & \left(\dgc/\A\right)_*\\
\left(\B \to \A \right) & \to & \left(\A \to \B\oinfty \A \to \A\right).
\end{eqnarray*}

\begin{definition}
Recall the adjunction (\ref{eqn:1form}). For $f:\B \to \A$ any map in $\dgc,$ we say that the module $\Omega^1_{f}\left(\A\right)$ is the \textbf{module of differential $1$-forms relative to $f$}. When $f=id_\A,$ we denote $\Omega^1_{f}\left(\A\right)$ simply by $\Omega^1\left(\A\right),$ and call it the \textbf{module of differential $1$-forms of $\A.$}
\end{definition}

\begin{remark}
$\Omega^1\left(\A\right)$ is the \emph{cotangent complex} $\mathbb{L}_{X}$ of $X=\Speci\left(\A\right).$ More generally, $\Omega^1_{f}\left(\A\right)$ is the \emph{relative cotangent complex}.
\end{remark}

\begin{proposition}\cite[Example 2.2.18]{Nuiten},
Let $\cM$ be a supermanifold, then $$\Omega^1\left(\Ci\left(\cM\right)\right) \simeq \Omega^1_{dR}\left(\cM\right).$$
\end{proposition}

Due to the adjunction (\ref{eqn:1form}), we have that
$$\Der\left(A,M\right)_f\simeq \Map_{\Mod_\A}\left(\Omega^1_f\left(\A\right),M\right).$$
Note that the stable $\i$-category $\Mod_{\A}$ is canonically enriched in itself. For $N$ and $M$ two $\A$-modules, denote their mapping module by $\Mod_{\A}\left(N,M\right).$ In light of this, and the above, we make the following definition:

\begin{definition}
Let $M$ be an $\A$-module and let $f:\B \to \A.$ The \textbf{$\A$-module of derivations of $\A$ along $f$ with values in $M$} is 
$$\underline{\Der}\left(\A,M\right)_f:=\Mod_{\A}\left(\Omega^1_f\left(\A\right),M\right),$$ and the \textbf{$\A$-module of derivations of $\A$} is
$$\underline{\Der}\left(\A\right):=\Mod_{\A}\left(\Omega^1\left(\A\right),\A\right).$$
We also denote $\underline{\Der}\left(\A\right)$ by $\mathfrak{X}\left(\A\right)$ and refer to it as the \textbf{$\A$-module of vector fields on $\A.$}
\end{definition}

\begin{remark}\label{rmk:DKmap}
The space of derivations is related to the module of derivations in the following way
$$\Der\left(\A,M\right)_f \simeq \mathcal{DK}\left(\tau_{\le 0}\underline{\Der}\left(\A,M\right)_f\right),$$ where $\mathcal{DK}$ of a connective $\A$-module is the corresponding simplicial $\A$-module via the Dold-Kan correspondence.
\end{remark}

\begin{remark}
If $\cM$ is a $\mathbb{N}$-graded supermanifold, $\mathfrak{X}\left(\Ci\left(\cM\right)\right) \simeq \mathfrak{X}\left(\cM\right),$ that is the $\Ci\left(\cM\right)$-module of derivations of $\Ci\left(\cM\right)$ is equivalent to the $\Ci\left(\cM\right)$-module of graded vector fields on $\cM,$ in the usual sense.
\end{remark}

\subsection{Tangent categories}

\begin{definition}
Let $\sC$ be an $\i$-category with finite limits. The \textbf{tangent category} to $\sC$ is the $\i$-category of excisive functors $$\T\sC:=\mathbf{Exc}\left(\Spc^{\fin}_*,\sC\right)$$ from finite pointed spaces to $\sC.$
\end{definition}

 Let $F:\Spc^{\fin} \to \sC$ be any functor. Then it is the same data as the induced functor $$\Spc^{\fin} \to \sC/F\left(*\right),$$ which sends a space $X$ to $F\left(X\right) \to F\left(*\right).$ Now, suppose instead that $$F:\Spc^{\fin}_* \to \sC.$$ Consider the induced functor $$\tilde F:\Spc^{\fin}_* \to \sC/F\left(*\right).$$ If $* \to X$ is a pointed space, then as we have $* \to X \to *$ is the identify of $*,$ we have that $\tilde F\left(X\right)$ is canonically pointed by $$\tilde F\left(* \to X\right):id_{F\left(*\right)} \to \tilde F\left(X\right).$$ So we have an induced functor $$\Spc^{\fin}_* \to \left(\sC/F\left(*\right)\right)_*$$ which is reduced. This functor moreover is excisive if and only if $F$ is, in which case we get a reduced excisive functor $$\Spc^{\fin}_* \to \left(\sC/F\left(*\right)\right)_*,$$ i.e. an object of $\Stab\left(\sC/F\left(*\right)\right).$ Now suppose that $f:C \to D,$ then since $\sC$ has finite limits, there exists a functor $$f^*:\sC/D \to \sC/C$$ given by pulling back along $f,$ right adjoint to the functor induced by composition with $f.$ Hence $f^*$ preserves limits and thus induces a functor $$\Stab\left(f\right):\Stab\left(\sC/D\right) \to \Stab\left(\sC/C\right).$$
Now suppose that $\alpha:F \Rightarrow G$ is an arrow in $\T \sC,$ and $Z$ is an object in $\Spc^{\fin}_*.$ Then consider the naturality square for $\alpha$ at $Z$:
$$\xymatrix{F\left(Z\right) \ar[rrr]^-{\alpha\left(Z\right)} \ar[ddd] \ar@{-->}[rd] & & & G\left(Z\right)  \ar[ddd]\\
& \alpha\left(\ast\right)^*\left(G\left(Z\right)\right) \ar[ldd] \ar[rru] & & \\
& & & \\
F\left(\ast\right) \ar[rrr]^-{\alpha\left(\ast\right)} & & & G\left(\ast\right)}$$
This part of an explicit unwinding of a canonical equivalence $$\T\sC \simeq \int_{\sC} \Stab\left(\sC/\blank\right).$$ In particular, evaluation at the terminal object $\ast,$
\begin{equation}\label{eq:q}
q:=\mathbf{ev}_\ast:\T\sC  \to \sC
\end{equation}
is a Cartesian fibration. 

\begin{corollary}
$$\T\left(\dgc\right) \simeq \int_{\dgc} \Mod^{\ast}_{\left(\blank\right)},$$
where $\Mod^{\ast}_{\left(\blank\right)}:\left(\dgc\right)^{op} \to \widehat{\mathbf{Cat}}_\i$
is the functor induced by restriction of scalars.
\end{corollary}

\begin{remark}\label{rmk:cocart}
Note that, since taking the induced module is left adjoint to restriction of scalars, it follows that there is a canonical equivalence of $\i$-categories
$$\T\left(\dgc\right) \simeq \int_{\dgc} \Mod^{!}_{\left(\blank\right)},$$
where for each $\varphi:\A \to \B$
\begin{eqnarray*}
\Mod^{!}_{\left(\varphi\right)}:\Mod_{\A} & \to & \Mod_{\B} \\ M &\mapsto & \B \underset{\A} \otimes M.
\end{eqnarray*}
In particular, $q$ is also a coCartesian fibration, and hence a bifibration.
\end{remark}

Let us consider also the canonical functor induced by evaluation at $S^0$ 
$$\pi_\sC:=\mathbf{ev_{S^0}}:\T\sC \to \sC.$$ Note that for each $C,$ $\mathbf{ev_\ast}^{-1}\left(C\right) \simeq \Stab\left(\sC/C\right),$ and thus $\pi_\sC$ restricts to a functor
$$\Stab\left(\sC/C\right) \to \sC/C,$$
which can be identified with $\Omega_{\sC}^\i.$

\begin{remark}
If follows that when $\sC=\dgc,$ for each $\A,$ the restriction of $\pi_{\sC}$ can be identified with
$$\A\left[\blank\right]:\Mod_{\A} \to \dgc/A.$$
\end{remark}
 
Note that the $1^{st}$ Taylor approximation $P_1$ (in the sense of Goodwillie calculus) shows that $\T\sC=\mathbf{Exc}\left(\Spc^{\fin}_*,\sC\right)$ is a reflective subcategory of $\Fun\left(\Spc^{\fin}_*,\sC\right),$ and hence presentable and closed under limits in $\Fun\left(\Spc^{\fin}_*,\sC\right)$. Moreover, the functor $\pi_{\sC}$ is given by evaluation at $S^0,$ and hence preserves limits. By the adjoint functor theorem, this implies that there exist a left adjoint $\mathbb{L} \dashv \pi_{\sC}.$ 

\begin{remark}\label{rmk:tan_topos}
	It is well known that if $\cE$ is an $\i$-topos, then $\T\cE$ is also. By the above, this follows from the observation that the Taylor approximation functor $P_1$ is left exact.
	\end{remark}

\begin{definition}
For $\sC=\dgc,$ $\mathbb{L}$ is the absolute cotangent complex functor.
\end{definition}

\begin{proposition}
For $\A \in \dgc,$ $$\mathbb{L}\left(\A\right)=\left(\A,\Omega^1\left(\A\right)\right).$$
\end{proposition}

\begin{proof}
Let $\A \in \dgc,$ and consider $\mathbb{L}\left(\A\right) \in \T \dgc.$ Notice that $\mathbb{L}\left(\A\right)$ is in the fiber of
$$\pi:\tau \left(\dgc\right) \to \dgc$$
over $\A,$ which can be identified with $\Mod_{\A}.$ Now, let $M \in \Mod_{\A}.$ We have a pullback diagram
$$\xymatrix{\Map_{\Mod_{\A}}\left(\mathbb{L}\left(\A\right),M\right) \ar[r] \ar[d] & \Map_{\T \dgc}\left(\mathbb{L}\left(\A\right),M\right) \ar[d] \\
\ast \ar[r]^-{id_\A} & \Map_{\dgc}\left(\A,\A\right).}$$
But since $$\Map_{\T \dgc}\left(\mathbb{L}\left(\A\right),M\right) \simeq \Map_{\dgc}\left(\A,\pi_{\sC}\left(M\right)\right),$$ and $\pi_{\sC}\left(M\right) \simeq \A\left[\blank\right],$ we can identify $\Map_{\Mod_{\A}}\left(\mathbb{L}\left(\A\right),M\right)$ with $$\Map_{\dgc/A}\left(id_\A,\A\left[\blank\right]\right)\simeq \Map_{\Mod_{\A}}\left(\Omega^1\left(\A\right),M\right).$$
\end{proof}

\subsection{Sheaves of modules}\label{sec:shvmod}
\subsubsection{$\SCi$-ringed topoi}
Recall (Definition \ref{dfn:sctopos}) that an $\SCi$-ringed $\i$-topos is a pair $\left(\cE,\O_{\cE}\right)$ of an $\i$-topos $\cE$ with a sheaf $\O_{\cE}$ on $\cE$ with values in $\dgc.$

\begin{remark}
If $\cE=\Spc,$ $\cO_{\cE}$ may be identified with a $\SCi$-ring object in spaces, i.e. an object of $\dgc.$
\end{remark}

\begin{definition}
(Recall the coCartesian fibration $q$ (\ref{eq:q}).) Let $\left(\cE,\cO_{\cE}\right)$ be an $\SCi$-ringed $\i$-topos. An \textbf{$\cO_{\cE}$-module} is a sheaf $\mathscr{M}$ on $\cE,$ with values in $\T\dgc,$ which is a lift of the structure sheaf:
$$\xymatrix@C=2cm@R=2cm{ & \T\dgc \ar[d]^-{q}\\ \cE^{op} \ar[r]^-{\cO_{\cE}} \ar@{-->}[ur]^-{\mathscr{M}}  & \dgc.}$$ I.e., the $\i$-category fits into a pullback diagram
$$\xymatrix{ \Mod_{\cO_{\cE}} \ar[r] \ar[d] & \Shv\left(\cE,\T\dgc\right) \ar[d]^-{q_!}\\
\ast \ar[r]^-{\cO_{\cE}} & \Shv\left(\cE,\dgc\right).}$$
\end{definition} 

Less formally, an $\cO_{\cE}$-module is a sheaf $\mathscr{M}$ of $\R$-modules on $\cE$ together with the structure of an $\cE\left(E\right)$-module on $\mathscr{M}\left(E\right)$ for each $E \in \cE$, such that for each $f:E \to F$ in $\cE,$ the canonical map $$\mathscr{M}\left(f\right):\mathscr{M}\left(F\right) \to \mathscr{M}\left(E\right)$$ is a map of $\cE\left(F\right)$-modules.

\begin{remark}
Notice that we have a pullback diagram
$$\xymatrix{\T\dgc \ar[d] \ar[r] & \T\mathsf{ComAlg} \ar[d]\\
\dgc \ar[r] & \mathsf{ComAlg},}$$
where $\mathsf{ComAlg}$ is the $\i$-category of $\mathbb{E}_\i$-ring spectra. As such, an $\cO_{\cE}$-module sheaf is precisely the same as a module sheaf for the underlying sheaf of $\mathbb{E}_\i$-rings. Such a sheaf is the same as an $\mathbb{E}_\i$-ring object in $\cE.$ A module for it is a module for this $\mathbb{E}_\i$-ring object. This point of view is useful in establishing functoriality results.
\end{remark}

\begin{remark}
For any $\SCi$-ringed $\i$-topos $\left(\cE,\cO_{\cE}\right),$ there is a canonical limit-preserving functor
$$U_{\cE}:\Mod_{\cO_{\cE}} \to \Stab\left(\cE\right)\simeq \Shv\left(\cE,\Spt\right),$$ which object-wise assigns to $E \in \cE,$ the underlying spectrum of $\mathscr{M}\left(E\right),$ i.e. the underlying $\mathbb{S}$-module of the underlying $H\mathbb{R}$-module (equivalently chain complex in $\mathbb{R}$-modules) of $\mathscr{M}\left(E\right),$ where $\mathbb{S}$ is the sphere spectrum. To be more precise, if the structure sheaf is $\Z_2$-graded, we have
$$U_{\cE}:\Mod_{\cO_{\cE}} \to \Stab\left(\cE\right)\simeq \Shv\left(\cE,\Spt\right)^2,$$ associating a spectrum object to both the even and odd components of the module sheaf, but we will not dwell on this.
\end{remark}

\begin{definition}
Let  $\cX$ be an object in an $\SCi$-ringed $\i$-topos $\left(\cE,\O\right).$ Then the functor
$$\Tot_{\cX}:=\Omega_{\cE/\cX}^{\i} \circ U_{\cE/\cX}:\Mod_{\cO_{\cX}} \to \cE/\cX,$$
where $\Omega_{\cE/\cX}^{\i}:\Spt\left(\cE\right) \to \cE$ is the right adjoint to $\Sigma^\i_{\cE/\cX,+},$
is called the \textbf{total space} functor.
\end{definition}

\begin{definition}
Given a derived stack $\cX \in \bH,$ an \textbf{$\cO_\cX$-module} is a module sheaf for the $\SCi$-ringed $\i$-topos $\left(\bH/\cX,\cO|_\cX\right).$
\end{definition}

\subsection{Functoriality of module sheaves}\label{sec:funmod}
The $\i$-category of $\SCi$-ringed $\i$-topoi may be defined informally as the $\i$-category whose objects are $\SCi$-ringed $\i$-topoi, and whose morphisms $\left(\cE,\cO_{\cE}\right) \to \left(\cF,\O_{\cF}\right)$ consists of a geometric morphism $$f:\cE \to \cF,$$ together with a morphism of sheaves of $\SCi$-algebras over $\cF$ $$\alpha:\O_{\cF} \to f_*\O_{\cE}.$$ This can be rigorously constructed as a Cartesian fibration over $\Topi,$ c.f. \cite[Section 2.4]{higherme}.

Given such a morphism, if $\mathscr{M}$ is an $\O_{\cE}$-module, then $$f_{\star}\left(\mathscr{M}\right):=\alpha^{\star}\left(f_*\mathscr{M}\right)$$ is an $\cO_{\cF}$-module, where $\alpha^{\star}$ is the functor given by restriction of scalars. This assembles into a functor
$$f_{\star}:\Mod_{\O_\cE} \to \Mod_{\O_\cF}.$$
Notice also that if $\mathscr{N}$ is an $\O_\cF$-module, then $\mathscr{N}$ has the canonical structure of a $f^*\O_\cF$-module, and hence 
$$f^\star\left(\mathscr{N}\right):=\O_\cE \underset{f^*\O_\cF} \otimes f^*\mathscr{N}$$
has the canonical structure of a $\cE$-module. This assembles into a functor
$$f^\star:\Mod_{\O_\cF} \to \Mod_{\O_\cE},$$ which is readily checked to be left adjoint to $f_\star$ \cite[Section 2.5]{LurieSAG}.

\subsubsection{Working internally}

Note that this is not special for for $\SCi$-rings. There is a completely analogous story where $\cE$ is allowed to take values in not-necessarily commutative algebras. We will make the convention that in the non-commutative case, $\Mod_{\O_\cE}$ is the $\i$-category of left $\O_\cE$-modules. The $\i$-category of right $\O_\cE$-modules is thus $\Mod_{\left(\O_\cE\right)^{op}}.$ This will be important later when we discuss $\cD$-modules.

Here is a special case. Suppose that $\left(\cE,\O_{\cE}\right)$ is an $\SCi$-ringed $\i$-topos, and that $\cR$ and $\cS$ are sheaves of not necessarily commutative or bounded $\O_{\cE}$-algebras. An $\cO_{\cE}$-algebra homomorphism $\varphi:\cS \to \cR$ is the same data as a morphism of ringed $\i$-topoi $$f:\left(\cE,\cR\right) \stackrel{\left(id,\alpha\right)}{\longlonglongrightarrow} \left(\cE,\cS\right).$$ A pertinent example to have in mind is $\cE=\Sh\left(M\right)$ for a manifold $M,$ $\cE=\Delta\left(\R\right)$---the constant sheaf with value $\R,$ $\cS=\Ci_M$---smooth functions on $M,$ and $\cR=\cD_M$---the sheaf of linear differential operators on $M,$ locally of finite degree. In this case, there are the standard change of ring functors $$\alpha_! \dashv \alpha^\star \dashv \alpha_\star,$$ with $\alpha_!$ the induced module construction $\cS \underset{\cR} \otimes \left(\blank\right),$ $\alpha^\star$ restriction of scalars, and $\alpha_\star$ the co-induced module functor. Given a $\cS$-module $N,$ the underlying $\cO_{\cE}$-module of $\alpha_\star N$ is the same as $\Mod_\cS\left(\cR,N\right),$ but where the left $\cR$-module structure is induced from that on $\cS$ via $\alpha.$ Notice that we have $$f^\star=\alpha_!\mbox{,   }f_\star=\alpha^\star.$$ As such, we will make the definition
$$f^!:=\alpha_\star.$$


\subsubsection{Module Spec}
Consider the special case when $\left(\cF,\O_{\cF}\right)=\left(\Spc,\Gamma\left(\cO_{\cX}\right)\right),$ where we are using the canonical identification $$\Shv\left(\Spc,\dgc\right)\simeq \dgc,$$ and where $f:\cE \to \Spc$ is the (homotopically) unique map, and $\alpha$ is the identity (since $f_*\O_{\cE}=\Gamma\O_{\cE}$). In this case
$$f_\star=\Gamma:\Mod_{\O_\cE} \to \Mod_{\Gamma\O_\cE}$$ is the global sections functor.

\begin{definition}
With $f$ as in the preceding paragraph, we call $f^\star$ the \textbf{module spec} functor, and denote it $$\Mspec_\cE:\Mod_{\Gamma\O_\cE} \to \Mod_{\O_\cE}.$$
\end{definition}

\begin{lemma}\label{lem:stars}
Suppose that $$\xymatrix{\cE \ar[rd]^-{g} \ar[d]_-{f} & \\ \cF \ar[r]^-{h} & \mathcal{G}}$$ is a commutative diagram of $\SCi$-ringed $\i$-topoi. Then
$$g_\star \simeq h_\star \circ f_\star \mbox{ and } g^{\star} \simeq f^\star \circ h^\star.$$
\end{lemma} 

\begin{proof}
The former follows immediately from the fact that for the direct image functors for the corresponding geometric morphisms satisfies $g_*\simeq h_* \circ f_*.$ The latter then follows from uniqueness of adjoints.
\end{proof}

\begin{corollary} \label{cor:mspec}
For any morphism $f:\cE \to \cF$ of $\SCi$-ringed $\i$-topoi, the following diagram commutes up to homotopy:
$$\xymatrix@C=2.5cm@R=2cm{\Mod_{\Gamma\left(\O_\cF\right)} \ar[r]^-{\Mspec_\cF} \ar[d]_-{\Gamma\left(f\right)_{!}} & \Mod_{\O_\cF} \ar[d]^-{f^\star} \\ \Mod_{\Gamma\left(\O_\cE\right)} \ar[r]^-{\Mspec_\cE} & \Mod_{\O_\cE},}$$ where $$\Gamma\left(f\right)_!\left(M\right)=\Gamma\left(\O_\cE\right) \underset{\Gamma\left(\O_\cF\right)} \otimes M.$$
\end{corollary}

\subsection{Quasi-coherent sheaves}\label{sec:qc}

\begin{definition}\label{def:qc}
	Let $\A \in \dgc$ be a $\Ci$-algebra. We can regard $\Speci\left(\A\right)$ as a $\Ci$-ringed $\i$-topos, rather than simply a topological space, i.e. we can consider the pair $$\left(\Shv\left(\underline{\Speci\left(\A\right)}\right),\cO_A\right).$$ Denote by 
	$$\Mspec_\A:\Mod_\A \to \Mod_{\cO_\A}$$
	the functor $\Mspec_\cE$ for $\left(\cE,\cO_\cE\right)$ the above $\Ci$-ringed $\i$-topos. An $\cO_\A$-sheaf over $\underline{\Speci\left(\A\right)}$ is \textbf{quasi-coherent} if is in the essential image of $\Mspec_\A.$ Denote the full subcategory of $\cO_\A$-modules on the quasi-coherent sheaves by $\QC\left(\Speci\left(\A\right)\right).$
	\end{definition}

\begin{proposition}\label{prop:concqc}
For a complete $\Ci$-ring $\A$ and a given $\A$-module $M,$ $\Mspec\left(M\right)$ is the sheafification of the presheaf $\tilde M$ which assigns to an open subset $U$ the module $$\cO_\A\left(U\right) \underset{\A} \otimes M.$$
\end{proposition}

\begin{proof}
	Since $\A$ is complete, $\A\simeq\Gamma\left(\O_\A\right),$ so the unique geometric morphism 
	$$\Shv\left(\underline{\Speci\left(\A\right)}\right) \to \Spc$$
	canonically gives rise to a morphism of $\SCi$-ringed $\i$-topoi 
	$$t:\left(\Shv\left(\underline{\Speci\left(\A\right)}\right),\cO_A\right) \to \left(\Spc,\A\right).$$
	Denoting by $i$ the inclusion of the sheaf $\i$-category
$\Shv\left(\underline{\Speci\left(\A\right)}\right)$ into the presheaf $\i$-category $\Psh\left(\mathsf{Op}\left(\underline{\Speci\left(\A\right)}\right)\right),$
and denoting by $a$ its left adjoint (sheafification), the pair $\left(i,a\right)$ constitutes a geometric morphism $$\Shv\left(\underline{\Speci\left(\A\right)}\right) \to \Psh\left(\mathsf{Op}\left(\underline{\Speci\left(\A\right)}\right)\right).$$ By abuse of notation, denote the induced map of $\SCi$-ringed $\i$-topoi
$$i:\left(\Shv\left(\underline{\Speci\left(\A\right)}\right),O_\A\right) \to \left(\Psh\left(\mathsf{Op}\left(\underline{\Speci\left(\A\right)}\right)\right),i_*\O_\A\right).$$
Since $\Spc$ is the terminal $\i$-topos, there is a unique factorization of $t$ as
$$\xymatrix{\left(\Shv\left(\underline{\Speci\left(\A\right)}\right),O_\A\right) \ar[d]_-{i} \ar[r]^-{t} & \left(\Spc,\A\right)\\
\left(\Psh\left(\mathsf{Op}\left(\underline{\Speci\left(\A\right)}\right)\right),i_*\O_\A\right). \ar[ru]_-{s} &}$$
Note that $s^*$ is the constant presheaf functor, so for $M$ any $\A$-module and $U$ any open subset,
$$s^\star\left(M\right)\left(U\right)=\O_\A\left(U\right) \underset{\A} \otimes M.$$
Now, for $\cF$ any $i_*\O_\A$-module, since $i^*i_*\O\A  \simeq \O_\A$ since $i^*=a$ and $\O_\A$ is a sheaf, it follows that $i^\star \cF \simeq a\left(\cF\right),$ i.e. $i^\star \cF$ is simply the sheafification of $\cF.$ Combining these two observations with the fact that
$$\Mspec\left(M\right)=t^\star M \simeq i^\star \left(s^\star M\right),$$
yields the desired result.
\end{proof}

\begin{theorem} \label{thm:5.1.33}\cite[Lemma 5.1.33]{Nuiten},
	Let $\A$ be a complete $\SCi$-ring. Then every $\cO_\A$-module is quasi-coherent, and hence $$\QC\left(\Speci\left(\A\right)\right) = \Mod_{\cO_\A}.$$ Moreover, the right adjoint $\Gamma_\A \vdash \Mspec_\A$ is fully faithful, and the adjunction $\Mspec_\A \dashv \Gamma_\A$ restricts to an equivalence of $\i$-categories between $\Mod_{\cO_\A}$ and the full subcategory of $\Mod_{\A}$ on those $\A$-modules $M$ such that the unit $$\eta_{M}:M \to \Gamma_\A\left(\Mspec\left(M\right)\right)$$ is an equivalence.
	\end{theorem}

\begin{definition}
The full subcategory of $\Mod_{\A}$ on those sheaves for which the unit is an equivalence is called the subcategory of \textbf{complete} $\A$-modules, which we denote by $\widehat{\Mod}_{\A}$. We also denote the reflector $\Gamma \circ \Mspec$ by $\widehat{\left(\blank\right)}.$
\end{definition}

\begin{remark}\label{rmk:presMod}
	Since $\widehat{\Mod}_{\A}$ is a reflective localization of the presentable $\i$-category $\Mod_{\A},$ $\widehat{\Mod}_{\A}$ is also presentable.
	\end{remark}

\begin{remark}\cite[Example 5.1.36]{Nuiten}
Since the subcategory $\widehat{\Mod}_{\A},$ is reflective, it is closed under limits and retracts Moreover, for complete $\SCi$-algebras $\A,$ 
\begin{eqnarray*}
\Gamma \cO_{\Speci\left(\A\right)} &\simeq& \A \\
&\simeq & \Gamma \Mspec \A,
\end{eqnarray*}
so $\A \in \widehat{\Mod}_{\A}.$ It follows that $\widehat{\Mod}_{\A}$ contains all 
finitely presented $\A$-modules.
\end{remark}

\begin{remark}
As $\widehat{\Mod}_{\A} \simeq \Mod_{\Speci\left(\A\right)},$ it follows that it is a symmetric monoidal with an inherited tensor product $\widehat{\otimes}$ from that of $\Mod_{\Speci\left(\A\right)}$. By checking on stalks, one can see that in fact
$$M \widehat{\otimes} N \simeq \widehat{M \otimes N}.$$
\end{remark}

Let $\varphi:\A \to \cB$ be a morphism of $\SCi$-rings. Then we have the induced morphism of $\SCi$-ringed $\i$-topoi $$\Speci\left(\varphi\right):\left(\Shv\left(\underline{\Speci\left(\cB\right)}\right),\O_\cB\right) \to \left(\Shv\left(\underline{\Speci\left(\cA\right)}\right),\O_\cA\right).$$ Consider the functor
$$\Speci\left(\varphi\right)^\star:\Mod_{\O_\A} \to \Mod_{\O_\cB}.$$ Under the identifications
$\Mod_{\O_\A} \simeq \widehat{\Mod}_\A,$ and $\Mod_{\O_\cB} \simeq \widehat{\Mod}_{\cB}$, this induces a functor $$\widehat{\Mod}_\A \to \widehat{\Mod}_\cB$$

\begin{definition}\label{dfn:comind}
	For a complete $\A$-module $M,$ we denote the image of $M$ under the above functor by
	$$\cB \underset{\A} {\widehat{\otimes}} M,$$ and call it the \textbf{completed induced module of $M$}.
	\end{definition}

\begin{proposition}
	There is a canonical identification
	$$\widehat{\cB \underset{\A} {\otimes} M} \simeq \cB \underset{\A} {\widehat{\otimes}} M.$$
	\end{proposition}

\begin{proof}
	On one hand, $\widehat{\cB \underset{\A} {\otimes} M}$ is global sections of the sheafification of the presheaf on $\mathbf{Bas}\left(\Speci\left(\cB\right)\right)$ which assigns to $U_b=\Speci\left(\cB\left[b^{-1}\right]\right)$ the module
	$$\cB\left[b^{-1}\right] \underset{\cB} {\otimes} \left[\cB \underset{A} {\otimes} M\right].$$ On the other hand, $\cB \underset{\A} {\widehat{\otimes}} M$ is global sections of the sheafification of the presheaf which assigns to $U_b$ the module $$\cB\left[b^{-1}\right] \underset{\A} {\otimes} M,$$ but these presheaves are canonically equivalent.
	\end{proof}

\begin{lemma}\label{lem:sheafslice}
	Let $\left(\cE,\O_\cE\right)$ be an $\SCi$-ringed $\i$-topos. Then the composite functor
\begin{eqnarray*}
	\cE^{op} &\to& \widehat{\mathbf{Cat}}_\i\\
	E &\mapsto& \Mod_{\O_{\cE}|_E}
	\end{eqnarray*}
is a sheaf (i.e. preserves small limits).
\end{lemma}

\begin{proof}
This follows almost verbatim by the proof of \cite[Remark 2.1.11]{dagviii}.
\end{proof}	

Denote by $\QC$ the composite functor
$$\sD^{op} \hookrightarrow \left(\dgc\right) \stackrel{\Speci}{\longlongrightarrow} \left(\mathfrak{Top}_\i^{\SCi}\right)^{op} \stackrel{\Mod}{\longlongrightarrow} \widehat{\mathbf{Cat}}_\i$$ sending an algebra $\A$ to the $\i$-category $\Mod_{\O_A} \simeq \widehat{\Mod}_\A.$

\begin{proposition}\label{prop:QCsheaf}
	The above presheaf of large $\i$-categories is a sheaf.
	\end{proposition}

\begin{proof}
Let $\sC$ be an $\i$-category with small limits. Define a functor $G:\left(\mathfrak{Top}_\i^{\SCi}\right)^{op} \to \sC$ to be a sheaf if and only if for each $\left(\cE,\O_\cE\right),$ $\chi_{\cE}^*G:\cE^{op} \to \sC$ preserves small limits, where
\begin{eqnarray*}
	\chi_{\cE}:\cE &\to& \mathfrak{Top}_\i^{\SCi}\\
	E &\mapsto& \left(\cE/E,\O_{\cE}|_E\right).
	\end{eqnarray*}
	 For each $\left(\cE,\O_{\cE}\right),$ the sheaf from Lemma \ref{lem:sheafslice} is precisely $\chi_{\cE}^*\Mod,$ hence $\Mod$ is a sheaf on $\mathfrak{Top}_\i^{\SCi}.$ Now let $D$ be an element of $\sD,$ and let $\left(U_\alpha \hookrightarrow D\right)$ be an open cover by elements of $\sD.$ Then in the topos of sheaves $\Shv\left(\underline{\Speci}\left(\A\right)\right)$ on the underlying space of $D,$ the colimit of the associated \v{C}ech nerve of the cover is the terminal object (the whole space $D$). This means that $D$ is the \emph{limit} in the opposite category, and since the restriction of $\Mod$ preserves limits, we have that $\QC\left(D\right)$ is the limit of the cosimplicial diagram of $\QC\left(U_\alpha\right).$ Thus $\QC$ is a sheaf.
	\end{proof}
	
	The sheaf $\QC$ of $\i$-categories on $\sD$ is equivalent to a sheaf of $\i$-categories on $\bH.$ Concretely, one may take the left Kan extension of $\QC$ along the Yoneda embedding of $\sD$ into $\bH$ of $\QC.$ We will abuse notation and write this also as
	$$\QC:\bH^{op} \to \widehat{\Cat}_\i.$$ 
	
\begin{definition}
Let $E \in \bH$ be a derived stack. Then the $\i$-category of \textbf{quasi-coherent sheaves} on $E$ is the $\i$-category $\QC\left(E\right).$ For $f: Y \to E$ any morphism in $\bH$, we denote by $f^\star$ the functor
$$\QC\left(f\right):\QC\left(E\right) \to \QC\left(F\right)$$
and call it the \textbf{inverse image} functor.
\end{definition}

\begin{remark}\label{rmk:comind}
Let $f:\Speci\left(\cB\right) \to \Speci\left(\A\right)$ be a morphism in $\sD.$ Then, upon the identifications of quasi-coherent sheaves with complete modules:
$$\QC\left(\Speci\left(\cB\right)\right)\simeq \widehat{\Mod}_{\cB}$$ and
$$\QC\left(\Speci\left(\cA\right)\right)\simeq \widehat{\Mod}_{\cA},$$
$$f^\star:\widehat{\Mod}_{\cA} \to \widehat{\Mod}_{\cB}$$ can be identified with the completed induced module functor (Definition \ref{dfn:comind}).
\end{remark}

\begin{corollary}\label{cor:limrep}
	Let $\cX \in \bH,$ then the $\i$-category of quasi-coherent sheaves can be presented as
	$$\QC\left(\cX\right) \simeq \underset{\Speci\left(\A\right) \to X} \lim \widehat{\Mod}_\A,$$ with $\A$ ranging over all of $\sD.$
\end{corollary}

\begin{proof}
	The sheaf $\cO_{\bH}^*\QC$ on $\bH$ is a limit preserving functor
	$$\bH^{op} \to \widehat{\mathbf{Cat}}_\i.$$ Since $\sD$ is a site for $\bH,$ it follows that $$\cX \simeq \underset{\Speci\left(\A\right) \to \cX} \colim \Speci\left(\A\right),$$ and by the preservation of limits we then have an identification
	\begin{eqnarray*}
		\QC\left(\cX\right) &\simeq & \underset{\Speci\left(\A\right) \to \cX} \lim \QC\left(\Speci\left(\A\right)\right)\\
		&\simeq& \underset{\Speci\left(\A\right) \to \cX} \lim \widehat{\Mod}_\A
	\end{eqnarray*}
\end{proof}

We now wish to show that for any $E,$ $\QC\left(E\right)$ has the canonical presentation as a full subcategory of $\Mod_{\O_E}.$ We will start with the case that $E=\Speci\left(\A\right)$ is an affine object of $\sD.$ 

\begin{lemma}\label{lem:geomorph}
	For $\A$ any $\Ci$-ring, there is pair of geometric morphisms $$\Lambda:\bH/\Speci\left(\A\right) \to \Shv\left(\underline{\Speci\left(\A\right)}\right)$$ and $$\Psi:\Shv\left(\underline{\Speci\left(\A\right)}\right) \to \bH/\Speci\left(\A\right)$$ such that $\Lambda_*\simeq \Psi^*$ and $\Lambda^*$ is fully faithful.
\end{lemma}

\begin{proof}
	Without loss of generality, assume that $\A$ is complete. Denote by $$j:\mathbf{Bas}\left(\Speci\left(\A\right)\right) \to \sD/\Speci\left(\A\right)$$ the functor sending each basic open subset $U_a$ to $$\Speci\left(\A\left[a^{-1}\right]\right) \to \Speci\left(\A\right).$$ Since this functor preserves covers, there is a restriction functor $$j^*:\bH/\Speci\left(\A\right) \to \Shv\left(\underline{\Speci\left(\A\right)}\right)$$
	with a left adjoint $j_!.$ Moreover, it follows from \cite[Proposition 5.27]{higherme} that $j_!$ preserves finite limits, hence the pair $\left(j^*,j_!\right)$ constitutes a geometric morphism $\Lambda$ with $\Lambda^*=j_!$ and $\Lambda_*=j^*.$ The functor $j^*$ also has a natural \emph{right} adjoint $j_*,$ and the pair $\left(j_*,j^*\right)$ is a geometric morphism $\Psi$ with $$\Psi^*=j^*$$ and $$\Psi_*=j_*.$$ The fully faithfulness of $\Lambda^*$ follows by an analogous proof to that of \cite[Proposition 5.27]{higherme}.
\end{proof} 

\begin{corollary}
	With the notation of the preceding theorem, $$\Lambda \circ \Psi \simeq id.$$
\end{corollary}

\begin{proof}
	Since $\Lambda^*$ is fully faithful, it follows that the unit $$id \to \Lambda_*\Lambda^*$$ is an equivalence, but since $\Lambda_* = \Psi^*,$ we have
	\begin{eqnarray*}
		id & \simeq & \Psi^* \Lambda^*\\
		& \simeq & \left(\Lambda \circ \Psi\right)^*
	\end{eqnarray*}
	It follows that $\Lambda \circ \Psi \simeq id.$
\end{proof}

\begin{corollary}
	Let $X=\Speci\left(\A\right) \in \bH.$  Then there is a triple adjunction between $\Mod_{\cO_\A}$ and $\Mod_{\cO_X}$ given by $$\Lambda^\star \dashv \Lambda_\star=\Psi^\star \dashv \Psi_\star.$$ Moreover, $\Lambda^\star$ and $\Psi_\star$ are both fully faithful.
\end{corollary}

\begin{proof}
	
	Firstly, notice that 
	\begin{eqnarray*}
	\Lambda_*\cO_X&=& j^*\cO_X\\
	&\simeq& \cO_A.
	\end{eqnarray*}
	This canonical equivalence $\cO_\A \to \Lambda_*\cO_X$ equips the geometric morphism $\Lambda$ with the structure of a map of $\SCi$-ringed $\i$-topoi. The inverse equivalence, after identifying $\Psi^*$ with $\Lambda_*$ is a morphism
	$$\Psi^*\cO_X \to \cO_\A$$ and is adjoint to a map $$\cO_\cX \to \Psi_*\cO_\A$$ which equips $\Psi$ with the structure of a map of $\SCi$-ringed $\i$-topoi.

	The first statement of the corollary now follows immediately from the fact that $\Psi^*=\Lambda_*.$ The second statement follows from the fact that the unit map $$id \to \Psi^\star \Psi_\star$$ can be identified with the canonical equivalence $\Lambda_\star \Psi_\star \simeq id$ and the counit map $$\Lambda^*\Lambda_* \to id$$ can be identified with the canonical equivalence $$\Lambda^\star \Psi^\star \simeq id.$$
\end{proof}

\begin{theorem}\label{thm:QCff}
	For $X=\Speci\left(\A\right) \in \bH,$ the fully faithful functor $$\Lambda^\star:\Mod_{\cO_\A} \hookrightarrow \Mod_{\cO_X}$$ identifies $$\QC\left(\Speci\left(\A\right)\right)=\Mod_{\cO_\A}$$ with the subcategory $\Mod_{\cO_X}$ spanned by those $\O_\cX$-modules of the form $\Mspec_{\bH/X}\left(M\right)$ for an $\A$-module $M.$
\end{theorem}

\begin{proof}
	Denote by $$t_\A:\left(\Shv\left(\underline{\Speci\left(\A\right)}\right),\cO_\A\right) \to \left(\Spc,\A\right)$$ and
	$$t_X:\left(\bH/X,\cO_X\right) \to \left(\Spc,\A\right)$$ the canonical morphisms. Observe that the composite
	$$\left(\Shv\left(\underline{\Speci\left(\A\right)}\right),\cO_\A\right) \stackrel{\Psi}{\longrightarrow} \left(\bH/X,\cO_X\right) \stackrel{t_X}{\longrightarrow} \left(\Spc,\A\right)$$ is equivalent to $t_\A$ and the composition
	$$\left(\bH/X,\cO_X\right) \stackrel{\Lambda}{\longrightarrow} \left(\Shv\left(\underline{\Speci\left(\A\right)}\right),\cO_\A\right) \stackrel{t_\A}{\longrightarrow} \left(\Spc,\A\right)$$ is equivalent to $t_X.$ 
	
	By Theorem \ref{thm:5.1.33}, $t_\A^\star=\Mspec_{\A}$ is essentially surjective, so the essential image of $\Lambda^\star$ is the same as the essential image of $\left(t_\A \circ \Lambda\right)^\star.$ But this is the same as the essential image of $t_X^\star=\Mspec_{\bH/X}.$
	\end{proof} 

\begin{proposition}\label{prop:blabla}
Let $\cX \in \bH$ be arbitrary and let $\cF$ be an $\O_\cX$-module. Then, for any $X=\Speci\left(\A\right) \in \sD$ (with $\A$ complete) and $f:\Speci\left(\A\right) \to \cX,$ $\cF\left(f\right)$ is a complete $\A$-module.
\end{proposition}

\begin{proof}
Note that $\cF\left(f\right)$ is equivalent to global sections of $f^\star \cF,$ but this is the same as global sections of $\Lambda_\star f^* \cF$ over $\Shv\left(\underline{\Speci\left(\A\right)}\right),$ and thus is contained in the essential image of 
$$\left(t_{\A}\right)_\star:\Mod_{\O_\A} \to \Mod_\A,$$ which is precisely the complete $\A$-modules.
\end{proof}

\begin{corollary}\label{cor:mspecon}
Let $X=\Speci\left(\A\right).$ Then if $\cF \in \QC\left(X\right),$ if $$f:\Speci\left(\cB\right) \to \Speci\left(\A\right)$$ is in $\sD,$ then $\Lambda_X^\star \cF\left(f\right)$ is the complete $\cB$-module $$\cB \underset{\A} {\widehat{\otimes}} \cF\left(\Speci\left(\A\right)\right).$$
\end{corollary}

\begin{proof}
By Theorem \ref{thm:QCff}, $$\Lambda_X^\star \cF \simeq \Mspec_{\bH/X}\left(\cF\left(\Speci\left(\A\right)\right)\right).$$ Following the proof of Proposition \ref{prop:concqc}, for $M$ any $\A$-module, we can identify $\Mspec_{\bH/X}\left(M\right)$ with the sheafification over $\sD/X$ of the presheaf $\tilde M$ that assigns each $f:\Speci\left(\cB\right) \to X$ the $\cB$-module
$\cB \underset{\A} \otimes M.$ The value of the sheafification of $\tilde M$ on $f$ can be identified with the global sections of the sheafification of the restriction of $f^*\tilde M$ to the poset of open subsets of $\Speci\left(\cB\right).$  This is readily seen to be equivalent to the completed induced module $$\cB \underset{\A} {\widehat{\otimes}} M.$$ Applying this observation to $M=\cF\left(\Speci\left(\A\right)\right)$ yields the desired result.
\end{proof}

\begin{proposition}\label{prop:qcafcl}
Let $X=\Speci\left(\A\right),$ then $\cF \in \Mod_{\O_X}$ is quasi-coherent if and only if for any $f:\Speci\left(\B\right) \to X$ in $\sD,$ the canonical map
$$\cB \underset{\A} {\widehat{\otimes}} \cF\left(X\right) \to \cF\left(f\right)$$ is an equivalence.
\end{proposition}

\begin{proof}
The module sheaf $\cF$ is quasi-coherent if and only if $\cF \simeq \Lambda^\star_{X} \Lambda_\star*^X \cF.$ So, applying Corollary \ref{cor:mspecon} to $\Lambda_\star \cF$ implies that if $\cF$ is quasi-coherent, the canonical map in the proposition is an equivalence. 

Conversely, by Theorem \ref{thm:5.1.33}, $\MSpec_\A =t_\A^\star$ is fully faithful, so the co-unit $$t_\A^\star t^\A_\star \to id$$ is an equivalence. Moreover, recall from the proof of Theorem \ref{thm:QCff} that $t_X \simeq t_\A \circ \Lambda.$ It follows that
\begin{eqnarray*}
t^\star_X t^X_\star &\simeq & \left(t_\A \circ \Lambda\right)^\star \circ \left(t_\A \circ \Lambda\right)_\star\\
&\simeq & \Lambda^\star \circ t_\A^\star \circ t^\A_\star \circ \Lambda_\star\\
& \simeq & \Lambda^\star \Lambda_\star.
\end{eqnarray*}
Thus, $$\Lambda^\star_{X} \Lambda_\star^X \simeq \MSpec\left(\cF\left(X\right)\right).$$ By the proof of Corollary \ref{cor:mspecon}, for any $f:\Speci\left(\B\right) \to X$ in $\sD,$ we have
$$\MSpec\left(\cF\left(X\right)\right)\left(f\right) \simeq \cB \underset{\A} {\widehat{\otimes}} \cF\left(X\right).$$ The result now follows.
\end{proof}



\begin{lemma}\label{lem:orbqc}
Let $\left(\cX,\O_{\cX}\right)$ be an $\SCi$-Deligne-Mumford stack and let $Y\left(\cX\right) \in \bH$ be its functor of points. Then
$$\QC\left(Y\left(\cX\right)\right) \simeq \Mod_{\O_{\cX}}.$$
\end{lemma}

\begin{proof}
By Lemma \ref{lem:coldmsame}, there exists a unique colimit preserving functor $$\bH \to \Topi^{\SCi}$$ which sends $Y\left(\cX\right)$ to $\left(\cX,\O_{\cX}\right)$ for all $\SCi$-Deligne-Mumford stacks $\left(\cX,\O_{\cX}\right).$ For any such $\cX,$ 
the composite
$$\cX^{op} \simeq \left(\DMSCi^{\et}/\cX\right)^{op} \to \left(\DMSCi\right)^{op} \hookrightarrow \bH^{op} \stackrel{\Mod}{\longlongrightarrow} \widehat{\Cati}$$ preserves limits as it can be identified with $\Mod_{\O_{\cX}},$ which is a sheaf on $\cX.$ So the assignment $$\left(\cX,\O_{\cX}\right) \mapsto \Mod_{\O_{\cX}}$$ is a sheaf on $\DMSCi.$ By \cite[Lemma 3.1]{higherme}, the restriction functor
$$\Shv\left(\DMSCi\right) \to \Shv\left(\sD\right)=\bH$$ is an equivalence, and the restriction of this sheaf agrees with $\QC$ on affine $\SCi$-schemes. The result now follows.
\end{proof}

By Lemma \ref{lem:sheafslice}, the functor 
$$\Mod_\O:\bH^{op} \to \widehat{\Cati}$$ C sending $E \in \bH$ to $\Mod_{\O_{E}}$ is a sheaf. Moreover, $\QC$ is a sheaf by definition. Both of these functors take values in presentable $\i$-categories, and send morphisms to accessible colimit-preserving functors. It follows from \cite[Proposition 5.5.3.13]{htt} that they both lift to sheaves with values in $\PrL$--- the $\i$-category of presentable $\i$-categories and accessible colimit-preserving functors.

\begin{proposition}\label{prop:lamlam}
There is a canonical morphism $$\Lambda^\star:\QC \to \Mod_\O$$ in $\Shv\left(\bH,\PrL\right)$ which restricts to $\Lambda_X^\star$ for all $X \in \sD.$ Moreover, for any object $E \in \bH,$ $\Lambda_E^\star$ is fully faithful and thus $\QC\left(E\right)$ can be realized as a coreflective subcategory of $\Mod_{\O_E}$.
\end{proposition}

\begin{proof}
By Lemma \ref{lem:wdd}, there is a colimit preserving functor
$$\bH \to \Fun\left(\Delta\left[1\right],\Topi^{\SCi}\right)$$ which sends each affine element $X=\Speci\left(\A\right) \in \sD,$ to
$$\Lambda_X:\left(\bH/X,\O_X\right) \to \left(\Shv\left(\underline{\Speci\left(\A\right)}\right),\O_\A\right).$$ Composition with the functor
$$\Mod:\left(\Topi^{\SCi}\right)^{op} \to \PrL$$ induces a functor
$$\bH^{op} \to \Fun\left(\Delta\left[1\right],\PrL\right),$$
which we can canonically identify with a functor
$$\mu:\bH^{op} \times \Delta\left[1\right] \to \PrL,$$
in other words a natural transformation
$$\mu:\mu_0=\mu\left(\blank,0\right) \Rightarrow \mu\left(\blank,1\right)=\mu_1.$$
We will prove that $\mu_0=\QC,$ $\mu_1=\Mod_{\O},$ and that for all $X \in D,$
$\mu_X=\Lambda^\star_X.$ The latter is true by construction, and in particular, for all $X \in D,$ $\mu_0\left(X\right)=\QC\left(X\right)$ and $\mu_1\left(X\right)=\Mod_{\O_X}.$
Thus, it suffices to prove that $\mu_0$ and $\mu_1$ are sheaves. However, this is automatic since by construction, both functors are limit preserving functors from $\bH^{op}$ to $\PrL.$ Thus $\Lambda^\star:=\mu$ above is the desired morphism.

Now lets show that $\Lambda^\star$ is component-wise fully faithful. Since both $\QC$ and $\Mod$ are sheaves, regarded as functors $$\bH^{op} \to \widehat{\Cati},$$ they are limit preserving. Given any $E \in \bH,$ we can write $$E \simeq \underset{D \to E} \colim D,$$ with the colimit is indexed by $\sD/E.$ It follows that
$$\QC\left(E\right) \simeq \underset{E \to D} \lim \QC\left(D\right)$$ and
$$\Mod_{\O_E} \simeq \underset{E \to D} \lim \Mod_{\O_D}.$$
It is well-known that mapping spaces in limits of $\i$-categories can be computed as limits of mapping spaces. Using this, the naturality of $\Lambda,$ and the fact that $\Lambda_D^\star$ is fully faithful for all $D \in \sD,$ we have for any $\cF,\cG \in \QC\left(E\right),$ the following string of natural equivalences:
\begin{eqnarray*}
\Map_{\QC\left(E\right)}\left(\cF,\cG\right) &\simeq & \underset{D \stackrel{f}{\to} E}{\lim} \Map_{\QC\left(D\right)}\left(f^\star \cF,f^\star \cG\right)\\
&\simeq& \underset{D \stackrel{f}{\to} E}{\lim} \Map_{\Mod_{\O_D}}\left(\Lambda_D^\star f^\star \cF,\Lambda_D^\star f^\star \cG\right)\\
&\simeq & \underset{D \stackrel{f}{\to} E}{\lim} \Map_{\Mod_{\O_D}}\left(f^\star \Lambda_E^\star \cF,f^\star \Lambda_E^\star \cG\right)\\
&\simeq & \Map_{\Mod_{\O_E}}\left(\Lambda_E^\star \cF,\Lambda_E^\star \cG\right).
\end{eqnarray*}
Moreover, since $\Lambda_E^\star$ is a morphism in $\PrL,$ it has a right adjoint $\Lambda^E_\star,$ so $\QC\left(E\right)$ is coreflective in $\Mod_{\O_E}.$
\end{proof}

\begin{remark}
Given that for any stack $\cX,$ the functor
$$\Lambda_{\cX}^\star:\QC\left(\cX\right) \hookrightarrow \Mod_{\O_{\cX}}$$ is fully faithful, we will speak of an $\O_\cX$-module $\cF$ as \emph{being} quasicoherent if it's in the essential image of $\Lambda^\star_{\cX}.$
\end{remark}

The fact that $\Lambda$ is a natural transformations means that for any morphism $f:\cX \to \cY$ of derived stacks, there is a commutative diagram
	$$\xymatrix@C=2.5cm{\QC\left(\cX\right) \ar@{^{(}->}[r]^-{\Lambda_{\cX}^\star} & \Mod_{\O_\cX}\\
	\QC\left(\cY\right) \ar[u]^-{\QC\left(f\right)} \ar@{^{(}->}[r]^-{\Lambda_{\cY}^\star} & \Mod_{\O_\cY} \ar[u]_{f^\star}},$$ 
with the horizontal functors fully faithful. This justifies our abuse of notation in writing $\QC\left(f\right)$ as $f^\star.$
	
We warn that $f_\star$ \textbf{need not} send quasi-coherent sheaves, to quasi-coherent sheaves. Nonetheless, $\QC\left(f\right)$ does admit a right adjoint.

\begin{lemma}\label{lem:wadj}
Let $R:\mathscr{C} \hookrightarrow \mathscr{D}$ be a reflective subcategory of an $\i$-category $\mathscr{D},$ with reflector $L,$ and suppose that $\tilde R:\widetilde{\mathscr{C}} \hookrightarrow \widetilde {\mathscr{D}}$ is a reflective subcategory with reflector $\tilde L.$ Suppose further that there is a commutative diagram of functors
$$\xymatrix{\widetilde{\mathscr{C}} \ar[d]_-{F} \ar[r]^-{\tilde R} & \widetilde{\mathscr{D}} \ar[d]^-{G}\\
\mathscr{C} \ar[r]_-{R} & \mathscr{D},}$$
and that $G$ has a left adjoint $H$. Then the composite $\tilde L \circ H \circ R$ is left adjoint to $F.$
\end{lemma}

\begin{proof}
Let $C$ be an element of $\mathscr{C}$ and $Y$ an element of $\widetilde{\mathscr{C}}.$ Then we have
\begin{eqnarray*}
\Map_{\widetilde{\mathscr{C}}}\left(\tilde L H R\left(C\right),Y\right) &\simeq& \Map_{\widetilde {\mathscr{D}}}\left(H R\left(C\right),\tilde R\left(Y\right)\right)\\
&\simeq& \Map_{\mathscr{D}}\left(R\left(C\right),G \tilde R\left(Y\right)\right)\\
&\simeq& \Map_{\mathscr{D}}\left(R\left(C\right),R F\left(Y\right)\right)\\
&\simeq& \Map_{\mathscr{C}}\left(C,F\left(Y\right)\right),
\end{eqnarray*}
with the last equivalence following since $R$ is fully faithful.
\end{proof}

\begin{remark}
By the dual of Lemma \ref{lem:wadj}, we conclude that for any morphism $f:\cX \to \cY$ in $\bH,$ the functor
$$f^\star:\QC\left(\cY\right) \to \QC\left(\cX\right)$$ has a right adjoint $f^{\QC}_\star.$ Explicitly, it is given by $\Lambda^\cY_\star \circ f_\star \circ \Lambda_\cX^\star.$
\end{remark}

\begin{lemma}\label{lem:pbqce}
Let $\cX \in \bH$ and $\cF \in \Mod_{\O_\cX}.$ Then $\cF$ is quasi-coherent if and only if for all $\varphi:\Speci\left(\A\right) \to \cX$ with $\Speci\left(\A\right) \in \sD,$ $\varphi^\star\cF$ is quasi-coherent.
\end{lemma}

\begin{proof}
Since we have equivalences
$$\Mod_{\O_\cX} \simeq \underset{f:D \to \cX} \lim \Mod_{\O_D}$$
$$\QC\left(\cX\right) \simeq \underset{f:D \to \cX} \lim \QC\left(D\right),$$
we can think of an $\O_{\cX}$-module $\cG$ as a family of compatible $\O_{D}$-modules $\cG_f$ for each $f:D \to \cX,$ satisfying homotopy coherence, and similarly we can think of a quasi-coherent sheaf $\cH$ over $\cX$ as a family of compatible quasi-coherent sheaves $\cH_f$ over $D$ for each such $f.$ Since we can view quasi-coherent sheaves over $D$ as sitting inside $\O_D$-modules as a full subcategory, and the functor $f^\star$ for $\O_D$-modules restricts to one between quasi-coherent sheaves, we see that when viewing $\cF$ as a family $\left(\cF_f\right)_{f:D \to \cX},$ $\cF$ being quasi-coherent is just the condition that each $\cF_f,$ is. But under the equivalence 
$$\Mod_{\O_\cX} \simeq \underset{f:D \to \cX} \lim \Mod_{\O_D}$$ we have $\cF_f \simeq f^\star \cF.$ This can be made formal by identifying $\Mod_{\O_{\cX}}$ and $\QC\left(\cX\right)$ with the $\i$-categories of Cartesian sections of the Cartesian fibrations
$$\int_{\sD/\cX} \pi^*_\cX\Mod \to \sD/\cX$$
and
$$\int_{\sD/\cX} \pi^*_\cX\QC \to \sD/\cX$$
respectively.

\end{proof}

\begin{theorem}\label{thm:qccond}
Let $\cX \in \bH$ be an arbitrary derived stack. Then the following conditions are equivalent for an $\O_\cX$-module $\cF$:
\begin{itemize}
	\item[i)] $\cF$ is quasi-coherent (i.e. in the essential image of the fully-faithful functor $\Lambda^\star_\cX$).
	\item[ii)] For all morphisms
	$$\xymatrix{\Speci\left(\cB\right) \ar[rr]^-{\varphi} \ar[rd]_-{g} & & \Speci\left(\A\right) \ar[ld]^-{f}\\
		& \cX & }$$ in $\sD/\cX,$ the canonical map
	$$\cB \underset{\A} {\widehat{\otimes}} \cF\left(f\right) \to \cF\left(g\right)$$ is an equivalence.
	\end{itemize}
	(Note that by Proposition \ref{prop:blabla}, $\cF$ always takes values in complete modules, and above we are using the completed induced module functor $\cB \underset{\A} {\widehat{\otimes}} \left(\blank\right).$)
\end{theorem}

\begin{proof}
Suppose that $\cF=\Lambda_\cX^\star \cG$ is a quasi-coherent sheaf on $\cX$. Then 
\begin{eqnarray*}
\cF\left(g\right)&\simeq &\cF\left(f \circ \varphi\right)\\
&\simeq & \Gamma_{\Speci\left(\cB\right)}\left(\varphi^\star \left(f^\star\cF\right)\right)\\
&\simeq & \left(t_{\cB}\right)_\star \Lambda^{\Speci\left(\cB\right)}_\star \left(\varphi^\star \left(f^\star\cF\right)\right) \in \widehat{\Mod}_{\cB}.
\end{eqnarray*}
Under the identification $$\widehat{\Mod}_{\cB} \simeq \QC\left(\Speci\left(\cB\right)\right),$$ the complete $\cB$-module on the last line above corresponds to 
\begin{eqnarray*}
\Lambda^{\Speci\left(\cB\right)}_\star \left(\varphi^\star \left(f^\star\cF\right)\right) &\simeq &\Lambda^{\Speci\left(\cB\right)}_\star \left(\varphi^\star \left(f^\star \Lambda_{\cX}^\star \cG\right)\right)\\
&\simeq & \Lambda^{\Speci\left(\cB\right)}_\star \left(\varphi^\star \left(\Lambda_{\Speci\left(\A\right)}^\star f^\star  \cG\right)\right)\\
&\simeq & \Lambda^{\Speci\left(\cB\right)}_\star \Lambda_{\Speci\left(\cB\right)}^\star \varphi^\star f^\star \cG\\
& \simeq &  \varphi^\star f^\star \cG.
\end{eqnarray*}
Switching back to complete modules, the quasi-coherent sheaf $f^\star \cG$ on $\Speci\left(\A\right)$ is identified with its global sections, which is nothing but the complete $\A$-module $\cF\left(f\right).$ The result now follows since the operation $\varphi^\star$ on quasi-coherent sheaves corresponds to the complete induced module construction, by Remark \ref{rmk:comind}.

Conversely, notice that condition $ii)$ is equivalent to stating that for all $f:D \to \cX,$ $f^\star$ satisfies the necessary and sufficient conditions of Proposition \ref{prop:qcafcl} for quasi-coherence. Hence, we are done by Lemma \ref{lem:pbqce}.
\end{proof}

\section{The de Rham Stack and Formal Neighborhoods}\label{sec:deRham}

\begin{definition} \label{dfn:deRham}
Given $\cX \in \bH,$ its \textbf{de Rham stack} $\cX_{dR}$ is $k_*k^*\cX.$ Note that the unit of the adjunction $k^* \dashv k_*$ furnishes a canonical map $$\eta_{\cX}:\cX \to \cX_{dR}.$$
\end{definition}

Lets unwind this definition a bit:

\begin{eqnarray*}
\cX_{dR}\left(\sp A\right) & \simeq & \Map\left(y\left(\sp A\right),k_*k^*\cX\right)\\
& \simeq & \Map\left(k^*y\left(\sp A\right),k^*\cX\right)\\
&\simeq& \Map\left(p_!y\left(\sp A\right),k^*\cX\right),
\end{eqnarray*}

and since $p_! \simeq \Lan_y \left(y_{red} \circ p\right),$ we have

\begin{eqnarray*}
\cX_{dR}\left(\sp A\right) &\simeq& \Map\left(y_{red}\left(p\left(\sp A\right)\right),k^*\cX\right)\\
&\simeq& k^*\cX\left(p\left(\sp A\right)\right)\\
&\simeq& \cX\left(kp\left(\sp A\right)\right).
\end{eqnarray*}

In other words, the de Rham stack $\cX_{dR}$'s value on $\sp$ of an algebra $A$ is the same as the value of $\cX$ on $\sp$ of the reduction of $A.$
As a consequence, the functor $$\left(\blank\right)_{dR}:\bH \to \bH$$ is idempotent, i.e. there is a canonical equivalence $$\left(\blank\right)_{dR} \circ \left(\blank\right)_{dR} \simeq \left(\blank\right)_{dR}.$$

\begin{remark}
The functor $\left(\blank\right)_{dR}:=k_*k^*$ is a comonad. It moreover has both a left and right adjoint given by $k_!k^*$ and $p^*p_*$ respectively. Consequently, it preserves all small limits and colimits.
\end{remark}

\subsection{Formal completions and formal neighborhoods}

\begin{definition}
Let $f:\cY \to \cX$ be a map in $\bH.$ The \textbf{formal completion of $\cY$ along $f$,} denoted by $\widehat{\cY}_f,$ is the pullback
$$\xymatrix{\widehat{\cY}_f \ar[d]_-{\hat f} \ar[r] & \cY_{dR} \ar[d]^-{f_{dR}}\\
\cX \ar[r]_-{\eta_{\cX}} & \cX_{dR}.}$$
The naturality square for $\eta$ yields a canonical map $\iota_f:\cY \to \widehat{\cY}_f$ factorizing $f$ as
$$\cY \stackrel{\iota_f}{\longrightarrow} \widehat{\cY}_f \stackrel{\hat f}{\longrightarrow} \cX.$$ Viewing $\hat f$ as a generalized point, the map 
$$\hat f:\widehat{\cY}_f \to \cX$$ is defined to be the \textbf{formal neighborhood of $\cX$ at $f$}, for which we introduce another notation
$$\cX_{\dpr{f}}=\widehat{\cY}_f.$$
\end{definition}

\begin{proposition}\label{prop:drsame}
The canonical map $$\left(\iota_f\right)_{dR}:\cY_{dR} \to \left(\widehat{\cY}_f\right)_{dR}$$ is an equivalence.
\end{proposition}

\begin{proof}
This follows formally from the fact that the de Rham functor is idempotent and preserves limits:
\begin{eqnarray*}
\left(\widehat{\cY}_f\right)_{dR} &\simeq& \left(\cX\times_{\cX_{dR}} \cY\right)_{dR}\\
&\simeq& \cX_{dR} \times_{\left(\cX_{dR}\right)_{dR}} \cY_{dR}\\
&\simeq& \cX_{dR} \times_{\cX_{dR}} \cY_{dR}\\
&\simeq& \cY_{dR}
\end{eqnarray*}
\end{proof}

The following proposition follows by the universal property of pullback diagrams together with the fact that the de Rham functor preserves finite products:

\begin{proposition} \label{prop:prodcomplete}
Let $f=\left(f_0,f_1\right):\cY \to \cX_0 \times \cX_1.$ Then the canonical commutative diagram
$$\xymatrix{\widehat{\cY}_{\left(f_0,f_1\right)} \ar[r] \ar[d] & \widehat{\cY}_{f_0} \times \widehat{\cY}_{f_1} \ar[d]\\
\cY_{dR} \ar[r]^-{\Delta} & \cY_{dR} \times \cY_{dR},}$$
is a pullback diagram.
\end{proposition}\

The following corollary is a special case of Proposition \ref{prop:prodcomplete}:

\begin{corollary} \label{cor:prodcomplete}
Let $x=\left(x_0,x_1\right):* \to \cX_0 \times \cX_1$ be a geometric point. Then
$$\left(\cX_0\times\cX_1\right)_{\dpr{x}} \simeq \left(\cX_0\right)_{\dpr{x_0}} \times \left(\cX_1\right)_{\dpr{x_1}}.$$
\end{corollary}

\begin{definition}
Let $k$ be a non-negative integer. The \textbf{$k$-dimensional formal disk} is the colimit
$$\bbD^k:=\underset{n} \colim \sp\left(C^\i\left(t_1,t_2,\ldots,t_k\right)/\left(t_1,t_2,\ldots,t_k\right)^n\right),$$ where the colimit is taken in $\bH.$
\end{definition}

\begin{definition}
Let $\A$ be an $\SCi$-algebra. Let $k$ be a non-negative integer. Let $N_k \subset \Nil_{\Ci}\left(\pi_0\left(\A\right)\right)$ be the subset of nilpotent elements of degree $k$. 
Define $$\Nil_k\left(\A\right):=\A \times_{\pi_0 \A} N_k$$ be the \textbf{space of nilpotent elements of degree $k$ of $\A$}, and let $$\Nil_{\Ci}\left(\A\right):=\A \times_{\pi_0 \A}  \Nil_{\Ci}\left(\pi_0 \A \right)$$ be the \textbf{space of locally nilpotent elements of $\A$}. Denote by 
 $$\Nil_k:\cD^{op} \to \Spc$$ the functor $\sp\left(\A\right) \mapsto \Nil_k\left(\A\right),$ and by
$$\Nil_{\Ci}:\cD^{op} \to \Spc$$ the functor $\sp\left(\A\right) \mapsto \Nil_\Ci\left(\A\right).$
\end{definition}

\begin{lemma}
The functor $$\Nil_k:\cD^{op} \to \Spc$$ is representable by $\sp\left(C^\i\left(t\right)/t^k\right).$
\end{lemma}

\begin{proof}
Define $P$ to be the pushout in $\dgc$
$$\xymatrix{C^\i\left(t\right) \ar[r] \ar[d]_-{x^k} & C^\i\left(t\right)/t=\R \ar[d]\\
C^\i\left(x\right) \ar[r] & P.}$$ 
Clearly, $\sp\left(P\right)$ represents $\Nil_k.$ It suffices to prove that  $P \simeq C^\i\left(t\right)/t^k.$ Since $\pi_0$ preserves colimits, we have that $\pi_0 P \simeq C^\i\left(t\right)/t^k.$ The model structure of dg-$C^\i$-algebras gives a presentation for $\dgc,$ consequently, we can represent $P$ as the derived tensor product with respect to this model structure. Consider the Koszul algebra $K\left(t\right)$ associated to the element $t$ in $C^\i\left(t\right)$. Since the singleton $t$ is a regular sequence, $K\left(t\right)$ is acyclic, and $C^\i\left(t\right) \to K\left(t\right)$ is a quasi-free extension, hence a cofibration. Therefore the derived tensor product is simply $P=C^\i\left(x\right) \oinfty K\left(t\right)$ which can be identified with the Koszul algebra of $C^\i\left(x\right)$ with respect to the element $x^k.$ Since $x^k$ is not a zero divisor, the canonical map $P \to C^\i\left(x\right)/x^k$ is a quasi-isomorphism.
\end{proof}

\begin{corollary}
$\Nil_k$ is a sheaf for all $k.$
\end{corollary}

\begin{proposition}
$\Nil_\Ci$ is a sheaf and $$\Nil_\Ci \simeq \underset{k} \colim \Nil_k.$$
\end{proposition}

\begin{proof}
It suffices to prove that the restriction of $\Nil_{\Ci}$ to the small site of $\sp\left(\A\right)$ is a sheaf for all $\A.$ As in the proof of Lemma \ref{lem:nilradsame}, $\Nil_{\Ci}\left(A\right)$ can be identified with global sections over $\sp\left(\A\right)$ of $\mathbf{MSpec}\left(\Nil\left(\A\right)\right),$ which is a sheaf. Hence $\Nil_{\Ci}$ is a sheaf, and in fact the sheafification of the presheaf $\Nil$ assigning the space of genuinely nilpotent elements. Since 
$$\Nil\left(A\right) \simeq \underset{k} \colim \Nil_k\left(A\right),$$ and colimits in sheaves are computed by first computing them object-wise in presheaves and then sheafifying, the result follows.
\end{proof}

\begin{corollary}\label{cor:repnil}
The sheaf $\Nil_\Ci$ is represented by $\bbD^1.$
\end{corollary}

\begin{lemma}
Let $x:* \to \R$ be a point in $\R.$ Then the formal neighborhood around $x$ is equivalent to $\bbD^1$
\end{lemma}

\begin{proof}
Firstly, since there is a diffeomorphism $$\R \to \R$$ carrying $x$ to the origin, we can assume without loss of generality that $x$ is the origin.

Notice that $$\Map\left(\sp\left(A\right),\R\right) \simeq A$$ and $$\Map\left(\sp\left(A\right),\R_{dR}\right) \simeq \pi_0\left(A\right)/\Nil_{\Ci}\left(\pi_0 A\right),$$ hence
 $\Map\left(\sp\left(A\right),\R_{\dpr{0}}\right)$ can be identified with the fiber over $0$ of $$A \to \pi_0\left(A\right)/\Nil_{\Ci}\left(\pi_0 A\right).$$ By pasting pullback diagrams

$$\xymatrix{\Nil_{\Ci}\left(A\right) \ar[r] \ar[d] & A \ar[d]\\
\Nil_{\Ci}\left(\pi_0 A\right) \ar[d] \ar[r] & \pi_0\left(A\right) \ar[d]\\
0 \ar[r] & \pi_0\left(A\right)/\Nil_{\Ci}\left(\pi_0 A\right),}$$
we conclude that $$\Map\left(\sp\left(A\right),\R_{\dpr{0}}\right)\simeq \Nil_{\Ci}\left(A\right).$$ The result now follows from Corollary \ref{cor:repnil}.
\end{proof}

\begin{proposition}\label{prop:fmrn}
Let $x:* \to \R^n$ be a point in $\R^n.$ Then the formal neighborhood around $x$ is equivalent to $\bbD^n$.
\end{proposition}

\begin{proof}
By Corollary \ref{cor:prodcomplete}, the formal neighborhood of $x$ in $\R^n$ is equivalent to $\left(\bbD^1\right)^n.$ It therefore suffices to prove that $$\left(\bbD^1\right)^n \simeq \bbD^n.$$ By Proposition \ref{prop:closedprod} and the fact that $\bHc$ is coreflective in $\bH$ and hence closed under colimits, both $\left(\bbD^1\right)^n$ and $\bbD^n$ are in $\bHc.$ Therefore, it suffices to show that $$i^*\left(\bbD^1\right)^n \simeq i^*\bbD^n.$$ Let $\sp\left(A\right)$ be in $\sD_0.$ Then $A$ is an ordinary $C^\i$-ring. Hence, on one hand, $$i^*\left(\bbD^1\right)^n\left(\sp\left(A\right)\right) \cong \Nil_{\Ci}\left(A\right)^n.$$ On the other hand $i^*\bbD^n$ is the sheafification of the presheaf $F$
$$A \mapsto  \underset{k} \colim \Hom\left(C^\i\left(t_1,t_2,\ldots,t_n\right)/\left(t_1,t_2,\ldots,t_n\right)^k,A\right).$$ Since $t_i^k \in \left(t_1,t_2,\ldots,t_n\right)^k$ for all $i,$ it follows that $$F\left(\sp\left(A\right)\right) \subseteq \Nil\left(A\right)^n.$$ Let us now show the reverse inclusion. Let $a_1,a_2,\ldots a_n \in \Nil\left(A\right)^n.$ Suppose that $a_i$ is nilpotent of order $n_i.$ Let $N:=\sum_{i=1}^{n} n_i.$ The ideal $\left(t_1,t_2,\ldots t_n\right)^N$ is generated by the binomials of the form $t_{1}^{k_{1}}t_{2}^{k_{2}}\cdots t_{m}^{k_{m}},$ where $k_{1}+k_{2}+\cdots +k_{m}=N.$ Fix such a monomial. For a given $j,$ 
$$k_j=N-\sum_{i \ne j} k_i.$$ So, if each $i \ne j,$ $k_i< n_j,$ then $k_j=n_j +\sum_{i \ne j} \left(n_i-k_i\right) > n_j,$ hence $a_j^{k_j}=0,$ and consequently $a_{1}^{k_{1}}a_{2}^{k_{2}}\cdots a_{m}^{k_{m}}=0.$ Otherwise, there exists $i \ne j$ such that $k_i \ge n_i,$ and again we conclude that $a_{1}^{k_{1}}a_{2}^{k_{2}}\cdots a_{m}^{k_{m}}=0$. It follows that $$\Nil\left(A\right)^n \subseteq \underset{k} \colim \Hom\left(C^\i\left(t_1,t_2,\ldots,t_n\right)/\left(t_1,t_2,\ldots,t_n\right)^k,A\right)=F\left(\sp\left(A\right)\right).$$ Since sheafification is left exact, the result follows.
\end{proof}

\begin{definition}
Let $n$ and $k$ be non-negative integers. The \textbf{$k^{th}$ order $\left(0|n\right)$-dimensional formal disk} is
$$\bbD^{0|n}_{\left(k\right)}:=\sp\left(C^\i\left(0;\eta_1,\eta_2,\ldots,\eta_n\right)/\left(\eta_1,\eta_2,\ldots,\eta_n\right)^k\right).$$
The \textbf{$\left(0,n\right)$-dimensional formal disk} is the colimit
$$\bbD^{0|n}:=\underset{k} \colim \sp\left(C^\i\left(0;\eta_1,\eta_2,\ldots,\eta_n\right)/\left(\eta_1,\eta_2,\ldots,\eta_n\right)^k\right),$$ where the colimit is taken in $\bH.$
\end{definition}

\begin{remark}
More generally, define the \textbf{$k^{th}$ order $\left(m|n\right)$-dimensional formal disk} as
$$\bbD^{m|n}_{\left(k\right)}:=\sp\left(C^\i\left(x_1,\ldots,x_m;\eta_1,\eta_2,\ldots,\eta_n\right)/\left(x_1,\ldots,x_m\eta_1,\eta_2,\ldots,\eta_n\right)^k\right).$$
Then we have $$\bbD^{m|n}_{\left(k\right)} \simeq \bbD^{m}_{\left(k\right)} \times \bbD^{0|n}_{\left(k\right)}.$$
\end{remark}

\begin{proposition}
There is a canonical equivalence $$\bbD^{0|n} \simeq \R^{0|n}.$$
\end{proposition}

\begin{definition}
For $n$ and $m$ non-negative integers, define the \textbf{$\left(n|m\right)$-dimensional formal disk} to be $$\bbD^{m|n}:=\bbD^m \times \bbD^{0|n}.$$
\end{definition}

\begin{proof}
Let $F$ be an arbitrary object of $\bH.$ Then we have
$$\Map\left(\bbD^{0|n},F\right) \simeq \underset{k} \lim \Map\left(\bbD^{0|n}_{\left(k\right)},F\right).$$
Notice that for all $k \ge n,$
$\bbD^{0|n}_{\left(k\right)} = \R^{0|n}.$ Thus the diagram in the above limit has $\Map\left(\R^{0|n},F\right)$ as an initial object, and thus the limit is $\Map\left(\R^{0|n},F\right).$ We are done by the Yoneda lemma.
\end{proof}

\begin{proposition}
Let $\A$ be a (complete) $\SCi$-algebra such that $\A_{red}=\R.$ Then
\begin{itemize}
\item[i)] $\Map\left(*,\sp\left(\A\right)\right) \simeq *,$
\item[ii)] $\sp\left(\A\right)_{dR}\simeq *,$
\item[iii)] The formal completion $\sp\left(\A\right)_{\left(\left(x\right)\right)}$ at the unique point $x$ is equivalent to $\sp\left(\A\right)$ itself.
\end{itemize}
\end{proposition}

\begin{proof}
Consider the following string of natural equivalences
\begin{eqnarray*}
\Map\left(*,\sp\left(\A\right)\right) &\simeq& \Map\left(\sp\left(\R\right),\sp\left(\A\right)\right)\\
&\simeq& \Map\left(\A,\R\right)\\
&\simeq& \Map\left(\A_{red},\R\right)\\
&\simeq& \Map\left(\R,\R\right)\\
&\simeq& *.
\end{eqnarray*}
This establishes $i)$
Let $\B$ be an arbitrary (complete) $\SCi$-ring in $\cD$. Then we have the following natural equivalences
\begin{eqnarray*}
\Map\left(\sp\left(\B\right),\sp\left(\A\right)_{dR}\right) &\simeq&  \Map\left(\sp\left(\B_{red}\right),\sp\left(\A\right)\right)\\
&\simeq& \Map\left(\A,\B_{red}\right)\\
&\simeq& \Map\left(\A_{red},\B_{red}\right)\\
&\simeq& \Map\left(\R,\B_{red}\right)\\
&\simeq& *.
\end{eqnarray*}
It follows that $\sp\left(\A\right)_{dR}$ is the terminal object, establishing $ii).$ 
Finally, we have that
\begin{eqnarray*}
\sp\left(\A\right)_{\left(\left(x\right)\right)} &\simeq& * \times_{\sp\left(\A\right)_{dR}} \sp\left(\A\right)\\
&\simeq&	 * \times_{*} \sp\left(\A\right)\\
&\simeq& \sp\left(\A\right),
\end{eqnarray*}
which establishes $iii).$
\end{proof}

\begin{corollary}\label{cor:fnbds}
Let $x:* \to \R^{n|m}$ be a geometric point. Then its formal neighborhood around $x$ is equivalent to $$\bbD^{m|n}=\bbD^n \times \R^{0|m}.$$
\end{corollary}
 
\begin{proof}
The de Rham functor preserves all limits, hence 
$$\left(\R^{n|m}\right)_{dR} \simeq \left(\R^n\right)_{dR} \times * \simeq \left(\R^n\right)_{dR}.$$ It follows that the canonical map
$$\R^{n|m}=\R^n \times \R^{0|m} \to \left(\R^{n|m}\right)_{dR} \times \left(\R^n\right)$$ factors as
$$\R^{n} \times \R^{0|m} \stackrel{pr_1}{\longlongrightarrow} \R^n \stackrel{\eta}{\longrightarrow} \left(\R^n\right)_{dR}.$$
By pasting pullback diagrams we get a diagram of pullbacks
$$\xymatrix{\bbD^n \times \R^{0|m} \ar[r] \ar[d] & \R^n \times \R^{0|m} \ar[d]\\
\bbD^n \ar[r] \ar[d] & \R^n \ar[d]\\
\ast \ar[r] & \left(\R^n\right)_{dR}.}$$
The bottom square is a pullback diagram by virtue of Proposition \ref{prop:fmrn}.
\end{proof}

\subsection{Formally \'etale maps}

\begin{definition}
A morphism $\varphi:\cZ \to \cX$ in $\bH$ is \textbf{formally \'etale} if $\varphi \simeq \hat f$ for some $f:\cY \to \cX.$
\end{definition}

\begin{proposition}
$\varphi:\cZ \to \cX$ in $\bH$ is formally \'etale if and only if the naturality square
$$\xymatrix{\cZ \ar[d]_-{\varphi} \ar[r]^-{\eta_\cZ} & \cZ_{dR}\ar[d]^-{\varphi_{dR}}\\
\cX \ar[r]_-{\eta_\cX} & \cX_{dR}}$$
is Cartesian.
\end{proposition}

\begin{proof}
If the above diagram is a pullback square, it implies that $\varphi \simeq \hat \varphi,$ so $\varphi$ is formally \'etale. Conversely, suppose that $\varphi=\hat f,$ for $f:\cY \to \cX.$ Let $\cZ=\widehat{\cY}_f.$ Since
$$\xymatrix{\widehat{\cY}_f \ar[d]_-{\hat f} \ar[r] & \cY_{dR} \ar[d]^-{f_{dR}}\\
\cX \ar[r]_-{\eta_{\cX}} & \cX_{dR}}$$
is a pullback diagram, and $\left(\blank\right)_{dR}$  is idempotent and preserves pullbacks, the two right most squares in the following diagram are pullbacks:
$$\xymatrix@R=2cm@C=2cm{ \widehat{\cY}_f= \cZ \ar[d]_-{\hat f} \ar[r]^-{\eta_\cZ} & \left(\widehat{\cY}_f\right)_{dR}=\cZ_{dR} \ar[d]_-{\left(\hat f\right)_{dR}} \ar[r] & \left(\cY_{dR}\right)_{dR} \ar[d]_-{\left(f_{dR}\right)_{dR}} \ar[r]^-{\sim} & \cY_{dR} \ar[d]^-{f_{dR}}\\
\cX \ar[r]_-{\eta_\cX} & \cX_{dR} \ar[r]_-{\left(\eta_\cX\right)_{dR}} & \left(\cX_{dR}\right)_{dR} \ar[r]_-{\sim} & \cX_{dR}.}$$
The outermost square is also a pullback, by definition, hence so is the left most square.
\end{proof}

\begin{lemma}\label{lem:festab}
	Formally \'etale maps are stable under pullback.
	\end{lemma}

\begin{proof}
	Let $f:X \to Y$ be formally \'etale, and let $\varphi:Z \to Y.$ We wish to show that $pr_Z:Z \times_{Y} X \to Z$ is formally \'etale. Since $f$ is formally \'etale, all squares in the following diagram are pullbacks:
	$$\xymatrix{Z \times_{Y} X \ar[d] \ar[r] & X \ar[r]^-{\eta_X} \ar[d]_-{f} & X_{dR} \ar[d]^-{f_{dR}}\\
		Z \ar[r]^-{\varphi} & Y \ar[r]^-{\eta_Y} & Y_{dR}.}$$
	By the naturality of $\eta,$ this implies that all the squares in the following diagram are pullbacks:
	$$\xymatrix{Z \times_{Y} X \ar[d] \ar[r] & Z_{dR} \times_{Y_{dR}} X_{dR} \ar[d] \ar[r] & X_{dR} \ar[d]\\
		Z \ar[r]^-{\eta_Z} & Z_{dR} \ar[r]^-{\varphi_{dR}} & Y_{dR}.}$$
	Since $\left(\blank\right)$ preserves limits, we have $\left(Z \times_Y X\right)_{dR} \simeq Z_{dR} \times_{Y_{dR}} X_{dR}$ and the morphism $$Z_{dR} \times_{Y_{dR}} X_{dR} \to Z_{dR}$$ in the above diagram can be identified with $\left(pr_Z\right)_{dR}.$ Hence $pr_Z$ is formally \'etale.
	\end{proof}

\begin{lemma}\label{lem:repfet}
	The following are equivalent for a morphism $f:\cY \to \cX$:
	\begin{itemize}
		\item[i)] $f$ is formally \'etale
		\item[ii)] for all $\varphi:D \to \cX$ with $D \in \sD,$ $\cY \times_{\cX} D \to D$ is formally \'etale.
		\end{itemize}
	\end{lemma}
\begin{proof}
	$i) \Rightarrow ii)$ follows from the fact that formally \'etale maps are stable under pullback (Lemma \ref{lem:festab}). Conversely, suppose that $f$ satisfies $ii).$ We wish to show that the canonical map
	$$\xymatrix{\cY \ar@{-->}[rd] \ar@/^2.0pc/[rrd]^-{\eta_{\cY}}   \ar@/_2.0pc/[rdd]_-{f}& &\\
	& {\cX \times_{\cX_{dR}} \cY_{dR}} \ar[r] \ar[d]_-{pr_\cX} & \cY_{dR} \ar[d]^-{f_{dR}} \\
	& \cX \ar[r]^-{\eta_\cX} & \cX_{dR},}$$
	denoted by a dotted arrow above, is an equivalence. For this, it suffices to prove that for any $D \in \sD$ and $g:D \to \cX,$ the induced map
	$$\operatorname{hofib}_g\left(\cY\left(D\right) \to \cX\left(D\right)\right) \to \operatorname{hofib}_g\left({\cX \times_{\cX_{dR}} \cY_{dR}}\left(D\right) \to \cX\left(D\right)\right)$$ is an equivalence. This is equivalent to showing that the induced map
	$$\Map_{\bH/\cX}\left(g,f\right) \to \Map_{\bH/\cX}\left(g,pr_\cX\right)$$ is an equivalence.
	Notice that 
	$$\Map_{\bH/\cX}\left(g,f\right) \simeq \Gamma_D\left(\alpha:\cY \times_{\cX} D \to D\right).$$ By hypothesis, this is sections of a formally \'etale map. This implies that the left most square in the following diagram is is a pullback,
	$$\xymatrix{\cY \times_{\cX} D \ar[d]_-{\alpha} \ar[r] & \cY_{dR}\times_{\cX_{dR}} D_{dR} \ar[d]_-{\alpha_{dR}} \ar[r] & \cY_{dR} \ar[d]^-{f_{dR}}\\
		D \ar[r]^-{\eta_D} & D_{dR} \ar[r]^-{g_{dR}} & \cX_{dR}.}$$
		Therefore all of the squares above are pullback squares. By the naturality of $\eta$, it follows all the squares in the following diagram are also pullbacks
		$$\xymatrix{\cY\times_{\cX} D \ar[d] \ar[r] & \cX \times_{\cX_{dR}} \cY_{dR} \ar[d]_-{pr_{\cX}} \ar[r]& \cY_{dR} \ar[d]^-{f_{dR}}\\
			D \ar[r]^-{g} & \cX \ar[r]^-{\eta_{\cX}} & \cX_{dR}.}$$ 
	Hence, the space of sections of $\alpha$ can be canonically identified with $\Map_{\bH/\cX}\left(g,pr_\cX\right),$ completing the proof.
	\end{proof}

\begin{proposition} \label{prop:etaleisfet}
	If $f:X \to Y$ is an \'etale map (a.k.a local diffeomorphism) of $\SCi$-schemes, then $f$ is formally \'etale.
	\end{proposition}
\begin{proof}
We want to show that the canonical map $$X \to Y\times_{Y_{dR}} X_{dR}$$ is an equivalence. Let $Z$ be any $\SCi$-scheme. Then a morphism to the above fibered product is the same as a pair of morphisms making a commutative square
$$\xymatrix{Z_{red} \ar@{-->}[r]^-{g} \ar[d]_-{\pi} & X \ar[d]^-{f}\\
	Z \ar@{-->}[r]^-{h} & Y.}$$
Since $f$ is \'etale, we have that $f^*\O_Y \simeq \O_X,$ where here we mean the structure sheaf of the $\SCi$-schemes as ringed spaces. Notice that $\pi$ is the identity morphism on the underlying spaces of $Z$ and $Z_{red}.$ In particular, the spatial component of the map $h:Z \to Y$ is fixed by the fact that we must have $\underline{h}=\underline{f} \circ \underline{g},$ where the underline denotes the forgetful functor to the category of topological spaces. Since $f$ is \'etale, this means that $h^*\O_Y \simeq g^*f^*\O_Y \simeq g^* \O_X,$ hence the morphism $h=\left(\underline{h},\alpha\right)$ with $\alpha:h^*\O_Y \to \O_Z,$ gives rise to a morphism $$g':=\left(g,\alpha\right):Z \to X$$ such that $f \circ g \simeq h$ and $g' \circ \pi \simeq g.$ So, $g'$ determines the pair $\left(g,h\right)$ up to a contractible space of choices. It follows that the induced map $$\Map\left(Z,X\right) \to \Map\left( Y\times_{Y_{dR}} X_{dR}\right)$$ is an equivalence.
\end{proof}

\begin{corollary}\label{cor:repetisfet}
	Any representable \'etale map is formally \'etale.
	\end{corollary}

\begin{proof}
Let $f:\cX \to \cY$ be a representable \'etale map. Then, for any $D \in \sD$ and $\varphi:D \to \cY,$ we have that $\cX \times_{\cY} D$ is an \'etale map of $\SCi$-schemes, hence formally \'etale by Proposition \ref{prop:etaleisfet}. The result now follows from Lemma \ref{lem:repfet}.
\end{proof}

\begin{proposition}
A morphism $f: \cM \to \cN$ of supermanifolds is formally \'etale if and only if it is a local diffeomorphism.
\end{proposition}

\begin{proof}
	If $f$ is a local diffeomorphism, then it is \'etale, and hence formally \'etale by Proposition \ref{prop:etaleisfet}. Conversely, suppose that $f:\cM \to \cN$ is a formally \'etale map of supermanifolds. To show that $f$ is a local diffeomorphism, it suffices to show that for each point $x \in \cM,$ the tangent map
	$$Tf_x:T_x \cM \to T_x \cN$$ is an isomorphism of super vector spaces. Let $D$ be either $\mathbb{D}_{\left(1\right)}$ or $\mathbb{D}^{0|1}.$ In either case, for any $\cX,$ $$\cX_{dR}\left(D\right) \simeq \cX\left(*\right).$$ Since $f$ is formally \'etale, there is a pullback diagram
	$$\xymatrix{\cM \ar[d]_-{f} \ar[r]^-{\eta_\cM} & \cM_{dR} \ar[d]^-{f_{dR}}\\
		\cN \ar[r]^{\eta_\cN} & \cN_{dR}.}$$
	Then, for $x \in \cM$ of the induced map between the fiber over $x$ of the map
	$$\Map\left(D,\cM\right) \to \Map\left(*,\cM\right)$$ with the fiber over $f\left(x\right)$ of the map $$\Map\left(D,\cN\right) \to \Map\left(*,\cN\right)$$ is a bijection. But these maps can be identified with the even/odd components of $Tf_x$ when $D=\mathbb{D}_{\left(1\right)}$ and $D=\mathbb{D}^{0|1}$ respectively.
	\end{proof}

 \subsection{Formal smoothness}
 
 \begin{definition}
 	An object $\cX \in \bH$ is \textbf{formally smooth} if the unit map
 	$$\eta_\cX:\cX \to \cX_{dR}$$
 	to the de Rham stack is an epimorphism.
 \end{definition}
 
 \begin{lemma}
 	Suppose that $\left(X_\alpha\right)_\alpha$ is a collection of formally smooth objects. Then the coproduct $\underset{\alpha} \coprod X_\alpha$ is also formally smooth.
 \end{lemma}
 
 \begin{proof}
 	Recall from Definition \ref{dfn:deRham} and Section \ref{sec:adj} that $$\left(\blank\right)_{dR}=k_*k^*\simeq \left(kp\right)^*.$$ It follows that the de Rham functor is a left adjoint, and hence preserves coproducts. The result now follows.
 \end{proof}
 
 \begin{lemma}
 	If $p:\cY \to \cX$ is an epimorphism and $\cY$ is formally smooth, then so is $\cX.$
 \end{lemma}
 
 \begin{proof}
 	Since the de Rham functor is a left adjoint, it preserves epimorphisms. Consider the naturality square for $p$
 	$$\xymatrix@C=1.3cm@R=1.3cm{\cY \ar@{->>}[r]^-{\eta_\cY} \ar@{->>}[d]_-{p} & \cY_{dR} \ar@{->>}[d]^-{p_{dR}}\\
 		\cX \ar[r]_-{\eta_\cX} & \cX_{dR}.}$$
 	It follows from \cite[Corollary 6.2.3.12]{htt} that $\eta_\cX$ is also an epimorphism.
 \end{proof}
 
 \begin{proposition}
 	For all $\left(n,m\right)$ $\R^{\left(n,m\right)}$ is formally smooth.
 \end{proposition}
 
 \begin{proof}
 	The map $\eta:\R^{\left(n,m\right)} \to \R^{\left(n,m\right)}_{dR}$ is an epimorphism if and only if any map $f:D \to \R^{\left(n,m\right)}_{dR}$ with $D \in \sD$ locally factors through $\eta.$ We will prove something stronger, namely that any such maps globally factors through $\eta.$ Let $D=\Speci\left(\A\right).$ The map $f$ corresponds to a homomorphism
 	$$f:\SCi\left(x_1,\ldots,x_n;\zeta_1,\ldots,\zeta_m\right) \to \A_{red}.$$ As the algebra $\SCi\left(x_1,\ldots,x_n;\zeta_1,\ldots,\zeta_m\right)$ is free, this corresponds to picking $n$ even elements of $\A_{red}$ and $m$ odd ones, but $\A_{red}$ is purely even so this is the same as picking $n$ elements. We want a lift of this map to $\A,$ but this is easy--- just pick $n$ even elements of $\A$ whose class in $\A_{red}$ agree with the chosen ones, and pick $m$ arbitrary odd elements of $\A.$
 \end{proof}
 
 \begin{corollary}\label{cor:mfdsmth}
 	If $\cM$ is a supermanifold, then it is formally smooth.
 \end{corollary}

\section{Group objects}\label{sec:group}
Here, we briefly recall the notion of a group object in an $\i$-topos as introduced in \cite{pb}. We claim no originality for the ideas.

\begin{definition}
A \textbf{classifying stack for a group object} in an $\i$-topos $\cX$ is an object $G$ together with an effective epimorphism $$\pi_X:* \to X$$ from the terminal object $*.$ The \textbf{group object} associated to this classifying stack is $G:=\Omega_*\left(X\right).$
\end{definition}

\begin{remark}
If $H$ is a Lie group, then $\pi:* \epi \cB H$ is an effective epimorphism in $\bH,$ making it a classifying stack for a group object. Notice that
$$\Omega_*\left(\cB H\right) \simeq H.$$ Moreover, we claim that the group structure on $H$ is completely encoded by the fact that $H$ can be identified with the based loop object on $\cB H.$ Indeed, the \v{C}ech nerve of $\pi$ is a simplicial object
$$\mbox{\v{C}}_\pi:\Delta^{op} \to \bH$$ with $\mbox{\v{C}}_n=H^n,$ and $d_1:H \times H \to H$ can be identified with group multiplication.
\end{remark}

More generally, for any group object $G,$ the \v{C}ech nerve of $\pi_G$ produces a simplicial object
$$\mbox{\v{C}}_{\pi_G}:\Delta^{op} \to \cX$$ with $\left(\mbox{\v{C}}_{\pi_X}\right)_0=*.$ This is the usual definition of a group object given in \cite[Definition 7.2.2.1]{htt}. Conversely, the if $$\cG:\Delta^{op} \to \cX$$ is a group object in Lurie's sense, $\cG$ is equivalent to the \v{C}ech nerve of $*=\cG_0 \to \colim \cG,$ so the two definitions agree.

Notice since $\pi_G:* \epi X$ is an effective epimorphism, $X$ is equivalent to the colimit of the \v{C}ech nerve of $\pi_G.$ Since $\Omega_*\left(X\right)=G,$ we have $$X\simeq \underset{n \in \Delta^{op}} \colim \pi_X=\colim\left(\xymatrix{\ast & \ar@<-.5ex>[l] \ar@<.5ex>[l] G & G \times G \ar@<-.5ex>[l] \ar@<.5ex>[l] \ar[l]&  \ar@<-.7ex>[l]  \ar@<-.25ex>[l]\ar@<.25ex>[l] \ar@<0.7ex>[l] G \times G \times G \cdots}\right)$$
I.e. $X \simeq \cB G$--- the classifying stack of principal $G$-bundles.

Moreover, notice that, just as for a Lie group, the structure of the simplicial object $\mbox{\v{C}}_{\pi_G}$ encodes the group structure. E.g., the face map
$d_1:G \times G \to G$ is the group multiplication, and the other simplicial maps encode higher coherency information.

\begin{definition}
Given a group object $G,$ denote by $N\left(G\right)$ the simplicial object $\mbox{\v{C}}_{* \epi \cB G}.$
\end{definition}

\begin{definition}
A \textbf{principal $G$-bundle} or \textbf{$G$-torsor} over an object $S$ is a pullback square of the form
$$\xymatrix{P \ar[r] \ar[d] & \ast \ar[d]^-{\pi_G}\\
S \ar[r] & \cB G.}$$ $P$ is called the \textbf{total space} of the bundle.
\end{definition}

\begin{definition}\label{dfn:action}
An \textbf{action} of a group object $G$ on an object $V$ in $\cX$ is a pullback square
$$\xymatrix{V \ar[r]^-{\pi_{\rho}} \ar[d] & \ast \ar[d]^-{\pi_G}\\
S \ar[r] & \cB G}$$
exhibiting $V$ as the total space of a principal bundle over some object $S.$ In this case the object $S$ is called the \textbf{quotient stack} of $V$ by $G$ and is denoted by $\q{V}{G}.$
\end{definition}

\begin{remark}
Notice that, as a pullback of an effective epimorphism, $$\pi_{\rho}:V \to \q{V}{G}$$ is an effective epimorphism, and hence
$$\q{V}{G} \simeq \colim \mbox{\v{C}}_{\pi_\rho}.$$ By stacking pullback diagrams 
\begin{eqnarray*}
V \times_{\q{V}{G}} V &\simeq &* \times_{\cB G} V\\
&\simeq&\left(* \times_{\cB G} *\right) \times_* V\\
&\simeq& G \times V,
\end{eqnarray*}
and more generally, the $\left(n+1\right)^{st}$-fold fiber product of $V$ with itself over $\q{V}{G}$ is equivalent to $G^n \times V.$
So
$$\q{V}{G} \simeq \colim\left(\xymatrix{V & \ar@<-.5ex>[l] \ar@<.5ex>[l] G \times V & G \times G \times V \ar@<-.5ex>[l] \ar@<.5ex>[l] \ar[l]&  \ar@<-.7ex>[l]  \ar@<-.25ex>[l]\ar@<.25ex>[l] \ar@<0.7ex>[l] G \times G \times G \times V \cdots}\right)$$
The map $$\rho:=d_0:G \times V \to V$$ exhibits the action map of $G$ on $V,$ and the rest of the simplicial structure encodes the higher coherencies. Denote the above groupoid object by $N\left(G \ltimes V\right).$ In particular, for an ordinary smooth action of a Lie group $G$ on a manifold $V,$ this agrees with the actual nerve of the action groupoid.
\end{remark}

In particular, we proved the following proposition:

\begin{proposition}\label{prop:objprod}
For $G$ any group object with an action on an object $V,$ we have a canonical identification: $$V \times_{\q{V}{G}} V \simeq G \times V.$$
\end{proposition}

\section{Jet Bundles and Differential Operators}\label{sec:jet}
\subsection{Jet bundles and disk bundles}
\begin{definition}
Let $\pi:E \to M$ be an object of $\bH/M.$ Then the \textbf{formal disk bundle of $\pi$}, $\bT^\i_M\left(\pi\right) \to M$  is $$\eta^*\eta_!\left(\pi\right),$$ where $\eta:M \to M_{dR}$ is the unit map to the de Rham stack.
\end{definition}

More concretely, there is a pullback diagram
$$\xymatrix{\bT^\i_M\left(\pi\right) \ar[dd] \ar[r] & E \ar[d]^-{\pi}\\
& M \ar[d]^-{\eta}\\
M \ar[r]^-{\eta} & M_{dR}.}$$

\begin{definition}
	Given $M \in \bH,$ the \textbf{formal disk bundle of $M$,} denoted by $$\bT^\i M \to M$$ is the formal disk bundle of $id_M,$ i.e. the pullback $M \times_{M_{dR}} M \to M.$
	\end{definition}

The following proposition follows immediately from the fact that the de Rham functor $\left(\blank\right)_{dR}$ preserves finite products:

\begin{proposition}\label{prop:diagcomp}
	The formal disk bundle of $M$ is the formal completion of $M$ along the diagonal 
	$$\Delta:M \to M \times M,$$ that is $\bT^\i M \simeq \widehat{M}_{\Delta}.$ 
	\end{proposition}

\begin{remark}\label{rmk:124}
	For any $\pi:E \to M,$ there is a factorization
	$$\xymatrix{\bT^\i_M\left(\pi\right) \ar[d]_-{\pi_E} \ar[r]^-{\ev_E} & E \ar[d]^-{\pi}\\
		\bT^\i M  \ar[d]_-{\pi_M} \ar[r]^-{\ev} & M \ar[d]^-{\eta}\\
		M \ar[r]^-{\eta} & M_{dR}.}$$
	\end{remark}

The following proposition is immediate:

\begin{proposition} \label{prop:propdisk}
	Given any $F \to M$ in $\bH/M,$ and $D \in \bH,$ then $\bT^\i_M$ applied to $$F \times D \stackrel{pr}{\longrightarrow} F \to M$$ is equivalent to $$\bT^\i_M\left(F\right) \times D \stackrel{pr}{\longrightarrow} \bT^\i_M\left(F\right) \to M.$$ 
	\end{proposition}

\begin{lemma}\label{lem:fsmgp}
Let $G$ be a formally smooth group object. Then there is a canonical equivalence
$$G_{dR} \simeq \q{G}{\widehat{G}_e},$$ where $\widehat{G}_e$ is the formal completion of the identity.
\end{lemma}

\begin{proof}
Since the de Rham functor preserves limits and effective epimorphisms, $$*\simeq *_{dR} \to \left(\cB G\right)_{dR}$$ is an effective epimorphism. Due again to the preservation of limits, the fibered product $* \times_{\left(\cB G\right)_{dR}} * \simeq G_{dR}.$ It follows that $G_{dR}$ is a group object and $$\left(\cB G\right)_{dR} \simeq \cB\left(G_{dR}\right).$$ Denote by $P$ the fibered product
$$\xymatrix{P \ar[r] \ar[d] & \cB G \ar[d]^-{\eta}\\
\ast \ar[r] & \left(\cB G\right)_{dR}.}$$
Since $G$ is formally smooth, $$\eta_G:G \to G_{dR}$$ is an effective epimorphism, and hence
$$\pi_0\left(G\right) \to \pi_0\left(G_{dR}\right)$$ is as well. We can identify the latter map with the induced map
$$\pi_1\left(\cB G\right) \to \pi_1\left(\left(\cB G\right)_{dR}\right).$$ Analyzing the long exact sequence in homotopy sheaves from the fiber sequence
$$P \to \cB G \to \left(\cB G\right)_{dR},$$ we conclude that $\pi_0\left(P\right) \simeq *.$ Thus, $P$ is the classifying stack of a group object, and this group object is $\Omega_* P.$ Consider diagram of pullback squares
$$\xymatrix{\Omega_\ast P \ar[r] \ar[d] & \ast\simeq \ast_{dR} \ar[d] & \\
G \ar[r]_-{\eta_G} \ar[d] & G_{dR} \ar[r] \ar[d] & \ast \ar[d]\\
\ast \ar[r] & P \ar[r] \ar[d] & \cB G \ar[d]\\
& \ast \ar[r] & \cB\left(G_{dR}\right).}$$
It follows that there is a canonical identification $$\Omega_* P \simeq \widehat{G}_e,$$ and thus $$P \simeq \cB\left(\widehat{G}_e\right).$$ Therefore there is a canonically induced action of $\widehat{G}_e$ on $G,$ and we can identify $$G_{dR} \simeq \q{G}{\widehat{G}_e}.$$
\end{proof}


\begin{corollary}\label{cor:gpderham}
Let $G$ be a formally smooth group object. Denote by $\widehat{G}_{e}$ the formal completion of $G$ at the identity. Then there is a canonical equivalence
$$T^\i\left(G\right) \simeq G \times \widehat{G}_{e}.$$
\end{corollary}

\begin{proof}
By Lemma \ref{lem:fsmgp}, we can identify $\eta_G$ with the quotient map $$G \to \q{G}{\widehat{G}_e}.$$ The result now follows from Proposition \ref{prop:objprod}.
\end{proof}

\begin{corollary}
	There is a canonical equivalence $$\bT^\i\left(\R^{\left(m,n\right)}\right) \cong \R^{\left(m,n\right)} \times \bbD^{\left(m,n\right)}.$$
			\end{corollary}
			
\begin{proof}			
For all $n$ and $m,$ $\R^{\left(m,n\right)}$ is an abelian group object, so by Corollary \ref{cor:gpderham}, $$\bT^\i\left(\R^{\left(m,n\right)}\right) \cong \R^{\left(m,n\right)} \times \widehat{\R^{\left(m,n\right)}}_{0}.$$
The result now follows from Corollary \ref{cor:fnbds}.
\end{proof}




\begin{definition}
	Let $\A$ be a classical $\SCi$-algebra. Consider the codiagonal map $$\nabla:\A \oinfty \A \to \A,$$ and let $\I$ be its kernel. Then the \textbf{$k^{th}$-order infinitesimal neighborhood of the diagonal of $\Speci\left(\A\right)$} is $\Speci\left(\left(\A \oinfty \A\right)/\I^{k+1}\right).$ For $X$ an affine $\SCi$-scheme, we denote this by $X_{\left(k\right)},$ and denote the canonical maps $X_{\left(k\right)} \to X$ coming from first and second projection respectively by $\alpha^k$ and $\beta^k.$ 
	\end{definition}

\begin{lemma}\label{lem:nbdktriv}
	For any non-negative integers $n$ an $m$ and $k,$ $$\left(\R^{n|m}\right)_{\left(k\right)} \simeq \R^{n|m} \times \mathbb{D}_{\left(k\right)}^{n|m}.$$
\end{lemma}

\begin{proof}
	Notice that the codiagonal map $$\Ci\left(x_1,x_2,\ldots,x_n|\eta_1,\ldots,\eta_m\right) \oinfty \Ci\left(y_1,y_2,\ldots,y_n| \xi_1,\ldots,\xi_m\right) \to \Ci\left(t_1,t_2,\ldots,t_n|\theta_1,\ldots,\theta_m\right)$$ sends $x_i$ and $y_i$ to $t_i,$ and $\eta_j$ and $\xi_j$ to $\theta_j.$ It follows that the homogeneous ideal $\I=\ker\left(\nabla\right)$ is generated by $\left(x_1-y_1,\ldots,x_n-y_n|\eta_1-\xi_1,\ldots,\eta_m-\xi_m\right).$ Define $$z_i:=x_i-y_i$$ and $$\theta_i:=\eta_i-\xi_i.$$
	Then we have the following string of natural equivalences:

		\begin{eqnarray*}
\resizebox{!}{0.1in}{$\left(\Ci\left(\mathbf{x}|\mathbf{\eta}\right) \oinfty \Ci\left(\mathbf{y}|\mathbf{\xi}\right)\right)/\I^k$} &\cong& \resizebox{!}{0.13in}{$\Ci\left(\mathbf{x}| \mathbf{\eta}\right) \oinfty \Ci\left(\mathbf{y}|\mathbf{\xi}\right)/\left(x_1-y_1,\ldots,x_n-y_n|\eta_1-\xi_1,\ldots,\eta_m-\xi_m\right)^k$}\\
	&\cong& \resizebox{!}{0.15in}{$\left(\Ci\left(z_1,\ldots,z_n|\theta_1,\ldots,\theta_m\right) \oinfty \Ci\left(\mathbf{y}|\mathbf{\xi}\right)\right)/\left(z_1,\ldots,z_n,\theta_1,\ldots,\theta_m\right)^k$}\\
	&\cong& \resizebox{!}{0.15in}{$\left(\Ci\left(z_1,\ldots,z_n|\theta_1,\ldots,\theta_m\right)/\left(z_1,\ldots,z_n,\theta_1,\ldots,\theta_m\right)^k\right) \oinfty \Ci\left(\mathbf{y}|\mathbf{\xi}\right)$}\\
	&\simeq & \Ci\left(\mathbb{D}_{\left(k\right)}^{n|m} \times \R^{n|m}\right) 
	\end{eqnarray*}
	
\end{proof}

\begin{proposition}\label{prop:bTloc}
	Let $\cM$ be a finite dimensional supermanifold. Then $$\bT^\i\left(\cM\right)= \underset{k} \colim \cM_{\left(k\right)}.$$
	\end{proposition}

\begin{proof}
		Write $\cM \simeq \underset{\alpha} \colim U_\alpha$  in $\bH.$ with each $U_\alpha$ diffeomorphic to $\R^{n|m}.$ It follows from Proposition \ref{prop:nbdcosh} and Lemma \ref{lem:nbdktriv} that
	\begin{eqnarray*}
		\cM_{\left(k\right)} &\simeq& \underset{\alpha} \colim \left(U_\alpha\right)_{\left(k\right)}\\
		&\simeq & \underset{\alpha} \colim U_\alpha \times \mathbb{D}_{\left(k\right)}^{n|m}.
	\end{eqnarray*}
Hence
\begin{eqnarray*}
	\underset{k} \colim \cM_{\left(k\right)} & \simeq & \underset{k} \colim \underset{\alpha} \colim \left(U_\alpha \times \mathbb{D}_{\left(k\right)}^{n|m}\right)\\
	&\simeq &   \underset{\alpha} \colim \underset{k} \colim \left(U_\alpha \times \mathbb{D}_{\left(k\right)}^{n|m}\right)\\
	&\simeq & \underset{\alpha} \colim \mathbb{T}^\i\left(U_\alpha\right)\
		\end{eqnarray*}
		Note however, that since each open inclusion $U_\alpha \hookrightarrow \cM$ is formally \'etale, it follows that we have two pullback squares
		$$\xymatrix{\bT^\i U_\alpha \ar[d] \ar[r] & U_\alpha \ar[d]\\
			U_\alpha \ar[r] \ar[d] & U_{dR} \ar[d]\\
			\cM \ar[r] & \cM_{dR},}$$
		and hence $$\bT^\i U_\alpha \simeq \bT^\i_\cM U_\alpha \simeq \bT^\i \cM \times_\cM U_\alpha.$$
		It follows thus that
		\begin{eqnarray*}
			\underset{k} \colim \cM_{\left(k\right)} & \simeq & \underset{\alpha} \colim \bT^\i \cM \times_\cM U_\alpha\\
			&\simeq & \bT^\i\c M \times_\cM \cM\\
			&\simeq & \bT^\i \cM.
			\end{eqnarray*}
		\end{proof}
	
\begin{lemma}\label{lem:suppsh}
	If $\cM$ is a supermanifold, then there is an isomorphism of presheaves on $\cM_0$  $$U \mapsto \Gamma_\cM\left(\J^k\left(U,\R\right)\right) \cong \Ci\left(U_{\left(k\right)}\right).$$
	\end{lemma}
\begin{proof}
For simplicity, we will prove this for an even manifold $M$.
Fix $U \subseteq M$ an open subset. Define
\begin{eqnarray*}
\psi_U:\Ci\left(U_{\left(k\right)}\right) &\to& \Gamma_\cM\left(\J^k\left(U,\R\right)\right)\\
\left[f\right] &\mapsto& \left(x \mapsto j^1_{\left(x,x\right)}\left(f|_{\left\{x\right\}\times X}\right) \right),
\end{eqnarray*}
where
\begin{eqnarray*}
	\Ci\left(U \times U\right) & \to & \Ci\left(U \times U\right)/\I^{k+1}\\
	f &\mapsto& \left[f\right].
	\end{eqnarray*}
This clearly defines a natural transformation of presheaves. We claim that $\psi_U$ is an isomorphism. We will construct a map in the opposite direction. Fix a locally finite atlas $\left(V_i \hookrightarrow U\right),$ with $\varphi_i:V_i \stackrel{\cong}{\longrightarrow} \R^n$  and a partition of unity $\lambda$ subordinate to it. Note that each coordinate patch $\varphi_i$ induces a trivialization $$\J^k\left(U,\R\right)|_{V_i} \cong V_i \times \J^k_0\left(\R^n\right),$$ where $\J^k_0\left(\R^n\right)$ is the set of $k$-jets of smooth functions on $\R^n$ at the origin. I.e., these are induced jet coordinates on $\J^k\left(U,\R\right),$ $\left(x^1,\ldots,x^n,u^I\right),$ with $u^I$ ranging over all multi-indices with $|I| \le k.$ By multiplying with $1=\underset{i} \sum \lambda_i,$ we get that
$$\sigma=\underset{i} \sum \sigma^i,$$ with each $\sigma^i$ a section supported on $V_i.$ Hence, in local coordinates we can write
$$\sigma^i= \underset{|I|\le k} \sum \sigma^i_I u^I,$$ with each $\sigma^i_I$ a smooth function on $V_i.$ Define $f^i_\sigma:V_i \times V_i \to \R$ by
$$f^i_\sigma\left(x,y\right)=\underset{|I| \le k} \sum \sigma^i_I\left(x\right)\left(x-y\right)^I.$$
Then $\underset{i} \sum f^i_\sigma\left(x,y\right)$ is a smooth function defined on some neighborhood $W$ of the diagonal in $U \times U,$ independent of $\sigma.$ Let $\beta$ be a bump function supported in $W.$ Then $$f_\sigma:=\beta \cdot \underset{i} \sum f^i_\sigma\left(x,y\right)$$ is a global function on $U \times U.$ Fixing the choice of coordinate charts and W, define $\theta\left(\sigma\right)=\left[f_\sigma\right].$ Notice that $\psi_U\left(\theta\left(\sigma\right)\right) =\sigma$ and $\theta\left(\psi_U\left(\left[f\right]\right)\right)=\left[f\right].$ It follows that $\psi_U$ is an isomorphism.
\end{proof}

\begin{proposition}\label{prop:nbdcosh}
	The assignment $U \mapsto U_{\left(k\right)}$ is a cosheaf on $\cM_0$ with values in $\bH.$
	\end{proposition}
\begin{proof}
	The assignment $U \mapsto \Ci\left(U_{\left(k\right)}\right)$ can be identified the sheaf of sections of $\J^k\left(\cM,\R\right),$ since for all $U,$ 
	$$\J^k\left(\cM,\R\right)|_{U}\cong \J^k\left(U,\R\right).$$ Applying $\Speci,$ the result now follows.
	\end{proof}

More generally, there is the following result:

\begin{proposition}\cite[Chapter 7]{SYN} \label{prop:SYN}
	Let $\pi:E \to M$ be a surjective submersion between finite dimensional manifolds. Then $\J^k\left(M,E\right) \to M$ can be identified with $\beta^k_*\alpha^*_k\left(\pi\right)$ in $\bH/M.$
	\end{proposition}

\begin{corollary}\label{cor:jetfrech}
	Let $\pi:E \to M$ be a surjective submersion between finite dimensional manifolds. Then the following are equivalent:
	\begin{itemize} 
		\item[1.] $S_{\mathsf{Conv}}\left(\J^\i\left(\pi\right)\right) \to M,$ where $\J^\i\left(\pi\right)$ is given its usual Frech\'et manifold structure.
		\item[2.] The limit $\underset{k} \lim \beta^k_*\alpha^*_k\left(\pi\right).$
		\end{itemize}
\end{corollary}

\begin{proof}
This follows immediately from Corollary \ref{cor:colimfre} and Proposition \ref{prop:SYN}.
\end{proof} 
	
\begin{proposition}
Let $$\xymatrix{P\ar[r]^-{b} \ar[d]_-{a} \ar[r] & Z \ar[d]^-{g} \\
X \ar[r]^-{f} & Y}$$
be a pullback diagram in an infinity topos $\cX.$ Then there is a canonical equivalence of endofunctors of $\cX/Z$ 
$$b_*a^* \simeq g^*f_*.$$
\end{proposition}

\begin{proof}
Let $w:W \to Z$ be an object of $\cX/Z.$
Consider the diagram of pullback squares
$$\xymatrix{P \times_{Z} W \ar[r]^-{\pi_2} \ar[d]_-{\pi_1} & W \ar[d]^-{w}\\
P \ar[r]^-{b} \ar[d]_-{a} & Z \ar[d]^-{g}\\
X \ar[r]_-{f} & Y}.$$
On one hand, $b^*w$ is $\pi_1:P \times_{Z} W \to P \in \cX/P,$ and hence
$a_!b^*w$ is just $a \circ \pi_1 \in \cX/X.$ On the other hand, $g_!w$ is $g \circ w \in \cX/Y,$ and, by pasting pullback diagrams, we see that $f^*g_!w$ is thus 
$X \times_{Y} W \to X.$ However, this is just the left vertical map in the outer pullback square (formed by composition), and hence is canonically equivalent to $a \circ \pi_1$ also. Hence, there is a canonical equivalence of functors
$$a_!b^*w \simeq f^*g_!.$$ Both of these functors are left adjoints, so there corresponding right adjoint must also be equivalent, and these are $b_*a^*$ and $g^*f_*$ respectively.
\end{proof}

\begin{lemma}\label{lem:wdw}
Suppose that $$\xymatrix{P\ar[r]^-{b} \ar[d]_-{a} \ar[r] & Z \ar[d]^-{g} \\
X \ar[r]^-{f} & Y}$$
is a pullback diagram in an $\i$-topos $\cX$ and
$$P \simeq \underset{k\in {\cK}_0}{\colim} P_k$$ is an expression of $P$ as a colimit over some small $\i$-category $\cK.$ For each $k \in \cK_0,$ consider the induced maps $$b_k:P_k \to Z$$ and $$a_k:P_k \to X.$$ Notice that for each such $k,$ there is a commutative square
$$\xymatrix{P_k\ar[r]^-{b_k} \ar[d]_-{a_k} \ar[r] & Z \ar[d]^-{g} \\
X \ar[r]^-{f} & Y}$$
(which is no longer necessarily a pullback diagram). Then there is a canonical equivalence of endofunctors of $\cX/X$
$$\underset{k \in \cK_0} \lim b^k_*a^*_k \simeq b^*a_*.$$
\end{lemma}

\begin{proof}
Let $w:W \to Z \in \cX/Z$ be arbitrary. By the universality of colimits, we have a pullback diagram
$$\xymatrix{{\underset{k \in \cK_0} \colim \left(P_k \times_{Z} W\right)} \ar[r] \ar[d] & W \ar[d]^-{w}\\
P \ar[r]_-{b} & Z.}$$
Let $\mu$ be a colimiting cocone with components $\mu_k.$ Then we have 
$$b^*w \simeq \underset{k \in \cK_0} \colim \left(\mu_k\right)_!\left(b_k^* w\right).$$ 
Since $a_!$ preserves colimits,
\begin{eqnarray*}
a_!b^*w \simeq \underset{k \in \cK_0} \colim a_!\left(\mu_k\right)_!\left(b_k^* w\right) &\simeq & \underset{k \in \cK_0} \colim \left(a\circ \mu_k\right)_!\left(b_k^* w\right)\\
\end{eqnarray*}
So, for any $w$ in $\cX/W$ and $q:Q \to X$ in $\cX/X,$
\begin{eqnarray*}
\Map_{\cX/X}\left(a_!b^*w,q\right) & \simeq & \Map_{\cX/X}\left(\underset{k \in \cK_0} \colim \left(a_k\right)_!\left(b_k^* w\right),q\right)\\
& \simeq &\underset{k \in \cK_0} \lim  \Map_{\cX/X}\left(\left(a_k\right)_!\left(b_k^* w\right),q\right)\\
&\simeq & \underset{k \in \cK_0} \lim  \Map_{\cX/W}\left(w,b^k_*a_k^* q\right)\\
&\simeq & \Map_{\cX/W}\left(w,\underset{k \in \cK_0} \lim b^k_*a_k^* q\right).
\end{eqnarray*}
The result now follows by the uniqueness of adjoints.
\end{proof}

\begin{corollary}\label{cor:frechmfd}
Let $\pi:E \to M$ be a surjective submersion between finite dimensional manifolds. Then there is a canonical identification
$$\eta_M^*\eta^M_*\left(\pi\right) \simeq S_{\mathsf{Conv}}\left(\J^\i\left(\pi\right)\right) \to M,$$ where $\J^\i\left(\pi\right)$ is the infinite jet bundle regarded as a Frech\'et manifold.
\end{corollary}

\begin{proof}
By definition of $T^\i M,$ we have a pullback diagram
$$\xymatrix{T^\i M \ar[d] \ar[r] & M \ar[d]^-{\eta_M}\\
M \ar[r]_-{\eta_M} & M_{dR}.}$$
By Proposition \ref{prop:bTloc}, we have
$$T^\i M \simeq \underset{k} \colim M_{(k)},$$
so we are in the situation of Lemma \ref{lem:wdw}.
Therefore 
$$\eta^*\eta_* \left(\pi\right) \simeq \underset{k} \lim b^k_*a_k^*\left(\pi\right).$$
So we are done by Corollary \ref{cor:jetfrech}.
\end{proof}

This justifies the following definition in general:

\begin{definition}
	Let $\pi:E \to M$ be an object of $\bH/M,$ and consider the unit map $$\eta:M \to M_{dR}.$$ Then \textbf{the infinite jet bundle of $\pi$} is the object $\eta^*\eta_*\left(\pi\right)$ in $\bH/M.$
	\end{definition}

\begin{remark}
Observe that we have an adjunction:
$$\bT^\i_M \dashv \J^\i.$$
\end{remark}

\subsubsection{Jet prolongations}
Given a fiber bundle $\pi:E \to M,$ and a section $\sigma,$ for any $0 \le k \le \i,$ there is its $k^{th}$-jet prolongation $j^k\left(\sigma\right)$ which is a section of $\J^k\left(\pi\right).$ This construction leads to a set-theoretic map $$j^k:\Gamma\left(\pi\right) \to \Gamma\left(\J^k\left(\pi\right)\right)$$ and one way wonder if this map is in a suitable sense smooth, or better, in topos-theoretic language, whether or not this is just the evaluation on the point of a map $$j^k:\Sec_M\left(\J^k\left(\pi\right)\right) \to \Sec_M\left(\J^k\left(\pi\right)\right)$$ in $\bH.$ We shall show that this is indeed the case.

We will firstly dispense with the case $k=\i,$ which will hold for an arbitrary $\pi:E \to M$ in $\bH.$

\begin{lemma}
	There is a canonical natural transformation $$J^\i:\Sec_M\left(\blank\right) \times M \Rightarrow \J^\i.$$
	\end{lemma}

\begin{proof}
Consider the diagram
$$\xymatrix{E \ar[rd]^-{\pi} & &\\
	& M \ar[ld]_-{t} \ar[r]^-{\eta} & M_{dR} \ar[lld]^-{s}\\
	\ast & &}$$
Notice that $t=s \eta,$ so on one hand we have
\begin{eqnarray*}
	\Sec_M\left(\blank\right) \times M &\simeq& t^*t_*\\
	&\simeq&  \eta^*s^*s_*\eta_*.
	\end{eqnarray*}
On the other hand, we have that $\J^\i=\eta^*\eta_*,$ so the counit $$s^*s_* \Rightarrow id_{\bH/M_{dR}}$$ induces a canonical natural transformation $$J^\i:\Sec_M\left(\blank\right) \times M \Rightarrow \J^\i.$$
\end{proof}

Now for any $E \to M$ in $\bH/M,$ $\Sec_M\left(E\right)=t_*\left(E\right),$ and there is a canonical map
$$t_*\left(E\right) \to t^*t_*t^*\left(E\right)$$ induced by the unit map $$id \Rightarrow t_*t^*.$$ We can now consider the composite 
$$t_*\left(\pi\right) \to t_*t^*t_*\left(\pi\right) \stackrel{t_*J^\i}{\longlongrightarrow} t_*\eta^*\eta_*\left(E\right),$$ which we denote by $j^\i\left(E\right).$ Clearly, these maps assemble into a natural transformation $$j^\i:\Sec_M \Rightarrow \Sec_M \circ \J^\i.$$
\begin{definition}
The above transformation $j^\i$ is called \textbf{the $\i$-jet prolongation}
\end{definition}


Now, let $M$ be a manifold, but allow $E$ to be arbitrary. Consider the diagram
$$\xymatrix{E \ar[rd]^-{\pi} & & \\
	M_{\left(k\right)} \ar@<-0.5ex>[r]_-{\beta} \ar@<+0.5ex>[r]^-{\alpha} & M \ar[r]^-{t} & \ast.}$$
Notice that since $\ast$ is terminal, $t\circ \alpha \simeq t \circ \beta.$ Consequently
\begin{eqnarray*}
	\Gamma_M\left(\J^k\left(\pi\right)\right) &\simeq & t_* \beta_* \alpha^*\left(\pi\right)\\
	& \simeq & \left(t \beta\right) _* \alpha^*\left(\pi\right)\\
	& \simeq & \left(t \alpha\right)_* \alpha^* \left(\pi\right)\\
	& \simeq & t_* \alpha_* \alpha^*\left(\pi\right).
	\end{eqnarray*}
Denote by $\lambda$ the unit of the adjunction $\alpha^* \dashv \alpha_*.$ Then we have the canonical map
$$t_*\left(\lambda_\pi\right):t_*\left(\pi\right) \to t_* \alpha_* \alpha^*\left(\pi\right),$$
which can be identified with a map $$\Sec_M\left(\pi\right) \to \Sec_M\left(\J^k\left(\pi\right)\right),$$ which can easily be checked to be a natural transformation
$$j^k:\Sec_M \Rightarrow \Sec_M \circ \J^k.$$
\begin{definition}
	The above transformation $j^k$ is called \textbf{the $k$-jet prolongation}.
	\end{definition}


\subsection{Differential operators}

Notice that if $D:\J^\i E \to E$ is a morphism in $\bH/M,$ there is an induced map $\widehat{D}$ of stacks of sections, defined as the composite
$$\Sec_M\left(E\right) \stackrel{j^\i}{\longrightarrow} \Sec_M\left(\J^\i E\right) \stackrel{\Sec_M\left(D\right)}{\longlongrightarrow} \Sec_M\left(F\right).$$

\begin{definition}
	Given $E,F \in \bH/M,$ the \textbf{space of (non-linear) differential operators} from $\Sec_M\left(E\right)$ to $\Sec_M\left(F\right)$ is the essential image of
	$$\Map_{\bH/M}\left(\J^\i E , F\right) \stackrel{\Sec_M}{\longlongrightarrow} \Map_{\bH}\left(\Sec_M\left(\J^\i E\right) , \Sec_M\left(F\right)\right) \stackrel{\left(j^\i\right)^*}{\longlongrightarrow}  \Map_{\bH}\left(\Sec_M\left(E\right) , \Sec_M\left(F\right)\right).$$
	The space of \textbf{formal differential operators} is simply the space $\Map_{\bH/M}\left(\J^\i E , F\right).$
	A $\widehat{D}$ is a differential operator \textbf{of order $k$} if there is a factorization of $D$ as
	$$\J^\i\left(E\right) \to \J^k\left(E\right) \to \F.$$
		\end{definition}
	
	Note that since $\J^\i=\eta^*\eta_*,$ and $\eta^* \dashv \eta_*,$ $\J^\i$ inherits the structure of a comonad. Consider the comonadic adjunction
	$$\xymatrix{\bH/M \ar@<-0.5ex>[r]_-{C} & \mathbf{CoAlg}_\J^\i\left(\bH/M\right) \ar@<-0.5ex>[l]_-{U}}$$ between the $\i$-category of Eilenberg-Moore coalgebras for $\J^\i$ and $\bH/M,$ where $U$ is the forgetful functor, and $C$ is the cofree coalgebra functor, with $C \vdash U.$ Notice that for $X,Y \in \bH/M,$
	\begin{eqnarray*}
	\Map\left(CX,CY\right) &\simeq& \Map\left(UCX,Y\right)\\
	&\simeq & \Map\left(\J^\i X,Y\right),
	\end{eqnarray*}
	that is the space of $\J^\i$-coalgebra maps from $CX$ to $CY$ is the same as the space of formal differential operators from $X$ to $Y$.
	
	\begin{definition}
		The \textbf{$\i$-category of formal differential operators on M,} $\mathsf{FDiff_M},$ is the essential image of $C$ above.
		\end{definition}
	
	\begin{lemma}\label{lem:diffcomp}
		The composition of two differential operators is a differential operator.
		\end{lemma}
	\begin{proof}
		The follows by using the composition in the $\i$-category $\mathbf{CoAlg}_\J^\i\left(\bH/M\right).$ Suppose that $D:\J^\i\left(E\right) \to F$ and $D':\J^\i\left(F\right) \to G$ are formal differential operators. Define $D' \star D$ to be the following composition
	$$\J^\i\left(E\right) \stackrel{\mu_E}{\longlongrightarrow} \J^\i \J^\i\left(E\right) \stackrel{\J^\i\left(D\right)}{\longlongrightarrow} \J^\i\left(F\right) \stackrel{D'}{\longrightarrow} G,$$ where $\mu$ comes from the comonad structure on $\J^\i=\eta^*\eta_*.$ A careful unwinding of the definitions shows that the induced differential operator $$\Sec_M\left(E\right) \to \Sec_M\left(G\right)$$ is equivalent the composite $\hat{D'} \circ \hat{D}.$
	\end{proof}

Often, but certainly not always, a given differential operator $B$ has a unique formal differential operator $D$ inducing it. For example, suppose that $M$ is a manifold, and $E \to M$ is a surjective submersions onto $M,$ then this is a classical fact.

\begin{proposition}
	Let $E \to M$ be in $\bH/M$. All differential operator from $\Sec_M\left(E\right)$ to $\Sec_M\left(F\right),$ for all $F \to M$ in $\bH/M,$ uniquely determine a corresponding formal differential operator inducing them, (up to equivalence) if and only if $$J^\i_M:M \times \Sec_M\left(E\right) \to \J^\i E$$ is an epimorphism.
	\end{proposition}

\begin{proof}
Suppose that $$D,D':\J^\i E \to F$$ are two formal differential operators that $\widehat{D} \simeq \widehat{D'}.$ Notice that $$\widehat{D} \simeq \Sec_M\left(D \circ \phi\right),$$ and similarly for $D'.$ Moreover, we have a commutative diagram
$$\xymatrix{M \times \Sec_M\left(E\right) \ar[r]^-{J^\i_M} \ar[rd]_-{id_M \times j^\i} & \J^\i E \ar[r]^{D} & F \\
	& M \times \Sec_M\left(\J^\i E\right) \ar[u]^-{\epsilon_{\J^\i E}} \ar[r]_-{\widehat{D}} & M \times \Sec_M\left(F\right) \ar[u]_-{\epsilon_F},}$$
where $\epsilon$ is the counit of the adjunction $t^* \dashv t_*=\Sec_M$ (and similarly for $D'$). It follows that $D \circ J^\i_M$ is completely determined by $\widehat{D}.$ This completely determines $D$ for all differential operators $D$ if and only if $J^\i_M$ is an epimorphism (by definition).
\end{proof}

\begin{proposition}\label{prop:surjdiffop}
	If $E \to M$ is a surjective submersion of manifolds, and let $F \to M$ be such that $F$ is a concrete sheaf (aka diffeological space), then any differential operator 
	$$B:\Sec_M\left(E\right) \to \Sec_M\left(F\right)$$ is induced by a unique formal differential operator
	$$D:\J^\i E \to F.$$
	\end{proposition}

\begin{proof}
Each point of $\J^\i E$ is an infinite jet class of a section of $E$ at a point, so it follows that $\phi$ is surjective on points, or, more precisely, that $J^\i_M\left(\ast\right)$ is surjective. However, both $\J^\i E$ and $F,$ are concrete sheaves, so any differential operator (i.e. any map) $D:\J^\i E \to F$ is uniquely determined by
$$D\left(\ast\right):\J^\i E\left(\ast\right) \to F\left(\ast\right),$$ hence $D=D'$ if and only if $$D\left(\ast\right)=D'\left(\ast\right).$$ Now suppose that $D \circ J^\i_M = D' \circ J^\i_M,$ then $$D\left(\ast\right) \circ J^\i_M\left(\ast\right)=D'\left(\ast\right) \circ J^\i_M\left(\ast\right),$$
and since $J^\i_M\left(\ast\right)$ is an epimorphism of sets, this is implies that $D\left(\ast\right) =D'\left(\ast\right),$ and hence $D=D'.$
\end{proof}

\begin{corollary}\label{cor:diffokpoints}
	If $E \to M$ is a surjective submersion of manifolds, and let $F \to M$ be such that $F$ is a concrete sheaf (aka diffeological space), then any differential operator 
	$$B:\Sec_M\left(E\right) \to \Sec_M\left(F\right)$$ is uniquely determined by its value on the point, i.e. the map of sets
	$$B\left(\ast\right):\Gamma_M\left(E\right) \to \Gamma_M\left(F\right).$$
	\end{corollary}

\begin{proof}
	This follows immediately from the proof of Proposition \ref{prop:surjdiffop} since $B\left(\ast\right)$ uniquely determines $D\left(\ast\right)$.
	\end{proof}

\section{Monadic and comonadic descent}\label{sec:monad}
Let $\sC$ be an $\i$-category. Recall that a \emph{monad} on $\sC$ is a monoid object in the monoidal $\i$-category $\Fun\left(\sC,\sC\right),$ with monoidal product given by composition.

Recall that given an adjunction
$$\xymatrix{\sD \ar@<-0.5ex>[r]_-{R} & \sC \ar@<-0.5ex>[l]_-{L},}$$
	it produces a monad $T:=RL$ on the category $\sC,$ and a comonad $C:=LR$ on $\sD.$ Given any object $X \in \sD,$ $RX$ carries the canonical structure of $T$-algebra. This produces a canonical functor
	$$k:\sD \to \Alg_{T}\left(\sC\right)$$ which commutes over $\sC.$ The adjunction is called \emph{monadic} if $k$ is an equivalence of $\i$-categories. Dually, given an object $Y \in \sC,$ the object $LY$ carries the canonical structure of a $C$-coalgebra, and the adjunction is called \emph{comonadic} if the induced functor
	$$\sC \to \CoAlg_{C}\left(\sD\right)$$ is an equivalence.
	
	We have the following theorem characterizing when this happens:

\begin{theorem}[$\i$-Categorical Barr-Beck Theorem]\label{thm:BarrBeck} \cite[Theorem 4.7.3.5]{LurieHA}
An adjunction 
$$\xymatrix{\sC \ar@<-0.5ex>[r]_-{G} & \sD \ar@<-0.5ex>[l]_-{F}}$$
is monadic, if and only if

\begin{itemize}
\item[i)] The functor $G:\sD \to \sC$ is conservative
\item[ii)] If $$V^\bullet:\Delta^{op} \to \sD$$ is a simplicial object of $\sD$ which is $G$-split, then a colimit for $V^\bullet$ exists and $G$ preserves it.
\end{itemize}
\end{theorem}

\begin{definition}
	For a monad $T$ on an $\i$-category $\sC,$ the \textbf{Kleisli category} of $T$ is the full subcategory $\Kl\left(T\right)$ of $\Alg_T\left(\sC\right)$ on the free $T$-algebras, i.e. the essential image of the left adjoint $F$ to the forgetful functor $$U:\Alg_T\left(\sC\right) \to \sC.$$ The objects of $\Kl\left(T\right)$ may identified with the objects of $\sC,$ and under this identification the morphisms $X \to Y$ correspond to morphism $X \to TY$ in $\sC.$ The composition of two such morphisms $$f:X \to TY$$ and $$g:Y \to TZ$$ is given as the composition
	$$X \stackrel{f}{\longrightarrow} TY \stackrel{Tg}{\longrightarrow} T^2 Z \stackrel{\mu_Z}{\longrightarrow} TZ,$$ where $$\mu:T^2 \to T$$ is the monad product.
	
	Note that since the essential image of $F$ factors through $\Kl\left(T\right),$ there is an induced adjunction 
	$$\xymatrix@C=2cm{\Kl\left(T\right) \ar@<-0.5ex>[r]_-{U_{\Kl}} & \sC.  \ar@<-0.5ex>[l]_-{F_{\Kl}}}$$
	We make an analogous (dual) definition of $\Kl\left(C\right)$ is the case of a comonad, where the morphisms $X \to Y$ correspond to morphism $CX \to Y$ in the underlying $\i$-category.
\end{definition}

The following lemma gives as useful description of the $\i$-category of algebras for a monad in terms of its Kleisli category:

\begin{lemma}\label{lem:AlgKL}
	Given a monad $T$ on an $\i$-category $\sC,$ there is a pullback diagram of $\i$-categories	$$\xymatrix{\Alg_T\left(\sC\right) \ar@{^(_->}[r] \ar[d] & \Psh\left(\Kl\left(T\right)\right) \ar[d]^-{\left(F_\Kl\right)^\ast} \\
		\sC \ar@{^(_->}[r]^-{y} & \Psh\left(\sC\right).}$$
	\end{lemma}

\begin{proof}
	Since $\Kl\left(T\right)$ is locally small, the above pullback remains unchanged under change to a larger universe. Make this change if necessary, so that $\Kl\left(T\right)$ is essentially small with respect to this universe. Consider the functor $$\left(F_\Kl\right)^*:\Psh\left(\Kl\left(T\right)\right) \to \Psh\left(\sC\right).$$ It has a left adjoint $$\left(F_\Kl\right)_!=\Lan_{F_\Kl}\left(\blank\right)$$ and a right adjoint
	$$\left(F_\Kl\right)_*=\operatorname{Ran}_{F_\Kl}\left(\blank\right)$$ given by global left and right Kan extension respectively. Similarly, there is a sequence of adjunctions
	$$\left(U_\Kl\right)_! \dashv \left(U_\Kl\right)^* \dashv \left(U_{\Kl}\right)_*.$$ Since $F_{\Kl} \dashv U_{\Kl},$ it readily follows that $$\left(F_\Kl\right)_!\simeq \left(U_\Kl\right)^*$$ and $$\left(F_\Kl\right)^* \simeq \left(U_{\Kl}\right)_!.$$ Since 
	$$\left(F_\Kl\right)_! \dashv \left(F_\Kl\right)^* \simeq \left(U_\Kl\right)_!,$$
	it follows that the induced monad from the above adjunction can be identified with
	$$\left(U_\Kl\right)_! \circ \left(F_\Kl\right)_! \simeq T_!.$$ We also have the description
	$$T_! \simeq \Lan_{y} \left(y \circ T\right),$$ and so $T_!$ is the unique colimit preserving endofunctor of $\Psh\left(\sC\right)$ which satisfies
	$$T_!y\left(C\right) \simeq y\left(TC\right)$$ for all $C \in \sC.$ It follows that the Yoneda embedding induces a fully faithful inclusion
	$$\Alg_{T}\left(\sC\right) \hookrightarrow \Alg_{T_!}\left(\Psh\left(\sC\right)\right).$$
	Observe that since $F_{\Kl}$ is essentially surjective, it follows immediately that $\left(F_\Kl\right)^*$ is conservative, and since it has a right adjoint $\left(F_\Kl\right)_*$ it preserves all colimits. It follows from Theorem \ref{thm:BarrBeck}, that the adjunction
	$$\left(F_\Kl\right)_! \dashv \left(F_\Kl\right)^* \simeq \left(U_\Kl\right)_!,$$ is monadic, and hence there is a canonically induced equivalence of $\i$-categories
	$$k:\Psh\left(\Kl\left(T\right)\right) \stackrel{\sim}{\longrightarrow} \Alg_{T_!}\left(\Psh\left(\sC\right)\right).$$
	Using the above equivalence $k,$ we can identify the essential image of $\Alg_{T}\left(\sC\right)$ under the fully faithful inclusion into $\Alg_{T_!}\left(\Psh\left(\sC\right)\right)$ with the full subcategory of $\Psh\left(\Kl\left(T\right)\right)$ on those presheaves $X$ such that $\left(F_{\Kl}\right)^*$ is representable. The result now follows.
	\end{proof}

\begin{lemma} \label{lem:epicon}
	Let $f:E \to F$ be an epimorphism in an $\i$-topos $\cE$. Then 
	$$f^*:\cE/F \to \cE/E$$ is conservative.
	\end{lemma}

\begin{proof}
	Suppose that $a:A \to Y$ $b:B \to Y$ and $h:a \to b$ in $\cE/Y$ are such that $f^*h$ is an equivalence. Consider the following commutative diagram in which each square is a pullback square
	$$\xymatrix@R=0.5cm@C=1.2cm{ & f^\ast B \ar@{->>}[r]^-{q} \ar[lddd] & B \ar[lddd]  \\
		f^\ast A \ar[dd] \ar[ru]^-{\underset{\sim}{f^\ast h}} \ar@{->>}[r]^-{p} & A \ar[dd] \ar[ru]^-{h}&\\
		& & \\
		X \ar@{->>}[r]^-{f} & Y. &}$$
	The double headed arrows are effective epimorphisms. The result now follows from \cite[Lemma 6.2.3.16]{htt}.
	\end{proof}

\begin{theorem}
	Let $f:E \to F$ be an epimorphism in an $\i$-topos $\cE.$ Let $M$ be the monad $f_*f^*$ and let $C$ be the comonad $f^*f_!,$ both of which are on $\cE/E$. Then there are canonical equivalences
	$$\mathbf{CoAlg}_{C}\left(\cE/E\right) \simeq \cE/F \simeq \mathbf{Alg}_M\left(\cE/E\right)$$ between the Eilenberg-Moore category of coalgebras for $C,$ and the slice $\i$-topos $\cE/F,$ and the Eilenberg-Moore category of algebras for $M.$ More precisely, the adjunction $f_! \dashv f^*$ is monadic and $f^* \dashv f_*$ is comonadic.
	\end{theorem}

\begin{proof}
	By the dual to \cite[Theorem 4.7.3.5]{LurieHA}, to show that the adjunction
$$\xymatrix@1{\cE/E \ar@<-0.5ex>[r]_-{f_*} & \cE/F \ar@<-0.5ex>[l]_-{f^*}}$$
is comonadic, it suffices to show that certain types of limits are preserved by $f^*$--- but here there is nothing to show since $f^*$ is right adjoint to $f_!,$ and $f^*$ is conservative--- which is the content of Lemma \ref{lem:epicon}. The dual argument is analogous: we need that $f^*$ is conservative--- which we just proved--- and we need $f^*$ to preserve certain colimits, but it preserves all of them since it is a left adjoint.
\end{proof}

\begin{corollary}
	If $\cX \in \bH$ is formally smooth, then there are canonical equivalences of $\i$-categories
	$$\mathbf{Alg}_{\bT^\i_M}\left(\bH/E\right) \simeq \mathbf{CoAlg}_{\J^\i_{\cX}}\left(\bH/\cX\right) \simeq \bH/\cX_{dR}$$
	\end{corollary}

\begin{lemma}\label{lem:cons_spt}
	Let $f:E \to F$ be an epimorphism in an $\i$-topos $\cE$. Then the induced functor
	$$f^\star:\Stab\left(\cE/F\right) \to \Stab\left(\cE/E\right)$$ is conservative.
	\end{lemma}

\begin{proof}
	Informally, objects of the tangent $\i$-topos consist of pairs $\left(E,\cG\right)$ of an object $E \in \cE,$ and $\cG \in \Stab\left(\cE/E\right),$ and morphisms $$\left(E,\cG\right) \to \left(F,\cF\right)$$ consists of pairs $\left(f,\alpha\right)$ where $$g:E \to F$$ and $$\alpha:g^*\cF \to \cG.$$ Consider the functor $$\zeta:\cE \to \T\cE$$ which sends an object $E$ to $\left(E,0\right).$ It is left adjoint to the forgetful functor $$q:\T\cE \to \cE.$$ by the initiality of the zero spectrum. Notice that $\Stab\left(\cE/E\right)$ can be identified with the fiber
	$$\xymatrix{\Stab\left(\cE/E\right) \ar[r]^-{\mu} \ar[d] & \T\cE \ar[d]^-{q}\\
		\ast \ar[r]^-{E} & \cE.}$$
	Observe that a morphism in $\T\cE$ corresponding to a pair $\left(g,\alpha\right)$ is an equivalence if and only if both $g$ and $\alpha$ are. It follows that the above canonical functor
	\begin{eqnarray*}
		\mu_E:\Stab\left(\cE/E\right) &\to& \T\cE\\
		\cG &\mapsto& \left(E,\cG\right)
	\end{eqnarray*}
	is conservative. Let $\cF$ be an object in the stabilization of $\cE/F.$ Then we have a canonical pullback diagram in $\T\cE$
	$$\xymatrix@C=2.5cm{\mu_E\left(\eta^\star\cF\right)=\left(E,\eta^\star \cF\right) \ar[d] \ar[r] & \left(F,\cF\right)=\mu_F\left(\cF\right) \ar[d]\\
		\left(E,0\right)=\mu_E\left(0\right)=\zeta\left(E\right) \ar[r]^-{\zeta\left(f\right)} & \zeta\left(F\right)=\left(F,0\right)=\mu_F\left(0\right).}$$
In other words, $$\mu_E\left(f^\star \cF\right)=\zeta\left(f\right)^*\mu_F\left(\cF\right),$$
where $$\zeta\left(f\right)^*:\T\cE/\zeta\left(F\right) \to \T\cE/\zeta\left(E\right).$$ Since both $\mu_E$ and $\mu_F$ are conservative, it suffices to show that $\zeta\left(f\right)^*$ is conservative. However, since $\zeta$ is a left adjoint, it preserves epimorphisms, and hence $\zeta\left(f\right)$ is an epimorphism, and we are done by Lemma \ref{lem:epicon}.
\end{proof}

\begin{lemma}
	Let $F:\cE \to \cF$ be an essential geometric morphism of $\i$-topoi, i.e. the inverse image functor $F^*$ has a left adjoint $F_!.$ Then for any presentable $\i$-category $\sC,$ there is a triple of adjunctions
$$
\xymatrix@R=3cm{
	\Shv\left(\cE,\sC\right)
	\ar@<9pt>[rr]|-{f_!}
	\ar@<0pt>@{<-}[rr]|-{f^\ast}
	\ar@<-9pt>[rr]|-{f_\ast}
	&&
	\;\Shv\left(\cF,\sC\right)}.$$
\end{lemma}

\begin{proof}
	Notice that $\Shv\left(\cE,\sC\right)$ is just the tensor product of $\cE$ and $\sC$ in the symmetric monoidal $\i$-category $\mathfrak{Pr}^L$ of presentable $\i$-categories and left adjoint functors. The inverse image functor $$F^*:\cF \to \cE$$ is an arrow in this category, so define $$f^*:=F^*\otimes \sC.$$ This functor is a left adjoint, so denote its right adjoint by $f_*.$ The functor $F_!:\cE \to \cF$ is also an arrow in $\mathfrak{Pr}^L.$ Denote by
	$$f_!:=F_! \otimes \sC.$$ This functor has a right adjoint, and to finish proving the lemma, we need to show it can be identified with $f^*.$ Consider the canonical functor
	$$\Fun^L\left(\cE,\cE\right) \to \Fun^L\left(\cE \otimes \sC,\cE \otimes \sC\right)$$ induced by $$\left(\blank\right) \otimes \sC:\mathfrak{Pr}^L \to \mathfrak{Pr}^L,$$ and consider the analogous functor $$\Fun^L\left(\cF,\cF\right) \to \Fun^L\left(\cF \otimes \sC,\cF \otimes \sC\right).$$
The diagrams expressing the triangle equations for the unit $\eta$ and counit $\varepsilon$ for the adjunction $F_! \dashv F^*$ are commutative diagrams in $\Fun^L\left(\cE,\cE\right)$ and $\Fun^L\left(\cF,\cF\right)$. Their images in  $\Fun^L\left(\cE \otimes \sC,\cE \otimes \sC\right)$ and $\Fun^L\left(\cF \otimes \sC,\cF \otimes \sC\right)$ respectively are the triangle equations expressing $\eta \otimes \sC$ and $\varepsilon \otimes \sC$ as the unit and co-unit for an adjunction between $f_!$ and $f^*.$
\end{proof}

\begin{corollary}
	If $f:E \to F$ is an epimorphism in an $\i$-topos $\cE$, the adjunction
$$\xymatrix@1{\Stab\left(\cE/E\right) \ar@<+0.5ex>[r]^-{f_!} & \Stab\left(\cE/F\right) \ar@<0.5ex>[l]^-{f^\star}}$$
is monadic, and the adjunction
$$\xymatrix@1{\Stab\left(\cE/E\right) \ar@<-0.5ex>[r]_-{f_\star} & \Stab\left(\cE/F\right) \ar@<-0.5ex>[l]_-{f^\star}}$$
is comonadic.
\end{corollary}

\begin{proof}
	Note that for an $\i$-topos $\cF,$ there are canonical equivalences
	$$\Shv\left(\cF,\Spt\right) \simeq \cF \otimes \Spt \simeq \Stab\left(\cF\right).$$
	By Lemma \ref{lem:cons_spt}, $f^\star$ is conservative, and as it has both a left and a right adjoint, the result follows from Theorem \ref{thm:BarrBeck}.
	\end{proof}

\begin{corollary}\label{cor:modcom}
	Let $\left(\cE,\O\right)$ be a $\SCi$-ringed $\i$-topos. Suppose that $f:E \to F$ in $\cE$ is an epimorphism. Then the adjunction
	$$\xymatrix@1{\Mod_{\O_E} \ar@<-0.5ex>[r]_-{f_\star} & \Mod_{\O_F} \ar@<-0.5ex>[l]_-{f^\star}}$$
is comonadic.
\end{corollary}

\begin{proof}
Notice that $f^\star \O_F \simeq \O_E,$ so that $f^\star\cF \simeq f^* \cF$ for all $\O_F$-modules $\cF.$ It follows that this $f^\star$ is the restriction of the one between stabilizations. Observe that
$$\Mod_{\O_E} \simeq \Mod_{H\O_E}\left(\Stab\left(\cE/E\right)\right),$$ hence $\Mod_{\O_E}$ is monadic over $\Stab\left(\cE/E\right)$ via $$H\O_E \wedge \left(\blank\right) \dashv U,$$ where $U$ is the forgetful functor. It follows that $U$ is conservative. The same is true for $F.$ Hence $f^\star$ is conservative. It also preserves limits, since it does on stabilizations and $U$ reflects limits. We are now done by Theorem \ref{thm:BarrBeck}.
\end{proof}

\begin{lemma}\label{lem:fstarqc}
	Let $f:E \to F$ be an epimorphism in $\bH$. Then if for an $\O_F$-module $\cF,$ $f^\star \cF$ is a quasi-coherent sheaf over $E,$ then $\cF$ is quasi-coherent over $F.$
	\end{lemma}

\begin{proof}
Notice that by Lemma \ref{lem:pbqce}, we can reduce to the case when $F \in \sD.$ In this case, since $f$ is an epimorphism, there is a cover $\left(U_\alpha \hookrightarrow F\right)_\alpha$ by affine open $\SCi$-subschemes over which $f$ admits local sections $\lambda_\alpha.$ It follows that for all $\alpha,$ $$\cF|_{U_\alpha} \simeq \lambda_\alpha^\star f^\star \cF,$$ and hence $\cF|_{U_\alpha}$ is quasi-coherent for all $\alpha.$ Let $M:=\Gamma_F\left(\cF\right) \in \widehat{\Mod}_{\A},$ for $F=\Speci\left(\A\right).$ We want to show that the canonical map
$$\epsilon:\MSpec_{\bH/F}\left(M\right) \to \cF$$ is an equivalence. Since
$$\pi:\underset{\alpha} \coprod U_\alpha \to \cF$$ is an epimorphism, it follows from Corollary \ref{cor:modcom} that $\pi^\star$ is conservative, so it suffices to check that $\pi^\star\left(\epsilon\right)$ is an equivalence. By hypothesis, letting $M_\alpha:=\cF\left(U_\alpha\right) \in \widehat{\Mod}_{\A_\alpha},$ where $U_\alpha=\Speci\left(\A_\alpha\right)$ for each $\alpha,$ we have that the canonical maps
$$\MSpec_{\bH/U_\alpha}\left(M_\alpha\right) \to \cF|_{U_\alpha}$$ are equivalences. But since the restriction of $\cF$ to the topos of sheaves over the underlying space of $F$ is an $\O_\A$-modules sheaf, and all such module sheaves are quasi-coherent, it follows that for all $\alpha,$
$$\cF\left(U_\alpha\right) \simeq \A_\alpha \underset{\A} {\widehat{\otimes}} M$$
and thus 
$$\MSpec_{\bH/U_\alpha}\left(M_\alpha\right) \simeq \MSpec_{\bH/F}\left(M\right)|_{U_\alpha},$$ and we are done.
\end{proof}

\begin{corollary}\label{cor:pbqc}
Let $f:E \to F$ be an epimorphism in $\bH$. Then the following diagram of $\i$-categories is a pullback:
$$\xymatrix@C=1.5cm{\QC\left(F\right) \ar[r]^-{\Lambda^\star\left(F\right)} \ar[d]_-{f^\star} & \Mod_{\O_F} \ar[d]^-{f^\star}\\
\QC\left(E\right) \ar[r]_-{\Lambda^\star\left(E\right)} & \Mod_{\O_E}}$$
\end{corollary}

Let $L:\sC \to \sD$ be a functor between $\i$-categories associated to a coCartesian fibration $$p:\mathcal{M} \to \Delta^1,$$ with $p^{-1}\left(0\right)=\sC$ and $p^{-1}\left(1\right)=\sD.$ By the universal property of pullbacks, there is a canonical isomorphism of simplicial sets,
$$\Fun_{\Delta^1}\left(\sC,\cM\right) \cong \Fun\left(\sC,\sC\right),$$ and similarly
$$\Fun_{\Delta^1}\left(\sD,\cM\right) \cong \Fun\left(\sD,\sD\right).$$ Denote by
$$i_0:\sC \hookrightarrow \cM$$ and $$i_1:\sD \hookrightarrow \cM$$ the canonical inclusions. Then restriction along these functors determine canonical maps $$i^*_0\Fun_{\Delta^1}\left(\cM,\cM\right) \to \Fun_{\Delta^1}\left(\sC,\sM\right)$$ and 
$$i_1^*:\Fun_{\Delta^1}\left(\cM,\cM\right) \to \Fun_{\Delta^1}\left(\sD,\sM\right).$$

\begin{lemma}\label{lem:adf}
In the situation above, suppose that $p$ is also a Cartesian fibration, exhibiting a right adjoint $R$ of $L.$ Then $i_1^*$ has a fully faithful right adjoint.
\end{lemma}

\begin{proof}
If we take $\cM$ to explicitly be relative nerve (\cite[Definition 3.2.5.2]{htt}) of the functor
$$\Delta^1 \to \Set^{\Delta^{op}}$$ whose image is $$L:\sC \to \sD,$$ then it follows by direct inspection that the canonical dotted map $\Psi$ below is a trivial Kan fibration:
$$\xymatrix{\Fun_{\Delta^1}\left(\cM,\cM\right) \ar@{-->}[rd]^-{\Psi} \ar@/^2.0pc/[rrrd] \ar@/_{2.0pc}/[rdd] &  & &\\
& \cP \ar[rr] \ar[d] & & \Fun\left(\sC,\sD\right)^{\Delta^1} \ar[d]^-{\langle ev_0,ev_1\rangle} \\
& \Fun\left(\sC,\sC\right) \times \Fun\left(\sD,\sD\right) \ar[rr]_-{L_\ast \times L^\ast} && \Fun\left(\sC,\sD\right) \times \Fun\left(\sC,\sD\right),}$$
where $\cP$ is the pullback in simplicial sets. Informally, an object of $\cP$ is a diagram of the form
$$\xymatrix{\sC \ar[r]^-{F} \ar[d]_-{L} & \sC \ar[d]^-{L} \ar@{=>}[ld]^-\alpha\\
\sD \ar[r]_-{G}& \sD.}$$
Notice that the right most arrow in the pullback diagram defining $\cP$ is a categorical fibration, hence the diagram is also a homotopy pullback in the Joyal model structure. In particular, this is a pullback in the $\i$-category of $\i$-categories. It follows that mapping spaces may be computed as pullbacks (in the $\i$-category $\Spc$)
$$\xymatrix{\Map_{\cP}\left(\left(F,G,\alpha\right),\left(F',G',\alpha'\right)\right) \ar[d] \ar[r] & \Map_{\Fun\left(\sC,\sD\right)^{\Delta^1}}\left(\alpha,\alpha'\right) \ar[d]\\
\Map_{\Fun\left(\sC,\sC\right)}\left(F,F'\right) \times \Map_{\Fun\left(\sD,\sD\right)}\left(G,G'\right) \ar[r] & \Map_{\Fun\left(\sC,\sD\right)}\left(LF,LF'\right) \times \Map_{\Fun\left(\sC,\sD\right)}\left(GL,GL'\right).}$$
Moreover, we have a pullback diagram
$$\xymatrix{\Map_{\Fun\left(\sC,\sD\right)^{\Delta^1}}\left(\alpha,\alpha'\right) \ar[r] \ar[d] & \Map_{\Fun\left(\sC,\sD\right)}\left(GL,G'L\right) \ar[d]^-{\alpha^\ast}\\
\Map_{\Fun\left(\sC,\sD\right)}\left(LF,LF'\right) \ar[r]_-{\alpha'_\ast} & \Map_{\Fun\left(\sC,\sD\right)}\left(LF,G'L\right).}$$
It follows that we have a pullback diagram
$$\xymatrix{\Map_{\cP}\left(\left(F,G,\alpha\right),\left(F',G',\alpha'\right)\right) \ar[d] \ar[r] & \Map_{\Fun\left(\sD,\sD\right)}\left(G,G'\right) \ar[d]^-{\alpha^\ast \circ L}\\
\Map_{\Fun\left(\sC,\sC\right)}\left(F,F'\right) \ar[r]_-{\alpha'_\ast \circ L} & \Map_{\Fun\left(\sC,\sD\right)}\left(LF,G'L\right).}$$
Now suppose $G:\sD \to \sD$ is a fixed functor. Then there is a canonical diagram
$$\xymatrix@C=2cm{\sC \ar[r]^-{RGL} \ar[d]_-{L} & \sC \ar[d]^-{L} \ar@{=>}[ld]^-{\varepsilon_{GL}}\\
\sD \ar[r]_-{G}& \sD,}$$
where $\varepsilon:LR \Rightarrow id_{\sD}$ is the co-unit of the adjunction. This defines a functor 
\begin{eqnarray*}
Q_1:\Fun\left(\sD,\sD\right) &\longrightarrow & \cP\\
G &\mapsto & \left(RGL,G,\varepsilon_{GL}\right).
\end{eqnarray*}
We will show that $Q_1$ is right adjoint to functor
\begin{eqnarray*}
pr_1:\cP &\longrightarrow & \Fun\left(\sD,\sD\right)\\
\left(F,G,\alpha\right) &\mapsto & G.
\end{eqnarray*}
Since $pr_1 \circ \Psi=i_1^*,$ this will prove that $i_1^*$ has a right adjoint. Notice that since $\varepsilon$ is the co-unit of the adjunction $L \dashv R,$ for any $$G:\sD \to \sD$$ and $$F':\sC \to \sC,$$
the composite
$$\xymatrix{\Map_{\Fun\left(\sC,\sC\right)}\left(F',RGL\right) \ar[r]^-{L} & \Map_{\Fun\left(\sC,\sD\right)}\left(LF',LRGL\right) \ar[r]^-{\left(\epsilon_{GL}\right)_\ast} & \Map_{\Fun\left(\sC,\sD\right)}\left(LF',GL\right),}$$
is a equivalence. It follows that for any $\left(F',G',\alpha'\right)$ in $\cP,$ since we have a pullback diagram
$$\xymatrix{\Map_{\cP}\left(\left(F',G',\alpha'\right),\left(RGL,G,\varepsilon_{GL}\right)\right) \ar[d] \ar[r] & \Map_{\Fun\left(\sD,\sD\right)}\left(G',G\right) \ar[d]^-{\alpha^\ast \circ L}\\
\Map_{\Fun\left(\sC,\sC\right)}\left(F',RGL\right) \ar[r]_-{\left(\varepsilon_{GL}\right)_\ast \circ L} & \Map_{\Fun\left(\sC,\sD\right)}\left(LF',GL\right),}$$
the map
$$\Map_{\cP}\left(\left(F',G',\alpha'\right),\left(RGL,G,\varepsilon_{GL}\right)\right) \longrightarrow \Map_{\Fun\left(\sD,\sD\right)}\left(G',G\right)$$ is an equivalence. Thus $Q_1$ is right adjoint to $i_1^*$

Notice moreover that for any endofunctor $G$ on $\sD,$
\begin{eqnarray*}
i_1^* \circ Q_1\left(G\right) &\simeq & i_1^*\left(RGL,G,\varepsilon_{GL}\right)\\
&\simeq & G.
\end{eqnarray*}
Hence the co-unit of the adjunction $i_1^* \dashv Q_1$ is an equivalence, and we deduce that $Q_1$ is fully faithful.
\end{proof}

\begin{proposition}\label{prop:lxmon}
Let $L \dashv R$ be an adjunction, with $L:\sC \to \sD$ a functor between $\i$-categories. Then the functor
\begin{eqnarray*}
\Xi:\Fun\left(\sD,\sD\right) &\to & \Fun\left(\sC,\sC\right)\\
G & \mapsto & R\circ G \circ L
\end{eqnarray*}
is lax monoidal with respect to the composition monoidal structures.
\end{proposition}

\begin{proof}
If $p:\cM \to \Delta^1$ is a biCartesian fibration exhibiting the adjunction, the functor
$$i^*_0:\Fun_{\Delta^1}\left(\cM,\cM\right) \to \Fun_{\Delta^1}\left(\sC,\cM\right)\cong \Fun\left(\sC,\sC\right)$$ is a strict map of simplicial monoids, and thus a monoidal functor between the associated monoidal $\i$-categories. Similarly for $i_1^*.$ By Lemma \ref{lem:adf}, $i_1^*$ has a right adjoint $\tilde Q_1.$ Under the equivalence $\Psi$ from the proof of Lemma \ref{lem:adf}, the composite of $i_0^* \tilde Q_1$ can be identified with $pr_0 \circ Q_1,$ where $Q_1$ is the explicit right adjoint to $pr_1$ given by
$$Q_1\left(G\right)=\left(LGR,G,\varepsilon_{GL}\right).$$ Thus, the functor $\Xi$ can be identified with $i_0^* \tilde Q_1.$ Since $\tilde Q_1$ is right adjoint to the monoidal functor $i_1^*,$ it is lax monoidal. Since $i_0^*$ is monoidal, it follows that $\Xi$ is also lax monoidal.
\end{proof}

\begin{lemma}\label{lem:monsup}
Let $L \dashv R$ exhibit $\sC$ as a coreflective subcategory of $\sD.$ Denote the co-unit of this adjunction by $\varepsilon.$ Suppose that $J$ is an endofunctor of $\sD$ such that for all $n,$ $RJ\left(\varepsilon_{J^nLX}\right)$ is an equivalence. Denote by $\Fun_J\left(\sD,\sD\right)$ the full subcategory of $\Fun\left(\sD,\sD\right)$ on $J^n,$ for all $n \ge 0,$ (with $J^0=id_{\sD}$). Then the composite
$$\Fun_J\left(\sD,\sD\right) \hookrightarrow \Fun\left(\sD,\sD\right) \stackrel{\Xi}{\longlongrightarrow} \Fun\left(\sC,\sC\right)$$ is monoidal.
\end{lemma}

\begin{proof}
The unit of the adjunction $L \dashv R$ is what is responsible for the non-strictness of the unit identities of the lax monoidal functor $\Xi.$ Since $L$ is fully faithful, the unit of the adjunction $L \dashv R$ is an equivalence. Hence, to show that $\Xi|_{\Fun_J\left(\sD,\sD\right)}$ is monoidal, it suffices to prove that all the comparison maps
$$\mu_n:\Xi\left(J\right) \circ \Xi\left(J\right) \circ \ldots \circ \Xi\left(J\right) \to \Xi\left(J \circ J \circ \ldots \circ J\right)$$ 
are equivalences. These maps are induced by the co-unit, and we have equivalences $$\mu_{n+1} \simeq RJ\left(\varepsilon_{J^nLX}\right) \circ RJL\left(\mu_n\right)$$ for all $n \ge 0.$ Moreover, $\mu_0$ is the unit, which is an equivalence.
\end{proof}

\begin{remark}\label{rmk:monom}
By essentially the same proof, since the monoidal structure on $$\Fun_{\Delta^1}\left(\cM,\cM\right) \simeq \left[\Fun\left(\sC,\sC\right) \times \Fun\left(\sD,\sD\right)\right] \times_{\Fun\left(\sC,\sD\right) \times \Fun\left(\sC,\sD\right)} \Fun\left(\sC,\sD\right)^{\Delta^1}$$ is given 	``pointwise,'' i.e.
$$\left(F',G',\alpha'\right) \otimes \left(G,G,\alpha\right)=\left(F\circ F', G\circ G',G'\left(\alpha\right) \circ \alpha'_F\right),$$ under the same assumptions as Lemma \ref{lem:monsup}, the composite
$$\Fun_J\left(\sD,\sD\right) \hookrightarrow \Fun\left(\sD,\sD\right) \stackrel{Q_1}{\longlongrightarrow} \Fun_{\Delta^1}\left(\cM,\cM\right)$$ is also monoidal.
\end{remark}

\begin{lemma}\label{lem:amthm}
Let $L \dashv R$ exhibit $\sC$ as a coreflective subcategory of $\sD.$ Denote the co-unit of this adjunction by $\varepsilon.$ Let $J$ be a comonad on $\sD$ such that for all $n,$ $RJ\left(\varepsilon_{J^nLX}\right)$ is an equivalence. Denote by $C=LR$ the comonad induced by the adjunction. Then $CJC$ has a canonical structure of a comonad and there is a commutative diagram of comonads (monoid objects in the opposite category of $\Fun\left(\sD,\sD\right)$)
$$\xymatrix{CJC \ar[d] \ar[r] & J \ar[d]\\
C \ar[r] & id_\sD.}$$
\end{lemma}

\begin{proof}
By Lemma \ref{lem:adf}, $Q_1$ is fully faithful, so by Remark \ref{rmk:monom}, it follows that the composite $$\Fun_J\left(\sD,\sD\right) \hookrightarrow \Fun\left(\sD,\sD\right) \stackrel{Q_1}{\longlongrightarrow} \Fun_{\Delta^1}\left(\cM,\cM\right)$$ is a fully faithful monoidal functor. Thus, $Q_1\left(J\right)=\left(RJL,J,\varepsilon_{JL}\right)$ is a comonoid object of $\Fun_{\Delta^1}\left(\cM,\cM\right).$ Properly dualizing Proposition \ref{prop:lxmon}, $i_0^*$ has a \emph{left} adjoint $Q_0$ given by 
$$Q_0\left(F\right)=\left(F,LFR,LF\left(\eta\right)\right).$$ (Note moreover, $\eta$ is an equivalence). Hence there is a comonad $Q_0 \circ i_0^*$ on $\Fun_{\Delta^1}\left(\cM,\cM\right),$ and since $i_0^*$ is monoidal (as a strict map of simplicial monoids), $Q_0 \circ i_0^*$ is oplax monoidal. It follows that the canonical (essentially unique) map of comonoids
$$\left(RJL,J,\varepsilon_{JL}\right) \to \left(id_\sC,\id_\sD,LF\left(\eta\right)\right)=1\!\!1$$ induces a map of comonoids
$$Q_0 \circ i_0^*\left(RJL,J,\varepsilon_{JL}\right) \to Q_0 \circ i_0^*\left(id_\sC,\id_\sD,LF\left(\eta\right)\right).$$ Notice
\begin{eqnarray*}
Q_0 \circ i_0^*\left(RJL,J,\varepsilon_{JL}\right) &\simeq & Q_0\left(RJL\right)\\
&\simeq & \left(RJL,LRJLR,LF\left(\eta\right)\right)\\
&=& \left(RJL,CJC,LF\left(\eta\right)\right)
\end{eqnarray*}
and 
\begin{eqnarray*}
Q_0 \circ i_0^*\left(id_\sC,\id_\sD,L\left(\eta\right)\right) &\simeq & Q_0\left(id_\sC\right)\\
&\simeq & \left(id_\sC,LR,L\left(\eta\right)\right)\\
&=& \left(id_\sC,C,L\left(\eta\right)\right).
\end{eqnarray*}
So there is a canonical map of comonoid objects
$$\left(RJL,CJC,LF\left(\eta\right)\right) \to \left(id_\sC,C,L\left(\eta\right)\right).$$ Since $i_1^*$ is monoidal, we get a canonical map of comonads
$CJC \to C$ as desired.

Since the comonad $Q_0\circ i_0^*$ is the composition of two oplax monoidal functors, it becomes a comonoid object in the monoidal $\i$-category $$\Fun^\mathsf{lax}\left(\Fun_{\Delta^1}\left(\cM,\cM\right)^\otimes,\Fun_{\Delta^1}\left(\cM,\cM\right)^\otimes\right).$$ It follows that the co-unit of $Q_0\circ i_0^*$ is a map of oplax monoidal functors $$Q_0\circ i_0^* \to id_{\Fun_{\Delta^1}\left(\cM,\cM\right)^\otimes}.$$ Therefore, the co-unit evaluated at $\Xi\left(J\right)$ gives a map of comonoids
$$Q_0\circ i_0^*\Xi\left(J\right)=\left(RJL,CJC,LF\left(\eta\right)\right) \to \left(RJL,J,\varepsilon_{JL}\right)=\Xi\left(J\right).$$
Applying the monoidal functor $i_1^*,$ this induces a map of comonads $CJC \to J.$ Finally, the diagram commutes since $id_\sD$ is the terminal comonad on $\sD.$
\end{proof}

\begin{remark}
Let $p:\cM \to \Delta^1$ be a biCartesian fibration exhibiting an adjunction $L \dashv R$ with $$L:\sC \hookrightarrow \sD.$$ 
We can identify $\cM$ with the (contravariant) relative nerve (\cite[Definition 3.2.5.2]{htt}) of the functor
$$R:\left(\Delta^1\right)^{op} \to \Set^{\Delta^{op}}$$ whose image is $$R:\sD \to \sC.$$ The commutative diagram of simplicial sets
$$\xymatrix{\sD \ar[r]^-{R} \ar[d]_-{R} & \sC \ar[d]^-{id_\C}\\
\sC \ar[r]_-{id_\sC} & \sC}$$
is a natural transformation $$\mu:R \Rightarrow \Delta_{\sC}.$$ Since the relative nerve construction if functorial, it induces a morphism
$$N_{\mu}\left(\Delta^1\right):\cM=N_{R}\left(\Delta^1\right) \to N_{\Delta_{\sC}}\left(\Delta^1\right)=\Delta \times \sC$$ in $\Set^{\Delta^{op}}/\Delta^1.$ Composition with the projection to $\sC$ yields a functor $$\tau:\cM \to \sC$$ such that $\tau \circ i_0=id_{\sC}.$ Informally, it is given by $$\tau\left(0,C\right)=C$$ and
$$\tau\left(1,D\right)=R\left(D\right).$$
\end{remark}

It is convenient to recall the framework of homotopy coherent monads of \cite{RiehlVerityMon}. We will adapt the set up for comonads. Let $\underline{\mathsf{Adj}}$ denote the $2$-category with a generic adjunction. It has two objects $+$ and $-.$ Moreover,
$$\xymatrix{\underline{\mathsf{Adj}}\left(+,+\right)=\Delta_+ & \underline{\mathsf{Adj}}\left(-,-\right)=\Delta_+^{op}\\
\underline{\mathsf{Adj}}\left(-,+\right)=\Delta_\i & \underline{\mathsf{Adj}}\left(+,-\right)=\Delta_\i^{op},}$$
where $\Delta_\infty$ is the full wide subcategory of $\Delta$ on those morphisms which preserves top elements. Denote by $\underline{\mathsf{CMnd}}$ the full subcategory on the object $-.$ This is the same as the $2$-category with one object corresponding to the monoidal structure on $\Delta_+^{op}$ given by the join construction. Regard this strict $2$-category as a category enriched in simplicial sets, via the nerve construction. Denote by $\underline{\Set}^{\Delta^{op}}$ the category simplicial sets equipped with its canonical self-enrichment given by the fact that it is Cartesian closed. Furthermore, denote by $\underline{\mathsf{QCat}}$ the full subcategory on the quasicategories. Composition in the $2$-category $\underline{\mathsf{Adj}}$ equips $\Delta_\i^{op}$ with the canonical structure of a left module for the monoidal category $\Delta_+^{op},$ witnessed by a canonical $2$-functor $$\mathbb{W}:\underline{\mathsf{CMnd}} \to \Cat$$ sending $-$ to $\Delta_\i^{op},$ where the induced strict monoidal functor
$$\Delta_+^{op} \to \Fun\left(\Delta_\i^{op},\Delta_\i^{op}\right)$$
encodes the action. Using the nerve construction, this can be regarded as an enriched functor
$$\mathbb{W}:\underline{\mathsf{CMnd}} \to \underline{\Set}^{\Delta^{op}}$$
of simplicial categories. A \emph{homotopy coherent comonad} on a quasicategory $\sD$ is a simplicially enriched functor
$$\overline{J}:\underline{\mathsf{CMnd}} \to \underline{\mathsf{QCat}}$$ which sends $-$ to $\sD.$ This is the same data as a strict map of simplicial monoids
$$N\left(\Delta_+^{op}\right) \to \Fun\left(\sD,\sD\right).$$ The $\i$-category $\CoAlg_J\left(\sD\right)$ can be modeled by the quasicategory obtained as the weighted limit $\left\{\mathbb{W},J\right\}$ in the simplicially enriched category $\underline{\mathsf{QCat}}.$ It has the following universal property:

For any quasicategory $\cK,$ there is a canonical isomorphism of simplicial sets
$$\Fun\left(\cK,\CoAlg_J\left(\sD\right)\right) \cong
\mathsf{Nat}_{\Delta}\left(\mathbb{W}\left(\blank\right),\Fun\left(\cK,\overline{J}\left(\blank\right)\right)\right),$$ where the right hand side is the simplicial set of simplicially enriched natural transformations of simplicial functors
$$\underline{\mathsf{CMnd}} \to \underline{\Set}^{\Delta^{op}}.$$
This universal property allows a direct description of the simplicial set $\CoAlg_J\left(\sD\right),$ by letting $\cK=\Delta^n,$ and taking $0$-simplices of functor quasicategories. Unwinding the definitions, an $n$-simplex of $\CoAlg_J\left(\sD\right)$ can be described as a map
$$\sigma:\Delta_\i^{op} \to \Fun\left(\Delta^n,\sD\right)$$ such that the following diagram commutes
$$\xymatrix@C=2cm{\Delta_+^{op} \ar[r]^-{\Fun\left(\Delta^n,\overline{J}\left(\blank\right)\right)} \ar[d]_-{\mathbb{W}} & \Fun\left( \Fun\left(\Delta^n,\sD\right),\Fun\left(\Delta^n,\sD\right)\right) \ar[d]^-{\left(\blank\right) \circ \sigma}\\
\Fun\left(\Delta_+^{op},\Delta_+^{op}\right) \ar[r]_-{\sigma \circ \left(\blank\right)} & \Fun\left(\Delta_\i^{op},\Fun\left(\Delta^n,\sD\right),\Fun\left(\Delta^n,\sD\right)\right).}$$

The comonad $J$ is encoded as a strict map of simplicial monoids $$\overline{J}:\Delta_+^{op} \to \Fun\left(\sD,\sD\right).$$ For a given object $D$ in $\sD,$ we get by composition a functor
$$\mathsf{ev}_D \circ \overline{J}:\Delta_+^{op} \to \sD.$$ Unwinding the definitions above even further, the structure of a $J$-coalgebra on $D$ is a lift
$$\xymatrix@C=2cm{\Delta_+^{op}  \ar@{^{(}->}[d] \ar[r]^-{\mathsf{ev}_D \circ \overline{J}} & \sD\\
\Delta_\i^{op} \ar@{-->}[ru]_-{A} &}$$
such that if $\beta=\left(D=A_0 \to JD=A_1\right),$ then (the $1$-skeletal image of) $A$ can be described as
$$\xymatrix@C=75pt{ D \ar@<-.75ex>@{-->}[r]|(.4)\beta  & JD \ar@<-.75ex>[l]|(.4){\epsilon_D} \ar@<+.75ex>[r]|(.4){\Delta_D} \ar@<-2.25ex>@{-->}[r]|(.4){J\left(\beta\right)} & J^2D \ar@<+.75ex>[l]|(.4){J\left(\epsilon_D\right)} \ar@<-2.25ex>[l]|(.4){\epsilon_{JD}} \ar@<2.25ex>[r]|(0.4){\Delta_{JD}} \ar@<-.75ex>[r]|(0.4){J\left(\Delta_D\right)} \ar@<-3.75ex>@{-->}[r]|(0.4){J^2\left(\beta\right)}& J^3D \ar@<-3.75ex>[l]|(.4){\epsilon_{J^2D}} \ar@<-.75ex>[l]|(.4){J\left(\epsilon_{JD}\right)} \ar@<2.25ex>[l]|(0.4){J^2\left(\epsilon_D\right)}\cdots}$$
Notice that the diagram without the dotted arrows is precisely $\mathsf{ev}_D \circ \overline{J},$ and the dotted arrows are provided by the lift. We stress that one needs the full diagram, not just $\beta,$ as the lift also includes higher simplices that provide coherence data for the algebraic relations needed for the above to describe a homotopy coherent functor of $\i$-categories $\Delta_\i^{op} \to \sD.$ Nonetheless, we will often denote a $J$-coalgebra informally as $$\beta:D \to JD.$$

\begin{remark}\label{rmk:goon}
Let $\cK$ be an $\i$-category and let $J$ be a comonad on $\sD.$ Then composition with $J$ induces a comonad $J_*$ on $\Fun\left(\cK,\sD\right).$ Unwinding the definitions, one has a canonical isomorphism of simplicial sets
$$\CoAlg_{J_*}\left(\Fun\left(\cK,\sD\right)\right) \cong \Fun\left(\cK,\CoAlg_J\left(\sD\right)\right).$$
Applying this to $\cK=\Delta^1$ implies that for $\left(X,\beta\right)$ and $\left(X',\beta'\right)$ $J$-coalgebras, we have a pullback diagram
$$\xymatrix{\Map_{\CoAlg_J\left(\sD\right)}\left(\left(X,\beta\right),\left(X',\beta'\right)\right) \ar[d] \ar[r] & \CoAlg_{J_*}\left(\sD^{\Delta^1}\right) \ar[d]\\
\ast \ar[r]_-{\left(\left(X,\beta\right),\left(X',\beta'\right)\right)} & \CoAlg_J\left(\sD\right).}$$
This means an $n$-simplex of $\Map_{\CoAlg_J\left(\sD\right)}\left(\left(X,\beta\right),\left(X',\beta'\right)\right)$ can be identified with an $n$-simplex of $\CoAlg_{J_*}\left(\sD^{\Delta^1}\right)$ which upon evaluation and $0$ and $1$ is a degenerate $n$-simple on $\left(X,\beta\right)$ and $\left(X',\beta'\right)$ respectively. But an $n$-simplex of $\CoAlg_{J_*}\left(\sD^{\Delta^1}\right)$ is the same data as a $J_*$-coalgebra in $\Fun\left(\Delta^n \times \Delta^1,\sD\right).$ This is the same as a functor
$$\Delta_\i^{op} \to \Fun\left(\Delta^n \times \Delta^1,\sD\right)$$ which is pointwise a $J$-coalgebra. This is the same as an $n$-simplex in
$\Fun\left(\Delta_\i^{op},\Fun\left(\Delta^1,\sD\right)\right)$ which is pointwise a $J$-coalgebra, and that is the same as an $n$-simplex in $\Fun\left(\Delta^1,\Fun\left(\Delta_\i^{op},\sD\right)\right)$ which takes values in $J$-coalgebras. It follows that $\CoAlg_J\left(\sD\right)$ is a full subcategory of $\Fun\left(\Delta_\i^{op},\sD\right).$
\end{remark}


\begin{theorem}\label{thm:odpsec}
Let $L \dashv R$ exhibit $\sC$ as a coreflective subcategory of $\sD.$ Denote the co-unit of this adjunction by $\varepsilon.$ Let $J$ be a comonad on $\sD$ such that for all $n,$ $RJ\left(\varepsilon_{J^nLX}\right)$ is an equivalence. Then there is a canonical structure on the functor $RJL$ of a comonad on $\sC$ and there is a pullback diagram
$$\xymatrix{\CoAlg_{RJL}\left(\sC\right) \ar[r] \ar[d]_-{U_{RJL}} & \CoAlg_{J}\left(\sD\right) \ar[d]^-{U_J}\\
\sC \ar@{^{(}->}[r]_-{L} & \sD.}$$
\end{theorem}

\begin{proof}
By the proof of Lemma \ref{lem:amthm}, $Q_1\left(J\right)=\left(RJL,J,\varepsilon_{JL}\right)$ is a comonoid object of $\Fun_{\Delta^1}\left(\cM,\cM\right),$ and since $i_0^*$ is monoidal,
$$RJL=i_0^*Q_1\left(J\right)$$ is a comonad on $\sC.$ By the proof of the same lemma, the map
$$CJC \stackrel{C\left(\epsilon_C\right)}{\longlongrightarrow} C$$ is a morphism of comonads, where $\epsilon$ is the co-unit of $J$. Thus there is an induced functor
$$\CoAlg_{CJC}\left(\sD\right) \to \CoAlg_{C}\left(\sD\right).$$
This implies that if $\beta:X \to CJCX$ is a co-algebra for $CJC,$ that the composite
$$\xymatrix{X \ar[r]^-{\beta} \ar@{-->}[rd]_-{\sim} & CJCX \ar[d]^-{C\left(\epsilon_{CX}\right)}\\
& CX}$$
is an a co-algebra for $C.$ But $C$ is an idempotent comonad, so this implies that the dotted arrow is an equivalence. It follows that the underlying object of any $CJC$ co-algebra lies in the essential image of $C,$ or equivalently, the essential image of $L.$ (More precisely, a $CJC$ co-algebra has higher coherency data involving higher powers of $CJC,$ and each of these maps to higher powers of $C$ in a coherent way, exhibiting each $CJC$-coalgebra as a $C$-coalgebra.)

Notice that we have a commutative diagram (up to canonical homotopy)
$$\xymatrix{\sC \ar[r]^-{RJL} \ar@{^{(}->}[d]_-L & \sC \ar@{^{(}->}[d]^-L\\
\sC \ar[r]_-{CJC} & \sD,}$$
since
\begin{eqnarray*}
CJC \circ L &\simeq & CJLRL\\
&\simeq & CJL\\
&\simeq & L \circ RJL
\end{eqnarray*}
as $RL \simeq id_{\sC}.$
Since $L$ is fully faithful, this implies that $CJC$ restricts to $RJL$ on the subcategory $\sC,$ and thus $\CoAlg_{RJL}\left(\sC\right)$ is equivalent to the full subcategory of $\CoAlg_{CJC}\left(\sD\right)$ on those in the essential image of $\sC,$ and this equivalence is realized at the level of objects by sending a $RJL$-coalgebra $X$ in $\sC$ to $LX.$ But we've already seen that all $CJC$-coalgebras have their underling object in the essential image of $L.$ Thus we have commutative diagram
$$\xymatrix{\CoAlg_{RJL}\left(\sC\right) \ar[d]_-{U_{RJL}} \ar[r]^-{\sim} & \CoAlg_{CJC}\left(\sD\right) \ar[d]^-{U_{CJC}}\\
\sC \ar[r]_-{\sim} & \CoAlg_C{\sD}.}$$
Since we also have a map of comonads $CJC \to J,$ there is an induced functor
$$\rho:\CoAlg_{CJC}\left(\sD\right) \to \CoAlg_{J}\left(\sD\right).$$
Informally, given any $CJC$-coalgebra $\beta:X \to CJCX,$ the composite
$$X \stackrel{\beta}{\to} CJCX \stackrel{\varepsilon_{JCX}}{\longlongrightarrow} JCX \stackrel{J\left(\varepsilon_{X}\right)}{\longlongrightarrow} JX$$ exhibits $X$ as a $J$-coalgebra. Combining this observation with the functor $\rho,$ we conclude that if $$X \to RJLX$$ is an $RJL$-coalgebra, then $$LX \to JLX$$ is a $J$-coalgebra where the latter map is induced by the adjunction $L \dashv R.$ In particular, there is a functor
$$\tilde \rho:\CoAlg_{RJL}\left(\sC\right) \to \CoAlg_J\left(\sD\right).$$

We will now show the converse, namely, if
$$\beta:LX \to JLX$$
then the adjoint map $$\tilde \beta:X \to RJLX$$ is an $RJL$-coalgebra. Recall that $Q_1\left(J\right)$ is a comonoid object in $\Fun_{\Delta^1}\left(\cM,\cM\right).$ Encode this by a strict map of simplicial monoids
$$\mathbb{J}:\Delta_+^{op} \to \Fun_\Delta^1\left(\cM,\cM\right).$$ Then $i_0^* \circ \mathbb{J}=:\overline{RJL}$ and $i_1^* \circ \mathbb{J}=:\overline{J}$ encode the comonad structure of $RJL$ and $J$ respectively. Moreover, it follows from Remark \ref{rmk:monom} that for any object $X$ of $\sC,$ the diagram
$$\xymatrix{\Delta_+^{op} \ar[d]_-{\mathbb{J}} \ar[r]^-{\mathbb{J}} & \Fun_{\Delta^1}\left(\cM,\cM\right) \ar[rd]^-{i_1^\ast} & & \\
\Fun_{\Delta^1}\left(\cM,\cM\right) \ar[rd]_-{i_0^\ast} & & \Fun\left(\sD,\sD\right) \ar[rd]^-{\mathsf{ev}_{LX}}  &\\
& \Fun\left(\sC,\sC\right)\ar[rd]_-{\mathsf{ev}_X} & & \sD \ar[ld]^-{R}\\
 & & \sC& .}$$
commutes up to a canonical homotopy
$$\mu_X:\mathsf{ev}_{X} \circ i_0^* \circ \mathbb{J} \stackrel{\sim}{\Rightarrow} R \circ \mathsf{ev}_{LX} \circ i_1^* \circ \mathbb{J}.$$ Suppose that $$B:\Delta^{op}_\i \to \cD$$ encodes the $J$-coalgebra on $LX,$ with $$\beta=\left(B_0=LX \to B_1=JLX\right).$$ 
Then $$B|_{\Delta_+^{op}} = \mathsf{ev}_{LX} \circ i_1^* \circ \mathbb{J}.$$ It follows that 
$$\mu_X:\mathsf{ev}_{X} \circ i_0^* \circ \mathbb{J} \stackrel{\sim}{\Rightarrow} \left(R \circ B\right)|_{\Delta^{op}_+}.$$ Since $\mu_X$ is component-wise an equivalence in $\sC,$ there exists a lift
$$\xymatrix@C=2cm{\Delta_+^{op}  \ar@{^{(}->}[d] \ar[r]^-{\mathsf{ev}_X \circ i_0^\ast \circ \mathbb{J}} & \sC\\
\Delta_\i^{op} \ar@{-->}[ru]_-{A} &}$$
and an equivalence $$\widetilde{\mu}_X:A \to R \circ B$$ whose components are the same as $\mu_X$ and are the vertical maps in the following diagram, which commutes up to canonical coherent homotopies, with the dashed arrows indicating the parts of the diagram $A$ which are not included in $\mathsf{ev}_{X} \circ i_0^* \circ \mathbb{J}$:
\begin{center}
\resizebox{6in}{!}{$\xymatrix@C=150pt@R=5cm{ X  \ar[d]_-{\eta_X} \ar[d]^-{\rotatebox{90}{$\sim$}} \ar@<-.75ex>@{.>}[r]|(.4){\widetilde{\beta}}  & RJLX \ar[d]_-{id}  \ar@<-.75ex>[l]|(.4){\epsilon'_X} \ar@<+.75ex>[r]|(.4){\Delta'_X} \ar@<-2.25ex>@{.>}[r] & RJL\left(RJLX\right) \ar[d]_-{\mu_1=RJ\left(\varepsilon_{JLX}\right)} \ar[d]^-{\rotatebox{90}{$\sim$}}  \ar@<+.75ex>[l]|(.4){RJL\left(\epsilon'_X\right)} \ar@<-2.25ex>[l]|(.4){\epsilon'_{RJLX}} \ar@<2.25ex>[r]|(0.4){\Delta_{RJLX}} \ar@<-.75ex>[r]|(0.4){RJL\left(\Delta'_X\right)} \ar@<-3.75ex>@{.>}[r] & \left(RJL\right)^3 X \ar[d]_-{\mu_2=RJ\left(\varepsilon_{J^2LX}\right) \circ RJL\left(RJ\left(\varepsilon_{JLX}\right)\right)} \ar[d]^-{\rotatebox{90}{$\sim$}} \ \ar@<-3.75ex>[l]|(.4){\epsilon'_{RJLRJLX}} \ar@<-.75ex>[l]|(.4){RJL\left(\epsilon'_{RJLX}\right)} \ar@<2.25ex>[l]|(0.4){RJLRJL\left(\epsilon'_X\right)}\cdots\\
RLX \ar@<-.75ex>[r]|(.4){R\left(\beta\right)}  & RJLX \ar@<-.75ex>[l]|(.4){R\left(\epsilon_{LX}\right)} \ar@<+.75ex>[r]|(.4){\Delta_{LX}} \ar@<-2.25ex>[r]|(.4){RJ\left(\beta\right)} & RJ^2LX \ar@<+.75ex>[l]|(.4){RJ\left(\epsilon_{LX}\right)} \ar@<-2.25ex>[l]|(.4){R\left(\epsilon_{JLX}\right)} \ar@<2.25ex>[r]|(0.4){R\left(\Delta_{JD}\right)} \ar@<-.75ex>[r]|(0.4){RJ\left(\Delta_{LX}\right)} \ar@<-3.75ex>[r]|(0.4){RJ^2\left(\beta\right)}& RJ^3LX \ar@<-3.75ex>[l]|(.4){R\left(\epsilon_{J^2LX}\right)} \ar@<-.75ex>[l]|(.4){RJ\left(\epsilon_{JLX}\right)} \ar@<2.25ex>[l]|(0.4){RJ^2\left(\epsilon_{LX}\right)}\cdots
}$}
\end{center}
where $\epsilon$ and $\Delta$ are the co-unit and co-multiplication of $J$ and $\epsilon'$ and $\Delta'$ are the co-unit and co-multiplication of $RJL.$ The notation $\mu_n$ refers to the proof of Lemma \ref{lem:monsup}. The first dashed arrow is $R\left(\beta\right) \circ \eta_X,$ which is $\tilde \beta$ by definition. We will inductively build the desired lift by defining a lift
 $$\xymatrix@C=2cm{\left(\Delta_+\right)_{\le n}^{op} \ar@{^{(}->}[d] \ar[r]^-{\mathsf{ev}_X \circ i_0^\ast \circ \mathbb{J}} & \sC\\
\left(\Delta_\i\right)_{\le n}^{op} \ar@{-->}[ru]_-{\tilde A_n} &}$$
for each $n$ by induction.
Recall from the proof of Lemma \ref{lem:monsup} that
$$\mu_n=RJ\left(\varepsilon_{J^{n-1}LX} \circ L\left(\mu_{n-1}\right)\right).$$
We will take $\tilde A_1$ to be induced in the obvious way from the first commuting square in the diagram. Suppose the lift $\tilde A_n$ has been defined. By naturality and inductive hypothesis, the following diagram commutes up to canonical homotopy
$$\xymatrix@C=3.5cm@R=2cm{L\left(\left(RJL\right)^{n-1}X\right) \ar@{.>}[r]^-{L\left(\left(RJL\right)^{n-1}\left(\tilde \beta\right)\right)} \ar[d]_-{L\left(\mu_{n-1}\right)} & L\left(\left(RJL\right)^{n}X\right) \ar[d]^-{L\left(\mu_n\right)}\\
L\left(RJ^{n-1}LX\right) \ar[r]^-{L\left(RJ^{n-1}\left(\beta\right)\right)} \ar[d]_-{\varepsilon_{J^{n-1}LX}} & L\left(RJ^nLX\right) \ar[d]^-{\varepsilon_{J^nLX}}\\
J^{n-1}LX \ar[r]_-{J^{n-1}\left(\beta\right)} & J^nLX}$$
The image of this diagram under $RJ$ yields a commuting diagram
$$\xymatrix@C=3cm@R=1.5cm{\left(RJL\right)^{n}X \ar@{.>}[r]^-{\left(\left(RJL\right)^{n}\left(\tilde \beta\right)\right)} \ar[d]_-{RJL\left(\mu_{n-1}\right)} & \left(\left(RJL\right)^{n+1}X\right) \ar[d]^-{RJL\left(\mu_n\right)}\\
\left(RJLJ^{n-1}LX\right) \ar[r]^-{\left(RJ^{n}\left(\beta\right)\right)} \ar[d]_-{RJ\left(\varepsilon_{J^{n-1}LX}\right)} & \left(RJLJ^{n}LX\right) \ar[d]^-{R\left(\varepsilon_{J^nLX}\right)}\\
RJ^{n}LX \ar[r]_-{RJ^{n}\left(\beta\right)} & RJ^{n+1}LX.}$$
Taking the outer square induces a lift $\tilde A_{n+1}.$ Thus, there exists a lift $A$ as claimed. Notice that the $n^{th}$ dotted arrow is $\left(RJL\right)^{n-1}\left(\tilde \beta\right).$ It follows that the lift $A$ exhibits $$\tilde \beta:X \to RJLX$$ as an $RJL$-coalgebra. Notice that for any $\i$-category $\cK,$ the same logic can be applied to construct in the same inductive fashion an $RJL_*$-coalgebra in $\Fun\left(\cK,\sC\right)$ on $R\ \circ G$ for any $J_*$-coalgebra $G$ in $\Fun\left(\cK,\sD\right)$ for which $G=L \circ F,$ for some $$F:\cK \to \sC.$$ Moreover, this can be done in such a way that it agrees with the induced $RJL$-coalgebra structure on $RLF\left(x\right)$ from the $J$-coalgebra structure on $LF\left(x\right),$ for every object $x$ of $\cK.$ Applying this to $\cK=\Delta^n$ implies that the construction induces a functor of $\i$-categories
$$\kappa:\sC \times_{\sD} \CoAlg_J\left(\sD\right) \to \CoAlg_{RJL}\left(\sC\right),$$
where
$$\xymatrix{\sC \times_{\sD} \CoAlg_J\left(\sD\right) \ar[r] \ar[d]_-{U^{\sC}_J} & \CoAlg_J\left(\sD\right) \ar[d]^-{U_J} \\
\sC \ar[r]_-{L} & \sD}$$
pullback diagram of $\i$-categories.
Moreover, the functor $$\tilde \rho:\CoAlg_{RJL}\left(\sC\right) \to \CoAlg_J\left(\sD\right)$$ induces a functor
$$\widehat{\rho}:\CoAlg_{RJL}\left(\sC\right) \to \sC \times_{\sD} \CoAlg_J\left(\sD\right).$$ It is easy to check that these two functors are inverse to one another.
\end{proof}

\section{Linear differential operators and D-modules}\label{sec:DMod}

\begin{definition}
Let $\cX \in \bH$ and consider the morphism $\eta:\cX \to \cX_{dR}$ to its de Rham stack and the induced adjunction between the $\i$-categories of modules: 
$$\xymatrix{\Mod_{\O_\cX} \ar@<-0.5ex>[r]_-{\eta_\star} & \Mod_{\O_{\cX_{dR}}}. \ar@<-0.5ex>[l]_-{\eta^\star}}$$
Denote the comonad
$$\eta^\star \eta_\star:\Mod_{\O_\cX} \to \Mod_{\O_\cX}$$ by $\J^\i.$ For $\cF$ an $\O_\cX$-module, $\J^\i \cF$ is called the \textbf{jet module} of $\cF.$
\end{definition}

\begin{definition}
Let $\cX$ be a derived stack and let $\cF$ and $\cG$ be $\O_\cX$-modules A \textbf{linear differential operator} from $\cF$ to $\cG$ is a morphism $\J^\i\cF \to \cG$. We will call the $\O_{\cX}$-module $$\Diff\left(\cF,\cG\right):=\Mod_{\cO_X}\left(\J^\i\cF,\cG\right)$$ the \textbf{module of linear differential operators} from $\cF$ to $\cG.$ By the \textbf{sheaf of differential operators on $\cX$} we mean the $\O_\cX$-module $$\Diff\left(\cX\right):=\Diff\left(\O_\cX,\O_\cX\right).$$
 \end{definition}
 
 \begin{lemma}
 With $\cF$ and $\cG$ as above, there is a canonical equivalence of $\O_{\cX_{dR}}$-modules
 $$\Mod_{\O_{\cX_{dR}}}\left(\eta_\star \cF,\eta_\star \cG\right) \simeq \eta_\star \Diff\left(\cF,\cG\right).$$
 \end{lemma}
 \begin{proof}
 Let $\cH$ be any $\O_{\cX_{dR}}$-module. Then we have the following string of natural equivalences
 \begin{eqnarray*}
 \Map_{\O_{\cX}}\left(\cH,\Mod_{\O_{\cX_{dR}}}\left(\eta_\star \cF,\eta_\star \cG\right)\right) &\simeq &  \Map_{\O_{\cX_{dR}}}\left(\cH \underset{\O_{\cX_{dR}}} \otimes \eta_\star \cF,\eta_\star \cG\right)\\
 &\simeq& \Map_{\O_{\cX}}\left(\eta^\star\left(\cH \underset{\O_{\cX_{dR}}} \otimes \eta_\star \cF\right),\cG\right)\\
 &\simeq& \Map_{\O_{\cX}}\left(\eta^\star\cH \underset{\O_{\cX}} \otimes \eta^\star \eta_\star \cF,\cG\right)\\
 &\simeq& \Map_{\O_{\cX}}\left(\eta^\star\cH,\Mod_{\O_{\cX}}\left(\J^\i\cF,\cG\right)\right)\\
 &\simeq& \Map_{\O_{\cX_{dR}}}\left(\cH,\eta_\star\Mod_{\O_{\cX}}\left(\J^\i\cF,\cG\right)\right).
 \end{eqnarray*}
 \end{proof}
 
\begin{remark}\label{rmk:assalg}
Notice that $\Mod_{\O_{\cX_{dR}}}\left(\eta_\star \O_{\cX},\eta_\star \cO_{\cX} \right)$ is the internal endomorphism algebra of $\eta_\star \O_{\cX},$ hence has the structure of an associative algebra object. In particular, it has the structure of a sheaf on $\bH/\cX_{dR}$ with values in associative $\mathbb{R}$-algebras.
\end{remark}
 
\begin{proposition}
Let $u:U \to \cX$ be formally \'etale. Then $\Diff\left(\cX\right)\left(U\right)$ has the canonical structure of an associative $\mathbb{R}$-algebra.
\end{proposition}

\begin{proof}
Since $u$ is formally \'etale, we have a canonical identification $U \simeq \eta^*\left(U_{dR}\right).$ Hence
\begin{eqnarray*}
\Diff\left(\cX\right)\left(U\right) &\simeq& \Diff\left(\cX\right)\left(\eta^*\left(U_{dR}\right)\right) \\
&\simeq& \left(\eta_\star \Diff\left(\cX\right)\right)\left(U_{dR}\right)\\
&\simeq& \Mod_{\O_{\cX_{dR}}}\left(\eta_\star \O_{\cX},\eta_\star \cO_{\cX} \right)\left(U_{dR}\right).
\end{eqnarray*}
The result now follows form Remark \ref{rmk:assalg}.
\end{proof}

\begin{corollary}
For $\cX$ any $\SCi$-Deligne Mumford stack, for example a derived orbifold, then $\Diff\left(\cX\right)$ restricts to a sheaf $\mathcal{D}_\cX$ of associative $\mathbb{R}$-algebras on the small site of $\cX.$
\end{corollary}

\begin{remark}\label{lem:dmodqc}
For $\cX$ any $\SCi$-Deligne Mumford stack, since $\Diff\left(\cX\right)$ is a sheaf of $\O^{\bH}_{\cX}$-modules (on the large site), its restriction to the small site is a sheaf of $\O_{\cX}$-modules, hence a quasicoherent sheaf by Lemma \ref{lem:orbqc}.
\end{remark}

\begin{definition}
Let $\cX$ any $\SCi$-Deligne Mumford stack. The $\i$-category $\DMod\left(\cX\right)$ of (left) $\cD$-modules over $\cX,$ is the $\i$-category of left modules of $\cD_{\cX}$ in $\QC\left(\cX\right).$
\end{definition}

Let $\left(\tilde \cX,\O_{\tilde \cX}\right)$ be a $\SCi$-Deligne Mumford stack with functor of points $\cX$. Notice that there is a canonical map of sheaves of $\R$-algebras $$b:\O_{\tilde\cX} \hookrightarrow \cD_{\cX}$$ over $\tilde \cX$ which hence induces an adjunction
$$\xymatrix@C=2cm{\QC\left(\cX\right)=\Mod_{\O_{\tilde \cX}} \ar@<-0.5ex>[r]_-{b_!} & \DMod\left(\cX\right) \ar@<-0.5ex>[l]_-{b^*},}$$
with $b_! \dashv b^*,$ where $b^*$ is the restriction of scalars functor, and where $b_!$ is the induced module functor. Thus, $b^*b_!:=\De_\cX$ is a monad on $\QC\left(\cX\right).$ Explicitly, it is given by
$$\De_{\cX}\left(\cF\right)=\mathscr{D}_{\cX} \underset{\O_{\tilde \cM}} \otimes \cF.$$

\begin{proposition}\label{prop:locfinjet}
Let $M$ be a smooth manifold. Then $\cD_M$ is the sheaf which assigns each open subset $U$ of $M$ the ring of linear differential operators of $\Ci\left(U\right)$ which are \underline{locally} of finite order.
\end{proposition}

\begin{proof}
By Remark \ref{lem:dmodqc}, $\cD_M$ is quasicoherent. Since for each $U \subseteq M$, $U$ is also a manifold and $\cD_U=\cD_M|_{U},$ it suffices to prove the statement at the level of global sections. Global sections of $\cD_M$ is the same as global sections of $\Diff\left(M\right)$ on the large site. Since $\cD_M$ and $\J^\i \O_M$ are concentrated in degree zero, this is the same as the subset of maps
$$\xymatrix{\J^\i_M\left(M \times \R\right) \ar[rr]  \ar[rd] & & M \times \R \ar[ld]\\
& M,}$$
i.e. elements of $\Map_{\bH/M}\left(\J^\i_M\left(M\times \R\right),M \times \R\right),$ which respect the $\O_{M}$-module structure. By Corollaries \ref{cor:frechmfd} and \ref{cor:frechff}, this is the same as morphisms of Frech\'et manifolds over $M,$ which are fiberwise linear, i.e. morphisms of Frech\'
et vector bundles from $\J^\i_M\left(M \times \R\right)$ to $\R \times M.$ These can be identified with the linear differential operators that are locally of finite order, i.e. a locally finite sum
$$\sum\limits_{k=0}^{\i} f_k D_k,$$ with each $D_k$ a linear operator of order $k.$
\end{proof}

\begin{corollary}
Let $M$ be a manifold, then there is an equivalence of quasicoherent sheaves
$$\cD_M \simeq \underset{k} \colim \left(\J^k\O_M\right)^\vee,$$
where $\left(\J^k\O_M\right)^\vee$ is dual to the module of sections of the $k^{th}$ order jet bundle (of the trivial $\R$-bundle).
\end{corollary}

\begin{proof}
The description of $\cD_M$ from Proposition \ref{prop:locfinjet} identifies $\cD_M$ with global sections of the sheafification of the pointwise filtered colimit of the modules of finite order linear differential operators. Since the dual space to $\J^k\O_M$ is the space of linear differential operators of degree $k,$ this completes the proof.
\end{proof}

\begin{proposition}\label{prop:Dsuper}
Suppose that $\cM$ is a supermanifold. Then 
$$\cD_\cM \simeq \underset{k} \colim \left(\J^k\O_\cM\right)^\vee.$$
\end{proposition}

\begin{proof}
As this is a statement about sheaves, it suffices to check it locally, so it suffices to prove it for $\cM=\R^{p|q}.$ By Lemma \ref{lem:wdw}, it follows that there is a canonical equivalence
$$\Gamma\left(\J^\i_{\R^{p|q}}\left(\R \times \R^{p|q}\right)\right) \cong \underset{k} \lim \Ci\left(\R^{p|q} \times \left(\R^{p|q}\right)_{(k)}\right).$$
Since for $k \ge q,$ $\left(\R^{p|q}\right)_{(k)} \simeq \left(\R^p\right)_{\left(k\right)} \times \R^{0|q},$ it follows that
$$\Gamma\left(\J^\i_{\R^{p|q}}\left(\R \times \R^{p|q}\right)\right) \cong \underset{k} \lim \left[\Ci\left(\R^{p|q}\right) \underset{\R} \otimes \Ci\left(\left(\R^{p}\right)_{\left(k\right)}\right) \underset{\R} \otimes \Ci\left(\R^{0|q}\right)\right].$$
Thus
$$\J^\i\O_{\R^{p|q}}  \simeq \underset{k} \lim \left[\O_{\R^{p|q}} \underset{\R} \otimes \left(\Ci\left(\R^p_{\left(k\right)}\right) \underset{\R} \otimes \Ci\left(\R^{0|q}\right)\right)\right].$$
Note that $\Ci\left(\R^p_{\left(k\right)}\right)$ is a finite dimensional vector space and $\Ci\left(\R^{0|q}\right)$ is a finite dimensional super vector space, and we are inducing these $\R$-module to $\Ci\left(\R^{p|q}\right).$ It follows that we can rewrite this as
$$\J^\i\O_{\R^{p|q}}  \simeq \left(\underset{k} \lim \left[\O_{\R^{p|q}} \underset{\R} \otimes \Ci\left(\R^p_{\left(k\right)}\right)\right]\right) \otimes \left[\O_{\R^{p|q}} \underset{\R} \otimes \Ci\left(\R^{0|q}\right)\right].$$
Note that
\begin{eqnarray*}
 \Gamma\left(\underset{k} \lim \left[\O_{\R^{p|q}} \underset{\R} \otimes \Ci\left(\R^p_{\left(k\right)}\right)\right]\right) &\simeq & \underset{k} \lim \left[\Ci\left(\R^{p|q}\right) \underset{\R} \otimes \Ci\left(\R^p_{\left(k\right)}\right)\right]\\
 &\simeq & \underset{k} \lim \left[\Ci\left(\R^{p}\right) \underset{\R} \otimes \Ci\left(\R^{0|q}\right) \underset{\R} \otimes \Ci\left(\R^p_{\left(k\right)}\right)\right]\\
  &\simeq &\Ci\left(\R^{0|q}\right) \underset{\R} \otimes \left(\underset{k} \lim \left[\Ci\left(\R^{p}\right) \underset{\R} \otimes  \Ci\left(\R^p_{\left(k\right)}\right)\right]\right)\\
 & \simeq & \Ci\left(\R^{p|q}\right) \underset{\Ci\left(\R^{p}\right)} \otimes \left(\underset{k} \lim \left[\Ci\left(\R^{p}\right) \underset{\R} \otimes  \Ci\left(\R^p_{\left(k\right)}\right)\right]\right)\\
& \simeq & \Ci\left(\R^{p|q}\right) \underset{\Ci\left(\R^{p}\right)} \otimes \Gamma\left(\J^\i \O_{\R^p}\right).
\end{eqnarray*}
So if $pr:\R^{p|q} \to \R^p$ denotes the canonical projection, we have
$$\underset{k} \lim \left[\O_{\R^{p|q}} \underset{\R} \otimes \Ci\left(\R^p_{\left(k\right)}\right)\right] \simeq pr^\star \J^\i \O_{\R^p},$$ and thus
$$\J^\i\O_{\R^{p|q}} \simeq pr^\star\left(\J^\i \O_{\R^p}\right) \otimes \left(\O_{\R^{p|q}} \underset{\R} \otimes \Ci\left(\R^{0|q}\right)\right).$$
Since $\Ci\left(\R^{0|q}\right)$ is a finite dimensional super vector space, we have
\begin{eqnarray*}
\cD_{\R^{p|q}} &\simeq& \Mod_{\QC\left(\R^{p|q}\right)}\left(\J^\i \O_{\R^{p|q}},\O_{\R^{p|q}}\right)\\
&\simeq & \Mod_{\QC\left(\R^{p|q}\right)}\left(pr^\star\left(\J^\i \O_{\R^p}\right) \otimes \left(\O_{\R^{p|q}} \underset{\R} \otimes \Ci\left(\R^{0|q}\right)\right),\O_{\R^{p|q}}\right)\\
&\simeq & \Mod_{\QC\left(\R^{p|q}\right)}\left(pr^\star\left(\J^\i \O_{\R^p}\right) ,\Mod_{\QC\left(\R^{p|q}\right)}\left(\O_{\R^{p|q}} \underset{\R} \otimes \Ci\left(\R^{0|q}\right),\O_{\R^{p|q}}\right)\right)\\
&\simeq & \Mod_{\QC\left(\R^{p|q}\right)}\left(pr^\star\left(\J^\i \O_{\R^p}\right) ,\left(\O_{\R^{p|q}} \underset{\R} \otimes \Ci\left(\R^{0|q}\right)\right)^\vee\right)\\
&\simeq & \Mod_{\QC\left(\R^{p|q}\right)}\left(pr^\star\left(\J^\i \O_{\R^p}\right),\O_{\R^{p|q}}\right) \otimes \left(\O_{\R^{p|q}} \underset{\R} \otimes \Ci\left(\R^{0|q}\right)\right)^\vee\\
&\simeq & \left(pr^\star\left(\J^\i \O_{\R^p}\right)\right)^\vee \otimes \left(\O_{\R^{p|q}} \underset{\R} \otimes \Ci\left(\R^{0|q}\right)\right)^\vee.
\end{eqnarray*}
Notice that the underlying topological space of $\R^{p|q}$ is $\R^p,$ so $pr_\star$ is simply the restriction of scalars for the morphism of sheaves of $\SCi$-rings
$$pr:\O_{\R^{p|q}} \to \O_{\R^p}$$ on $\R^p.$
It follows that
\begin{eqnarray*}
\Gamma \left[\left(pr^\star\left(\J^\i \O_{\R^p}\right)\right)^\vee\right] &\simeq & \Map_{\QC\left(\R^{p|q}\right)}\left(pr^\star\left(\J^\i \O_{\R^p}\right),\O_{\R^{p|q}}\right)\\
&\simeq & \Map_{\QC\left(\R^{p}\right)}\left(\J^\i \O_{\R^p},pr_\star \O_{\R^{p|q}}\right)\\
&\simeq & \Map_{\QC\left(\R^{p}\right)}\left(\J^\i \O_{\R^p},\O_{\R^{p}}\underset{\R} {\otimes}\Ci\left(\R^{0|q}\right)\right)\\
&\simeq & \Map_{\QC\left(\R^{p}\right)}\left(\J^\i \O_{\R^p},\O_{\R^{p}}\right) \underset{\R} {\otimes}\Ci\left(\R^{0|q}\right)\\
&\simeq & \Map_{\QC\left(\R^{p}\right)}\left(\J^\i \O_{\R^p},\O_{\R^{p}}\right) \underset{\Ci\left(\R^{p}\right)} {\otimes}\Ci\left(\R^{p|q}\right)\\
&\simeq & \Gamma\left(\left(\J^\i \O_{\R^p}\right)^\vee\right) \underset{\Ci\left(\R^{p}\right)} {\otimes}\Ci\left(\R^{p|q}\right)\\
&\simeq & \Gamma\left( \left(\J^\i \O_{\R^p}\right)^\vee \underset{\O_{\R^{p}}} {\otimes} \O_{\R^{p|q}}\right)\\
&\simeq &\Gamma\left( pr^\star\left(\left(\J^\i \O_{\R^p}\right)^\vee\right) \right),
\end{eqnarray*}
and hence 
$$\left(pr^\star\left(\J^\i \O_{\R^p}\right)\right)^\vee \simeq \left(pr^\star\left(\J^\i \O_{\R^p}\right)^\vee\right).$$
Thus we have that 
\begin{eqnarray*}
\cD_{\R^{p|q}} &\simeq & \left(pr^\star\left(\J^\i \O_{\R^p}\right)^\vee\right) \otimes \left(\O_{\R^{p|q}} \underset{\R} \otimes \Ci\left(\R^{0|q}\right)\right)^\vee\\
&\simeq & pr^\star\left(\underset{k} \colim \left(\J^k\O_{\R^p}\right)^\vee\right) \otimes \left(\O_{\R^{p|q}} \underset{\R} \otimes \Ci\left(\R^{0|q}\right)\right)^\vee\\
&\simeq & \underset{k} \colim \left[\left(pr^\star\left(\left(\J^k\O_{\R^p}\right)^\vee\right)\right) \otimes \left(\O_{\R^{p|q}} \underset{\R} \otimes \Ci\left(\R^{0|q}\right)\right)^\vee\right]\\
&\simeq & \underset{k} \colim \left[\left(\O_{\R^{p|q}} \underset{\O_{\R^p}} \otimes \left(\J^k\O_{\R^p}\right)^\vee\right) \otimes \left(\O_{\R^{p|q}} \underset{\R} \otimes \Ci\left(\R^{0|q}\right)\right)^\vee\right]\\
&\simeq & \underset{k} \colim \left[\left(\O_{\R^{p|q}} \underset{\R} \otimes \Ci\left(\R^{p}_{\left(k\right)}\right)^\vee\right) \otimes \left(\O_{\R^{p|q}} \underset{\R} \otimes \Ci\left(\R^{0|q}\right)\right)^\vee\right]\\
&\simeq & \underset{k} \colim \left[\left(\O_{\R^{p|q}} \underset{\R} {\otimes} \left(\Ci\left(\R^{p}_{\left(k\right)}\right) \otimes \O_{\R^{p|q}} \underset{\R} \otimes \Ci\left(\R^{0|q}\right)\right)\right)^\vee\right]\\
&\simeq & \underset{k} \colim \left[\left(\J^k\O_{\R^{p|q}}\right)^\vee\right].
\end{eqnarray*}
\end{proof}

\begin{proposition}\label{prop:lamadj}
Let $\cM$ be a supermanifold. Then the functor $\Lambda_\star\J^\i\Lambda^\star$ is right adjoint to $\De_{\cM}.$
\end{proposition}

\begin{proof}
Let $\cF$ be an arbitrary quasicoherent sheaf on $\cM.$ Since $\Lambda_\star \J^\i \Lambda^\star \cF$ is also quasicoherent, it is determined by its global sections. These are the same as global sections of $\J^\i \Lambda^\star\cF.$ By Proposition \ref{prop:bTloc} and Lemma \ref{lem:suppsh} it follows that
\begin{eqnarray*}
\J^\i\Lambda^\star \cF\left(\cM\right) &=& \left(\eta^\star \eta_\star \Lambda^\star \cF\right)\left(\cM\right)\\
&\simeq & \Lambda^\star \cF\left(\eta^*\eta_!\left(id_M\right)\right)\\
&\simeq & \Lambda^\star \cF\left(T^\i \cM\right)\\
&\simeq & \Lambda^\star \cF\left(\underset{k} \colim \cM_{\left(k\right)}\right)\\
&\simeq & \underset{k} \lim \Lambda^\star \cF\left(\cM_{\left(k\right)}\right)\\
&\simeq & \underset{k} \lim \Ci\left(\cM_{\left(k\right)}\right) \underset{\Ci\left(\cM\right)} {\widehat{\otimes}} \cF\left(\cM\right)\\
&\simeq & \underset{k} \lim \Gamma_M\left(\J^k_M\left(M \times \R\right)\right) \underset{\Ci\left(\cM\right)} {\widehat{\otimes}} \cF\left(\cM\right).
\end{eqnarray*}
Thus, we deduce that there is an equivalence of quasicoherent sheaves:
$$\Lambda_\star\J^\i\Lambda^\star \cF \simeq \underset{k} \lim \left[\J^k\left(\O_{\cM}\right) \otimes \cF\right].$$
Since $\J^k\left(\O_{\cM}\right)$ is locally free of finite rank (it is a finite dimensional vector bundle), we have
$$\Lambda_\star\J^\i\Lambda^\star \cF \simeq \underset{k} \lim \Mod_{\QC\left(\cM\right)}\left(\left(\J^k\left(\O_{\cM}\right)\right)^{\vee},\cF\right).$$
Let $\cG$ be any quasicoherent sheaf. Then

\begin{eqnarray*}
\Map_{\QC\left(\cM\right)}\left(\cG,\Lambda_\star\J^\i\Lambda^\star \cF\right) & \simeq & \Map_{\QC\left(\cM\right)}\left(\cG,\underset{k} \lim \Mod_{\QC\left(\cM\right)}\left(\left(\J^k\left(\O_{\cM}\right)\right)^{\vee},\cF\right)\right)\\
&\simeq & \underset{k} \lim\ \Map_{\QC\left(\cM\right)}\left(\cG,\Mod_{\QC\left(\cM\right)}\left(\left(\J^k\left(\O_{\cM}\right)\right)^{\vee},\cF\right)\right)\\
&\simeq & \underset{k} \lim \Map_{\QC\left(\cM\right)}\left(\left(\J^k\left(\O_{\cM}\right)\right)^{\vee}\otimes \cG,\cF\right)\\
&\simeq &  \Map_{\QC\left(\cM\right)}\left(\underset{k} \colim\left[\left(\J^k\left(\O_{\cM}\right)\right)^{\vee}\otimes \cG\right],\cF\right)\\
&\simeq &  \Map_{\QC\left(\cM\right)}\left(\underset{k} \colim \left[\left(\J^k\left(\O_{\cM}\right)\right)^{\vee}\right] \otimes \cG,\cF\right)\\
&\simeq &  \Map_{\QC\left(\cM\right)}\left( \cD_{\cM} \otimes \cG,\cF\right)\\
&\simeq & \Map_{\QC\left(\cM\right)}\left(\De_{\cM}\left(\cG\right),\cF\right).
\end{eqnarray*}
\end{proof}

\begin{remark}\label{rmk:bees}
Recall that there is a canonical homomorphism of sheaves $$b:\O_{\tilde\cM} \hookrightarrow \cD_{\cM}$$ and we have
$$\xymatrix@C=2cm{\QC\left(\cM\right)=\Mod_{\O_{\tilde \cM}} \ar@<-0.5ex>[r]_-{b_!} & \DMod\left(\cM\right) \ar@<-0.5ex>[l]_-{b^*}.}$$
So $b_! \dashv b^*.$ There is a further adjoint $b_*$ such that
$$b_! \dashv b^* \dashv b_*,$$
where $$b_*:\QC\left(\cM\right) \to \DMod\left(\cM\right)$$ sends a quasicoherent sheaf $\cF$ to its \emph{coinduced $\cD$-module}. The underlying sheaf of vector spaces of $b_*\cF$ is the same as
$\Mod_{\O_\cM}\left(\cD_\cM,\cF\right),$ but the $\cD$-module structure is defined by
$$\left(D \cdot \varphi\right)\left(x\right)=\varphi\left(x \cdot D\right).$$
Notice that $$\De_{\cX}=b^*b_! \dashv b^*b_*,$$ so by Proposition \ref{prop:lamadj} and the uniqueness of adjoints, we must have that
$$\Lambda_\star\J^\i\Lambda^\star \simeq b^*b_*.$$ Notice that $b^*b_*$ has a canonical comonad structure.
\end{remark}

\begin{lemma}\label{lem:finn}
Let $\cM$ be a supermanifold and let $\cF$ be quasi-coherent sheaf over $\cF.$ Denote by $\varepsilon$ the counit of the adjunction
$$\Lambda^\star \dashv \Lambda_\star.$$ Then for all $n \ge 1,$

$$\Lambda_\star\left(\varepsilon_{\left(\J^\i\right)^n\Lambda^\star\cF}\right):\left(\Lambda_\star \J^\i \Lambda^\star\right)^n\left(\cF\right) \to \Lambda_\star \left(\J^\i\right)^n\Lambda^\star \cF$$
is an equivalence.
\end{lemma}

Similarly to the proof of Proposition \ref{prop:lamadj}, if $\cF$ is an arbitrary quasicoherent sheaf on $\cM,$ 
\begin{eqnarray*}
\Gamma \Lambda_\star \left(\J^\i\right)^n\Lambda^\star \cF &\simeq & \left(\Lambda^\star \cF\right)\left(\left(\eta_!\eta^*\right)^n \left(id_{\cM}\right)\right)\\
&\simeq & \left(\Lambda^\star \cF\right)\left(\left(T_{\cM}^\i\right)^n\left(id_{\cM}\right)\right)\\
&\simeq & \left(\Lambda^\star \cF\right)\left(T^\i \cM \times_{\cM} T^\i \cM \times _{\cM} \ldots \times_{\cM} T^\i \cM\right)\\
&\simeq & \left(\Lambda^\star \cF\right)\left(\underset{k_1,k_2,\ldots,k_n} \colim \left[\cM_{\left(k_1\right)} \times_\cM \cM_{\left(k_2\right)} \times _{\cM} \ldots \times_{\cM} \cM_{\left(k_n\right)}\right]\right)\\
&\simeq & \underset{k_1,k_2,\ldots,k_n} \lim\left(\Lambda^\star \cF\right)\left( \left[\cM_{\left(k_1\right)} \times_\cM \cM_{\left(k_2\right)} \times _{\cM} \ldots \times_{\cM} \cM_{\left(k_n\right)}\right]\right)\\
&\simeq & \underset{k_1,k_2,\ldots,k_n} \lim \left[\Ci\left(\cM_{\left(k_1\right)} \times_{\cM} \cM_{\left(k_2\right)} \times_{\cM} \ldots \times_{\cM} \cM_{\left(k_n\right)}\right) \underset{\Ci\left(\cM\right)} {\otimes} \cF\left(\cM\right)\right]
\end{eqnarray*}
Notice that each $\Ci\left(\cM\right)_{\left(k\right)}$ is a split nilpotent extension over $\Ci\left(\cM\right),$ hence the $\Ci$-tensor product agrees with that of commutative $\R$-algebras, and we deduce
\begin{eqnarray*}
\Gamma \Lambda_\star \left(\J^\i\right)^n\Lambda^\star \cF &\simeq & \underset{k_1,k_2,\ldots,k_n} \lim \left[\Ci\left(\cM_{\left(k_1\right)}\right) \underset{\Ci\left(\cM\right)} {\otimes}  \ldots \underset{\Ci\left(\cM\right)} {\otimes} \Ci\left(\cM_{\left(k_n\right)}\right) \underset{\Ci\left(\cM\right)} {\otimes} \cF\left(\cM\right)\right]\\
&\simeq & \underset{k_1,k_2,\ldots,k_n} \lim \left[\Gamma\left(\J^{k_1} \O_{\cM}\right) \underset{\Ci\left(\cM\right)} {\otimes}  \ldots \underset{\Ci\left(\cM\right)} {\otimes} \Gamma\left(\J^{k_n} \O_{\cM}\right) \underset{\Ci\left(\cM\right)} {\otimes} \cF\left(\cM\right)\right]
\end{eqnarray*}
Hence
$$\Lambda_\star \left(\J^\i\right)^n\Lambda^\star \cF \simeq  \underset{k_1,k_2,\ldots,k_n} \lim \left[\J^{k_1} \O_{\cM} \otimes \ldots \J^{k_n} \O_{\cM} \otimes \cF\right].$$
So, for any quasicoherent sheaf $\cF,$ we have
\begin{eqnarray*}
\Map_{\QC\left(\cM\right)}\left(\cG,\Lambda_\star \left(\J^\i\right)^n\Lambda^\star \cF\right) &\simeq & \Map_{\QC\left(\cM\right)}\left(\cG,\underset{k_1,k_2,\ldots,k_n} \lim \left[\J^{k_1} \O_{\cM} \otimes \ldots \otimes\J^{k_n} \O_{\cM} \otimes \cF\right]\right)\\
&\simeq & \underset{k_1,k_2,\ldots,k_n} \lim \Map_{\QC\left(\cM\right)}\left(\cG, \J^{k_1} \O_{\cM} \otimes \ldots \otimes\J^{k_n} \O_{\cM} \otimes \cF\right)\\
&\simeq & \underset{k_1,k_2,\ldots,k_n} \lim \Map_{\QC\left(\cM\right)}\left(\cG, \Mod_{\QC\left(\cM\right)}\left(\left(\J^{k_1} \O_{\cM}\right)^\vee \otimes \ldots \otimes\left(\J^{k_n} \O_{\cM}\right)^\vee,\cF\right)\right)\\
&\simeq & \underset{k_1,k_2,\ldots,k_n} \lim \Map_{\QC\left(\cM\right)}\left(\left(\J^{k_1} \O_{\cM}\right)^\vee \otimes \ldots\otimes \left(\J^{k_n} \O_{\cM}\right)^\vee \otimes \cG,\cF\right)\\
&\simeq & \Map_{\QC\left(\cM\right)}\left(\underset{k_1,k_2,\ldots,k_n} \colim\left[\left(\J^{k_1} \O_{\cM}\right)^\vee \otimes \ldots\otimes \left(\J^{k_n} \O_{\cM}\right)^\vee \otimes \cG\right], \cF\right)\\
&\simeq & \Map_{\QC\left(\cM\right)}\left(\left[\underset{k_1} \colim\left(\J^{k_1} \O_{\cM}\right)^\vee\right] \otimes \ldots\otimes \left[\underset{k_n} \colim \left(\J^{k_n} \O_{\cM}\right)^\vee\right] \otimes \cG, \cF\right)\\
&\simeq & \Map_{\QC\left(\cM\right)}\left(\cD_{\cM} \otimes \ldots\otimes \cD_{\cM}  \otimes \cG, \cF\right)\\
&\simeq & \Map_{\QC\left(\cM\right)}\left(\De^n\left(\cG\right), \cF\right)\\
&\simeq & \Map_{\QC\left(\cM\right)}\left(\cG, \left(\Lambda_\star \J^\i \Lambda^\star\right)^n\left(\cF\right)\right).
\end{eqnarray*}

\begin{theorem}\label{thm:DMod}
Let $\cM$ be a supermanifold. Then there is a canonical equivalence
$$\QC\left(\cM_{dR}\right) \simeq \DMod\left(\cM\right).$$
\end{theorem}

\begin{proof}
Since $\cM$ is formally smooth (Corollary \ref{cor:mfdsmth}), the adjunction $$\eta^\star \dashv \eta_\star$$ is comonadic by Corollary \ref{cor:modcom}. Hence there is a canonical equivalence
$$\Mod_{\O_{\cM_{dR}}} \simeq \CoAlg_{\J^\i}\left(\Mod_{\O_\cM}\right).$$
Consider the adjunction
$$\xymatrix@C=2cm{\QC\left(\cM\right) \ar@<-0.5ex>@{^{(}->}[r]_-{\Lambda^\star} & \Mod_{\O_\cM} \ar@<-0.5ex>[l]_-{\Lambda_\star}}$$
with $\Lambda^\star \dashv \Lambda_\star,$ with $\Lambda^\star$ the fully faithful inclusion. Recall that by Corollary \ref{cor:pbqc} that there is a pullback diagram
$$\xymatrix@C=2cm{\QC\left(\cM_{dR}\right) \ar@{^{(}->}[r]^-{\Lambda_{dR}^\star} \ar[d] & \Mod_{\O_{\cM_{dR}}} \ar[d]^{\eta^\star}\\
\QC\left(\cM\right) \ar@{^{(}->}[r]_-{\Lambda^\star} & \Mod_{\O_\cM}.}$$
Putting these two facts together implies that there is a pullback diagram
$$\xymatrix@C=2cm{\QC\left(\cM_{dR}\right) \ar@{^{(}->}[r]^-{\Lambda_{dR}^\star} \ar[d] &\CoAlg_{\J^\i}\left(\Mod_{\O_\cM}\right) \ar[d]^-{U_{\cJ^\i}}\\
\QC\left(\cM\right) \ar@{^{(}->}[r]_-{\Lambda^\star} & \Mod_{\O_\cM}.}$$
By Lemma \ref{lem:finn}, the adjunction $\Lambda^\star \dashv \Lambda_\star$ and the comonad $\J^\i$ satisfy the hypothesis of Theorem \ref{thm:odpsec}. It follows that $\J^\i$ induces comonad structure on $\Lambda^\star \cJ^\i \Lambda_\star$--- which by direct inspection agrees with the one coming from the identification
$$\Lambda^\star \cJ^\i \Lambda_\star \simeq b^*b_*$$ of Remark \ref{rmk:bees}--- and a pullback diagram
$$\xymatrix@C=2cm{\CoAlg_{\Lambda^\star \cJ^\i \Lambda_\star}\left(\QC\left(\cM\right)\right) \ar@{^{(}->}[r] \ar[d]_-{U_{\Lambda^\star \J^\i \Lambda_\star}} &\CoAlg_{\J^\i}\left(\Mod_{\O_\cM}\right) \ar[d]^-{U_{\cJ^\i}}\\
\QC\left(\cM\right) \ar@{^{(}->}[r]_-{\Lambda^\star} & \Mod_{\O_\cM}.}$$
Hence, we must have $$\QC\left(\cM_{dR}\right) \simeq \CoAlg_{\Lambda^\star \cJ^\i \Lambda_\star}\left(\QC\left(\cM\right)\right) \simeq \CoAlg_{b_*b^*}\left(\QC\left(\cM\right)\right).$$
Notice that there is a triple of adjunctions
$$\xymatrix@C=1.5cm{\DMod\left(\cM\right) \ar[d]|{b^\star}\\
\QC\left(\cM\right). \ar@<+2ex>[u]^-{b_!} \ar@<-2ex>[u]_-{b_\ast}}$$
The functor $b^*$ is both a left and right adjoint and clearly is conservative, hence $$b_! \dashv b^*$$ is monadic and $$b^* \dashv b_*$$ is comonadic. In particular, 
$$\CoAlg_{b_*b^*}\left(\QC\left(\cM\right)\right) \simeq \DMod\left(\cM\right).$$
This completes the proof.
\end{proof}

\begin{remark}
We can draw several observations from the proof of Theorem \ref{thm:DMod}. Firstly, under the equivalence
$$\QC\left(\cM_{dR}\right) \simeq \DMod\left(\cM\right),$$
we have an identification of 
$$b^*:\DMod\left(\cM\right) \to \QC\left(\cM\right)$$
with 
$$\eta_{\QC}^\star:\QC\left(\cM_{dR}\right) \to \QC\left(\cM\right),$$
and by uniqueness of adjoints, an identification of $\eta^{\QC}_\star$ with $b_*,$ where $b_*$ is the functor sending an $\tilde \O_{\cM}$-module to its \emph{co}-induced (left) $\mathcal{D}_{\cM}$-module. This means as a $\mathcal{D}_{\cM}$-module, we have
$$\eta^{\QC}_\star \cF = \Mod_{\O_{\cM}}\left(\mathcal{D}_{\cM},\cF\right),$$ with the obvious $\mathcal{D}_{\cM}$-module structure. This is commonly called the \emph{jet module} of $\cF$. Notice that since $\mathcal{D}_{\cM} \simeq \underset{k} \colim \left(\J^k\O_\cM\right)^\vee,$ we have the equivalent description $\underset{k} \lim \left(\J^k\O_\cM \otimes \cF\right).$

Moreover, since $b^*$ admits a left adjoint $b_!$, the \emph{induced module functor}, $\eta_{\QC}^\star$ also admits a left adjoint $\eta_!^{\QC},$ which can be identified with $b_!$ under the above equivalence. Moreover, the adjunction 
$$\eta_!^{\QC} \dashv \eta^\star_{\QC}$$
is monadic and there is a canonical identification of monads
$$\eta^\star_{\QC}\circ \eta_!^{\QC}   \simeq \De_{\cX} \simeq \mathcal{D}_{\cM} \underset{\tilde \O_{\cM}} \otimes \left(\blank\right)$$
\end{remark}

\bibliography{research}
\bibliographystyle{hplain}

\end{document}